\documentclass[%
	a4paper,
	twoside,%
	11pt,
	toc=listof,				
	toc=bibliography,
]{report}%

\usepackage[latin1]{inputenc}
\usepackage[german,english]{babel}
\usepackage[
autostyle,
german=quotes
]{csquotes}
\usepackage{marvosym}

\usepackage{tipa} 
\usepackage{pifont}
\usepackage{setspace}
\usepackage{scrhack}

\usepackage{xargs}
\usepackage[colorinlistoftodos,prependcaption,textsize=tiny]{todonotes}

\usepackage{array}

\usepackage{float}
\usepackage{graphicx}
\usepackage{caption}	
\usepackage{subfig}
\usepackage{amsmath}
\usepackage{amssymb}
\usepackage{empheq}
\usepackage{dsfont}		
\usepackage{amstext}
\usepackage{amsfonts}
\usepackage{amsthm}
\usepackage{wasysym}
\usepackage{color}
\usepackage{tabularx}
\newcolumntype{x}[1]{!{\centering\arraybackslash\vrule width #1}}
\usepackage{booktabs}
\usepackage{listings}
\usepackage{titlesec}
\titleformat{\chapter}[hang]
{\normalfont\LARGE\bfseries}
{\thechapter}{20pt}{\LARGE}
\titlespacing*{\chapter}{0pt}{0pt}{5pt}
\titleformat*{\section}{\normalsize\bfseries}
\titlespacing*{\section}{0pt}{10pt}{5pt}
\titleformat*{\subsection}{\normalsize\bfseries}
\titleformat*{\subsubsection}{\normalsize\bfseries}
\titleformat*{\paragraph}{\normalsize\bfseries}
\titleformat*{\subparagraph}{\normalsize\bfseries}

\usepackage{algorithm}
\usepackage{algpseudocode}
\usepackage{musicography}

\usepackage{scalerel}[2016-12-29]
\def\stretchint#1{\vcenter{\hbox{\stretchto[440]{\displaystyle\int}{#1}}}}

\usepackage{tikz-cd}
\usepackage{tikz} 
\usetikzlibrary{cd}
\usetikzlibrary{shapes}
\usetikzlibrary{calc}

\usepackage{parskip}		

\usepackage{fancyhdr}

\usepackage{ifthen}

\usepackage{tensor}

\usepackage[
bookmarks,
raiselinks,
pageanchor,
colorlinks,
citecolor=black,
linkcolor=black,
urlcolor=magenta,				
filecolor=cyan,
menucolor=red
]{hyperref}

\usepackage[
left=3cm,right=4cm,top=2.5cm,bottom=2cm,
,headsep=1.2cm,							
footskip=1.5cm,
includeheadfoot									
]{geometry}

\usepackage{cleveref}

\newcommand{\R}{\mathbb{R}}
\DeclareMathOperator{\Rm}{Rm} 
\DeclareMathOperator{\Ric}{Ric} 
\DeclareMathOperator{\Ad}{Ad} 
\DeclareMathOperator{\ad}{ad} 
\DeclareMathOperator{\tr}{tr} 

\newcommand{\son}{\mathfrak{so}(n)}  
\DeclareMathOperator{\Mat}{Mat} 
\DeclareMathOperator{\id}{id} 
\DeclareMathOperator{\Id}{Id} 
\DeclareMathOperator{\scal}{scal} 
\DeclareMathOperator{\SO}{SO}
\newcommand{\so}{\mathfrak{so}}
\newcommand{\su}{\mathfrak{su}}
\newcommand{\syp}{\mathfrak{sp}}
\renewcommand{\leq}{\leqslant}
\DeclareMathOperator{\Spin}{Spin}
\DeclareMathOperator{\Pin}{Pin}
\DeclareMathOperator{\SU}{SU}
\DeclareMathOperator{\Weyl}{Weyl}
\newcommand{\CP}{\mathbb{CP}}
\DeclareMathOperator{\pr}{pr}
\DeclareMathOperator{\Isom}{Isom}
\DeclareMathOperator{\Iso}{Iso}
\DeclareMathOperator{\Par}{Par}
\DeclareMathOperator{\Hol}{Hol}

\DeclareMathOperator{\vol}{vol}
\DeclareMathOperator{\Sn}{Sn}
\DeclareMathOperator{\dvol}{dvol}

\DeclareMathOperator{\adj}{adj}
\DeclareMathOperator{\diag}{diag}
\DeclareMathOperator{\GL}{GL}
\DeclareMathOperator{\tri}{tri}
\DeclareMathOperator{\nor}{nor}
\DeclareMathOperator{\crit}{crit}
\DeclareMathOperator{\sym}{sym}

\renewcommand{\angle}{\sphericalangle}

\newcommandx{\unsure}[2][1=]{\todo[linecolor=red,backgroundcolor=red!25,bordercolor=red,#1]{#2}}
\newcommandx{\change}[2][1=]{\todo[linecolor=blue,backgroundcolor=blue!25,bordercolor=blue,#1]{#2}}
\newcommandx{\question}[2][1=]{\todo[linecolor=green,backgroundcolor=green!25,bordercolor=green,#1]{#2}}
\newcommandx{\ask}[2][1=]{\todo[linecolor=Plum,backgroundcolor=Plum!25,bordercolor=Plum,#1]{#2}}

\newcommandx{\thiswillnotshow}[2][1=]{\todo[disable,#1]{#2}}

\theoremstyle{plain} \newtheorem{lem}{Lemma}[chapter]
\theoremstyle{plain} \newtheorem*{lem*}{Lemma}
\theoremstyle{definition} \newtheorem{rem}[lem]{Remark}	
\theoremstyle{remark} 	
\theoremstyle{definition} \newtheorem*{rem*}{Remark}
\theoremstyle{plain} \newtheorem{thm}[lem]{Theorem}
\theoremstyle{plain} \newtheorem*{thm*}{Theorem}
\theoremstyle{definition} \newtheorem{defi}[lem]{Definition}
\theoremstyle{definition} \newtheorem*{defi*}{Definition}
\theoremstyle{plain} \newtheorem{cor}[lem]{Corollary}
\theoremstyle{plain} \newtheorem*{cor*}{Corollary}
\theoremstyle{plain} \newtheorem{prop}[lem]{Proposition}
\theoremstyle{plain} \newtheorem*{prop*}{Proposition}
\theoremstyle{definition} \newtheorem{exa}[lem]{Example}
\theoremstyle{definition} \newtheorem*{exa*}{Example}
\theoremstyle{plain} \newtheorem*{cla*}{Claim}
\theoremstyle{plain} \newtheorem*{main}{Main Theorem}
\theoremstyle{plain} \newtheorem*{thma}{Theorem A}
\theoremstyle{plain} \newtheorem*{thmb}{Theorem B}
\theoremstyle{plain} \newtheorem*{thmc}{Theorem C}
\theoremstyle{plain} \newtheorem*{conjA}{Conjecture A}
\theoremstyle{plain} \newtheorem*{conjB}{Conjecture B}

\renewcommand{\headrulewidth}{0.4pt}	
\renewcommand{\footrulewidth}{0pt}
\fancypagestyle{plain}{ 
	\fancyfoot[C]{\small\ -- \thepage \hspace{0.25pt} --}
	\fancyhead[R]{}
	\fancyhead[L]{}
	\renewcommand{\headrulewidth}{0.4pt}
}
\def\Xint#1{\mathchoice
	{\XXint\displaystyle\textstyle{#1}}%
	{\XXint\textstyle\scriptstyle{#1}}%
	{\XXint\scriptstyle\scriptscriptstyle{#1}}%
	{\XXint\scriptscriptstyle\scriptscriptstyle{#1}}%
	\!\int}
\def\XXint#1#2#3{{\setbox0=\hbox{$#1{#2#3}{\int}$ }
		\vcenter{\hbox{$#2#3$ }}\kern-.6\wd0}}

\def\dashint{\Xint-}
\def\stretchdashint#1{\vcenter{\hbox{\stretchto[440]{\displaystyle\dashint}{#1}}}}

\def\bs{\mkern-12mu}

\begin{document}

\begin{titlepage}
	\begin{center}
		\vspace*{\fill}
		\begin{spacing}{2.5}
			\textbf{\huge Towards Finding the Second Best Einstein Metric in Low Dimensions}\\
			\vspace{1.5cm}
			Kevin Poljsak \\
			2022
		\end{spacing}
		\vspace{5cm}
		\vspace*{\fill}
	\end{center}
\end{titlepage}
\thispagestyle{empty}
In the following work we investigate the structure of Einstein manifolds with positive scalar curvature whose curvature operator is sufficiently close to the identity operator in dimensions below $12$. \\
An Einstein manifold with positive scalar curvature, that is not locally isometric to the round sphere, is called the second best Einstein manifold if its curvature operator minimizes the angle to the identity operator among all of these manifolds.
In dimensions above 11 the search for the second best Einstein manifold turns out to be of purely algebraic nature and is, by a conjecture of B\"ohm and Wilking, locally isometric to the product $S^{\lfloor n/2 \rfloor} \times S^{\lceil n/2 \rceil}$. \\ In lower dimensions an angle $\alpha_n$, that is induced by the curvature operator of $\mathbb{CP}^2$, is an obstruction for proving the same result with algebraic methods. We will use the divergence term of the parabolic partial differential equation
\[ \tfrac{2\scal}{n} \mathcal{R} = \Delta \mathcal{R} + 2 (\mathcal{R}^2+\mathcal{R}^{\#})
\]
for Einstein manifolds in order to get better quantitative estimates on the second best Einstein manifold in dimensions $10$ and $11$. We show that, assuming a conjecture of B\"ohm and Wilking, there exists an angle $\alpha_n < \alpha_0 < \angle(\mathcal{R}_{S^{\lfloor n/2 \rfloor} \times S^{\lceil n/2 \rceil}}, \Id_n)$ such that any simply connected Einstein manifold with positive scalar curvature whose angle of the curvature operator to the identity is smaller than $\alpha_0$ is isometric to the round sphere. We will also be able to compute this angle explicitely.

\textbf{Acknowledments:} This paper is mainly the authors PhD thesis. I would like to thank my advisor Burkhard Wilking for suggesting me this beautiful topic and for many helpful discussions during the last years. I am also grateful to Christoph B\"ohm for sharing his ideas on this topic with me.

The project was funded by the Deutsche Forschungsgemeinschaft (DFG, German Research Foundation) - Project-ID 427320536 - SFB 1442, as well as under Germany's Excellence Strategy EXC 2044 390685587, Mathematics Muenster: Dynamics-Geometry-Structure.
\newpage

\begingroup
\pagestyle{empty}
\tableofcontents
\clearpage
\endgroup

\pagenumbering{arabic}

\addtocontents{toc}{\protect\thispagestyle{empty}}
\chapter{Introduction}
Investigating the topological and geometrical structure of Riemannian manifolds with constant curvature is a fundamental question in differential geometry. While this is completely answered for sectional curvature by the model spaces $\R^n$, $S^n$ and $\mathbb{H}^n$ and for scalar curvature by the solution to the Yamabe problem (cf. \cite{Yamabe:1960tv}, \cite{ASNSP_1968_3_22_2_265_0}, \cite{Aubin1976TheSC}, \cite{Schoen:1984wk}), constant Ricci curvature manifolds are largely not understood. Their representative geometry is given by the notion of Einstein manifolds. A Riemannian manifold $(M,g)$ is called an Einstein manifold if its Ricci curvature is proportional to the metric tensor, i.e.
\[ \Ric_g = \lambda g
\]
for some $\lambda \in \R$. Analytically, this equation is a system of non-linear partial differential equations, which is, in general, hard to solve on arbitrary manifolds. In three and four dimensions there exist general topological obstructions to the existence of Einstein metrics. For instance, three dimensional Einstein manifolds are known to have constant sectional curvature and four dimensional Einstein manifolds have nonnegative Euler characteristic (cf. \cite{Berger:1961un}). \\ Although there exist these strict topological obstructions in low dimensions, in dimension greater than four it is still unknown whether there exist closed manifolds that do not admit an Einstein metric. Even for the sphere the variety of Einstein metrics is not fully understood. Besides the round metric, the spheres $S^{4m+3}$, $m >1$ admit a $\text{Sp}(m+1)$-homogeneous Einstein metric \cite{Jensen:1973aa} and $S^{15}$ admits another $\text{Spin}(9)$-homogeneous Einstein metric \cite{ASENS_1978_4_11_1_71_0}. In 1982, W. Ziller \cite{Ziller:1982wm} proved that these are in fact the only homogeneous Einstein metrics on spheres. Until this point those were the only known Einstein metrics on spheres. In 1998, C. B\"ohm \cite{Bohm:1998vc} proved the existence of an infinite sequence of non-isometric Einstein metrics of positive scalar curvature on $S^5, S^6, S^7, S^8$ and $S^9$. In 2005, Boyer, Galicki and Koll\'{a}r \cite{Boyer:2005tw} were able to construct many more Einstein metrics on the sphere. For example, they showed that for $n \geq 2$ the $(4n+1)$-dimensional spheres admit many families of inequivalent Sasakian-Einstein metrics. They were also the first to construct Einstein metrics on exotic spheres. \\
Besides the question of existence, it is natural to ask if existing Einstein metrics can be deformed. An Einstein metric is called rigid if it cannot be deformed in the space of Einstein metrics of a fixed volume. Bourguignon \cite[Cor 12.72]{besse} showed that Einstein structures with $\delta$-pinched sectional curvatures, $\tfrac{n-2}{3n} < \delta$, are rigid. In four dimensions, D. Yang \cite{Yang:2000vb} was able to give explicit curvature bounds for simply connected Einstein manifolds with nonnegative sectional curvature in order to show that it has to be isometric to $S^4$, $\mathbb{CP}^2$ or $S^2 \times S^2$. For instance, he showed that it is sufficient for $(M,g)$ with $\Ric_g =g$ to suppose that $\sec(M,g) \geq (\sqrt{1249}-23)/120 \approx 0.102843$ to be isometric to one of the previous spaces. \\
From the algebraic viewpoint, the Einstein equation translates into the study of curvature operators of the form $\mathcal{R} = \frac{\lambda}{n-1} \Id_{\son} + W$. Here the curvature operator is represented with respect to the $\text{O}(n)$-irreducible decomposition
\[ S_B^2(\son) = \langle \Id_{\son} \rangle \oplus \Ric_0 \oplus \Weyl_n
\]
of the space of algebraic curvature operators $S_B^2(\son)$, where $W \in \Weyl_n$ is some Weyl curvature operator. Hence the existence of Einstein metrics turns into the study of Weyl curvature operators. Note that among simply connected manifolds the round sphere is characterized by the fact that its curvature operator is given by the identity. A very coarse way of describing the distance between the round metric on the sphere and a certain curvature operator $\mathcal{R}$ is given by measuring the angle \[\angle(\mathcal{R}, \Id_n) = \arccos\left( \frac{\langle \mathcal{R}, \Id_n \rangle}{||\mathcal{R}|| \cdot || \Id_n ||} \right) \] between $\mathcal{R}$ and the identity operator. \\
An Einstein metric with positive scalar curvature on a manifold with non-constant sectional curvatures is said to be the \textit{second best Einstein metric}, if its curvature operator minimizes the angle above. \\
The thesis deals with finding the second best Einstein metric. In order to do so, we consider the elliptic partial differential equation
\[ \tfrac{2 \scal}{n} \mathcal{R} = \Delta \mathcal{R} + 2(\mathcal{R}^2+\mathcal{R}^{\#})
\]
arising from the evolution equation of the curvature tensor under the Ricci flow. After taking the scalar product with $\mathcal{R}$ and integrating this on the manifold we obtain the following identiy that turns out to be of major importance.
\begin{thma}
	Let $(M^n,g)$ be an Einstein manifold with $\Ric = \lambda g$ such that $\lambda > 0$. Then we have that
	\begin{align} \label{eq:key1}
		\resizebox{0.9\hsize}{!} {$\stretchint{6ex}_{\bs M} \stretchdashint{6ex}_{\bs B_r(0_p)} |\nabla \Rm|^2_{\exp_p(v)} d\lambda_n(v) \dvol_g(p) = 8 \stretchint{6ex}_{\bs M} ||\mathcal{R}_W||^3 \left( P_{\nor}(\mathcal{R}_W) - \sqrt{\tfrac{2(n-1)}{n}} \tfrac{1}{\tan(\alpha)} \right) \dvol_g,$}
	\end{align}
	where $\alpha(p)$ denotes the angle of the curvature operator to the identity at $p \in M$ and $P_{\nor}(\mathcal{R}) = \tfrac{1}{||R||^3} \langle \mathcal{R}^2+\mathcal{R}^{\#}, \mathcal{R} \rangle$ denotes the normalized potential. 
\end{thma}
Note that the double integral of the left hand side in (\ref{eq:key1}) is taken with respect to the usual Lebesgue measure $\lambda_n$ on $T_pM$ for each $p \in M$. Now the strategy for finding section best Einstein metrics lies in comparing the integrands of the left and the right hand side above. \\
Moreover, assuming that we know the maximum of $P_{\nor}(W)$ among all Weyl curvatures $W \in \Weyl_n$, we are able to obtain an lower bound for the maximum of $\angle(\mathcal{R}_p, \Id_n)$ among all $p \in M$. Indeed, the following conjecture of B\"ohm and Wilking asserts that \newpage
\begin{conjA}
	Let $W \in \Weyl_n$ be a unit Weyl curvature operator that is a critical point of 
	\[ P(\mathcal{R}) = \langle \mathcal{R}^2 + \mathcal{R}^{\#}, \mathcal{R} \rangle
	\] among $\Weyl_n^1= \{ W \in \Weyl_n \mid ||W|| = 1 \}$. If either $n \geq 12$ or $8 \leq n \leq 11$ and $W \notin \SO(n).W_{\mathbb{CP}^2}$, we have
	\[ P(W) \leq P(\mathcal{W}_{\sym}^{\crit}).
	\]
\end{conjA}
Here, we denote by $\mathcal{R}_{\sym,n}^{\crit}= \mathcal{R}_{\sym}^{\crit}$ the curvature operator of $S^{\lceil \tfrac{n}{2} \rceil} \times S^{\lfloor \tfrac{n}{2} \rfloor}$ endowed with its symmetric Einstein metric and by $\mathcal{W}_{\sym}^{\crit} = \left( \mathcal{R}_{\sym}^{\crit} \right)_{W}$ its unit Weyl curvature operator. \\
The previous conjecture directly implies in dimensions $12$ and above that the second best Einstein manifold is given by 
$S^{\lceil \tfrac{n}{2} \rceil} \times S^{\lfloor \tfrac{n}{2} \rfloor}$, since the right hand side of (\ref{eq:key1}) is nonpositive for $\alpha(p) \leq \angle (\mathcal{R}_{\sym}^{\crit}, \Id_n)$, with equality if and only if $\mathcal{R}_p = \mathcal{R}_{\sym}^{\crit}$ for all $p \in M$ and the left hand side is nonnegative. \\ 
However, in dimensions below $12$ the situation is different.
Let $W_{\mathbb{CP}^2}$ be the Weyl curvature operator of $\mathbb{CP}^2$ with unit length. Then we put \[\mathcal{R}^{\text{crit}}_{\mathbb{CP}^2} = \sqrt{\tfrac{3}{2}} \cdot \tfrac{1}{n-1} \Id_{\son} + W_{\mathbb{CP}^2}. \] It is straightforward to check that $\mathcal{R}_{\mathbb{CP}^2}^{\crit}$ is a critical point of $P_{\nor}(\mathcal{R})$. Moreover, we denote by
\begin{align*} & \alpha_n = \angle(\mathcal{R}_{\mathbb{CP}^2}^{\text{crit}}, \Id_{\son}),\\ &\beta_n = \angle(\mathcal{R}_{\sym}^{\crit}, \Id_{\son}).
\end{align*}
A simple computation shows that $\alpha_n < \beta_n$ for $n \leq 11$, while for $n \geq 12$ the inequality turns around. Moreover, $\beta_n \to \tfrac{\pi}{4}$ as $n \to \infty$. \\ Now, again using Conjecture A and (\ref{eq:key1}), we obtain in dimensions $8 \leq n \leq 11$ that any Einstein manifold $(M,g)$ with positive scalar curvature and $\angle(\mathcal{R}_p, \Id_n) \leq \alpha_n$ for all $p \in M$ has to be locally isometric to the sphere, but it is certainly not possible to go further, if one just uses the algebraic behaviour of (\ref{eq:key1}). This thesis also takes the study of the analytical behaviour of (\ref{eq:key1}) into account. We are able to find a quantitative estimate for the left hand side in order to extend the angle beyond $\alpha_n$ in dimensions below $12$. More precisely, assuming Conjecture A, we are able to prove the following
\begin{main}
	Let $n=10, 11$. Then there exists an angle $\alpha_0 > 0$ with \[\angle(\mathcal{R}^{\crit}_{\mathbb{CP}^2}, \Id_{\son}) < \alpha_0 < \angle(\mathcal{R}_{\sym}^{\crit}, \Id_{\son}) \] such that the following holds: Any simply connected Einstein manifold $(M,g)$ with positive scalar curvature that satisfies
	$\angle(\mathcal{R}_p, \Id_n) < \alpha_0$
	for all $p \in M$ is isometric to the round sphere up to scaling. Furthermore, $\alpha_0$ can be explicitly expressed by
	\[
	\alpha_0 = \angle(\mathcal{R}_{\mathbb{CP}^2}^{\crit}, \Id_n) + 1.015 \cdot 10^{-15}.
	\]
\end{main}
It seems like the choice of $\alpha_0$ is just curiously small, since we are just able to earn an amount of $1.015 \cdot 10^{-15}$ onto the angle of $\mathcal{R}_{\mathbb{CP}^2}$ to the identity operator. But we would like to comment that
\[ \angle(\mathcal{R}_{\sym,11}^{\crit}, \Id_{\so(11)}) - \angle(\mathcal{R}_{\mathbb{CP}^2}^{\crit}, \Id_{\so(11)}) \approx 3 \cdot 10^{-3}
\]
and it is certainly not possible to go beyond the angle of $\mathcal{R}_{\sym}^{\crit}$ to the identity operator. Furthermore, there is no principal reason that our proof is not able to push the angle all the way to $\beta_n$, except that our estimates are not yet good enough. \\
The main theorem directly causes that in dimensions $10$ and $11$ there exist no Einstein manifolds with curvature operator of the form $\tfrac{\lambda}{n-1} \Id_{\son} + W_{\mathbb{CP}^2}$ for $\lambda \geq \lambda_0$ for some constant $\lambda_0 < \sqrt{3/2}$ that can be explicitly computed.
\\ Results of this kind are not new. In 1985, Huisken \cite{Huisken:1985wh}, Margerin \cite{margerin} and Nishikawa \cite{nishikawa} have proven independently that for manifolds with positive scalar curvature, whose curvature operator $\mathcal{R}$ satisfy the condition $||\mathcal{R}_W||^2 + ||\mathcal{R}_0||^2 < \delta_n ||\mathcal{R}_I||^2$ for $\delta_n = \tfrac{2}{(n-2)(n+1)}$, the normalized Ricci flow evolves the metric into a metric of constant sectional curvature. This shows directly that our result holds if we choose $\alpha_0 = \arctan(\sqrt{\delta_n})$. Nevertheless, this angle is clearly much smaller.

We now start explaining the proof. After we might assume that there exists points $p \in M$ such that the right hand side of (\ref{eq:key1}) nonnegative, we analyze the behaviour of the potential close to the orbit of $W_{\mathbb{CP}^2}$ in order to show that, assuming Conjecture A, the curvature operator has to be close to $\mathcal{R}_{\mathbb{CP}^2}^{\crit}$. More explicitly, we prove the following explicit bound on the distance.
\begin{thmb}
	Let $10 \leq n \leq 11$ and let $\mathcal{R} \in S^2_B(\son)$ be an Einstein curvature operator with positive scalar curvature and $\alpha_n < \angle(\mathcal{R}, \Id_n) < \beta_n$, such that $\mathcal{R}_W$ satisfies $P_{\nor}(\mathcal{R}_W) > P_{\nor}(\mathcal{W}_{\sym}^{\crit})$. Then, up to the action of $\SO(n)$, we obtain that
	\begin{align*}
		\angle(\mathcal{R}, \mathcal{R}_{\mathbb{CP}^2}^{\crit}) \leq \begin{cases} \tfrac{\pi}{16}, & \text{ if } n = 10, \\
			\tfrac{\pi}{32}, & \text{ if } n = 11.
		\end{cases}
	\end{align*}
\end{thmb}
In fact, the proof of Theorem B implies that we can assume that any Einstein curvature operator $\mathcal{R}$ as in the previous result is of the form 
\begin{align}  \label{eq:key2} \mathcal{R} = \overline{\lambda} \Id_n + \cos(\varphi) W_{\mathbb{CP}^2} + \sin(\varphi) W_1
\end{align}
for suitable chosen $\overline{\lambda}$ and some small $\varphi \geq 0$, that we can be computed explicitely. \\
The main observation is now that on the one hand, the right hand side of (\ref{eq:key1}) shows that around those points the manifold has to be almost symmetric. On the other hand, for curvature operators of the form (\ref{eq:key2}) the antisymmetrical part of the second covariant derivatives of the curvature tensor, given by \begin{align} \label{eq:key3}
	(\nabla^2_{X,Y} - \nabla^2_{Y,X})\Rm = [\mathcal{R}, \ad_{\mathcal{R}(X \wedge Y)}], \end{align}
is quantitatively non vanishing. The fact that this expression turns out to indicate the rate of symmetry of a manifold, will lead us to a contradiction in the end. More precisely, we are able to find points $x,y \in S(T_pM)  = \{ v \in T_pM \mid ||v|| = 1\}$ such that
\[ \left[\mathcal{R}_{\mathbb{CP}^2}^{\crit}, \ad_{\mathcal{R}_{\mathbb{CP}^2}^{\crit}(x \wedge y)}\right] = \sqrt{2} \left( \tfrac{1}{2} - \tfrac{3}{2(n-1)}\right), 
\]
which is approximatly $0.49$ for $n = 11$, see Lemma \ref{algsym1}.
Using that, we are now able to obtain a lower bound on
\begin{align} \label{eq:key4} \dashint_{S(T_pM)} |\nabla^2_{v, \cdot} \Rm_p| d\lambda_{S(T_pM)}(v) \geq C(n)
\end{align}
for curvature operators that are of the form as in Theorem B, see Theorem \ref{thmd}. For instance, we can choose $C(11) \approx 0.0123$.
\\ In order to find the desired contradiction to conclude the main theorem we do a Taylor expansion along radial geodesics of the map $v \mapsto \nabla \Rm_{\exp(v)}$. More explicitely we show, assuming an upper bound on $r > 0$, that
\begin{align}\label{eq:key5}& \; \; \; \; \; \dashint_{B_r(0_p)} |\nabla \Rm|_{\exp_p(v)}^2 d\lambda^n(v)  \\ & \geq \tfrac{1}{\vol(B_r(0_p))} \int_0^r \int_{S(T_pM)} t^2 \left( |\nabla_v \nabla \Rm_p| - \tfrac{t}{2} |\nabla^2_{v,v} \nabla \Rm_{\exp_p(\xi \cdot v)}| \right)^2 t^{n-1} d\lambda_{S(T_pM)}(v) dt. \nonumber
\end{align}
Suprisingly, the linear term in the Taylor expansion vanishes because of symmetry reasons. \\
The error term arising here is the third covariant derivative of the curvature tensor. We will use the Shi estimates, a classical tool in Ricci flow theory, in order to prove quantitative estimates for that. More explicitly, we show
\begin{thmc}[Quantitative Shi estimates]
	Let $K, \lambda > 0$. Then there exist constants $\textbf{C}_i(n)$ for $i = 1,2,3$ such that any $n$-dimensional Einstein manifold $(M,g)$ with $\Ric = \lambda g$ and $|\Rm|_g \leq K$ satisfies
	\begin{itemize}
		\item[(i)] $|\nabla \Rm|_g \leq (2K- \lambda)^{3/2} \textbf{C}_1(n)$,
		\item[(ii)] $|\nabla^2\Rm|_g \leq (2K- \lambda)^2 \textbf{C}_2(n)$,
		\item[(iii)] $|\nabla^3\Rm|_g \leq (2K- \lambda)^{5/2} \textbf{C}_3(n)$.
	\end{itemize}
	Furthermore, we can choose the constants in dimensions $8 \leq n \leq 11$ as follows:
	\footnotesize
	\begin{center}
		\begin{tabular}[h]{l|lll}
			$n$ & $\textbf{C}_1(n)$ & $\textbf{C}_2(n)$ & $\textbf{C}_3(n)$  \\
			\hline
			$11$ & $18$ & $2050$ & $385661$  \\
			\hline
			$10$ & $18$  & $1990$ & $367142$  \\
			\hline
			$9$ & $18$  & $1920$ & $348265$  \\
			\hline
			$8$ & $18$ &  $1850$ & $328939$ \\
		\end{tabular}
	\end{center}
\end{thmc}
Using (\ref{eq:key4}) and Theorem B, the expression (\ref{eq:key5}) above will clearly be strictly positive if we choose $r>0$ quantitativly small enough. Thus we obtain an estimate
\[ \dashint_{B_r(0_{p})} |\nabla \Rm|^2_{\exp_p(v)} d\lambda^n(v) \geq C(r_0)
\] for any $r > r_0$ and any point $p \in M$ (compare Theorem \ref{last}). \\
This is a contradiction to the fact that the right hand side
\[ \int_M 8||\mathcal{R}_W||^3 \left( P_{\text{nor}}(\mathcal{R}_W) - \sqrt{\tfrac{2(n-1)}{n}} \tfrac{1}{\tan(\alpha)} \right) \dvol_g
\]  of $(\ref{eq:key1})$ is smaller if $\alpha < \alpha_0 $.\\ We now explain the structure of this thesis. In Chapter 2 we give a short introduction to upcoming notions. After doing so, we will explain the Main Theorem in full detail and also give a detailed strategy on the proof in Chapter 2.6, including the start of the proof of Theorem A. In Chapter 3 we show Theorem C. At first, we explain how the classical Shi estimates provide a-priori estimates for Einstein manifolds with bounded curvature. Then we analyze the classical proof in order to find explicit bounds. In Chapter 4 we deal with the proof of Theorem B. It turns out that this analytic result relies on a finer decomposition of the space of algebraic curvature operators. We use it to explicitly compute the Hessian of the potential at $W_{\mathbb{CP}^2}$. In Chapter 5 we assign to any algebraic curvature operator an \textit{algebraic symmetry operator} that corresponds to the antisymmetrical part of the second covariant derivative of the curvature tensor. Lastly, we will analyze the algebraic symmetry operator for curvature operators that are close to $\mathcal{R}_{\mathbb{CP}^2}^{\crit}$. In Chapter 6 we put everything together in order to prove the Main Theorem. In the last section we briefly comment on how better estimates in Theorem C could change the outcome. We would also like to point out that the results of Chapter 4 lead to an elementary proof of the $\SO(n)$-irreducibility of the space of Weyl curvature operators attached in the appendix. \\

The obvious upcoming conjecture is the following, that we were introduced to by Burkhard Wilking.
\begin{conjB}
	Let $n \geq 7$. Then the universal cover of the $n$-dimensional second best Einstein metric is isometric to $S^{\lceil \frac{n}{2} \rceil} \times S^{\lfloor \frac{n}{2} \rfloor}$ with the standard Einstein metric.
\end{conjB}
There is no principal reason that it is not possible prove this conjecture with the provided techniques, at least in dimensions $8 \leq n \leq 11$. The main problem is that the estimates we find in Theorem C are likely to be far away from being optimal. Using the techniques of this thesis it suffices to prove that
\[ \dashint_{B_r(0_{\vert p})} |\nabla \Rm|^2_{\exp_p(v)} d\lambda^n(v) \geq ||\mathcal{R}_W||^3 C(n)
\]
at any point $p \in M$ such that the curvature operator is of the form as in Theorem A. For instance, it suffices to take $C(11) \approx 7.58 \cdot 10^{-3}$. Here $r > 0$ might be chosen arbitrary large.

\chapter{The problem and the main players}
In this chapter we will describe the main problem, define the main objects within this thesis and present something that is currently known in this area. \\
On the one hand, Einstein manifolds are defined in a pretty geometric way, so that it seems quite unlikely that there is a way understanding them with non-geometric tools. Nevertheless, it is possible to assign an algebraic object to each Riemannian manifold, namely its algebraic curvature operator. In this setting, the property of being an Einstein manifold turns out to be of a very algebraic nature. We will combine both, in order to get an identity for Einstein manifolds, that will be presented in the end of this chapter.\\
Before starting we fix some notation. Let $M$ be a differentiable manifold, $\pi: E \to M$ be a vector bundle. We denote the vector bundle of $(k,l)$-tensors on $E$ by $T^l_k(E)$ and the $(k,l)$-tensor fields on $E$ by $\mathcal{T}^l_k(E) = \Gamma T^l_k(E)$. If $E = TM$ we write $\mathcal{T}^l_k(TM) = \mathcal{T}^l_k(M)$ for abbreviation. For instance, a Riemannian metric $g$ is the choice of a positive definite, symmetric $(2,0)$-tensor field.
\begin{section}{Einstein manifolds}
	In this section we give a short introduction to Einstein manifolds, since they are the objects that we will take into account for the whole thesis. Einstein manifolds are Riemannian manifolds with very special geometry. Although the definition is very easy, we are far away from a classification.
	\begin{defi}
		A Riemannian manifold $(M,g)$ is called Einstein, if there exists a constant $\lambda \in \mathbb{R}$, such that 
		\[ \Ric = \lambda g,
		\]
		where $\Ric \in \mathcal{T}^0_2(M)$ denotes its Ricci tensor. We call $\lambda$ the Einstein constant.
	\end{defi}
	Einstein manifolds can be seen as the manifolds with "constant" Ricci curvature. Even though their relatives with constant sectional curvatures and constant scalar curvature are well understood, we are, at least in general dimensions, not able to find a general class of differentiable manifolds that forms the set of manifolds, that admit an Einstein metric. By the Bonnet-Myers Theorem, any Einstein manifold with positive Einstein constant has to be compact with finite fundamental group, see \cite[p. 200]{do1992riemannian}.\\
	We start with the discussion of Einstein metrics in low dimensions. By the decomposition of the curvature tensor and Schurs Lemma, any three dimensional Einstein manifold has to have constant sectional curvatures, see \cite[Lem 3.1.4]{petersen20062nd}. \\
	For four dimensional Einstein manifolds there are plenty of topological obstructions that are mostly due to the existence of generalized Gauss-Bonnet formulas. A simple obstruction by Berger \cite{Berger:1961un} asserts that any compact differentiable four dimensional manifold, that admits an Einstein metric, has nonnegative Euler characteristic. If the Euler characteristic vanishes, the manifold has to be flat. The proof is a simple application of the Chern-Gauss-Bonnet formula
	\[ \mathcal{X}(M) = \frac{1}{8\pi^2} \int_M |\Rm|^2 - |\Ric_0|^2 \text{dvol}_g
	\]
	that holds for any compact $4$-dimensional Riemanian manifold $(M,g)$. Here, $\mathcal{X}(M)$ denotes the Euler characteristic of $M$, $\Rm$ denotes its Riemannian curvature tensor and $\Ric_0 = \Ric - \frac{\scal}{4}g$ denotes the traceless Ricci part. Since $\Ric_0$ vanishes exactly for Einstein manifold, as we will see in the next section, the claim follows. A refinement of this result was found independently by Hitchin and Thorpe, see \cite{hitchin} and \cite{thorpe}. They proved that, if $(M,g)$ is a four dimensional Einstein manifold, then
	\[ \mathcal{X}(M) \geq \frac{3}{2} |\tau(M)|,
	\]
	where $\tau(M)$ denotes the signature of $M$, which is also a topological quantity, that we will not explain in detail here. The significant input here is the Hirzebruch signature formula that asserts that
	\[ \tau(M) = \frac{1}{12\pi^2}\int_M \left( |W_+|^2 - |W_-|^2 \right) \text{dvol}_g,
	\]
	where $W_{\pm}$ denotes the self-dual and the anti self-dual parts of the Weyl curvature (that only exist in four dimensions). \\
	In higher dimensions there are currently no topological obstructions known.
	Note that, in both of these results, the idea is to connect two different approaches in understanding a manifold. Both, Berger and Hitchin-Thorpe, used the topological way on the one side and the geometric way on the other side. Our Main Theorem will be proven similary, besides that we will not use the topological approach but an algebraic one, the space of algebraic curvature operators. \\
	When understanding Einstein manifolds in arbitrary dimensions one usually restricts to several geometric properties, such as homogenecity. A classical approach here is to consider the Einstein Hilbert functional $\mathcal{S}: \mathcal{M} \to \R$, given by
	\[ g \mapsto \int_M \scal(g) \text{dvol}_g,
	\]
	where $\mathcal{M}$ denotes the space of metrics on $M$. Then, Einstein metrics on $M$ correspond to critical points of the Einstein Hilbert functional, restricted to $\mathcal{M}_1$, the space of metrics of volume $1$. A very classical result by Wang-Ziller \cite{Wang:1986vw} asserts that if $G$ is a connected, compact Lie group and $H$ a connected, closed subgroup, such that $G/H$ is effective, then the Einstein Hilbert functional, restricted to the space of $G$-invariant metrics of volume $1$, is bounded from above and proper if and only if $H$ is a maximal connected subgroup of $G$. In this case, $\mathcal{S}$ reaches its global maximum at a $G$-invariant metric which is Einstein. \\ Very recently, B\"ohm-Lafuente \cite{lafuenteboehm} have shown that any connected homogeneous Einstein manifold with negative Einstein constant has to be diffeomorphic to Euclidean space. This is known as the Alekseevskii conjecture. \\
	When restricting to symmetric spaces the story is much easier. This is mainly due to the fact that indecomposable symmetric spaces are classified. 
\end{section}
\begin{section}{The space of algebraic curvature operators}
	In this section we recall the definition of an algebraic curvature operator. Furthermore, we relate that notion to the Riemannian curvature tensor of a manifold. Algebraic curvature operators are an algebraic access to the geometric notion of curvature. By using this notion it is easier to define various algebraic curvature conditions. Moreover, it is advantageous to do computations in this algebraic space rather than calculations with tensors.
	\subsection{Algebraic Preliminaries}
	Denote by $\son = \{ A \in \Mat_n(\R) \mid A^t = -A \}$ the Lie algebra of skew symmetric matrices, endowed with the inner product $\langle A,B \rangle = -\frac{1}{2} \tr(AB)$. Note that this does not correspond to the standard inner product on $\Mat_n(\R)$, since there is an additional factor of $\frac{1}{2}$ involved here. Moreover, we denote by $\Lambda^2(\R^n)$ the set of bivectors on $\R^n$. We start with the following basic observation: If $\langle \cdot , \cdot \rangle_2$ denotes the standard inner product on $\R^n$, there is an inner product, also denoted by $\langle \cdot, \cdot \rangle_2$, on $\Lambda^2(\R^n)$ that is induced by 
	\[
	\langle v \wedge w , x \wedge y \rangle_2 = \langle v,x \rangle_2 \cdot \langle w,y \rangle_2 - \langle v, y \rangle_2 \cdot \langle w,x \rangle_2
	\]
	for $v,w,x,y \in \R^n$ such that the following holds:
	\begin{lem}\label{isomorphismson}
		The map $\varphi: (\Lambda^2(\R^n), \langle \cdot, \cdot \rangle_2) \to (\son, \langle \cdot, \cdot \rangle)$ that is induced by
		\[ \varphi(v \wedge w)x = \langle w, x \rangle_2 v - \langle v,x \rangle_2 w
		\]
		for $v,w,x \in \R^n$ is an isometry of vector spaces.
	\end{lem}
	\begin{proof}
		At first, we show that $\varphi(v \wedge w)$ is skew adjoint. For that let $x,y \in \R^n$. Then
		\begin{align*}
			\langle	\varphi(v \wedge w)x,y\rangle_2 & = \langle w,x \rangle_2 \cdot \langle v,y \rangle _2 - \langle v,x \rangle_2 \cdot \langle w,y \rangle_2 \\
			& = - \left( \langle w,y \rangle _2 \cdot \langle v,x \rangle _2 - \langle v,y \rangle_2 \cdot \langle w,x \rangle_2 \right) = - \langle \varphi(v \wedge w)y,x \rangle_2.
		\end{align*}
		Now it is left to show that $\varphi$ is an isometry. The easiest way for this is to write down orthonormal bases and show that they are mapped to each other. If $e_1, \dots , e_n$ denotes the standard basis of $\R^n$, then the set $\{ e_i \wedge e_j \mid 1 \leq i < j \leq n \}$ is an orthonormal basis for $\Lambda^2(\R^n)$. Furthermore, if $E_{ij} \in \Mat_n(\R)$ denotes the matrix which has a $1$ at the $i$-th row and $j$-th column and is 0 everywhere else, the matrices $S_{ij} = E_{ij} - E_{ji}$ form an orthonormal basis of $\son$ for $1 \leq i < j \leq n$. Then for $i < j$:
		\[ \varphi(e_i \wedge e_j) e_k = \langle e_j,e_k \rangle_2 e_i - \langle e_i, e_k \rangle_2 e_j = \begin{cases} e_i, & \text{ if } k = j \\
			-e_j, & \text{ if } k=i \\
			0, & \text {else.}
		\end{cases}
		\]
		This is exactly $S_{ij}$.
	\end{proof}
	Recall that a nonabelian Lie algebra is called \textit{simple} if it does not contain any nonzero proper ideal. It is fundamental to prove the following:
	\begin{prop}\label{simplenessofson}
		$\so(n)$ is simple for $n=3$ and $n \geq 5$. Furthermore, $\so(2)$ is abelian and there is an Lie algebra isomorphism $\so(3)_+ \oplus \so(3)_- \to \so(4)$ where the Lie algebras $\so(3)_{\pm}$ are isomorphic to $\so(3)$ and can be seen as two nontrivial ideals in $\so(4)$.
	\end{prop}
	\begin{proof}
		It is easy to see that $\so(2)$ is isomorphic to $\R$. In order to prove the \\$\so(3)_+ \oplus \so(3)_- \to \so(4)$ isomorphism consider the basis 
		\[ i_+= \left( \begin{array}{cccc} 0 & 1 & & \\ -1 & 0 & & \\ & & 0 & 1 \\ & & -1 & 0 \end{array} \right),  \; \; j_+ = \left( \begin{array}{cccc}  &  &1 & 0 \\  &  & 0& -1 \\ -1& 0 &  &  \\ 0&1 &  &  \end{array} \right), \; \; k_+ = \left( \begin{array}{cccc} &  &0  & -1 \\  &  & -1 & 0 \\ 0& 1 & &  \\ 1 & 0 & &  \end{array} \right)
		\] 
		and 
		\[ i_-= \left( \begin{array}{cccc} 0 & 1 & & \\ -1 & 0 & & \\ & & 0 & -1 \\ & & 1 & 0 \end{array} \right),  \; \; j_- = \left( \begin{array}{cccc}  &  &1 & 0 \\  &  & 0& 1 \\ -1& 0 &  &  \\ 0&-1 &  &  \end{array} \right), \; \; k_- = \left( \begin{array}{cccc} &  &0  & 1 \\  &  & -1 & 0 \\ 0& 1 & &  \\ -1 & 0 & &  \end{array} \right)
		\] 
		of $\so(4)$. We set $\so(3)_+ = \text{span}\{ i_+,j_+,k_+\}$ and $\so(3)_-= \text{span}\{i_-,j_-,k_-\}$ and immediatly see by a calculation that $\so(3)_+$ and $\so(3)_-$ are ideals of $\so(4)$ such that $[\so(3)_+,\so(3)_-] = 0$. Thus the canonical map $\so(3)_+ \oplus \so(3)_- \to \so(4)$ is a Lie algebra isomorphism. It is left to show that $\so(3)_{\pm}$ is isomorphic to $\so(3)$ as a Lie algebra. First note that $[i_{\pm},j_{\pm}] = 2k_{\pm}, [j_{\pm},k_{\pm}]= 2i_{\pm}$ and $[k_{\pm},i_{\pm}] =2 j_{\pm}$. For the basis
		\[ a = \sqrt2 \left( \begin{array}{ccc} 0 & 1 & \\ -1 & 0 & \\ & & 0 \end{array} \right), \; \; b = \sqrt2 \left( \begin{array}{ccc} 0 &  & 1 \\  & 0 & \\ -1 & & 0 \end{array} \right), \; \; c = \sqrt2 \left( \begin{array}{ccc} 0 &  & \\  & 0 & -1 \\ & 1 & 0 \end{array} \right)
		\]
		of $\so(3)$ we also compute that $[a,b] = 2c, [b,c] = 2a$ and $[c,a] = 2b$. Thus the maps
		$\varphi_{\pm} : \so(3) \to \so(3)_{\pm}$ that are induced by $\varphi_{\pm}(a) = i_{\pm}, \varphi_{\pm}(b)=j_{\pm}$ and $\varphi_{\pm}(c) = k_{\pm}$ are isomorphisms of Lie algebras.
		A proof that $\so(3)$ and $\son$ are simple for $n \geq 5$ can be found in \cite{goodman2009symmetry}.
	\end{proof}
	For multiple computations later we simplify the Lie algebras $\so(3)_{\pm}$ further. Consider 
	\[ \su(2) = \left\{ \left( \begin{array}{cc} ia & z \\ -\bar{z} & -ia \end{array} \right) \in \Mat_n(\mathbb{C}) \mid a \in \R, z \in \mathbb{C} \right\} \] 
	and also
	\[ \syp(1) = \text{span}_\R \{ i,j,k \},
	\]
	where $i,j,k$ shall be seen as the imaginary units in the quaternions, i.e. \[ i^2=j^2=k^2=ijk=-1. \]
	\\
	We introduce an inner product on these spaces as follows: For $v,w \in \su(2)$ we define $\langle v,w \rangle = -\tr(vw)$ and for $x,y \in \syp(1)$ we set $\langle x,y \rangle =2 \cdot \text{Re}(x\bar{y})$. We define a map
	\begin{align*}
		\Phi: \su(2) \to \syp(1); \; \; \; \Phi\left( \begin{array}{cc} ia & b+ic \\ -b+ic & -ia \end{array} \right) = ai+bj+ck
	\end{align*} 
	and note that this defines an isometry of Lie algebras. This basically follows from 
	\begin{align*} 2\text{Re}(-(ai+bj&+ck)\cdot(mi+nj+pk))  = 2(am +bn + cp) \\ & = -\tr\left( \left( \begin{array}{cc} ia & b+ic \\ -b+ic & -ia \end{array}\right) \cdot \left( \begin{array}{cc} im & n+ip \\ -n+ip & -im \end{array} \right) \right)
	\end{align*} and the comparison of the Lie brackets, which are just given by the standard commutators. Now we can compare these Lie algebras with $\so(3)_{\pm}$. 
	\begin{lem}\label{differentversionsofso3}
		The map $\phi_{\pm}: \so(3)_{\pm} \to \su(2)$, induced by 
		\[ 
		\phi_{\pm}(i_{\pm}) = \left(\begin{array}{cc} i & \\ & -i \end{array} \right),
		\; \; \phi_{\pm}(j_{\pm}) = \left( \begin{array}{cc} & 1 \\ -1 &  \end{array} \right),
		\; \; \phi_{\pm}(k_{\pm})= \left( \begin{array}{cc} & i \\ i & \end{array} \right),
		\]
		is an isometry of Lie algebras.
	\end{lem}
	\begin{proof}
		We first check that $\phi_{\pm}$ is an isomorphism of Lie algebras. In order to see that we just have to look at the Lie bracket of $\su(2)$ and compare it with the bracket on $\so(3)_{\pm}$. We have
		\[ \left[ \left(\begin{array}{cc} i & \\ & -i \end{array} \right),\left( \begin{array}{cc} & 1 \\ -1 &  \end{array} \right) \right] = 2 \left( \begin{array}{cc} & i \\ i & \end{array} \right),
		\]
		\[ \left[ \left( \begin{array}{cc} & 1 \\ -1 &  \end{array} \right), \left( \begin{array}{cc} & i \\ i & \end{array} \right) \right] = 2 \left(\begin{array}{cc} i & \\ & -i \end{array} \right)
		\]
		and
		\[ \left[ \left( \begin{array}{cc} & i \\ i & \end{array} \right) , \left(\begin{array}{cc} i & \\ & -i \end{array} \right) \right] = 2 \left( \begin{array}{cc} & 1 \\ -1 &  \end{array} \right).
		\]
		These are exactly the same relations, we computed for $\so(3)_{\pm}$ in Proposition \ref{simplenessofson}. In order to check that $\phi$ is an isometry we state orthormal bases for both spaces and show that they are mapped to each other. An orthonormal basis of $\so(3)_{\pm}$ is given by $\left\{ \tfrac{1}{\sqrt2} i_{\pm} , \tfrac{1}{\sqrt 2} j_{\pm}, \tfrac{1}{\sqrt2} k_{\pm} \right\}$. So we just have to check that an orthonormal basis for $\su(2)$ is given by \[ \Bigl\{ \tfrac{1}{\sqrt2}\phi_+(i_+), \tfrac{1}{\sqrt2}\phi_+(j_+), \tfrac{1}{\sqrt2}\phi_+(k_+) \Bigl\}. \]
		That this basis is orthogonal is clear from the calculations concerning the Lie bracket and the oberservation that the matrices anticommute with respect to the ordinary matrix multiplication. We compute for example that
		\[-\tr\left(\tfrac{1}{\sqrt2} \left(\begin{array}{cc} i & \\ & -i \end{array} \right) \cdot \tfrac{1}{\sqrt2} \left(\begin{array}{cc} i & \\ & -i \end{array} \right) \right) = -\frac{1}{2}\tr \left( \begin{array}{cc} -1 & \\ & -1 \end{array}  \right) = 1\]
		and similar for the other elements of the basis. 
	\end{proof}
	For later convenience, we identify the latter Lie algebras and denote them all by $\so(3)_{\pm}, \su(2)_{\pm}$ and $\syp(1)_{\pm}$. We write $i_{\pm}, j_{\pm}, k_{\pm}$ for the standard generators of all of these spaces. Note that in all cases the set $\Bigl\{ \frac{1}{\sqrt2}i_{\pm}, \frac{1}{\sqrt2}j_{\pm}, \frac{1}{\sqrt2}k_{\pm} \Bigl\}$ is a orthonormal basis for the corresponding Lie algebra.\\
	
	In the end, we state an explicit formula for the Lie bracket of $\son$ which will be used later. The proof is a basic computation with elementary matrices. Denote by $\ad: \son \times \son \to \son$ the adjoint action of the Lie algebra of $\son$, i.e. $\ad(v,w) = \ad_v(w) = [v,w]$ for $v,w \in \son$.
	\begin{lem}\label{formulabracket}
		The Lie bracket of $\son$ satisfies
		\[ \ad_{e_i \wedge e_j} (e_p \wedge e_q) = \delta_{jp} e_i \wedge e_q + \delta_{iq} e_j \wedge e_p + \delta_{jq} e_p \wedge e_i + \delta_{ip} e_q \wedge e_j
		\]
		for all $1 \leq i,j,p,q \leq n$ with respect to the isomorphism $\varphi: \Lambda^2(\R^n) \to \son$.
	\end{lem}
	\subsection{Preliminaries on algebraic curvature operators}
	Let $(M,g)$ be a Riemannian manifold. Denote by \[\Rm_p : T_pM \times T_pM \times T_pM \times T_pM \to \R\] its curvature tensor at $p \in M$ which is explicitely given by
	\[ \Rm_p(V,W,X,Y) = g_p(\nabla_W \nabla_V X(p) - \nabla_W \nabla_V X(p) + \nabla_{[V,W]}X(p), Y(p))
	\]
	for any vector fields $V,W,X,Y \in \Gamma(T^*M)$. The classical identities 
	\begin{align}
		&\nonumber	\Rm_p(v,w,x,y) = -\Rm_p(w,v,x,y) \\
		& \label{eq:identitesalgeb} \Rm_p(v,w,x,y) = - \Rm_p(v,w,y,x) \\
		& \nonumber \Rm_p(v,w,x,y) = \Rm_p(x,y,v,w) \\
		& \nonumber \Rm_p(v,w,x,y) + \Rm_p(x,v,w,y) + \Rm_p(w,x,v,y) = 0
	\end{align}
	for all $v,w,x,y \in T_pM$ show that the curvature operator can be interpreted as a section in the bundle $S_B^2(\Lambda^2TM)$, that is the bundle of all symmetric tensors $R : \Lambda^2(TM) \times \Lambda^2(TM) \to \R$, which also satisfy the first Bianchi identity. This is called the bundle of algebraic curvature operators. \\ In this subsection, we will recall the algebraic properties of this bundle and its underlying vector space $S_B^2(\Lambda^2\R^n)$, the space of algebraic curvature operators. Anything, that is mentioned here is common knowledge but will be used heavily during this thesis. A more detailed description and detailed proofs can be found in \cite{andrews} or \cite{Chow2007TheRF}. We will use the identification $\varphi: \son \xrightarrow{\cong} \Lambda^2(\R^n)$ frequently within this section and moreover the whole thesis.
	\begin{defi} Let $n \in \mathbb{N}$. An $n$-dimensional algebraic curvature operator is a map $\mathcal{R} : \Lambda^2(\R^n) \to \Lambda^2(\R^n)$ which satisfies the identities (\ref{eq:identitesalgeb}) above. We denote the space of algebraic curvature operators by $S_B^2(\Lambda^2(\R^n))$ or by $S_B^2(\son)$.
	\end{defi}
	Here, we just have identified the selfadjoint map $\mathcal{R}$ with the corresponding symmetric map $\Rm : \Lambda^2(\R^n) \times \Lambda^2(\R^n) \to \R$ via
	\[ \Rm(v \wedge w,x \wedge y) = \langle \mathcal{R}(v \wedge w), x \wedge y \rangle.
	\]
	The first question that arises is clearly the question of existence: Given an algebraic curvature operator $\mathcal{R} \in S_B^2(\son)$, does there exist a manifold $M$ and a Riemannian metric $g$ on $M$, such that its curvature tensor is given by $\mathcal{R}$ at some point (or even everywhere)? \\
	The question of local existence has been answered before. In \cite[Ch 3.1.3]{florianschmidt} it was shown that there exists a neighbourhood $U$ of $0_n \in \R^n$ and a Riemannian metric $g$ on $U$ such that its curvature tensor $\Rm$ at $0_n \in U$ is given by $\mathcal{R}$. Furthermore, the metric can be choosen such that $\nabla \Rm$ vanishes at $0_n \in \R^n$. \\
	The global question is probably harder and even not well-posed. A possible question would be the following: \\ Let $M$ be a differentiable manifold, $\mathcal{R}_p: \Lambda^2(T_pM) \times \Lambda^2(T_pM) \to \R$ be a symmetric map, that satisfies the first Bianchi identity as above, in a way that it depends differentiably on $p \in M$.  Does there exist a metric $g$ on $M$ such that its curvature tensor is given by $\mathcal{R}_p$ for all $p \in M$? \\
	In general, the answer to this is negative and closely related to the question what kind of metrics several manifolds admit, e.g. whether the manifold admits a metric of constant sectional curvature. \\
	Note that there is an action of $\text{O}(n) = \{ A \in \text{GL}_n(\R) \mid AA^T = \id_n \}$ on $S_B^2(\Lambda^2(\R^n))$, given by
	\begin{align} \label{eq:actionofon} g.\mathcal{R}(v \wedge w, x \wedge y) = \mathcal{R}(gv \wedge gw, gx \wedge gy)
	\end{align}
	for any $g \in \text{O}(n)$ and $\mathcal{R} \in S_B^2(\Lambda^2(\R^n))$. It is well known that $g.\mathcal{R} \in S_B^2(\son)$ and that this action is isometric with respect to the standard scalar product
	\[ \langle \mathcal{R}, \mathcal{S} \rangle = \text{tr}(\mathcal{R} \circ \mathcal{S})
	\]
	on $S_B^2(\son)$.
	\begin{rem} 
		\begin{itemize}
			\item[(i)] With respect to the isometry $\varphi: \Lambda^2(\R^n) \to \son$ the representation above corresponds to
			\[g.\mathcal{R} = \Ad_g^{\tr} \circ \mathcal{R} \circ \Ad_g \]
			for $\mathcal{R} : \son \to \son$, where $\Ad_g(v) =gvg^{-1}$ denotes the adjoint action of $\text{O}(n)$ on its Lie algebra $\son$. 
			\item[(ii)] It is worth noting that this action is in fact a right-action, i.e. $(gh).\mathcal{R} = h.g.\mathcal{R}$. This contradicts to the standard notation, but we will stick with this, since we find it more natural to consider (\ref{eq:actionofon}), instead of the corresponding left action
			\[ g \ast \mathcal{R} (v \wedge w, x \wedge y) = \mathcal{R}(g^{-1}v \wedge g^{-1} w , g^{-1} x \wedge g^{-1}y).
			\]
			Besides that, it does not make any difference, if we just keep track of it.
		\end{itemize}
	\end{rem}
	We would also like to mention that the usual scalar product of the space of $(4,0)$-tensor fields, given by
	\[ \langle R,S \rangle = g^{i_1j_1} \cdots g^{i_4j_4} R_{i_1, \dots, i_4} S_{j_1, \dots, j_4}
	\]
	does not coincide with the scalar product on $S_B^2(\so(T_pM))$. But we have the following relation.
	\begin{lem} \label{differencescalarproduct}
		Let $R,S$ be two $(4,0)$-tensor fields on $\R^n$, that satisfy (\ref{eq:identitesalgeb}) and \\ $\mathcal{R}, \mathcal{S} \in S_B^2(\son)$ the corresponding algebraic curvature operators. Then we have
		\[ 4 \langle \mathcal{R}, \mathcal{S} \rangle = \langle R, S \rangle.
		\]
	\end{lem}
	\begin{proof}
		Let $e_1, \dots e_n$ be an orthonormal basis on $\R^n$. Then we find that
		\begin{align*}
			\langle R, S \rangle & = \sum_{i,j,k,l = 1}^n R_{ijkl} S_{ijkl} = \sum_{i,j,k,l= 1}^n \mathcal{R}(e_i \wedge e_j, e_k \wedge e_l) \mathcal{S}(e_i \wedge e_j, e_k \wedge e_l) \\ & = 4 \sum_{i < j, k < l} \mathcal{R}(e_i \wedge e_j, e_k \wedge e_l) \mathcal{S}(e_i \wedge e_j, e_k \wedge e_l) \\ &
			= 4 \sum_{i < j, k < l} \langle \mathcal{R}(e_i \wedge e_j), e_k \wedge e_l \rangle  \langle \mathcal{S}(e_i \wedge e_j), e_k \wedge e_l \rangle \\&
			= 4 \sum_{i < j} \langle \mathcal{R}(e_i \wedge e_j), \sum_{k < l} \langle \mathcal{S}(e_i \wedge e_j), e_k \wedge e_l \rangle e_k \wedge e_l \rangle \\& 
			= 4 \sum_{i < j} \langle \mathcal{R}(e_i \wedge e_j), \mathcal{S}(e_i \wedge e_j) \rangle = 4 \sum_{i < j} \langle \mathcal{RS}(e_i \wedge e_j), e_i \wedge e_j \rangle \\
			& = 4 \langle \mathcal{R}, \mathcal{S} \rangle.
		\end{align*}
		The proof is completed.
	\end{proof}
	We will denote the norm on $S_B^2(\son)$ by $|| \cdot ||$ and the norm on $T_4^0(\R^n)$ by $| \cdot |$ in order to not cause any \vspace{5mm} confusion.\\
	We now come to the decomposition of $S_B^2(\son)$ into irreducible submodules. We consider the Bianchi-map $b: S^2(\Lambda^2(\R^n)) \to S^2(\Lambda^2(\R^n))$, given by
	\[ b\mathcal{R}(v \wedge w, x \wedge y) = \tfrac{1}{3} \left( \mathcal{R}(v \wedge w,x \wedge y) + \mathcal{R}(x \wedge v, w \wedge y) + \mathcal{R}(w \wedge x, v \wedge y) \right).
	\]
	Here, $S^2(\Lambda^2(\R^n))$ simply denotes the space of symmetric operators on $\Lambda^2(\R^n)$. By construction we obtain that $\ker(b) = S_B^2(\Lambda^2(\R^n))$. An easy calculation shows that $\text{Im}(b) = \Lambda^4(\R^n)$. The element $e_i \wedge e_j \wedge e_k \wedge e_l$ for $i < j < k < l$ is identified with the operator $R \in S^2(\Lambda^2(\R^n))$ that satisfies $R(e_i \wedge e_j) = e_k \wedge e_l$, $R(e_i \wedge e_k) = -e_j \wedge e_l$ and $R(e_i \wedge e_l) = e_j \wedge e_k$. Since $b$ is $\text{O}(n)$-equivariant, we obtain an invariant decomposition $S^2(\Lambda^2(\R^n)) = S_B^2(\Lambda^2(\R^n)) \oplus \Lambda^4(\R^n)$. The key oberservation, while decomposing the space of algebraic curvature operator into irreducible submodules is the existence of two maps which are adjoint up to scaling. \\ Define the Ricci operator $\Ric : S_B^2(\Lambda^2(\R^n)) \to S^2(\R^n)$, given by
	\[ \Ric(\mathcal{R})(v,w) = \sum_{i=1}^n \mathcal{R}(v, e_i, w, e_i)
	\]
	for some orthonormal basis $\{ e_i \mid i= 1, \dots, n\}$ of $\R^n$. Note that this is invariant under the choice of the orthonormal basis of $\R^n$. \\
	Define on the other side the wedge operator $A \wedge B: \Lambda^2(\R^n) \to \Lambda^2(\R^n)$ for two endomorphisms $A,B \in \text{End}(\R^n)$ by
	\[ (A \wedge B) (v \wedge w) = \frac{1}{2} \left( Av \wedge Bw + Bv \wedge Aw \right).
	\] 
	Note that $A \wedge B = B \wedge A$ and that for two symmetric operators $A, B \in S^2(\R^n)$ we immediatly get $A \wedge B \in S_B^2(\Lambda^2(\R^n))$. By \cite[Prop 11.19]{andrews} we find that the operator $2 \id_{\wedge}(A) = 2 A \wedge \id_{\R^n}$ is the adjoint of $\Ric$. Furthermore, since these maps are $\text{O}(n)$-equivariant, we obtain that
	\[ S_B^2(\son) = \text{ker}(\Ric) \oplus S^2(\R^n).
	\]
	We denote $\text{ker}(\Ric) = \Weyl_n$ the space of Weyl curvature operators and get the well known invariant decomposition $S^2(\R^n) = S^2_0(\R^n) \oplus \langle \id_{\R^n} \rangle$. This leads to 
	\begin{prop}[{\cite[Prop. 11.8]{Chow2007TheRF}}]\label{standarddecomp} The decomposition 
		\[ S^2(\Lambda^2(\R^n)) = \langle \Id_{\son} \rangle \oplus S^2_0(\R^n) \oplus \Weyl_n \oplus \Lambda^4(\R^n)
		\]
		is $\text{O}(n)$-invariant and irreducible. Here, for $A \in \id_{\R^n} \oplus S^2_0(\R^n)$ the corresponding algebraic curvature operator is given by $A \wedge \id_{\R^n}$. Furthermore, for any $\mathcal{R} \in S_B^2(\son)$ we have
		\begin{align} \label{eq:decompo} \mathcal{R} = \frac{\scal(\mathcal{R})}{n(n-1)} \Id_{\son} + \frac{2}{n-2} \Ric_0(\mathcal{R}) \wedge \id_{\R^n} + \mathcal{R}_W,
		\end{align}
		where $\scal = \tr\Ric$ denotes the scalar curvature of $\mathcal{R}$, $\Ric_0(\mathcal{R}) = \Ric(\mathcal{R}) - \frac{\scal}{n} \id_{\R^n}$ denotes the traceless Ricci curvature operator of $\mathcal{R}$ and $\mathcal{R}_W$ denotes the orthogonal projection of $\mathcal{R}$ onto the space of Weyl curvature operators.
	\end{prop}
	The previous Proposition asserts that any algebraic curvature operator $\mathcal{R} \in S_B^2(\son)$ can be decomposed as
	\[ \mathcal{R} = \mathcal{R}_I + \mathcal{R}_{S^2_0(\R^n)} + \mathcal{R}_W.
	\]
	Here, each summand denotes the orthogonal projection onto the associated submodule of $S_B^2(\son)$. \\
	We would like to note that the fact that $\langle \id_{\R^n} \rangle$, $S^2_0(\R^n)$ and $\Lambda^4(\R^n)$ are irreducible is rather standard. Nevertheless, it is not obvious that $\Weyl_n$ is irreducible and we found a new proof of this fact, using the methods of Chapter 4. This proof is attached in Appendix II. \\
	Now we explain the occuring parts in the decomposition (\ref{eq:decompo}). Let $\mathcal{R} \in S_B^2(\son)$ be a curvature operator.
	\begin{itemize} 
		\item[(a)] If $\Ric_0(\mathcal{R}) = 0$ and $\mathcal{R}_W = 0$, Schurs Lemma, c.f. \cite[Lemma 11.21.]{andrews} shows that any $n$-manifold that admits such a curvature operator has to have constant sectional curvatures for $n \geq 3$. These manifolds are called space forms. Well known representatives of this kind are for example the euclidean space, spheres and the hyperbolic space.
		\item[(b)] By construction, any manifold with $\Ric_0(\mathcal{R}) = 0$ is an Einstein manifold, i.e. has a metric $g$ such that $\Ric_g = \lambda g$ for some $\lambda \in \R$. Vice versa, this means, that the curvature operator of any Einstein manifold is of the form
		\[ \mathcal{R}= \frac{\scal(\mathcal{R})}{n(n-1)} \Id_{\son} + W
		\]
		for some Weyl curvature operator $W \in \Weyl_n$. It is easy to see, that the map $p \mapsto \scal(\mathcal{R})(p)$ is constant, but the Weyl curvature operator may vary. One main objective of this thesis is to ask, which Einstein curvature operators really can occur on manifolds. An Einstein manifold which has non-vanishing Weyl curvature is $(\mathbb{CP}^n, g_{FS})$, where $g_{FS}$ denotes the Fubini Study metric on the complex projective space. We will explain it in Section 2.4.
		\item[(c)] Curvature operators with vanishing Weyl curvature, i.e. $\mathcal{R}_W = 0$ are sometimes called of "Ricci-type". Those manifolds are conformally flat, i.e. there exists a smooth map $f: M \to \R_{> 0}$ such that $f \cdot g$ is flat. For $n = 3$ we just get by dimension comparison that all curvature operators are of Ricci-type.
	\end{itemize}
	In the end, we state a lemma which just follows by a direct computation that we will omit in this thesis. Later in this thesis, it will sometimes be useful to keep track of the occuring dimensions.
	\begin{lem}
		Let $n \geq 3$. Then we have that
		\begin{itemize}
			\item[(i)] $\dim(S_B^2(\son)) = \tfrac{1}{12}n^2(n^2-1).$ 
			\item[(ii)] $\dim(S^2_0(\R^n)) = {n+1 \choose 2} -1.$
			\item[(iii)]$\dim(\Weyl_n) = \tfrac{n-3}{2} {n+2 \choose 3}.$
		\end{itemize}
	\end{lem}
	\subsection{The potential of a curvature operator} In this section we will recall two quadratic maps on $S^2(\son)$, which are $\text{O}(n)$-equivariant, in order to define the potential of a curvature operator. These quadratic maps are essential in Ricci flow theory and describe the reaction term of the way the curvature tensor changes under the evolution of the Ricci flow.
	\begin{defi}
		Let $\mathcal{R},\mathcal{S} \in S^2(\son)$ be two self adjoint maps on $\son$. We define
		\[\mathcal{R} \# \mathcal{S}(v,w) = -\tfrac{1}{2} \tr (\ad_v \circ \mathcal{R} \circ \ad_w \circ \mathcal{S}) \]
		for any $v,w \in \son$ and denote $\mathcal{R} \# \mathcal{R} = \mathcal{R}^{\#}$. Here, $\ad: \Lambda^2(\son) \to \son$ denotes the adjoint action of the Lie algebra $\son$. 
	\end{defi}
	Now we will sum up some of the properties of the $\#$-product. Very detailed proofs of this can be found in \cite{jaeger} and \cite{florianschmidt}.
	\begin{lem}\label{properties of sharp} Let $\mathcal{R}, \mathcal{S} \in S^2(\son)$.
		\begin{itemize} 
			\item[(i)] $\#$ is bilinear and symmetric.
			\item[(ii)] $\#$ is $\text{O}(n)$-equivariant.
			\item[(iii)] $\langle \mathcal{R} \# \mathcal{S} (v),w \rangle = -\frac{1}{2} \text{tr}(\ad_w \mathcal{R} \ad_v \mathcal{S})$.
			\item[(iv)] $\langle \mathcal{R} \# \mathcal{S}(v),w \rangle = \sum_{\alpha, \beta} \langle \left[\mathcal{R}(b_{\alpha}), \mathcal{S}(b_{\beta})\right],v \rangle \langle \left[b_{\alpha}, b_{\beta} \right], w \rangle$ for any $v,w \in \son$ and some orthonormal basis $\left\{ b_{\alpha} \in \son \mid \alpha=1, \dots, \frac{n(n-1)}{2}\right\}$ of $\son$.
		\end{itemize}
	\end{lem}
	Note that for two algebraic curvature operators $\mathcal{R}, \mathcal{S} \in S_B^2(\son)$ their $\#$-product is not necessarily still an algebraic curvature operator. But we have the following result, which can be found in \cite{andrews}.
	\begin{lem}
		Let $\mathcal{R}, \mathcal{S} \in S_B^2(\son)$. Then 
		\[ Q(\mathcal{R},\mathcal{S}) = \frac{1}{2} \left( \mathcal{RS} + \mathcal{SR} \right) + \mathcal{R} \# \mathcal{S}
		\]
		is an algebraic curvature operator. Furthermore, if $\mathcal{R}, \mathcal{S} \in \Weyl_n$ are curvature operators of Weyl type, then $Q(\mathcal{R},\mathcal{S})$ is again a Weyl curvature operator. We denote by $Q(\mathcal{R}) = Q(\mathcal{R},\mathcal{R})$.
	\end{lem}
	Now we are directly connected to Ricci flow theory. Namely, let $(M,g(t))$ be a solution to the Ricci flow for $t \in [0,T)$, i.e.
	\[ g'(t) = -2 \Ric_{g(t)}
	\]
	for any $t \in [0,T)$. Then the curvature tensor $\Rm$ of $M$ satisfies
	\begin{align} \label{eq:RFevo} \nabla_{\partial_t} \Rm_{g(t)} = \Delta_{g(t)} \Rm_{g(t)} + 2 Q(\Rm_{g(t)}).
	\end{align}
	Here, $\nabla$ denotes a covariant derivative on the space of spatial vector fields 
	\[ \{ v \in T(M \times \R) \mid dt(v) = 0 \}. 
	\] This approach is also called the Uhlenbeck-trick. For a comprehensive explanation to this consider \cite[Ch. 5]{andrews} or \cite[Appendix F]{Chow2007TheRF}. \\
	Although the analysis of (\ref{eq:RFevo}) is crucial, when considering the Ricci flow, the equation is still not fully understood. This is because the Lie algebraic constants that obviously are contained in the $\#$-product can become pretty complicated. For example, it may occur that after diagonalizing a curvature operator $R$, the expression $Q(R)$ is not diagonal anymore. A fundamental property of the $\#$-operator, that is heavily used within this thesis and was discovered in \cite{spaceforms}, is the following
	\begin{lem}\label{BWlemma}
		Let $\mathcal{R} \in S_B^2(\son)$. Then
		\[ \mathcal{R} + \mathcal{R} \# I = (n-1) \mathcal{R}_I + \frac{n-2}{2} \mathcal{R}_{S^2_0(\R^n)}.
		\]
		In particular, if $W \in \Weyl_n$, we have that $W + W \# I = 0$.
	\end{lem}
	We will now define the potential.
	\begin{defi}
		Let $\mathcal{R} \in S_B^2(\son)$ be a curvature operator. We define the potential of $\mathcal{R}$ by
		\[ P(\mathcal{R}) = \langle Q(\mathcal{R}),\mathcal{R} \rangle.
		\]
		Furthermore, we define the normalized potential of $\mathcal{R}$ by $P_{\text{nor}}(\mathcal{R}) = \frac{1}{||\mathcal{R}||^3}P(\mathcal{R})$.
	\end{defi}
	It is obvious from the definition that $P(\lambda R) = \lambda^3 P(R)$, i.e. the normalized potential is made in the way, that it is invariant under scaling. We define Huiskens trilinear form \[\text{tri}: S_B^2(\son) \times S_B^2(\son) \times S_B^2(\son) \to \R\] by
	$\text{tri}(\mathcal{R},\mathcal{S},\mathcal{T}) = \langle Q(\mathcal{R},\mathcal{S}),\mathcal{T} \rangle$. Note that this is fully symmetric. 
	\begin{rem}\begin{itemize} 
			\item[(i)] Using Hamiltons maximum principle, it often suffices to consider curvature conditions that are invariant under the ordinary differential equation
			\begin{align} \label{eq:ordinaryde} \frac{d}{dt} R = Q(R)
			\end{align}
			in order to find curvature conditions that are invariant under the Ricci flow.
			In \cite{andrews} it is shown that the latter ODE is the gradient flow of the potential of $\frac{1}{3}P(R)$. Therefore it may plausible, that the Ricci flow pulls the metric in the direction of the critical points of the potential. A standard computation shows that if $R(t)$ is the solution to (\ref{eq:ordinaryde}), then the normalized potential is strictily increasing, except at points where $Q(R) = \lambda R$ for some $\lambda \in \R$. This emphazises that eigenvalues of $Q$ are of particular interest. 
			\item[(ii)] Critical points of the normalized potential are always eigenvectors of $Q$. Furthermore, if we restrict the potential to the space of unit curvature operators, it is easy to show that critical points of the potential are always eigenvectors of $Q$. Indeed, let $\mathcal{R} \in S_B^2(\son)$ be a critical point of $P$, restricted to the unit curvature operators and $\mathcal{R}_1 \in S_B^2(\son)$ be tangential to the unit curvature operators at $\mathcal{R}$. Then, by assumption
			\[ 0 = \frac{d}{dt}_{ \vert t=0} P(\cos(t) \mathcal{R} + \sin(t) \mathcal{R}_1) = 3 \langle Q(\mathcal{R}), \mathcal{R}_1 \rangle.
			\]
			Since $\mathcal{R}_1 \perp \mathcal{R}$ we find that $Q(\mathcal{R}) = \lambda \mathcal{R}$ for some $\lambda \in \R$.
		\end{itemize}
	\end{rem}
	In the end, we will show that the potential behaves extremely neat for Einstein curvature operators.
	\begin{lem}\label{potentialeinstein}
		Let $\mathcal{R} = \lambda \Id_n + W$ for some $\lambda > 0$ and $W \in \Weyl_n$. Then
		\[ P(\mathcal{R}) = P(\lambda \Id_n) + P(W).
		\]
	\end{lem}
	\begin{proof}
		This directly follows from Lemma \ref{BWlemma}, since \[Q(\mathcal{R}) = Q(\lambda \Id_n) + 2Q(\lambda \Id_n, W) + Q(W)\] and the fact that $Q(W) \in \Weyl_n$ for any Weyl curvature operator.
	\end{proof}
	\subsection{Curvature operators in four dimensions}
	Due to the decomposition $\so(4) = \syp(1)_+ \oplus \syp(1)_-$ curvature operators in four dimensions are quite special and much better understood. We will give a brief overview. Let $R \in S^2(\so(4))$ be a symmetric map on $\so(4)$. With respect to the latter decomposition we write
	\[ R = \left( \begin{array}{cc} A & B \\
		B^{\tr} & C \end{array} \right)
	\]
	where $A \in S^2(\syp(1)_+), C \in S^2(\syp(1)_-)$ and $B \in \syp(1)_+ \otimes \syp(1)_-$. Then one can understand $R$ with respect to the irreducible decomposition of $S^2(\so(4))$ as follows.
	\begin{lem} Let $R \in S^2(\so(4))$ be a symmetric map on $\so(4)$.
		\begin{enumerate}
			\item[(1)] $R$ is an algebraic curvature operator if and only if $\tr(A)= \tr(C)$.
			\item[(2)] $R \in S_B^2(\so(4))$ is an Einstein curvature operator if and only if $B=0$.
			\item[(3)] $R \in S_B^2(\so(4))$ has vanishing scalar curvature if and only if $\tr(A) = \tr(C) = 0$.
		\end{enumerate}
	\end{lem}
	\begin{proof}
		To show (1) we compute that
		\begin{align*} \tr(A) - \tr(C)  & = \langle R(\tfrac{1}{\sqrt2}i_+) , \tfrac{1}{\sqrt2}i_+ \rangle +\langle R(\tfrac{1}{\sqrt2}j_+) , \tfrac{1}{\sqrt2}j_+ \rangle + \langle R(\tfrac{1}{\sqrt2}k_+) , \tfrac{1}{\sqrt2}k_+ \rangle \\
			& \; \; \; \;	- \langle R(\tfrac{1}{\sqrt2}i_-) , \tfrac{1}{\sqrt2}i_- \rangle -\langle R(\tfrac{1}{\sqrt2}j_-) , \tfrac{1}{\sqrt2}j_- \rangle -\langle R(\tfrac{1}{\sqrt2}k_-) , \tfrac{1}{\sqrt2}k_- \rangle \\
			& = 2 \langle  R(e_1 \wedge e_2), e_3 \wedge e_4 \rangle - 2 \langle R(e_1 \wedge e_3), e_2 \wedge e_4 \rangle + 2 \langle R(e_2 \wedge e_3), e_1 \wedge e_4 \rangle.
		\end{align*}
		In order to prove (2), we consider $D \in S_0^2(\R^4)$ and observe that 
		\[ D \wedge \id_{\R^4} = \left( \begin{array}{cc} 0 & B \\ B^{\tr} & 0 \end{array} \right).
		\]
		for a suitable matrix $B$. In order to prove (3), observe that $\scal(R) = 2 \tr(R)$.
	\end{proof}
	Moreover, there is an explicit formula for the $\#$-product.
	\begin{lem}[\cite{Hamilton:1986tf}]\label{sharp4}
		Let $\mathcal{R} \in S^2_B(\so(4))$ be a four dimensional curvature operator. Then
		\[ \mathcal{R}^{\#} = \left( \begin{array}{cc} A^{\#} & B^{\#} \\
			(B^{\#})^{\tr} & C^{\#} \end{array} \right)
		\]
		where $A^{\#} = 2 \adj(A)$ corresponds to the adjugate matrix of $A$ for any $3$-dimensional matrix.
	\end{lem}
	We would like to remind that for
	\[ A = \left( \begin{array}{ccc} a & b & c \\ d & e & f \\ g & h & j \end{array} \right)
	\]
	the adjugate matrix is given by
	\[ \text{adj}(A) = \left( \begin{array}{ccc}  ej-fj & fg-dj & dh- eg \\ ch -bj & aj-cg & bg - ah \\ bf-ce & cd - af & ae - bd \end{array} \right).
	\]
	\subsection{Alternative description of standard diagonal curvature operators}
	In this section we will describe an alternative description of curvature operators, that can be diagonalized in a standard basis $e_1 \wedge e_2, e_1 \wedge e_3, \dots, e_1 \wedge e_n, e_2 \wedge e_3, \dots, e_{n-1} \wedge e_n$ for some orthonormal basis of $e_1, \dots , e_n$ of $\R^n$. These curvature operators are sometimes called \textit{pure}. This description will heavily be used in calculations with curvature operators.
	\begin{defi} Let $\mathcal{R} \in S_B^2(\son)$ be a curvature operator and $e_1, \dots , e_n$ be an orthonormal basis of $\R^n$ such that 
		\[
		\mathcal{R} = \text{diag}\left( \lambda_{(1,2)}, \lambda_{(1,3)}, \dots, \lambda_{(1,n)}, \lambda_{(2,3)}, \dots, \lambda_{(n-1,n)} \right)
		\]
		is diagonal in the corresponding basis $e_1 \wedge e_2, \dots, e_{n-1} \wedge e_n$ of $\son$. We define the alternative algebraic curvature operator $\tilde{R} \in S^2_0(\R^n)$ by
		\[ \langle \tilde{R}(e_i), e_j \rangle = \begin{cases} \langle \mathcal{R}(e_i \wedge e_j), e_i \wedge e_j \rangle, & \text{ if } i \neq j, \\
			0, & \text { else.} \end{cases}
		\]
	\end{defi}
	Note that this definition makes even sense for any curvature operator represented in a standard basis of $\son$, but it is not reasonable that it will give us any information about the curvature operator since it might have non diagonal entries. Clearly, any curvature operator can be diagonalized in a basis of $\son$ but for non-standard bases of $\son$ the definition of alternative curvature operators does not make any sense, so we stick with the special case of curvature operators that can be diagonalized with respect to a standard basis. We make the following oberservation that will help us reading off algebraic information from the curvature operator.
	\begin{lem}\label{StandardWeyl}
		Let $\mathcal{R}$ be a curvature operator that can be diagonalized in a standard basis $e_1 \wedge e_2, \dots, e_{n-1} \wedge e_n$ of $\son$. Then $\mathcal{R}$ is a Weyl curvature operator if and only if $\sum_{i=1}^n\langle \tilde{R}(e_i), e_j \rangle = 0$ for all $1 \leq j \leq n$. Visually speaking, that it the sum of each column of $\tilde{R}$ sums up to $0$.
	\end{lem}
	\begin{proof}
		Since $\mathcal{R}$ is diagonal in the basis $e_1 \wedge e_2, \dots , e_{n-1} \wedge e_n$ its Ricci curvature $\Ric(\mathcal{R})$ is diagonal in $e_1, \dots , e_n$. So we compute that for $1 \leq j \leq n$ 
		\[\Ric(\mathcal{R})(e_j, e_j) = \sum_{i=1}^n \mathcal{R}(e_i \wedge e_j,e_i \wedge e_j) = \sum_{i=1}^n \tilde{R}(e_i,e_j).
		\]
		The result follows since curvature operators that are purely Weyl, are exactly those whose Ricci curvature vanishes identically.
	\end{proof}
	The most important advantage is that for standard diagonalizable curvature operators the vector field $Q(\mathcal{R}) = \mathcal{R}^2+\mathcal{R}^{\#}$ is actually computable with basic methods from linear algebra. The following observation was communicated to us by Christoph B\"ohm and can be found in a slightly different form in \cite{jaeger}.
	\begin{lem}\label{StandardSquareandSharp}
		Let $\mathcal{R} \in S_B^2(\son)$ be a curvature operator that can be diagonalized in a standard basis $e_1 \wedge e_2, \dots, e_{n-1} \wedge e_n$ of $\son$. Then 
		\begin{enumerate}
			\item[(a)] $\mathcal{R}^2$ is diagonal and $\langle \mathcal{R}^2(e_i \wedge e_j), e_i \wedge e_j \rangle = (\widetilde{R}_{ij})^2$. That is, if we want to calculate $\mathcal{R}^2$, we just have to square the entries of $\widetilde{R}$.
			\item[(b)] $\mathcal{R}^{\#}$ is diagonal and $\langle \mathcal{R}^{\#}(e_i \wedge e_j), e_i \wedge e_j \rangle = \left( \widetilde{R}^2 \right)_{ij} = \sum_{k=1}^n \widetilde{R}_{ik} \cdot \widetilde{R}_{kj}$. That is, if we want to calculate $\mathcal{R}^{\#}$, we just have to square $\widetilde{R}$ and delete its diagonal. More explicitely, $\widetilde{\mathcal{R}^{\#}} = (\widetilde{R})^2 -\diag((\widetilde{R})^2)$
		\end{enumerate}
	\end{lem}
	\begin{proof} Let $\mathcal{R}(e_i \wedge e_j) = \lambda_{ij} e_i \wedge e_j$ for $i < j$.
		The first part is clear. In order to prove (b) we find that by Lemma \ref{properties of sharp}
		\begin{align*} \langle \mathcal{R}^{\#}(e_i \wedge e_j), & e_k \wedge e_l \rangle \\ & = \sum_{a < b, c < d} \langle [\mathcal{R}(e_a \wedge e_b), \mathcal{R}(e_c \wedge e_d)], e_i \wedge e_j \rangle \cdot \langle [e_a \wedge e_b, e_c \wedge e_d], e_k \wedge e_l \rangle \\
			& = \sum_{a < b, c < d} \lambda_{ab}\lambda_{cd} \langle [e_a \wedge e_b, e_c \wedge e_d], e_i \wedge e_j \rangle \cdot \langle [e_a \wedge e_b, e_c \wedge e_d], e_k \wedge e_l \rangle
		\end{align*}
		It is now clear that $\mathcal{R}^{\#}$ is still diagonal with respect to the basis above. We use Lemma \ref{formulabracket} and the definition of the scalar product to obtain the result with a rather lengthy but straightforward computation, cf. \cite{jaeger}.
	\end{proof}
\end{section}
\begin{section}{The curvature operator of the product of spheres}
	For later use one of the most important examples is the product of spheres $S^k \times S^l$ with their corresponding Einstein metric. This section provides an overview of these manifolds and their curvature operators. As usual, denote by 
	\[S^n_r= \{ x \in \R^{n+1} \mid ||x||^2=r^2 \} \]
	the sphere of radius $r > 0$ in the euclidean space. For abbreviation, we write $S^n_1 = S^n$. As an embedded submanifold, $S^n_r \subset \R^{n+1}$ is known to have constant sectional curvatures equal to $\frac{1}{r^2}$. Thus, its curvature operator $\mathcal{R}_{S^n_r}$ is given by $\mathcal{R}_{S^n_r}=\frac{1}{r^2}\Id_{\son}$. We will denote its induced metric by $g_r$. It is easy to see that the scaling map
	\[ \varphi_r : (S^n_1,r^2g_1) \to (S^n_r,g_r),
	\]
	given by $\varphi_r(x)=rx$ is an isometry. Furthermore, a computation shows that $(S^n,g_1)$ is an Einstein manifold with Einstein constant $\lambda=(n-1)$. Indeed, let $e_1, \dots, e_n$ be an orthonormal frame for $T_pM$. Then 
	\begin{align*}
		\Ric(e_i,e_j)= \sum_{k=1}^nR(e_i,e_k,e_j,e_k) = (n-1) \delta_{ij} = (n-1) g_1(e_i,e_j).
	\end{align*}
	For $k,l \in \mathbb{N}$ with $k+l =n$ consider the product $S_{k,l} = S^k \times S^l_r$ with $r=\sqrt{\frac{l-1}{k-1}}$. Denote its induced product metric by $g_{k,l}$.
	\begin{lem} $(S_{k,l},g_{k,l})$ is an Einstein manifold with Einstein constant $\lambda = k-1$. 
	\end{lem}
	\begin{proof}
		Let $(p,q) \in S^k \times S^l_r$. Clearly, for $v,w \in T_pS^k \subset T_{(p,q)}S_{k,l}$ we have \[\Ric_{g_{k,l}}(v,w) = \Ric_{g_1}(v,w) = (k-1) g_1(v,w) = (k-1) g_{k,l}(v,w)\]
		and by definition of the product metric for $v \in T_pS^k$ and $w \in T_qS^l_r$ we have $\Ric(v,w) = 0 = g_{k,l}(v,w)$. Now let $v, w \in T_qS^l_r$. Then
		\begin{align*} \Ric_{g_{k,l}}(v,w) & = \Ric_{g_r}(v,w) = \Ric_{g_1}(v,w) 
			= (l-1) g_1(v,w) \\ & =  \frac{l-1}{r^2}g_r(v,w) = (k-1) g_r(v,w) = (k-1)g_{k,l}(v,w),
		\end{align*}
		which proves the Lemma.
	\end{proof}
	A direct conclusion is that $\scal(S_{k,l}) = n(k-1)$. Now we will compute its curvature operator. With respect to the decomposition $\son=\so(k) \oplus \so(l) \oplus \left(\R^k \otimes \R^l \right)$ we have
	\begin{align}\label{eq:curvsphere}
		\mathcal{R}_{S_{k,l}} =	\scriptsize \left( \begin{array}{ccc} \Id_{\so(k)} & & \\  & \tfrac{k-1}{l-1} \Id_{\so(l)} & \\ & & 0 \end{array} \right).
	\end{align}
	Thus $||\mathcal{R}_{S_{k,l}}||^2 = \frac{k(k-1)}{2} + \left(\frac{k-1}{l-1}\right)^2 \cdot \frac{l(l-1)}{2} = \frac{(k-1)(2kl-n)}{2(l-1)}$. We denote its Weyl curvature operator by $\widetilde{W_{S_{k,l}}}$ and find that
	\begin{align*}
		\widetilde{W_{S_{k,l}}} & = \mathcal{R}_{S_{k,l}} - \frac{\scal(\mathcal{R}_{S_{k,l}})}{n(n-1)} \Id_{\son} = \mathcal{R}_{S_{k,l}} - \frac{k-1}{n-1} \Id_{\son} \\ & = \scriptsize	\left( \begin{array}{ccc} \frac{l}{n-1}\Id_{\so(k)} & & \\  & \tfrac{k(k-1)}{(l-1)(n-1)} \Id_{\so(l)} & \\ & & -\tfrac{k-1}{n-1}  \Id_{\R^k \otimes \R^l}\end{array} \right).
	\end{align*}
	A computation shows that $||\widetilde{W_{S_{k,l}}}||^2 = \frac{kl}{2} \cdot \frac{k-1}{l-1} \cdot \frac{n-2}{n-1}$. We denote the normalized Weyl curvature operator by
	\[ W_{S_{k,l}} = \frac{1}{||\widetilde{W_{S_{k,l}}}||} \widetilde{W_{S_{k,l}}}.
	\]
	For our purposes it is also convenient to introduce 
	\[ \widetilde{S_{\crit,n}} = \begin{cases}
		S_{\tfrac{n}{2}, \tfrac{n}{2}}, & \text{ if } n \text{ is even,} \\
		S_{\tfrac{n+1}{2}, \tfrac{n-1}{2}}, & \text{ if } n \text{ is odd.}
	\end{cases}
	\]
	and in addition to that we define $S_{\crit,n} = \left(\widetilde{S_{\crit,n}}, ||\widetilde{W_{S_n}}||\cdot g_{\widetilde{S_{\crit,n}}} \right)$. This is made in such a way that it is homothetic to $\widetilde{S_{\crit,n}}$, but with normalized Weyl curvature operator. We denote its curvature operator by 
	\[ \mathcal{R}_{\sym}^{\crit} = \mathcal{R}_{S_{\crit,n}}
	\]  and its Weyl curvature by $\mathcal{W}_{\sym}^{\crit}$. We have the following
	\begin{prop}
		$S_{\crit,n}$ is an Einstein manifold with Einstein constant \[\lambda = \begin{cases}
			\frac{1}{n} \sqrt{2(n-1)(n-2)}, & \text { if } n \text{ is even,} \\
			\sqrt{2 \cdot \frac{n-1}{n-2}} \cdot \frac{n-1}{n+1}, & \text{ if } n \text{ is odd.}
		\end{cases}\]
	\end{prop}
	In the end, we will compute the potential of the Weyl curvature of $S_{k,l}$.
	\begin{prop} \label{potentialofsphere}
		The Weyl curvature of $S_{k,l}$ satisfies
		\[ W_{S_{k,l}}^2 + W_{S_{k,l}}^{\#} = \theta_{k,l} \cdot W_{S_{k,l}},
		\]
		where $\theta_{k,l} = \sqrt{2} \cdot \sqrt{\frac{k-1}{k} \cdot \frac{l-1}{l} \cdot \frac{n-1}{n-2} }$. Moreover, $P(W_{S_{k,l}}) = \theta_{k,l}$.
	\end{prop}
	\begin{proof}
		Using Lemma \ref{StandardSquareandSharp}, the curvature operator $\mathcal{R}_{S_{k,l}}$ satisfies
		\[ \mathcal{R}_{S_{k,l}}^2 + \mathcal{R}_{S_{k,l}}^{\#} = (k-1) \mathcal{R}_{S_{k,l}}.
		\]
		Since the identity part of $S_{k,l}$ also satisfies $\left(\mathcal{R}_{S_{k,l}} \right)_I^2 + \left( \mathcal{R}_{S_{k,l}} \right)_I^{\#} = (k-1) \left(\mathcal{R}_{S_{k,l}} \right)_I$, we find that by Lemma \ref{potentialeinstein},
		\[ \widetilde{W_{S_{k,l}}}^2 + \widetilde{W_{S_{k,l}}}^{\#} = (k-1) \widetilde{W_{S_{k,l}}}.
		\]
		The proof is now finished by scaling the operator.
	\end{proof}
	Finally, we note that
	\[\theta_n = \begin{cases} \frac{1}{n} \sqrt{2(n-1)(n-2)} & \text{ if } n \text{ is even,} \\
		\sqrt{2} \cdot \sqrt{\frac{(n-1)(n-3)}{(n+1)(n-2)}} & \text{ if } n \text{ is odd,}
	\end{cases}
	\]
	where $\theta_n = \theta_{n/2,n/2}$, if $n$ is even and $\theta_n = \theta_{(n+1/2),(n-1/2)}$, if $n$ is odd.
\end{section}
\begin{section}{The curvature operator of the complex projective space}
	\subsection{The general case}
	The complex projective space $\mathbb{CP}^n$, endowed with the Fubini Study metric $g_{FS}$ is a manifold of dimension $2n$, which will be a very important example throughout this thesis, particulary in dimension $4$. Therefore, we will compute its curvature tensor completely and give all curvature theoretic information that will be needed. Note that the Fubini Study metric is made in such a manner, that the projection map $\pi : (S^{2n+1},g_1) \to (\mathbb{CP}^n,g_{FS})$ is a Riemannian submersion, i.e. its differential, restricted to the horizontal space $\mathcal{H}_p = \text{ker}d\pi_p^{\perp}$, is an isometry for each $p \in S^{2n+1}$. For Riemannian submersions there is an explicit formula for the curvature tensor, computed by O'Neill:
	\begin{thm}[\cite{submersion}]
		Let $\pi: (M, \bar{g}) \to (N,g)$ be a Riemannian submersion and let $X,Y,Z,W \in \Gamma(TN)$ be vector fields on $N$ with horizontal lifts $\bar{X}, \bar{Y}, \bar{Z}, \bar{W} \in \Gamma(TM)$. Then the following formula holds true:
		\begin{align*}
			\bar{g}(R(\bar{X},\bar{Y})\bar{Z},\bar{W}) =  &\;  g(R(X,Y)Z,W) \\
			& - \tfrac{1}{4} \bar{g}\left( [ \bar{X}, \bar{Z}]^{\mathcal{V}},[\bar{Y},\bar{W}]^{\mathcal{V}} \right) + \tfrac{1}{4} \bar{g} \left( [\bar{Y}, \bar{Z}]^{\mathcal{V}} , [\bar{X}, \bar{W} ]^{\mathcal{V}} \right) \\
			& - \tfrac{1}{2} \bar{g} \left( [ \bar{Y}, \bar{X} ]^{\mathcal{V}}, [\bar{W}, \bar{Z} ]^{\mathcal{V}} \right).
		\end{align*}
		Here $\cdot ^{\mathcal{V}} : TM \to TM$ denotes the projection onto the vertical space $\mathcal{V}M := \ker d\pi \subset TM$.
	\end{thm}
	Now we consider the special case of $\pi : (S^{2n+1}, g_1) \to (\mathbb{CP}^n, g_{FS})$. Denote by $V$ the standard unit vector field on $S^{2n+1}$ which is tangential to the $S^1$-action. Then, one computes that $iV$ is the unit inward-pointing vector field, when restricted to $S^{2n+1}$, i.e. $iV(z) = -z$ for $z \in S^{2n+1}$. Denote by $X_1, \dots, X_n, iX_1, \dots, iX_n$ unit vector fields on $\mathbb{CP}^n$ that arise by the complex structure, such that $\{ X_1, iX_1, \dots, X_n, iX_n \}$ forms a local orthonormal basis of $T\mathbb{CP}^n$. Then by definition their basic horizontal lifts $\overline{X_k}, i\overline{X_k}$ are made in such a way that $\{ \overline{X_1}, \dots \overline{X_n}, i\overline{X_1}, \dots, i\overline{X_n}, V\}$ is an orthonormal basis of $TS^{2n+1}$. By the Koszul-formula and the $\mathbb{C}$-linearity of the covariant derivative we find for two basic vector fields $\overline{X}, \overline{Y}$ on $S^{2n+1}$
	\[ g_1\left( \tfrac{1}{2} [\overline{X}, \overline{Y}], V \right) = g_1\left(\overline{Y}, i \overline{X}\right),
	\]
	i.e. $\tfrac{1}{2} [\overline{X}, \overline{Y}]^{\mathcal{V}} = g_1\left( \overline{Y}, i \overline{X} \right)V$.
	Now it is easy to compute the full curvature tensor of $\mathbb{CP}^n$. For vector fields $X,Y,Z,W$ on $\mathbb{CP}^n$ and their basic horizontal lifts $\overline{X}, \overline{Y}, \overline{Z}, \overline{W}$ on $S^{2n+1}$ we have
	\begin{align*}
		g_{FS}(R_{\mathbb{CP}^n}(X,Y)Z,W) = & g_1(R_{S^n}(\overline{X},\overline{Y})\overline{Z},\overline{W}) + g_1(\overline{Z},\overline{iX}) \cdot g_1(\overline{W}, \overline{iY} ) \\
		& - g_1( \overline{Z}, \overline{iY} ) \cdot g_1( \overline{W}, \overline{iX} ) + 2 g_1 (\overline{X}, \overline{iY} ) \cdot g_1 ( \overline{Z} , \overline{iW} )
	\end{align*}
	In order to write down the full curvature operator it is convenient to distinguish four cases for the first two entries of the curvature tensor: \\
	\underline{1st case: $X \wedge Y = X_k \wedge X_l$:} Here, the curvature tensor does not vanish exactly for $Z \wedge W = X_k \wedge X_l$ or $Z \wedge W = iX_k \wedge i X_l$. One computes that
	\[ \mathcal{R}_{\mathbb{CP}^n}(X_k \wedge X_l, X_k \wedge X_l) = \mathcal{R}_{\mathbb{CP}^n}(X_k \wedge X_l, iX_k \wedge iX_l) = 1.
	\]
	\underline{2nd case: $X \wedge Y = X_k \wedge iX_l$ for $k \neq l$:} Here, the curvature tensor does not vanish exactly for $Z \wedge W = X_k \wedge iX_l$ or $Z \wedge W = iX_k \wedge X_l$. One computes that
	\[ \mathcal{R}_{\mathbb{CP}^n}(X_k \wedge iX_l, X_k \wedge iX_l) = \mathcal{R}_{\mathbb{CP}^n}(X_k \wedge iX_l, X_l \wedge iX_k) = 1.
	\]
	\underline{3rd case: $X \wedge Y = iX_k \wedge iX_l$:} Here, the curvature tensor does not vanish exactly for $Z \wedge W = iX_k \wedge iX_l$ or $Z \wedge W = X_k \wedge X_l$. One computes that
	\[ \mathcal{R}_{\mathbb{CP}^n}(iX_k \wedge iX_l, iX_k \wedge iX_l) = \mathcal{R}_{\mathbb{CP}^n}(iX_k \wedge iX_l, X_k \wedge X_l) = 1.
	\]
	\underline{4th case: $X \wedge Y = X_k \wedge iX_k$:} Here, the curvature tensor does not vanish exactly for $Z \wedge W = X_l \wedge iX_l$ for each $l = 1, \dots, n$. One computes that
	\begin{align*}
		&\mathcal{R}_{\mathbb{CP}^n}(X_k \wedge iX_k, X_k \wedge iX_k) = 4, \\
		&\mathcal{R}_{\mathbb{CP}^n}(X_k \wedge iX_k, X_l \wedge iX_l) = 2,
	\end{align*}
	for $l \neq k$.
	\\
	This immediatly implies that $\text{sec}(\mathbb{CP}^n) \in [1,4]$.
	For $\mathbb{CP}^2$ the curvature operator $\mathcal{R}_{\mathbb{CP}^2} : \Lambda^2\mathbb{R}^4\to \Lambda^2\mathbb{R}^4$ looks like
	\begin{center}
		$\mathcal{R}_{\mathbb{CP}^2}^{\mathcal{A}} =\left( \begin{array}{rrrrrr}
			1 & 0 & 0 & 0 & 0 & 1 \\
			0 & 4 & 0 & 0 & 2 & 0 \\
			0 & 0 & 1 & 1 & 0 & 0\\
			0 & 0 & 1 & 1 & 0 & 0 \\
			0 & 2 & 0 & 0 & 4 & 0 \\
			1 & 0 & 0 & 0 & 0 & 1
		\end{array}\right) $ \end{center}
	in the ordered basis $\mathcal{A} = ( X \wedge Y , X \wedge iX, X \wedge iY, Y \wedge iX, Y \wedge iY, iX \wedge iY)$. Now we compute the eigenvalues of $\mathbb{CP}^n$ and their corresponding eigenspaces.
	Consider $X = \sum_{k=1}^n X_k \wedge i X_k$. Then
	\begin{align*}
		\mathcal{R}_{\mathbb{CP}^n}(X) = \sum_{k=1}^n \mathcal{R}_{\mathbb{CP}^n}(X_k \wedge iX_k)  & = \sum_{k=1}^n \left(4X_k \wedge iX_k + 2 \sum_{l=1, l \neq k}^n X_l \wedge iX_l\right) \\ & = 4 X + 2(n-1) X = (2n+2) X,
	\end{align*}
	i.e. $X$ is an eigenvector of eigenvalue $2n+2$. \\ One easily checks that for each $k =2, \dots, n$ the vectors $Y_k = X_1 \wedge iX_1 - X_k \wedge iX_k$ are eigenvectors of eigenvalue $2$. Moreover, these vectors span the subspace of $\Lambda^2(\R^n)$ dealt with in the 4th case above. \\ Now consider the vectors $Z_{k,l}^\pm = X_k \wedge X_l \pm iX_k \wedge iX_l$ for $k < l$. One directly sees that
	\begin{align*} & \mathcal{R}_{\mathbb{CP}^n}(Z_{k,l}^+) = 2 Z_{k,l}^+ \\
		&	\mathcal{R}_{\mathbb{CP}^n}(Z_{k,l}^-) = 0 
	\end{align*}
	Moreover, these vectors span the subspace considered in the first and the third case above. In the end, consider the vectors $W_{k,l}^{\pm} = X_k \wedge iX_l \pm X_l \wedge iX_k$ for $k < l$. As above we see that
	\begin{align*}
		& \mathcal{R}_{\mathbb{CP}^n}(W_{k,l}^+) = 2 W_{k,l}^+, \\
		& \mathcal{R}_{\mathbb{CP}^n}(W_{k,l}^-) = 0.
	\end{align*}
	Summarized, this gives the following 
	\begin{thm}\label{eigenvalues of CPn}
		The curvature operator $\mathcal{R}_{\mathbb{CP}^n} : \Lambda^2(\R^{2n}) \to \Lambda^2(\R^{2n})$ of $\mathbb{CP}^n$ has eigenvalues $0, 2, 2n+2$. The dimensions of the eigenspaces are $n(n-1)$, $n^2-1$, $1$, respectively.
	\end{thm}

	Furthermore, a direct computation shows also that $\text{Ric}(\mathcal{R}_{\mathbb{CP}^n}) = (2n+2) \text{id}_{\mathbb{R}^n}$ and $\scal(\mathbb{CP}^n) = 2n(2n+2)$, which shows that $\mathbb{CP}^n$ is an Einstein manifold. We also calculate that
	\[ ||\mathcal{R}_{\mathbb{CP}^n}||^2 = (2n+2)^2+ 4(n^2-1) = 2n(4n+4).
	\]
	We now turn back to $\mathbb{CP}^2$. If we rewrite $X = e_1, Y=e_2, iX=e_3, iY=e_4,$ we find that under the isomorphism $\so(4) \to \syp(1)_+ \oplus \syp(1)_-$ a representation of $\mathcal{R}_{\mathbb{CP}^2}$ with respect to the ordered basis $\mathcal{B} = (i_+, j_+, k_+, i_-,j_-,k_-)$ is given by
	\[ \mathcal{R}_{\mathbb{CP}^2} = \text{diag}(2,2,2,0,6,0).
	\]
	Since $\scal(\mathbb{CP}^2) = 24$ and $\mathbb{CP}^2$ is an Einstein manifold one sees directly that that its Weyl part is given by
	\begin{center}
		$W_{\mathbb{CP}^2}^{\mathcal{A}} =\left( \begin{array}{rrrrrr}
			-1 & 0 & 0 & 0 & 0 & 1 \\
			0 & 2 & 0 & 0 & 2 & 0 \\
			0 & 0 & -1 & 1 &0 & 0\\
			0 & 0 & 1 & -1 & 0 & 0 \\
			0 & 2 & 0 & 0 & 2 & 0 \\
			1 & 0 & 0 & 0 & 0 & -1
		\end{array}\right) $ \end{center}
	under the ordered basis $\mathcal{A}$ and given by $W_{\mathbb{CP}^2}^{\mathcal{B}}= \text{diag}(0, 0,0,-2,4,-2)$ under the ordered basis $\mathcal{B}$. \\ In the end, we write down the corresponding normalized Weyl curvature operator and compute its potential. Since $||\mathcal{W}_{\mathbb{CP}^2}||^2 = 24$, we obtain that the normalized Weyl curvature operator of $\mathbb{CP}^2$ under the ordered basis $\mathcal{B}$, that we denote by $W_{\mathbb{CP}^2}$, is given by $W_{\mathbb{CP}^2} = \sqrt{\tfrac{1}{6}} \text{diag}(0,0,0,-1,2,-1)$.
	Now,  using Lemma \ref{sharp4}, we have
	\begin{align*}Q(W_{\mathbb{CP}^2}) & = W_{\mathbb{CP}^2}^2 + W_{\mathbb{CP}^2}^{\#} = \tfrac{1}{6} \cdot \left( \text{diag}(0,0,0,1,4,1) + 2 \text{diag}(0,0,0,\text{adj}(-1,2,-1)) \right) \\
		& = \tfrac{1}{6}\text{diag}(0,0,0,-3,6,-3) = \sqrt{\tfrac{3}{2}} W_{\mathbb{CP}^2},
	\end{align*}
	i.e. $P(W_{\mathbb{CP}^2}) = \sqrt{\frac{3}{2}}$.
	\subsection{The critical $\mathbb{CP}^2$ in any dimension}
	In this section we introduce an abstract curvature operator in any dimension that has the same Weyl part as $\mathbb{CP}^2$, up to a rotation.
	Consider the natural action $\SO(n) \times S_B^2(\son) \to S_B^2(\son)$, given by $g.\mathcal{R} = \Ad_g^{\tr} \circ \mathcal{R} \circ \Ad_g$ for $g \in \SO(n)$ and $\mathcal{R} \in S_B^2(\so(n))$. Later on, many applications that we give will be made up to that action. For simplification we will change our notation for $W_{\mathbb{CP}^2}$ from now on. We denote \[
	W_{\mathbb{CP}^2} = \sqrt{\frac{1}{6}} \text{diag}(0,0,0,2,-1,-1).
	\]
	Note that $W_{\mathbb{CP}^2} = \tfrac{1}{\sqrt{2}}(1+k_-). \text{diag}(0,0,0,-1,2,-1)$ where $\tfrac{1}{\sqrt{2}}(1+k_-) \in \SU(2)_- \subset \SO(4)$, as described before. So this really corresponds to the rescaled Weyl curvature of $\mathbb{CP}^2$ up to the $\SO(n)$-action. Because of the $\SO(n)$-equivariance of $Q$, the potential of the redefined $W_{\mathbb{CP}^2}$ stays the same. With respect to the basis $\mathcal{A}$ the redefined $W_{\mathbb{CP}^2}$ is of the form
	\[ W^{\mathcal{A}}_{\mathbb{CP}^2} = \frac{1}{2\sqrt6} \left( \begin{array}{cccccc} 2 & & & & & -2 \\
		& -1 & & & -1 & \\ & &-1 & 1 & & \\ & & 1 & -1 & & \\ & -1 & & & -1 & \\ -2 & & & & & 2 \end{array} \right).
	\] 
	Since there is a natural embedding $\Weyl_{n-1} \hookrightarrow \Weyl_n$, we can extend any Weyl curvature operator just by zeros in the lower right corner in order to obtain a new Weyl curvature operator in higher dimensions. Therefore, it makes sense to write $W_{\mathbb{CP}^2} \in \Weyl_n$ for any $n \geq 4$. Now we consider for any $\lambda > 0$ the operator
	\[ \mathcal{R}_{\lambda,n} = \overline{\lambda} I_n + W_{\mathbb{CP}^2}
	\]
	where $\overline{\lambda} = \tfrac{\lambda}{n-1}$ is chosen in the following way: Suppose that $(M,g)$ is an Einstein manifold with curvature operator $\mathcal{R}_{\lambda,n}$ at some point in $M$ for $n \geq 4$. Then the Einstein constant of $M$ is $\lambda$. Note here that by Schurs Lemma it is not necessary to ask for a curvature operator $\mathcal{R}$ with $\Ric(\mathcal{R}) = \lambda \cdot \id_{\R^n}$ for each point in $M$. \\ Since $P(W_{\mathbb{CP}^2}) = \sqrt{\frac{3}{2}}$ we may denote by $\lambda_{\crit} = \sqrt{\frac{3}{2}}$ and see that $\mathcal{R}_{\lambda_{\crit},n}$ is an eigenvector of $Q(\mathcal{R})=\mathcal{R}^2+\mathcal{R}^{\#}$ with $Q(\mathcal{R}_{\lambda_{\crit},n}) = \sqrt{\frac{3}{2}} \mathcal{R}_{\lambda_{\crit},n}$. We will also denote 
	\[ \mathcal{R}_{\mathbb{CP}^2}^{\crit} = \mathcal{R}_{\lambda_{\crit},n} = \tfrac{1}{n-1} \sqrt{\tfrac{3}{2}} \Id_n + W_{\mathbb{CP}^2}.
	\]
	For a reason that will become clear later we define a whole region for $\lambda$ to be of a certain interest.
	\begin{defi}\label{peculiar}
		Let $n \in \mathbb{N}$, $n \geq 5$. We call a parameter $\lambda \in \mathbb{R}$  \textit{intermediate} if
		\begin{align*}
			&	\lambda \in \left[ \tfrac{1}{n}\sqrt{2(n-1)(n-2)}, \sqrt{\tfrac{3}{2}} \right], \text{ if } n \text{ is even} \\
			&	\lambda \in \left[ \sqrt{2}\cdot\sqrt{\tfrac{(n-1)(n-3)}{(n+1)(n-2)}}, \sqrt{\tfrac{3}{2}} \right], \text{ if } n \text { is odd.}
		\end{align*}
	\end{defi}
	Later in the thesis, the proof of the Main Theorem will heavily distinguish the cases whether $\lambda$ is intermediate or not. To be precise, the case, where $\lambda$ is not intermediate can be covered by Conjecture A. The intermediate case needs other methods, including the symmetry of the manifold. The whole theory will be covered throughout the next chapters.
	
\end{section}
\begin{section}{The angle of curvature operators}
	In this section we are going to introduce a tool for measuring how far away a curvature operator is from the curvature operator of the sphere. We define
	\begin{defi}
		Let $\mathcal{R} \in S_B^2(\son)$ be a curvature operator. We define
		\[ \angle(\mathcal{R},\Id_n) = \arccos \left( \frac{\langle \mathcal{R}, \Id_n \rangle}{||\mathcal{R}|| \cdot ||\Id_n||} \right)
		\]
		as the angle of $\mathcal{R}$ to the identity operator. 
	\end{defi}
	The set of all curvature operators with the same angle forms a cone in $S_B^2(\son)$. For instance, let $\mathcal{R},\mathcal{S} \in S_B^2(\son)$ be two perpendicular unit Weyl curvature operators, $\lambda > 0$. Consider
	\[ \mathcal{R}(\varphi) = \lambda \Id_n + \cos(\varphi) \mathcal{R} + \sin(\varphi) \mathcal{S}.
	\]
	Then $\angle(\mathcal{R}(\varphi), \Id_n)$ is constant in $\varphi$. Furthermore, the definition is invariant under scaling, as one might expect.\\
	In the following, we will compute the angle of $\mathcal{R}_{\sym}^{\crit}$ and the angle of $\mathcal{R}_{\lambda,n}$. \\
	For $\mathcal{R}_{\lambda,n}$ it is easy to see that
	\[ \cos^2(\angle(\mathcal{R}_{\lambda,n}, \Id_n)) = \frac{\langle \mathcal{R}_{\lambda,n}, \Id_n \rangle^2}{||\mathcal{R}_{\lambda,n} ||^2 || \Id_n ||^2} = \frac{\lambda^2n}{\lambda^2n +2(n-1)}. 
	\]
	In particular, \[\cos^2(\angle(\mathcal{R}_{\mathbb{CP}^2}^{\crit}, \Id_n)) = \frac{3n}{7n-4}. \] \\
	For $\mathcal{R}_{\sym,n}^{\crit}$ we have to destinguish between two cases: \\
	\underline{1st case. $n$ is even:} Since the angle is invariant under scaling we will consider $\mathcal{R}_{\sym,n}^{\crit}$ as in (\ref{eq:curvsphere}) with $k = l = \frac{n}{2}$. We immediatly see that $||\mathcal{R}_{\sym,n}^{\crit}||^2 = \frac{n(n-2)}{4}$. Thus we obtain that
	\begin{align*}
		\cos^2(\angle(\mathcal{R}_{\sym,n}^{\crit}, \Id_n)) = \frac{\langle \mathcal{R}_{\sym,n}^{\crit}, \Id_n \rangle^2}{||\mathcal{R}_{\sym,n}^{\crit}||^2 || \Id_n ||^2} = \frac{\scal^2 ||\Id_n||^2}{n^2(n-1)^2||\mathcal{R}_{\sym,n}^{\crit}||^2}.
	\end{align*}
	Using that $\scal(\mathcal{R}_{\sym,n}^{\crit}) = \frac{n(n-2)}{2}$ we finally find that
	\[ \cos^2(\angle(\mathcal{R}_{\sym,n}^{\crit}, \Id_n)) = \frac{n-2}{2(n-1)},
	\]
	if $n$ is even. \\
	\underline{2nd case. $n$ is odd:} Here we consider $\mathcal{R}_{\sym,n}^{\crit}$ as in (\ref{eq:curvsphere}) with $k = \frac{n+1}{2}$, $l = \frac{n-1}{2}$. We obtain $\scal(\mathcal{R}_{\sym,n}^{\crit}) = \frac{n(n-1)}{2}$ and $||\mathcal{R}_{\sym,n}^{\crit}||^2 = \frac{(n-1)(n^2-2n-1)}{4(n-3)}$ to find that
	\begin{align*}
		\cos^2(\angle(\mathcal{R}_{\sym,n}^{\crit}, \Id_n)) = \frac{\scal^2 ||\Id_n||^2}{n^2(n-1)^2||\mathcal{R}_{\sym,n}^{\crit}||^2} = \frac{n(n-3)}{2(n^2-2n-1)}
	\end{align*}
	Note that by the computations above we can show that for $5 \leq n \leq 11$ it holds that $\angle(\mathcal{R}_{\mathbb{CP}^2}^{\crit}, \Id_n) < \angle(\mathcal{R}_{\sym,n}^{\crit}, \Id_n)$, but for $n \geq 12$ the inequality turns around. \\
	\begin{rem}
		For $\mathcal{R}_{\lambda,n} = \overline{\lambda} \Id_n + W_{\mathbb{CP}^2}$, we find that $\lambda \in \R$ is by construction intermediate, if and only if \[\angle(\mathcal{R}_{\mathbb{CP}^2}^{\crit}, \Id_n) \leq \angle(\mathcal{R}_{\lambda,n}, \Id_n) \leq \angle(\mathcal{R}_{\sym}^{\crit}, \Id_n). \]
	\end{rem}
	This leads us to the next section, in which we describe the main problem of this thesis.
\end{section}
\begin{section}{The second best Einstein metric}
	In this section we use the preliminaries that we have collected during this introductory  chapter in order to state the main theorem of this thesis. Moreover, we will give a strategy on our approach of proving this. In the next chapters all heuristics explained here will be made rigorous.
	\subsection{Full statement of the Main Theorem}
	Let $(M,g)$ be an Einstein manifold with $\Ric = \lambda g$ where $\lambda > 0$. Its curvature operator is then given by
	\[ \mathcal{R} = \overline{\lambda} \Id_n + W,
	\]
	where $\overline{\lambda} = \frac{\lambda}{n-1}$ and $W: M \to S_B^2(\so(TM))$ is a section into the bundle of curvature operators such that $W(p)$ is a Weyl curvature operator for each $p \in M$. \begin{defi} Let $M$ be a differentiable $n$-manifold. A metric $g$ on $M$ is called the second best Einstein metric, if $(M,g)$ is an Einstein manifold with positive scalar curvature and the curvature operator $\mathcal{R}$ of $M$ attains
		\[ \min_{N \in X} \text{max} \{ \angle(\mathcal{R}_N(q), \Id_n) \mid q \in N \},
		\]
		where \[X = \{ (N,h) \mid (N,h) \text{ is Einstein with positive Einstein constant and } \mathcal{R}_N \neq (\mathcal{R}_N)_I. \} \]
	\end{defi}
	After passing to the universal cover, the quantity $\max \{ \angle (\mathcal{R}_N(q), \Id_n) \mid q \in N\}$ does not change. Thus it makes more sense to restrict the set $X$ to simply connected Einstein manifolds with positive Einstein constant that are not homothetic to the round sphere, i.e. are not isometric to the round sphere with sectional curvatures $1$ after rescaling.
	The second best Einstein metric should be understood as the answer to the following question which is clearly  striking in understanding what the class of Einstein manifolds looks  	\vspace*{0.5cm} like: \\
	Which simply connected Einstein manifold with positive Einstein constant, that is non-isometric to the sphere, has a curvature operator that is as close as we can get to the curvature operator of the  \vspace*{0.5cm} sphere? \\ 
	After rescaling this comes down to the question which kind of Weyl curvature operators can be really part of an Einstein manifold. In general, we will ask the following \vspace*{0.5cm} question: \\
	Let $\alpha > 0$ be some constant. Does there exist a simply connected Einstein manifold $(M,g)$ with $\angle(\mathcal{R}(p), \Id_n) \leq \alpha$ for all \vspace*{0.5cm} $p \in M$? \\
	For pratical use, the constant $\alpha$ will be related to the angle of some algebraic curvature operator or even the angle of some curvature operator that comes from an Einstein manifold. \\
	Finally, there is the following conjecture, that was introduced to us by B. Wilking.
	\begin{conjB}
		Let $n \geq 7$. Then the universal cover of the second best Einstein metric is isometric to $(S_{\sym,n}^{\crit},g_n)$ up to scaling, where \[S_{\sym,n}^{\crit} = S^{\lceil \frac{n}{2} \rceil} \times S^{\lfloor \frac{n}{2} \rfloor} \] and
		\[ g_{n} = \begin{cases} g_{S^{n/2}} + g_{S^{n/2}}, & \text{ if }n \text{ is even,} \\
			g_{S^\frac{n+1}{2}} + \frac{n-1}{n-3}  g_{S^\frac{n-1}{2}},  & \text{ if }n \text{ is odd.}
		\end{cases}
		\]
	\end{conjB}
	The conjecture implies that the second best Einstein metric among simply connected manifolds is attained by $(S_{\sym,n}^{\crit},g_n)$. Moreover, it shows that no simply connected Einstein manifold $M$, which is not isometric to the round sphere $S^n$ can posses an Einstein metric $g$ such that $\angle(\mathcal{R}(p), \Id_n) < \angle(\mathcal{R}_{\sym,n}^{\crit}, \Id_n)$ for any $p \in M$. \\
	Before we come to the main theorem of this thesis, we would finally like to explain several examples that come up.
	\begin{exa}
		Consider $S_{k,l}$ with $k+l = n$ as in Section 2.4. It is an easy exercise to see that $\angle(\mathcal{R}_{S_{k,l}}, \Id_n)$ is minimized if $k=l = \frac{n}{2}$ for even $n$ and if $k = \frac{n+1}{2}$ and $l = \frac{n-1}{2}$ if $n$ is odd.
	\end{exa}
	\begin{exa}
		Consider $(\mathbb{CP}^n, g_{FS})$ with the Fubini Study metric. By the calculations we did in Section 2.4 we obtain that
		\[ \cos^2(\angle (\mathcal{R}_{\mathbb{CP}^n}, \Id_{2n})) = \frac{n+1}{2(2n-1)}
		\]
		Plugging in $n = 3$, we see that $\angle(\mathcal{R}_{S_6}, \Id_6) = \angle(\mathcal{R}_{\mathbb{CP}^3}, \Id_6)$. Thus the conjecture is not that simple in dimension $n = 6$. Even though, a simple computation shows that $\angle(\mathcal{R}_{\sym,n}^{\crit}, \Id_{2n}) < \angle(\mathcal{R}_{\mathbb{CP}^n}, \Id_{2n})$ for any $n \geq 4$.
	\end{exa}
	\begin{exa} Consider the algebraic curvature operator $\mathcal{R}_{\lambda,n} = \overline{\lambda} \Id_n + W_{\mathbb{CP}^2}$ for any $\lambda > 0$. As noted before,
		\[ \cos^2(\angle(\mathcal{R}_{\lambda,n}, \Id_n)) = \frac{\lambda^2n}{\lambda^2n +2(n-1)}.
		\]
		Plugging in $\lambda_{\text{crit}} = \sqrt{\frac{3}{2}}$, one find that $\angle(\mathcal{R}_{\mathbb{CP}^2}^{\crit}, \Id_n) <  \angle(\mathcal{R}_{\sym}^{\crit}, \Id_n)$ for any $n < 12.$ For $n \geq 12$ the inequality turns around. As we will see later, there is a correspondence to the behaviour of the potential of these curvature operators. \\
		Any Einstein manifold that has the curvature operator of the form $\mathcal{R}_{\lambda,n}$ for some $\lambda \geq \lambda_{\text{crit}}$ is in principal a better candidate for the second best Einstein metric than $S_{\sym}^{\crit}$. We will show that these manifolds do not exist. 
	\end{exa}
	Now we state our main theorem. 
	\begin{main}
		Let $n=10, 11$. Then there exists an angle $\alpha_0 > 0$ with \[\angle(\mathcal{R}_{\mathbb{CP}^2}, \Id_\son) < \alpha_0 < \angle(\mathcal{R}_{\sym}^{\crit}, \Id_\son) \] such that the following holds: Any simply connected Einstein manifold $(M,g)$ with positive scalar curvature that satisfies
		$\angle(\mathcal{R}(p), \Id_n) < \alpha_0$
		for all $p \in M$ is isometric to the round sphere up to scaling. Furthermore, $\alpha_0$ can be explicitely expressed by
		\[
		\alpha_0 = \angle(\mathcal{R}_{\mathbb{CP}^2}, \Id_n) + 1.015 \cdot 10^{-15}.
		\]
	\end{main}
	\subsection{Strategy of the proof of the Main Theorem}
	Let $(M,g_0)$ be an Einstein manifold with $\Ric = \lambda g_0$ for some $\lambda > 0$. Then the solution $g(t)$ with initial condition $g_0$ to the Ricci flow $g'(t) = -2\Ric_{g(t)}$ is given by $g(t) = (1-2 \lambda t) g_0$. It clearly exists for $t \in \left[0,\tfrac{1}{2\lambda}\right)$. Although the Ricci flow changes $(M,g_0)$ only by scaling, we can use the evolution equation of the curvature operator $\mathcal{R} \in \Gamma(S_B^2(\Lambda^2(TM)))$, given by
	\[ \frac{d}{dt} \mathcal{R}_{g(t)} = \Delta \mathcal{R}_{g(t)} + 2 Q(\mathcal{R}_{g(t)})
	\]
	to gain more information. Here $\Delta \mathcal{R}_{g(t)}(p)$ is defined as follows: Choose an orthonormal basis of $S_B^2(\Lambda^2(T_pM))$ and extend it along radial geodesics emanating from $p$ by parallel transport with respect to the spatial covariant derivative $\nabla$ to an orthonormal basis $X_1(q), \dots, X_N(q))$ of $S_B^2(\Lambda^2T_qM)$ for all $q$ in a neighbourhood of $p$ with $N = \text{rank}(S_B^2(\Lambda^2(TM)))$. Now we may write 
	\[ \mathcal{R} = \sum_{i=1}^N f_i \cdot X_i
	\] in a neighbourhood of $p$. Then we define
	\[ \Delta \mathcal{R}_{g(t)}(p) = \sum_{i=1}^N \Delta f_i(p) \cdot X_i(p),
	\]
	where $\Delta$ denotes the usual Laplace-Beltrami operator with respect to $\nabla$, c.f. \cite{finitefundamentalgroups}. Let $\overline{g} = c^2g$ for some $c \in \R$. Then $\mathcal{R}_{\overline{g}} = \frac{1}{c^2} \mathcal{R}_{g}$. This shows that \[ \frac{d}{dt} \mathcal{R}_{g(t)} = \frac{d}{dt} \left( \frac{1}{1-2\lambda t} \mathcal{R}_{g_0}\right) = \frac{2\lambda}{(1-2\lambda t)^2} \mathcal{R}_{g_0}.\] We find that at $t =0$,
	\[ \frac{2 \scal}{n} \mathcal{R}_{g_0} = 2 \lambda \mathcal{R}_{g_0} = \Delta \mathcal{R}_{g_0} + 2 Q(\mathcal{R}_{g_0}). \]
	Now taking the inner product on $S^2(\son)$ with $\mathcal{R}_{g_0}$ and integrating this on the manifold, we have
	\[
	\int_{M} \frac{2 \scal}{n} ||\mathcal{R}_{g_0}||^2 \text{dvol}_{g_0} = \int_{M} \left( \langle \Delta \mathcal{R}_{g_0}, \mathcal{R}_{g_0} \rangle + 2 \langle Q(\mathcal{R}_{g_0}),\mathcal{R}_{g_0} \rangle \right) \text{dvol}_{g_0}.
	\]
	Observe that we can view $\Delta \mathcal{R}$ again as a $(4,0)$-tensor field, that coincides with $\Delta \Rm$, where $\Delta$ is the usual Laplace-Beltrami operator and $\Rm$ is the usual curvature tensor on $M$. We then use Lemma \ref{differencescalarproduct},
	the well known formula $\Delta |\Rm|^2_g = 2 \langle \Delta \Rm_g, \Rm \rangle_g + 2 |\nabla \Rm|_g^2$ and the divergence theorem to obtain
	\[ \int_{M} \frac{2 \scal}{n} ||\mathcal{R}_{g_0}||^2 \text{dvol}_{g_0} = \int_M \left( \frac{1}{4} |\nabla \Rm|^2_{g_0} + 2 \langle Q(\mathcal{R}_{g_0}),\mathcal{R}_{g_0} \rangle \right) \text{dvol}_{g_0}.
	\]
	We rearrange the terms now in order to obtain an equality, where one side is purely analytic, but the other side is purely algebraic. More explicitely,
	\[ \int_M |\nabla \Rm|_{g_0}^2 \text{dvol}_{g_0} = 8 \int_M \left( P(\mathcal{R}_{g_0}) - \frac{\scal}{n} ||\mathcal{R}_{g_0}||^2 \right) \text{dvol}_{g_0}.
	\]
	The next lemma ensures that the norms involving $||\mathcal{R}_I||$ cancel out.
	\begin{lem}
		Let $(M,g)$ be an Einstein manifold with curvature operator $\mathcal{R}$. Then 
		\[ P(\mathcal{R}_I) = \frac{\scal}{n} ||\mathcal{R}_I||^2.
		\]
	\end{lem}
	\begin{proof}
		We start with noting that by Lemma \ref{BWlemma} we have \[P(\mathcal{R}_I) = \left( \frac{\scal}{n(n-1)} \right)^3 P(\Id_n) =  \left( \frac{\scal}{n(n-1)} \right)^3 (n-1) ||\Id_n||^2 = \frac{\scal^3}{n^3(n-1)^2} ||\Id_n||^2. \]
		But on the other hand,
		\[ \frac{\scal}{n} ||\mathcal{R}_I||^2 = \frac{\scal^3}{n^3(n-1)^2} ||\Id_n||^2,
		\]
		which is due to the fact that $\mathcal{R}_I = \frac{\scal}{n(n-1)} \Id_n$. This proves the claim.
	\end{proof}
	When applying the latter lemma together with Lemma \ref{potentialeinstein} and the fact that $\mathcal{R}_I$ and $\mathcal{R}_W$ are perpendicular, we find 
	\[ \int_M |\nabla \Rm|^2 \text{dvol}_{g_0} = 8 \int_M \left( P(\mathcal{R}_W) - \frac{\scal}{n} ||\mathcal{R}_W||^2 \right) \text{dvol}_{g_0}.
	\]
	At next, again since $\mathcal{R}_W$ and $\mathcal{R}_I$ are perpendicular, we obtain for $\alpha = \angle(\mathcal{R}, \Id_n)$ that at points with $||\mathcal{R}_W|| \neq 0$ we have
	\[ \frac{\cos(\alpha)}{\sin(\alpha)} = \frac{||\mathcal{R}_I||}{||\mathcal{R}_W||} = \frac{\scal}{\sqrt{2n(n-1)}} \frac{1}{||\mathcal{R}_W||}.
	\]
	Note that, in general $\alpha$ depends on $p \in M$. \\
	Using this, we obtain the following result which is essentially the start of the proof of the Main Theorem.
	\begin{prop}\label{keyequality}
		Let $(M,g)$ be an Einstein manifold with $\Ric = \lambda g$, where $\lambda > 0$. Then we have
		\begin{align} \label{eq:keyequality} \int_M |\nabla \Rm|^2_g \dvol_g = 8 \int_M ||\mathcal{R}_W||^3 \left( P_{nor}(\mathcal{R}_W) - \sqrt{\frac{2(n-1)}{n}} \frac{\cos(\alpha)}{\sin(\alpha)} \right) \dvol_g
		\end{align}
		where $\alpha$ denotes the angle between the identity $\Id_n$ and the curvature operator $\mathcal{R}$. The integrand on the right hand side is interpreted as $0$ at points, where $||\mathcal{R}_W|| = 0$.
	\end{prop}
	The main strategy now is to compare the algebraic right hand side with the analytic left hand side and get a contradictory estimate for all Einstein manifolds with $\angle(\mathcal{R}, \Id_n) \leq \alpha_0 $ while using Conjecture A in order to estimate the potential. We will quickly note the main steps of the proof and explain how we are going to solve them heuristically. A rigorous proof of all of them can be found in the next \vspace{0.5cm} chapters. \\
	\textbf{1st step:} Any Einstein manifold with $\alpha(p) < \angle(\mathcal{R}_{\mathbb{CP}^2}^{\crit}, \Id_n)$ is isometric to the round sphere. This is a direct computation, using that the maximum of the potential is given by $W_{\mathbb{CP}^2}$ for $n \leq 11$ (Conjecture A). \\
	\textbf{2nd step:} The curvature operator $\mathcal{R}$ of an Einstein manifold $(M,g)$ with
	\begin{align} \label{eq:interestingpoints}	P_{\nor}(\mathcal{R}_W)(p) - \sqrt{\tfrac{2(n-1)}{n}} \tfrac{\cos(\alpha(p))}{\sin(\alpha(p))} > 0
	\end{align}
	and $\alpha(p) < \angle(\mathcal{R}_{\sym}^{\crit})$ for some $p \in M$ is, after rescaling, of the form \[\mathcal{R} = \tfrac{\lambda}{n-1} \Id_n + \cos(\varphi) W_{\mathbb{CP}^2} + \sin(\varphi) W_1, \] where $\lambda$ is intermediate (compare Definition \ref{peculiar}), $\varphi \geq 0$ is a small angle depending on $n$ and $W_1$ is a unit curvature operator, perpendicular to \\ $\R \cdot W_{\mathbb{CP}^2} \oplus T_{W_{\mathbb{CP}^2}} \SO(n).W_{\mathbb{CP}^2}$. \\
	\textbf{3rd step:} For curvature operators $\mathcal{R}$ of the form as above the quantity
	\[ v \mapsto [\mathcal{R}, \ad_{\mathcal{R}(v)}]
	\]
	is quantitatively large for many $v \in \son$. This quantity, that we call the algebraic symmetry operator, has an impact on the question whether the manifold is symmetric. As a $(4,0)$-tensor it corresponds to the antisymmetrical part of the second covariant derivative of the curvature tensor, i.e.
	\[ \left( \nabla^2_{x,y} - \nabla^2_{y,x}\right) \Rm = [\mathcal{R}, \ad_{\mathcal{R}(x \wedge y)}].
	\] Furthermore, for manifolds with curvature operator $\mathcal{R}$ at $p \in M$ with the previous property we obtain that there is a quantitative estimate from below on
	\[ \int_{S(T_pM)} | \nabla^2_{x, \cdot} \Rm |^2_{g_p} d\mu^{n-1}_{S(T_pM)}(x)
	\]
	where $S(T_pM) \subset T_pM$ denotes the unit sphere in $T_pM$. \\
	\textbf{4th step:} After introducing normal coordinates on a quantitatively controlled balls $B_r(0_p) \subset T_pM$ around $p \in M$, we are able to use a Taylor expansion of $\nabla \Rm$ along radial geodesic starting at $p \in M$. A quantitative version of the Shi estimates helps us to control the error term, in order to find a quantitative estimate on
	\[ \frac{1}{B_r(0_p)} \int_{B_r(0_p)} |\nabla \Rm|_{\exp_p(v)}^2 d\lambda^n(v)
	\]
	from below. \\
	\textbf{5th step:} Using the geodesic flow, it is possible to do a pointwise comparison between the left hand side and the right hand side of (\ref{eq:keyequality}) at points $p \in M$ such that (\ref{eq:interestingpoints}) holds. At points, where the inequality above is the other way around, we just use the trivial inequality $|\nabla \Rm| \geq 0$ in order to see that Einstein manifolds with $\alpha(p) \leq \alpha_0$ are isometric to the sphere. Here $\alpha_0$ is defined as in the Main Theorem.
\end{section}

\chapter{A quantified version of the Shi estimates for Einstein manifolds}
After the Ricci flow was first studied in the 1980's, it was crucial to ask for the long time behavior of solutions. In 1989, W. Shi \cite{shi-deforming} observed that the curvature tensor of solutions to the Ricci flow has to blow up, when the solution stops existing. This is done by showing that, as long as the curvature tensor is bounded, all derivatives of the curvature tensor have to stay bounded as well. Although the proof is just a simple, but clever application of the standard maximum principle, it still has extensive impact on todays research. \\
In this chapter we will use the proof of the Shi estimates in order to prove a priori estimates on the covariant derivatives of the curvature tensors for Einstein manifolds in terms of the bound of the curvature tensor. When fixing the Einstein constant and assuming a bound on the Weyl curvature part, this estimate can be expressed purely in terms of the Einstein constant. \\
Since there will be many computations with tensors in this chapter we want to remind the reader of the notation for tensors and tensor fields: Let $V$ be a real vector space of dimension $n < \infty$.
Let $T \in T_{k}^l(V)$ be a $(k,l)$-tensor on $V$ for some $k,l \in \mathbb{N}$, i.e. a multilinear map
\[ T: \underbrace{V \times \dots \times V}_{k \text{ times}} \times \underbrace{V^* \times \dots \times V^*}_{l \text{ times}}\to \R.
\]
For $1 \leq k_0 \leq k$ and $1 \leq l_0 \leq l$ we denote the contraction of $T$ with respect to the $k_0$-th and $l_0$-th entry by $\tr_{k_0l_0}(T)$, that is a $(k-1,l-1)$-tensor, given by
\begin{align*}\tr_{k_0l_0}(T)&(v_1, \dots,  v_{k-1}, \omega^1, \dots, \omega^{l-1}) \\ &= \sum_{p=1}^n T(v_1, \dots, v_{k_0-1}, e_p, v_{k_0}, \dots, v_{k-1}, \omega^1, \dots, \omega^{l_0-1}, e^p, \omega^{l_0}, \dots \omega^{l-1}),
\end{align*}
where $e_1, \dots, e_n$ is some basis of $V$ with corresponding dual basis $e^1, \dots, e^n$ and $v_1, \dots, v_{k-1} \in V, \omega^1, \dots, \omega^{l-1} \in V^*$. A $(k,l)$-tensor field $T$ on a manifold $M^n$ is a section in the tensor bundle $T_k^l(M)$, i.e. for each $p \in M$ a $(k,l)$-tensor $T_p$ on $T_pM$ that smoothly depends on $p \in M$.
Note that the Riemannian curvature tensor $\Rm$ is a $(3,1)$-tensor field and the corresponding curvature tensor  of type $(4,0)$ is given by $\tr_{5,1}(\Rm \otimes g)$. For abbreviation, we will denote both tensor fields simply by $\Rm$. A priori, it is not possible to perform contractions with respect to entries that are both covariant or contravariant, e.g. it is not possible to contract a $(k,0)$-tensor $T$. However, if $g$ is a scalar product on $V$, a metric contraction of $T$ with respect to $1 \leq k_0,k_1 \leq k$ is given by $\tr_{k_1,2}\left(\tr_{k_0,1}(T \otimes g^{-1}) \right)$. This technique is  used often in the curvature setting, e.g. in the definition of the scalar curvature. \\
Without further ado, we will use the Einstein sum convention frequently in this chapter.
\section{Shi's a priori derivative estimates}
In this section we will recall the work of Shi \cite{shi-deforming} and show how it can be used to obtain a priori estimates for the covariant derivatives of the curvature tensor for Einstein manifold only depending on the Einstein constant and an upper bound of the curvature. For our purposes, we will use the formulation in \cite{chowknopf}.
\begin{thm}[Shi]
	Let $(M^n,g(t))_{t \in [0,T)}$ be a closed maximal solution to the Ricci flow. Then for each $\alpha > 0$ and each $k \in \mathbb{N}_{> 0}$, there exists a constant $C=C(n,k, \alpha)$, depending only on $n, k, \max\{ \alpha, 1\}$, such that if
	\[ \left| \Rm_t(x) \right| _{g_t(x)} \leq K  \text{ for all } x \in M^n \text{ and } t \in \left[0, \frac{\alpha}{K} \right],
	\]
	then
	\[ \left|\nabla^k\Rm_t(x)\right|_{g_t(x)} \leq \frac{CK}{t^{m/2}}  \text{ for all } x \in M^n \text{ and } t \in \left( 0, \frac{\alpha}{K} \right].
	\]
\end{thm}
Note that this theorem is not able to make any conclusion on the norm of the covariant derivative at $t=0$. This is certainly not possible. For instance, there is a family of metrics $g_{\varepsilon}$ on $S^1 \times S^1$ with bounded curvature, such that $|\nabla \text{Rm}_{g_{\varepsilon}} |_{g_{\varepsilon}}$ tends to infinity as $\varepsilon \to 0$, cf. \cite[p.~215]{chowluni}.
The estimates given in Shi's Theorem do perfectly fit in the behaviour of the parabolic rescaling of the Ricci Flow. For $c > 0$ and $\bar{g}= cg$ we find that 
\begin{align}\label{eq:derivativescale} \left| \nabla^k\text{Rm}_{\bar{g}}\right|_{\bar{g}} = \frac{1}{c^{1+k/2}} \left| \nabla^k\text{Rm}_{g} \right|_{g}.
\end{align}
Indeed, choose an orthonormal basis $e_1, \dots ,e_n$ of $(T_pM,\bar{g}_p)$. Then $\sqrt c e_1, \dots, \sqrt c e_n$ is an orthonormal basis of $(T_pM, g_p)$. Since $\text{Rm}_{\bar{g}} = c \text{Rm}_g$ and $\nabla^{\bar{g}} = \nabla^g$, we obtain
\begin{align*}
	\left| \nabla^k\text{Rm}_{\bar{g}}\right|_{\bar{g}}^2 & = \sum_{i_1, \dots, i_k=1}^n \sum_{j,l,m,n=1}^n \left(\nabla^k_{e_{i_1}, \dots, e_{i_k}} \text{Rm}_{\bar{g}}(e_j,e_l,e_m,e_n) \right)^2  \\ & = \frac{1}{c^{2+k}}  \sum_{i_1, \dots, i_k=1}^n \sum_{j,l,m,n=1}^n \left(\nabla^k_{\sqrt{c}e_{i_1}, \dots, \sqrt{c}e_{i_k}} \text{Rm}_{g}(\sqrt{c}e_j,\sqrt{c}e_l,\sqrt{c}e_m,\sqrt{c}e_n) \right)^2 \\ 
	& =  \frac{1}{c^{2+k}}	\left| \nabla^k\text{Rm}_{g}\right|_{g}^2.
\end{align*}
While the primary application of the Shi estimates is the proof of the long time existence for the Ricci flow, another simple, but on the first sight non obvious application is often ignored. Since Einstein manifolds are basically fixed points of the Ricci flow, up to scaling, the Shi estimates provide a priori bounds for the derivatives of the curvature tensor in the following way:
\begin{thm}
	Let $(M^n,g_0)$ be an Einstein manifold with $\Ric_{g_0} = \lambda g_0$ for $\lambda > 0$. Let $K \in \mathbb{R}$, such that $\left| \Rm_{g_0} \right| \leq K$. Then, for each $k \in \mathbb{N}_{> 0}$ there exists a constant $C = C(n, \lambda, K, k)$, such that
	\[ \left| \nabla^k \Rm_{g_0} \right|_{g_0} \leq C.
	\]
\end{thm}
\begin{proof}
	Consider the solution $g(t)=(1-2 \lambda t) g_0$ on $t \in \left[ 0, \frac{1}{2 \lambda} \right)$ to the Ricci flow. Then $|\Rm_{g(t)} |_{g(t)} = \frac{1}{1-2 \lambda t} |\Rm_{g_0}|_{g_0}$, which means, by a direct computation, we obtain that $| \Rm_{g(t)} |_{g(t)} \leq 2K$ for $k \leq \frac{1}{4 \lambda}$. Since $\lambda \leq K$, we get that the solution $(M,g(t))$ to the Ricci flow with $g(t) = (1-2\lambda t)g_0$ for $t \in [0, \frac{1}{4K}]$ satisfies 
	\[ |\Rm_{g(t)}|_{g(t)} \leq 2K.
	\]
	On the one hand, by the using the Shi estimates, we find for all $k \in \mathbb{N}_{>0}$ a constant $\tilde{C}=\tilde{C}(n,k)$, such that
	\[ \left| \nabla^k \Rm_{g(t)} \right|_{g(t)} \leq \frac{2\tilde{C}K}{t^{k/2}}
	\]
	for all $t \in \left( 0, \frac{1}{4K} \right]$.
	On the other hand, since $M$ is an Einstein manifold, we have that
	\[ 	\left| \nabla^k \Rm_{g(t)} \right|_{g(t)} = \left(\frac{1}{1-2\lambda t}\right)^{1+k/2} \left| \nabla^k\text{Rm}_{g_0} \right|_{g_0}.
	\]
	If we bring this together, we obtain for all $ t \in [0, \frac{1}{4k}]$,
	\[ \left| \nabla^k \Rm_{g_0} \right|_{g_0} \leq \frac{(1-2 \lambda t)^{1+k/2}}{t^{k/2}} 2 \tilde{C}K
	\]
	Since the right hand side is minimized by $t= \frac{1}{4K}$ we get
	\[ \left| \nabla^k \Rm_{g_0} \right|_{g_0} \leq 2^{k/2}(2K- \lambda)^{1+k/2}  \tilde{C} \eqqcolon C.
	\]
\end{proof}
\section{Quantitative evolution equations}
This section will be rather technical. In order to prove a quantitative version of the a priori estimates we will have to compute the evolution equations for $\left| \nabla^k \Rm_g \right|^2_g$ very carefully. Note that in any literature that covers the Shi estimates, the notation $A \ast B$ is used for a dimension depending summation or contraction of $A \otimes B$ for two tensors $A$ and $B$. A direct consequence of that notation is, by using the Cauchy-Schwarz inequality, that
\[ | A \ast B | \leq C(n) |A| \cdot |B|
\]
for some constant $C(n)$ just depending on the dimension. The advantage of this notation is an enormous simplification of the evolution equations for $\left| \nabla^k \Rm_g \right|_g$, e.g.
\begin{align*} \frac{d}{dt} \left| \nabla \Rm_{g(t)}\right|_{g(t)}^2 & \leq \Delta |\Rm_{g(t)}|_{g(t)}^2 - 2 |\nabla^2 \Rm_{g(t)} |_{g(t)}^2+ | \Rm_{g(t)} \ast \nabla \Rm_{g(t)} \ast \nabla \Rm_{g(t)} | \\ &  \leq \Delta |\Rm_{g(t)}|_{g(t)}^2 - 2 |\nabla^2 \Rm_{g(t)} |_{g(t)}^2 + C(n) |\Rm_{g(t)}|_{g(t)}|\nabla \Rm_{g(t)}|_{g(t)}^2,
\end{align*}
as in \cite{chowknopf}.
In order to get a bound on $|\nabla \Rm |^2$ we can now just use the scalar maximum principle for the function $F(t) = t|\nabla \Rm |^2 + \beta |\Rm|^2$ for some constant $\beta$. Although the maximum principle may give quantitative bounds, this will fail with that notation, since the constant $C(n)$ is not explicit. In this section we will give an upper bound on such constants by a careful analysis of the evolution equations. During the whole section we will write $|A| = |A_t|_{g(t)}$ for a time dependent tensor $A=A_t$, if it is clear that the tensor is time dependent and the norm should be taken with respect to the metric at time $t$. \\
Before we come to the first derivative estimate, we prove a basic Lemma that compares the norm of a trace of a tensor and its actual norm.
\begin{lem} \label{tracelemma}
	Let $T$ be a $(k,l)$-tensor on $\R^n$, $S$ be a $(p,q)$-tensor on $\R^n$.
	\begin{enumerate}
		\item[(a)] For each contraction $\tr(T)$ of $T$ we have that
		\[ | \tr(T) | \leq \sqrt{n} |T|.
		\]
		\item[(b)] For each contraction $\tr_{mn}(T \otimes S)$ of $T \otimes S$ with $1 \leq m \leq k$ and $l+1 \leq n \leq l+q$ we have that
		\[ |\tr_{mn}(T \otimes S)| \leq |T||S|.
		\]
	\end{enumerate} 
\end{lem}
\begin{proof}
	(a)	Let $e_1, \dots, e_n$ be an orthonormal basis on $\R^n$. Denote by $e^1, \dots, e^n$ its dual basis. We can assume without loss of generality that \[\tr(T)(v_1, \dots, v_{k-1}, \omega^1, \dots, \omega^{l-1} ) = \sum_{i=1}^nT(e_i, v_1, \dots, v_{k-1}, e^i, \omega_1, \dots, \omega^{l-1}), \]
	i.e. the contraction is taken with respect to the first covariant and the first contravariant entry.
	Then
	\begin{align*}
		|\tr(T)|^2 & = \sum \left(\tr(T)(e_{i_1}, \dots, e_{i_{k-1}}, e^{j_1}, \dots, e^{j_{l-1}}) \right)^2 \\
		& = \sum \left( \sum_{i=1}^n T(e_i,e_{i_1}, \dots, e_{i_{k-1}}, e^i, e^{j_1}, \dots, e^{j_{l-1}})\right)^2 \\
		& \leq n \sum \sum_{i=1}^n T(e_i, e_{i_1}, \dots, e_{i_{k-1}}, e^i, e^{j_1}, \dots, e^{j_{l-1}})^2 \\ 
		& \leq n \sum T(e_{i_1}, \dots, e_{i_k}, e^{j_1}, \dots, e^{j_l})^2 = n |T|^2,
	\end{align*}
	where the sum is taken over all $1 \leq i_1, \dots i_k, j_1, \dots j_l \leq n$.\\
	(b) We can assume without loss of generality that the contraction is taken with respect to the first covariant and the $l+1$st contravariant entry, i.e. $m=1$, $n=l+1$. Then, by definition
	\begin{align*}
		\tr&_{mn}(T  \otimes S)(v_1, \dots, v_{k+p-1}, \omega^1, \dots \omega^{l+q-1})  \\
		& = \sum_{i=1}^n T(e_i, v_1, \dots, v_{k-1}, \omega^1, \dots, \omega^l) \cdot S(v_k, \dots, v_{k+p-1}, e^i, \omega^{l+1}, \dots, \omega^{l+q-1}) \\
		& = \langle T(\cdot, v_1, \dots, v_{k-1}, \omega^1, \dots, \omega^l), \tilde{S}(v_k, \dots, v_{k+p-1}, \cdot, \omega^{l+1}, \dots, \omega^{l+q-1}  \rangle,
	\end{align*}
	where $\tilde{S}$ denotes the operator, where we lowered the first contravariant index and $e_1, \dots, e_n$ is some basis of $\R^n$. Clearly, $|\tilde{S}| = |S|$. Thus, if we assume $e_1, \dots, e_n$ to be orthonormal, we obtain by the Cauchy Schwarz inequality
	\begin{align*}
		|\tr_{mn}(T \otimes S)|^2 & = \sum \left(\langle T_{\cdot, e_{i_1}, \dots, e_{i_{k-1}}}^{e^{j_1}, \dots, e^{j_{l}}}, \tilde{S}_{e_{i_k}, \dots e_{i_{k+p-1}}, \cdot}^{e^{j_{l+1}}, \dots , e^{j_{l+q-1}}}\rangle \right)^2 \\ & \leq \sum \left|T_{\cdot, e_{i_1}, \dots, e_{i_{k-1}}}^{e_{j_1}, \dots ,e_{j_l}}\right|^2 \cdot \left|\tilde{S}_{e_{i_k}, \dots, e_{i_{k+p-1}}, \cdot}^{e_{j_{l+1}}, \dots ,e_{j_{l+q-1}}}\right|^2 \\
		& =|T|^2 \cdot |\tilde{S}|^2 = |T|^2 \cdot |S|^2,
	\end{align*}
	where the sum is taken over all $1 \leq j_1, \dots, j_{l+q-1}, i_1, \dots i_{k+p-1} \leq n$. 
\end{proof}
As a warm-up we start with a well known result.
\begin{prop}
	Under the Ricci flow the norm of the Riemannian curvature operator $\Rm_t$ satisfies
	\[ \frac{d}{dt} | \Rm |^2 \leq \Delta | \Rm |^2 - 2 | \nabla \Rm |^2 + 16 | \Rm |^2.
	\]
\end{prop}
\begin{proof}
	A classical result states that
	\begin{align*} \frac{d}{dt} \Rm_{ijkl}  = & \Delta \Rm_{ijkl} + 2(B_{ijkl} -B_{ijlk} + B_{ikjl} - B_{iljk} ) \\
		& - \left( \Ric_i^p\Rm_{pjkl} + \Ric_j^p \Rm_{ipkl} + \Ric_k^p \Rm_{ijpl} + \Ric_l^p \Rm_{ijkp} \right),
	\end{align*}
	where $B_{ijkl} = - \Rm_{pij}^q \cdot \Rm_{qlk}^p$, see \cite{chowknopf}. Now,
	\begin{align*}
		\frac{d}{dt} |\Rm|^2 & = \frac{d}{dt} g^{ip}g^{jq}g^{ks}g^{lt} \Rm_{ijkl} \Rm_{pqst} \\ & = 
		2 g^{ip}g^{jq}g^{ks}g^{lt}\frac{d}{dt}\left( \Rm_{ijkl}\right)\Rm_{pqst} + \frac{d}{dt} \left( g^{ip} \right) g^{jq}g^{ks}g^{lt} \Rm_{ijkl} \Rm_{pqst} \\ & + g^{ip}\frac{d}{dt}\left(g^{jq} \right) g^{ks}g^{lt} \Rm_{ijkl} \Rm_{pqst} +  g^{ip}g^{jq} \frac{d}{dt} \left(g^{ks} \right)g^{lt} \Rm_{ijkl} \Rm_{pqst} \\ & + g^{ip}g^{jq}g^{ks}\frac{d}{dt}\left( g^{lt} \right) \Rm_{ijkl} \Rm_{pqst}
	\end{align*}
	Now we use the evolution equation of the inverse metric
	$\frac{d}{dt} g^{kl} = 2\Ric^{kl}$ to see that exactly the terms involving both, $\Ric$ and $\Rm$, cancel out. We obtain
	\begin{align*}
		\frac{d}{dt} |\Rm|^2 = 2\langle \Delta \Rm, \Rm \rangle + 4 \langle \mathcal{B}, \Rm \rangle,
	\end{align*}
	where $\mathcal{B}_{ijkl}= B_{ijkl} - B_{ijlk} + B_{ikjl} - B_{iljk}$. By Lemma \ref{tracelemma}, we obtain that $|\mathcal{B}| \leq 4 |\Rm|^2$. Finally, we use the identity $\Delta |\Rm|^2 = 2 \langle \Delta \Rm, \Rm \rangle + 2 | \nabla \Rm |^2$ to get the result.
\end{proof}
Before we achieve the first derivative estimate, it is worth noting that, when deriving $|\nabla^k\Rm|$, one has to derive not just $\Rm$ and the metric, but also the covariant derivative $\nabla = \nabla^{g(t)}$. Recall that, cf. \cite{chowknopf}:
\[ \frac{d}{dt} \Gamma_{ij}^k = -g^{kl} \left( \nabla_i \Ric_{jl} + \nabla_j\Ric_{il} - \nabla_l \Ric_{ij} \right).
\]
By the parallelity of the metric we get that the Christoffel symbols on Einstein manifolds satisfy $\frac{d}{dt} \Gamma_{ij}^k = 0$. This directly proves the following:
\begin{prop} \label{interchanging} Let $(M,g(t))$ be a Ricci flow on an Einstein manifold $(M,g_0)$. Then
	\[ \frac{d}{dt} \nabla^k\Rm = \nabla^k \left( \frac{d}{dt} \Rm \right).
	\]
\end{prop}
Another important step, when computing evolution equations of derivatives in a form, such that we are able to use the maximum principle, is interchanging the Laplacian $\Delta$ and the covariant derivative $\nabla$ of a tensor. This leads to the computation of the curvature of a tensor.
\begin{lem} \label{firstcommutator}
	Let $T \in \mathcal{T}_p^0(M)$ be a tensor field on an $n$-dimensional manifold $M$. Then in any local coordinate system $\partial_1, \dots, \partial_n$ we have that
	\begin{enumerate}
		\item[(i)] $[\nabla_{\partial_k}, \Delta]T =g^{ij}\left( [\nabla_{\partial_k}, \nabla_{\partial_i}] \nabla_{\partial_j} T + \nabla_{\partial_i}[\nabla_{\partial_k}, \nabla_{\partial_j}] T\right)$
		\item[(ii)] $ \left( \left[ \nabla_{\partial_k}, \nabla_{\partial_l} \right]T\right)(\partial_{i_1}, \dots, \partial_{i_p}) =
		\Rm_{kli_1}^j T_{j i_2 \dots i_p} + \Rm_{kli_2}^j T_{i_1 j i_3 \dots i_p} + \dots + \Rm_{kli_p}^j T_{i_1 \dots i_{p-1}j}$
		\item[(iii)] The commutator of a tensor T satisfies in local coordinates \begin{align*}
			\left([\nabla_{\partial_k}, \Delta]T \right)(\partial_{i_1}, \dots, \partial_{i_p}) & = 2\Rm_{ki_1}^{ja} \nabla_{\partial_a}T_{ji_2 \dots i_p}+ \dots + 2\Rm_{ki_p}^{ja} \nabla_{\partial_a}T_{i_1 \dots i_{p-1}j} \\
			& + \nabla_{\partial_a} \Rm_{ki_1}^{ja} \cdot T_{ji_2 \dots i_p} + \dots + \nabla_{\partial_a} \Rm_{ki_p}^{ja} \cdot T_{i_1 \dots i_{p-1}j} \\
			& - \Ric_{k}^j \nabla_{\partial_j}T_{i_1 \dots i_p}
		\end{align*}
	\end{enumerate}
\end{lem}
\begin{proof}
	(i): The claim directly follows by the definition of $\Delta T = g^{ij} \nabla_{\partial_i} \nabla_{\partial_j} T$. \\
	(ii): This follows when writing $T$ in local coordinates, i.e.
	\[ T = T_{i_1 \dots i_p} dx^{i_1} \otimes \dots \otimes dx^{i_p}.
	\]
	By the product rule for the covariant derivative it is sufficient to compute \\ $\left[ \nabla_{\partial_k}, \nabla_{\partial_l} \right] dx^{i_j}$, since the curvature of a function vanishes.
	We get
	\begingroup
	\allowdisplaybreaks
	\begin{align*}
		\left[ \nabla_{\partial_k}, \nabla_{\partial_l} \right] dx^j & =  \left[ \nabla_{\partial_k}, \nabla_{\partial_l} \right] dx^j (\partial_m) dx^m \\ &  = \left( \nabla_{\partial_k} \left( \nabla_{\partial_l} dx^j \right) \partial_m - \nabla_{\partial_l} \left( \nabla_{\partial_k} dx^j \right) \partial_m \right)  dx^m \\
		& = \left(\partial_k \left( \nabla_{\partial_l} dx^j \right) \partial_m - \nabla_{\partial_l} dx^j(\nabla_{\partial_k}\partial_m) \right) dx^m \\
		& \; \; \;  - \left(\partial_l \left( \nabla_{\partial_k} dx^j \right) \partial_m - \nabla_{\partial_k} dx^j(\nabla_{\partial_l}\partial_m) \right) dx^m \\
		& = \partial_k \partial_l dx^j(\partial_m) dx^m - \partial_k dx^j(\nabla_{\partial_l}\partial_m) dx^m \\ 
		& \; \; \;	-  \partial_l \partial_k dx^j(\partial_m) dx^m + \partial_l dx^j(\nabla_{\partial_k}\partial_m) dx^m \\
		& \; \; \; + \left( - \partial_l dx^j ( \nabla_{\partial_k} \partial_m ) + dx^j ( \nabla_{\partial_l} \nabla_{\partial_k} \partial_m) \right) dx^m \\
		& \; \; \; + \left( \partial_k dx^j(\nabla_{\partial_l} \partial_m) - dx^j(\nabla_{\partial_k} \nabla_{\partial_l} \partial_m) \right) dx^m \\
		& = dx^j(\Rm_{klm}^r\partial_r) dx^m = \Rm_{klm}^j dx^m
	\end{align*}
	\endgroup
	Now we obtain
	\begin{align*} \left[ \nabla_{\partial_k}, \nabla_{\partial_l} \right]T & = T_{i_1 \dots i_p} \Rm_{klm}^{i_1} dx^m \otimes dx^{i_2} \otimes \dots \otimes dx^{i_p} \\
		& \; \; \; + \dots + T_{i_1 \dots i_p} \Rm_{klm}^{i_p} dx^1 \otimes \dots \otimes dx^{i_{p-1}} \otimes dx^m.
	\end{align*}
	(iii): This is basically a combination of (i) and (ii). It is worth mentioning that one has to apply (ii) not only on the tensor field $T$ but also on the tensor field $\nabla T$, which is a $(p+1, 0)$-tensor. We obtain
	\begin{align*}
		\bigl([\nabla_{\partial_k}, & \Delta]T \bigl)(\partial_{i_1},  \dots, \partial_{i_p})  = \left(g^{ab} \left[ \nabla_{\partial_k}, \nabla_{\partial_a} \right] \nabla_{\partial_b}T \right)_{i_1 \dots i_p} + \left( g^{ab} \nabla_{\partial_a} \left[ \nabla_{\partial_k}, \nabla_{\partial_b} \right] T \right)_{i_1 \dots i_p} \\
		& =g^{ab} \left( \Rm_{kab}^{j} \nabla_{\partial_j}T_{i_1 \dots i_p} + \Rm_{kai_1}^j \nabla_{\partial_b} T_{ji_2 \dots i_p} + \dots +  \Rm_{kai_p}^j \nabla_{\partial_b} T_{i_1 \dots i_{p-1}j} \right) \\
		& \; \; \; + g^{ab} \nabla_{\partial_a} \left( \Rm_{kbi_1}^j T_{ji_2 \dots i_p} + \dots + \Rm_{kbi_p}^j T_{i_1 \dots i_{p-1}j} \right) \\
		& = 2 g^{ab} \left(\Rm_{kai_1}^j \nabla_{\partial_b} T_{ji_2 \dots i_p} + \dots + \Rm_{kai_p}^j\nabla_{\partial_b} T_{i_1 \dots i_{p-1} j} \right) \\
		& \; \; \; + g^{ab} \left( \nabla_{\partial_a} \Rm_{kbi_1}^j \cdot T_{ji_2 \dots i_p} + \dots + \nabla_{\partial_a} \Rm_{kbi_p}^j \cdot T_{i_1 \dots i_{p-1}j} \right) \\
		& \; \; \; + g^{ab} \Rm_{kab}^j \nabla_{\partial_j}T_{i_1 \dots i_p} \\
		& = 2 \Rm_{ki_1}^{ja} \nabla_{\partial_a} T_{ji_2 \dots i_p} + \dots + 2 \Rm_{ki_p}^{ja} \nabla_{\partial_a} T_{i_1 \dots i_{p-1}j}  \\
		& \; \; \; + \nabla_{\partial_a}\Rm_{ki_1}^{ja} \cdot T_{ji_2 \dots i_p} + \dots + \nabla_{\partial_a} \Rm_{ki_p}^{ja} \cdot T_{i_1 \dots i_{p-1}j} \\
		& \; \; \; + g^{ab} \Rm_{kab}^j \nabla_{\partial_j}T_{i_1 \dots i_p}.
	\end{align*}
	In the end, we note that $g^{ab} \Rm_{kab}^j = -\Ric_k^j$, which proves the claim.
\end{proof}
We are now able to compute the evolution equation for the first  covariant derivative of the curvature tensor.
\begin{prop}[Evolution equation of the first covariant derivative] \label{firstderivative} Let $(M^n,g_0)$ be an Einstein manifold with $\Ric = \lambda g_0$, where $\lambda > 0$. Let $(M, g(t))_{t \in [0,T)}$ with $T = \frac{1}{2\lambda}$ be a Ricci flow on $M$ with $g(0)= g_0$. Then the norm of the covariant derivative of the curvature tensor $|\nabla \Rm|^2$ satisfies
	\[ \frac{d}{dt} \left| \nabla \Rm \right|^2 \leq \Delta | \nabla \Rm | ^2 -2 |\nabla^2 \Rm | ^2 + \left( 48 + 8 \sqrt{n} \right) |\Rm| |\nabla \Rm|^2.
	\]
\end{prop}
\begin{proof}
	We denote by $\lambda_t = \lambda(t) = \frac{\lambda}{1-2 \lambda t}$ the Einstein constant of $(M,g(t))$.
	Using Proposition \ref{interchanging} and the evolution equation of the inverse metric we obtain
	\begingroup
	\allowdisplaybreaks
	\begin{align*}
		\frac{d}{dt} |\nabla \Rm |^2 & = \frac{d}{dt} \left( g^{i_1j_1} \cdots g^{i_5j_5} \nabla_{\partial_{i_1}} \Rm_{i_2 \dots i_5} \nabla_{\partial_{j_1}} \Rm_{j_2 \dots j_5} \right) \\
		&  = 2 \langle \nabla \frac{d}{dt} \Rm, \nabla \Rm \rangle + 10 \lambda_t |\nabla \Rm |^2 \displaybreak\\
		& = 2 \langle \nabla \Delta \Rm, \nabla \Rm \rangle + 10 \lambda_t |\nabla \Rm |^2 \\
		& \; \; \; + 4g^{i_1j_1} \cdots g^{i_5j_5} \nabla_{\partial_{i_1}} B_{i_2i_3i_4i_5} \nabla_{\partial_{j_1}} \Rm_{j_2 \dots j_5}  \\
		& \; \; \; + 4g^{i_1j_1} \cdots g^{i_5j_5} \nabla_{\partial_{i_1}} B_{i_2i_3i_5i_4} \nabla_{\partial_{j_1}} \Rm_{j_2 \dots j_5} \\
		& \; \; \; + 4g^{i_1j_1} \cdots g^{i_5j_5} \nabla_{\partial_{i_1}} B_{i_2i_4i_3i_5} \nabla_{\partial_{j_1}} \Rm_{j_2 \dots j_5} \\
		& \; \; \; + 4g^{i_1j_1} \cdots g^{i_5j_5} \nabla_{\partial_{i_1}} B_{i_2i_5i_3i_4} \nabla_{\partial_{j_1}} \Rm_{j_2 \dots j_5} \\
		& \; \; \; - 2g^{i_1j_1} \cdots g^{i_5j_5} \nabla_{\partial_{i_1}} \left( \Ric_{i_2}^p \Rm_{pi_3 i_4i_5} \right) \cdot \nabla_{\partial_{j_1}} \Rm_{j_2 \dots j_5} \\
		& \; \; \; - \dots - 2g^{i_1j_1} \cdots g^{i_5j_5} \nabla_{\partial_{i_1}}\left( \Ric_{i_5}^p \Rm_{i_2 i_3 i_4 p} \right) \cdot \nabla_{\partial_{j_1}} \Rm_{j_2 \dots j_5}.
	\end{align*}
	\endgroup
	Now we use again that $(M,g(t))$ is Einstein in order to simplify the last terms to $-8 \lambda_t |\nabla \Rm |^2$. Further, by using the product rule for the covariant derivative and Lemma \ref{tracelemma} (b) for the terms involving $B$ we find that
	\begin{align*}
		\frac{d}{dt} | \nabla &\Rm |^2  \leq 2 \langle \nabla \Delta \Rm, \nabla \Rm \rangle + 32 |\Rm| |\nabla \Rm|^2 + 2 \lambda_t |\nabla \Rm|^2 \\
		& = 2 \langle \Delta \nabla \Rm, \nabla \Rm \rangle + 2\langle [\nabla, \Delta] \Rm, \nabla \Rm \rangle + \left( 32| \Rm | +2 \lambda_t \right) |\nabla \Rm |^2 \\
		& = \Delta |\nabla \Rm |^2 -2 |\nabla^2 \Rm |^2 + 2 \langle [\nabla, \Delta] \Rm, \nabla \Rm \rangle + \left( 32| \Rm | +2 \lambda_t \right) |\nabla \Rm |^2.
	\end{align*}
	Now observe that by Lemma \ref{firstcommutator} (iii) for $T= \Rm$, Lemma \ref{tracelemma} and the Cauchy Schwarz inequality we have
	\[ \langle [ \nabla, \Delta] \Rm , \nabla \Rm \rangle \leq \left( (8 +4 \sqrt{n})|\Rm| - \lambda_t \right) |\nabla \Rm|^2.
	\]
	Thus, we get that
	\[ \frac{d}{dt} |\nabla \Rm|^2 \leq \Delta |\nabla \Rm |^2 -2 | \nabla^2 \Rm |^2 + \left( 48 + 8 \sqrt{n} \right) |\Rm|  |\nabla \Rm|^2,
	\]
	which is what we wanted to prove.
\end{proof}
As we have seen in the last proof, the crucial part of the computation of $\frac{d}{dt} |\nabla^k\Rm|^2$ is an estimate for $\langle [\nabla^k, \Delta ] \Rm, \nabla^k \Rm \rangle$. This will be done in a recursive way. At first, we will derive a formula for the commutator for a general tensor in normal coordinates. That turns out being sufficient in order to obtain an estimate for the commutator $[\nabla^k, \Delta] \Rm.$
\begin{lem} \label{commutatortwoandthree}
	Let $(M^n,g)$ be a Riemannian manifold and $T \in \mathcal{T}_p^0(M)$ be a $(p,0)$-tensor field on $M$. If we choose choose normal coordinates $(x_i)_{i=1, \dots, n}$ around some point $x \in M$ the following formulas hold:
	\begingroup
	\allowdisplaybreaks
	\begin{align*}
		(i) \;	[\nabla^2_{\partial_a, \partial_b}, \Delta] T_{i_1 \dots i_p} & 
		= 2 \nabla_{\partial_{a}} \Rm_{b i_1}^{jk} \nabla_{\partial_k} T_{ji_2 \dots i_p} + \dots + 2 \nabla_{\partial_{a}} \Rm^{jk}_{bi_p} \nabla_{\partial_k} T_{i_1 \dots i_{p-1}j} \\
		& \; \; \; + 2 \Rm_{b i_1}^{jk} \nabla^2_{\partial_{a}, \partial_k} T_{ji_2 \dots i_p} + \dots + 2  \Rm^{jk}_{bi_p} \nabla^2_{\partial_{a}, \partial_k} T_{i_1 \dots i_{p-1}j} \\
		& \; \; \; + \nabla^2_{\partial_{a}, \partial_k}  \Rm_{b i_1}^{jk} T_{ji_2 \dots i_p} + \dots + \nabla^2_{\partial_{a}, \partial_k} \Rm^{jk}_{bi_p} T_{i_1 \dots i_{p-1}j}  \\
		& \; \; \; +  \nabla_{\partial_k} \Rm_{b i_1}^{jk} \nabla_{\partial_{a}} T_{ji_2 \dots i_p} + \dots + \nabla_{\partial_k} \Rm^{jk}_{bi_p} \nabla_{\partial_{a}} T_{i_1 \dots i_{p-1}j} \\
		& \; \; \; + 2 \Rm_{a i_1}^{jk} \nabla^2_{\partial_k, \partial_b} T_{ji_2 \dots i_p} + \dots + 2 \Rm^{jk}_{ai_p} \nabla^2_{\partial_k, \partial_b} T_{i_1 \dots i_{p-1}j} \\
		& \; \; \; + \nabla_{\partial_k} \Rm_{a i_1}^{jk} \nabla_{\partial_{b}} T_{ji_2 \dots i_p} + \dots + \nabla_{\partial_k} \Rm^{jk}_{ai_p} \nabla_{\partial_{b}} T_{i_1 \dots i_{p-1}j} \\
		& \; \; \; + 2 \Rm_{ab}^{jk} \nabla^2_{\partial_k, \partial_j} T_{i_1 \dots i_p} + \nabla_{\partial_k} \Rm_{ab}^{jk} \nabla_{\partial_j} T_{i_1 \dots i_p} \\
		& \; \; \; - \nabla_{\partial_a} \Ric_b^j \nabla_{\partial_j} T_{i_1 \dots i_p} - \Ric_b^j \nabla^2_{\partial_a, \partial_j} T_{i_1 \dots i_p} - \Ric_a^j \nabla^2_{\partial_j, \partial_b} T_{i_1 \dots i_p}
	\end{align*}
	\endgroup
	\; \; \; at $x \in M$.
	\begingroup
	\allowdisplaybreaks
	\begin{align*}
		(ii) \; [\nabla^3_{\partial_a,\partial_b,\partial_c}, & \Delta] T_{i_1 \dots i_p} = 2 \nabla^2_{\partial_a, \partial_b} \Rm_{ci_1}^{jk} \nabla_{\partial_k} T_{ji_2 \dots i_p} + \dots + 2 \nabla^2_{\partial_a, \partial_b} \Rm_{ci_p}^{jk} \nabla_{\partial_k} T_{i_1 \dots i_{p-1}j} \\
		& \; \; \; +2 \nabla_{\partial_b} \Rm_{ci_1}^{jk} \nabla^2_{\partial_a, \partial_k} T_{j i_2 \dots i_p} + \dots + 2 \nabla_{\partial_b} \Rm_{ci_p}^{jk} \nabla^2_{\partial_a, \partial_k} T_{i_1 \dots i_{p-1}j} \\
		& \; \; \; +2 \nabla_{\partial_a} \Rm_{ci_1}^{jk} \nabla^2_{\partial_b, \partial_k} T_{ji_2 \dots i_p} + \dots + 2 \nabla_{\partial_a} \Rm_{ci_p}^{jk} \nabla^2_{\partial_b, \partial_k} T_{i_1 \dots i_{p-1}j} \\
		& \; \; \; +2 \Rm_{ci_1}^{jk} \nabla^3_{\partial_a, \partial_b, \partial_k} T_{ji_2 \dots i_p} + \dots 2 \Rm_{ci_p}^{jk} \nabla^3_{\partial_a, \partial_b, \partial_k} T_{i_1 \dots i_{p-1}j} \\
		& \; \; \; +\nabla^3_{\partial_a, \partial_b, \partial_k} \Rm_{ci_1}^{jk} T_{ji_2 \dots i_p} + \dots + \nabla^3_{\partial_a, \partial_b, \partial_k} \Rm_{ci_p}^{jk} T_{i_1 \dots i_{p-1}j} \\
		& \; \; \; +\nabla^2_{\partial_b, \partial_k} \Rm_{ci_1}^{jk} \nabla_{\partial_a} T_{ji_2 \dots i_p} + \dots + \nabla^2_{\partial_b, \partial_k} \Rm_{ci_p}^{jk} \nabla_{\partial_a} T_{i_1 \dots i_{p-1}j} \\
		& \; \; \; + \nabla^2_{\partial_a, \partial_k} \Rm_{ci_1}^{jk} \nabla_{\partial_b} T_{ji_2 \dots i_p} + \dots + \nabla^2_{\partial_a, \partial_k} \Rm_{ci_p}^{jk} \nabla_{\partial_b} T_{i_1 \dots i_{p-1}j} \\
		& \; \; \; + \nabla_{\partial_k} \Rm_{ci_1}^{jk} \nabla^2_{\partial_a, \partial_b} T_{ji_2 \dots i_p} + \dots + \nabla_{\partial_k} \Rm_{ci_p}^{jk} \nabla^2_{\partial_a, \partial_b} T_{i_1 \dots i_{p-1}j} \\
		& \; \; \; + 2 \nabla_{\partial_a} \Rm_{bi_1}^{jk} \nabla^2_{\partial_k, \partial_c} T_{ji_2 \dots i_p} + \dots + 2 \nabla_{\partial_a} \Rm_{bi_p}^{jk} \nabla^2_{\partial_k, \partial_c} T_{i_1 \dots i_{p-1}j} \\
		& \; \; \; + 2 \Rm_{bi_1}^{jk} \nabla^3_{\partial_a \partial_k, \partial_c} T_{ji_2 \dots i_p} + \dots + 2 \Rm_{bi_p}^{jk} \nabla^3_{\partial_a, \partial_k, \partial_c} T_{i_1 \dots i_{p-1}j} \\
		& \; \; \; + \nabla^2_{\partial_a, \partial_k} \Rm_{bi_1}^{jk} \nabla_{\partial_c} T_{ji_2 \dots i_p} + \dots + \nabla^2_{\partial_a, \partial_k} \Rm_{bi_p}^{jk} \nabla_{\partial_c} T_{i_1 \dots i_{p-1}j} \\
		& \; \; \; + \nabla_{\partial_k} \Rm_{bi_1}^{jk} \nabla^2_{\partial_a, \partial_c} T_{ji_2 \dots i_p} + \dots + \nabla_{\partial_k} \Rm_{bi_p}^{jk} \nabla^2_{\partial_a, \partial_c} T_{i_1 \dots i_{p-1}j} \\
		& \; \; \; + 2 \Rm_{ai_1}^{jk} \nabla^3_{\partial_k, \partial_b, \partial_c} T_{ji_2 \dots i_p} + \dots + 2 \Rm_{ai_p}^{jk} \nabla^3_{\partial_k, \partial_b, \partial_c} T_{i_1 \dots i_{p-1}j} \\
		& \; \; \; + \nabla_{\partial_k} \Rm_{ai_1}^{jk} \nabla^2_{\partial_b, \partial_c} T_{ji_2 \dots i_p} + \dots + \nabla_{\partial_k} \Rm_{ai_p}^{jk} \nabla^2_{\partial_b, \partial_c} T_{i_1 \dots i_{p-1}j} \\
		& \; \; \; + 2 \nabla_{\partial_a} \Rm_{bc}^{jk} \nabla^2_{\partial_k, \partial_j} T_{i_1 \dots i_p} + 2 \Rm_{bc}^{jk} \nabla^3_{\partial_a, \partial_k, \partial_j}T_{i_1 \dots i_p} \\
		& \; \; \; + \nabla^2_{\partial_a, \partial_k} \Rm_{bc}^{jk} \nabla_{\partial_j} T_{i_1 \dots i_p} + \nabla_{\partial_k} \Rm_{bc}^{jk} \nabla^2_{\partial_a, \partial_j} T_{i_1 \dots i_p} \\ 
		& \; \; \; +2 \Rm_{ab}^{jk} \nabla^3_{\partial_k, \partial_j, \partial_c} T_{i_1 \dots i_p} + 2 \Rm_{ac}^{jk} \nabla^3_{\partial_k, \partial_b, \partial_j} T_{i_1 \dots i_p} \\
		& \; \; \; + \nabla_{\partial_k} \Rm_{ab}^{jk} \nabla^2_{\partial_j, \partial_c} T_{i_1 \dots i_p} + \nabla_{\partial_k} \Rm_{ac}^{jk} \nabla^2_{\partial_b, \partial_j} T_{i_1 \dots i_p} \\
		& \; \; \; - \nabla^2_{\partial_a, \partial_b} \Ric_c^j \nabla_{\partial_j}T_{i_1 \dots i_p} - \nabla_{\partial_b} \Ric_c^j  \nabla^2_{\partial_a, \partial_j} T_{i_1 \dots i_p} \\
		& \; \; \; - \nabla_{\partial_a} \Ric_c^j \nabla^2_{\partial_b, \partial_j} T_{i_1 \dots i_p} - \Ric_c^j \nabla^3_{\partial_a, \partial_b, \partial_j} T_{i_1 \dots i_p} \\
		& \; \; \; - \nabla_{\partial_a} \Ric_b^j \nabla^2_{\partial_j, \partial_c} T_{i_1 \dots i_p} - \Ric_b^j \nabla^3_{\partial_a, \partial_j, \partial_c} T_{i_1 \dots i_p} \\
		&\; \; \; - \Ric_a^j \nabla^3_{\partial_j, \partial_b, \partial_c} T_{i_1 \dots i_p}
	\end{align*}
	at $x \in M$.
	\endgroup
\end{lem}
\begin{proof}
	(i): Since $g_{ij}(x) = \delta_{ij}$ and the Christoffel symbols $\Gamma_{ij}^k$ vanish at $x \in M$, the second covariant derivative simplifies as
	\[
	\nabla^2_{\partial_a, \partial_b} T (x) = \nabla_{\partial_a} (\nabla_{\partial_b}T)(x) - \nabla_{\nabla_{\partial_a}\partial_b}T(x) = \nabla_{\partial_a}(\nabla_{\partial_b}T)(x),
	\]
	because $\nabla_{\nabla_{\partial_a}\partial_b} T(x)$ just depends on $T$ and the value of $\nabla_{\partial_a}\partial_b$ at $x \in M$ that vanishes. This implies
	\begin{align*} [ \nabla^2_{\partial_a, \partial_b}, \Delta] T_{i_1 \dots i_p} & =  \nabla_{\partial_{a}}( \nabla_{\partial_{b}} \Delta T)_{i_1 \dots i_p} - \Delta \nabla_{\partial_{a}}(\nabla_{\partial_{b}}T)_{i_1 \dots i_p} \\
		& = \nabla_{\partial_{a}} [ \nabla_{\partial_{b}}, \Delta ] T_{i_1 \dots i_p} + [\nabla_{\partial_{a}}, \Delta] \nabla_{\partial_{b}} T_{i_1 \dots i_p} 
	\end{align*}
	at $x \in M$. Now we use Lemma \ref{firstcommutator} (iii) in order to see that
	\begin{align*}
		[\nabla^2_{\partial_a, \partial_b}, \Delta]T_{i_1 \dots i_p} & 
		= 2 \nabla_{\partial_{a}} \left( \Rm_{b i_1}^{jk} \nabla_{\partial_k} T_{ji_2 \dots i_p} + \dots +  \Rm^{jk}_{bi_p} \nabla_{\partial_k} T_{i_1 \dots i_{p-1}j} \right) \\
		& \; \; \; + \nabla_{\partial_{a}} \left( \nabla_{\partial_k} \Rm_{b i_1}^{jk} T_{ji_2 \dots i_p} + \dots + \nabla_{\partial_k} \Rm^{jk}_{bi_p} T_{i_1 \dots i_{p-1}j} \right) \\
		& \; \; \; - \nabla_{\partial_a} \left( \Ric_b^j \nabla_{\partial_j} T_{i_1 \dots i_p} \right) + 2 \Rm_{ab}^{jk} \nabla_{\partial_k} \nabla_{\partial_j} T_{i_1 \dots i_p} \\
		& \; \; \; + 2 \Rm_{a i_1}^{jk} \nabla_{\partial_k} \nabla_{\partial_{b}} T_{ji_2 \dots i_p} + \dots + 2 \Rm^{jk}_{ai_p} \nabla_{\partial_k} \nabla_{\partial_{b}} T_{i_1 \dots i_{p-1}j} \\
		& \; \; \; + \nabla_{\partial_k} \Rm_{ab}^{jk} \nabla_{\partial_j} T_{i_1 \dots i_p} \\
		& \; \; \; + \nabla_{\partial_k} \Rm_{a i_1}^{jk} \nabla_{\partial_{b}} T_{ji_2 \dots i_p} + \dots + \nabla_{\partial_k} \Rm^{jk}_{ai_p} \nabla_{\partial_{b}} T_{i_1 \dots i_{p-1}j} \\
		& \; \; \; - \Ric_a^j \nabla_{\partial_j} \nabla_{\partial_b} T_{i_1 \dots i_p} \\
		& = 2 \nabla_{\partial_{a}} \Rm_{b i_1}^{jk} \nabla_{\partial_k} T_{ji_2 \dots i_p} + \dots + 2 \nabla_{\partial_{a}} \Rm^{jk}_{bi_p} \nabla_{\partial_k} T_{i_1 \dots i_{p-1}j} \\
		& \; \; \; + 2 \Rm_{b i_1}^{jk} \nabla_{\partial_{a}} \nabla_{\partial_k} T_{ji_2 \dots i_p} + \dots + 2  \Rm^{jk}_{bi_p} \nabla_{\partial_{a}} \nabla_{\partial_k} T_{i_1 \dots i_{p-1}j} \\
		& \; \; \; + \nabla_{\partial_{a}}  \nabla_{\partial_k} \Rm_{b i_1}^{jk} T_{ji_2 \dots i_p} + \dots + \nabla_{\partial_{a}} \nabla_{\partial_k} \Rm^{jk}_{bi_p} T_{i_1 \dots i_{p-1}j}  \\
		& \; \; \; +  \nabla_{\partial_k} \Rm_{b i_1}^{jk} \nabla_{\partial_{a}} T_{ji_2 \dots i_p} + \dots + \nabla_{\partial_k} \Rm^{jk}_{bi_p} \nabla_{\partial_{a}} T_{i_1 \dots i_{p-1}j} \\
		& \; \; \; + 2 \Rm_{a i_1}^{jk} \nabla_{\partial_k} \nabla_{\partial_{b}} T_{ji_2 \dots i_p} + \dots + 2 \Rm^{jk}_{ai_p} \nabla_{\partial_k} \nabla_{\partial_{b}} T_{i_1 \dots i_{p-1}j} \\
		& \; \; \; + \nabla_{\partial_k} \Rm_{a i_1}^{jk} \nabla_{\partial_{b}} T_{ji_2 \dots i_p} + \dots + \nabla_{\partial_k} \Rm^{jk}_{ai_p} \nabla_{\partial_{b}} T_{i_1 \dots i_{p-1}j} \\
		& \; \; \; + 2 \Rm_{ab}^{jk} \nabla_{\partial_k} \nabla_{\partial_j} T_{i_1 \dots i_p} + \nabla_{\partial_k} \Rm_{ab}^{jk} \nabla_{\partial_j} T_{i_1 \dots i_p} \\
		& \; \; \; - \nabla_{\partial_a} \Ric_b^j \nabla_{\partial_j} T_{i_1 \dots i_p} - \Ric_b^j \nabla_{\partial_a} \nabla_{\partial_j} T_{i_1 \dots i_p} - \Ric_a^j \nabla_{\partial_j} \nabla_{\partial_b} T_{i_1 \dots i_p}
	\end{align*}
	at $x \in M$, which is exactly what we have stated. \\
	(ii): The proof is exactly the same as in (i), but we use here, since we are on normal coordinates, that
	\[ \nabla^3_{\partial_a, \partial_b, \partial_c}T(x) = \nabla_{\partial_a}(\nabla^2_{\partial_b, \partial_c}T).
	\]
	Thus we obtain
	\begin{align*}
		[\nabla^3_{\partial_a, \partial_b, \partial_c}, \Delta]T_{i_1 \dots i_p} & = \nabla_{\partial_a}\left( \nabla^2_{\partial_b, \partial_c} \Delta T \right)_{i_1 \dots i_p} - \Delta \nabla_{\partial_a} \left( \nabla^2_{\partial_b, \partial_c}T \right)_{i_1 \dots i_p} \\
		& = \nabla_{\partial_a} [\nabla^2_{\partial_b, \partial_c}, \Delta ]T_{i_1 \dots i_p} + [\nabla_{\partial_a}, \Delta]\nabla^2_{\partial_b, \partial_c}T_{i_1 \dots i_p}.
	\end{align*}
	Now we just use (i) and Lemma \ref{firstcommutator} (iii) again to obtain the result in a rather lengthly but completely analogous computation already performed in (i).
\end{proof}
\newpage
As a direct consequence as in the proof of Proposition \ref{firstderivative} we obtain the following
\begin{lem} \label{secondandthirdcommutatorestimate}
	Let $(M^n,g)$ be an Einstein manifold with $\Ric = \lambda g$. Then,
	\begin{align*}  \langle [\nabla^2, \Delta] \Rm, \nabla^2\Rm \rangle & \leq \left( (18+ 4\sqrt{n}) |\Rm| - 2 \lambda \right) |\nabla^2 \Rm |^2 \\ & \; \; \; + (8+ 9 \sqrt{n}) |\nabla \Rm |^2 | \nabla^2 \Rm |,
	\end{align*}
	\begin{align*}
		\; \; \; \;  \langle [\nabla^3, \Delta] \Rm, \nabla^3\Rm \rangle & \leq \left( (30+4 \sqrt{n})|\Rm| -3 \lambda \right) |\nabla^3 \Rm|^2 \\
		& \; \; \; + (34+ 28 \sqrt{n}) |\nabla \Rm | |\nabla^2 \Rm | | \nabla^3 \Rm|.
	\end{align*}
\end{lem}
With that in mind, we are now able to compute quantitative evolution equations for $| \nabla^k \Rm|^2$ for $k=2,3$.
\begin{prop}[Evolution equation for the second covariant derivative] \label{secondandthirdderivative}
	Let $(M^n,g_0)$ be an Einstein manifold with $\Ric = \lambda g_0$, where $\lambda > 0$. Let $(M,g(t))_{t \in [0,T)}$ with $T = \frac{1}{2 \lambda}$ be a Ricci flow on $M$ with $g(0)=g_0$. Then the norm of the second covariant derivative of the curvature tensor $|\nabla^2\Rm|^2$ evolves as
	\begin{align*}
		\frac{d}{dt} |\nabla^2 \Rm|^2 \leq \Delta|\nabla^2 \Rm|^2 - 2 | \nabla^3 \Rm |^2 &+ (68+8 \sqrt{n})|\nabla^2 \Rm |^2 |\Rm| \\
		& + (48+ 18 \sqrt{n}) |\nabla^2 \Rm | | \nabla \Rm |^2.
	\end{align*}
	Furthermore, the norm of the third covariant derivative of the curvature tensor $|\nabla^3 \Rm |^2$ evolves as
	\begin{align*}
		\frac{d}{dt} |\nabla^3 \Rm |^2 \leq \Delta |\nabla^3 \Rm |^2 - 2 | \nabla^4 \Rm |^2 & + (92+ 8 \sqrt{n}) |\nabla^3 \Rm |^2 |\Rm| 
		\\ & + (164+56 \sqrt{n}) |\nabla^3 \Rm | |\nabla^2 \Rm | \nabla \Rm |.
	\end{align*}
\end{prop}
Note that this result coincides with the non quantitative pendant in the standard literature, see \cite{Chow2007TheRF}.
\begin{proof}
	This is more or less a straightforward computation using Lemma \ref{secondandthirdcommutatorestimate}. For the first part we obtain, as in Proposition $\ref{firstderivative}$,
	\begin{align*}
		\frac{d}{dt} |\nabla^2 \Rm |^2 & = 2 \langle \nabla^2 \frac{d}{dt} \Rm , \nabla^2 \Rm \rangle + 12 \lambda_t |\nabla^2 \Rm |^2 \\
		& = 2 \langle \Delta \nabla^2 \Rm , \nabla^2 \Rm \rangle + 2 \langle [\nabla^2, \Delta] \Rm, \nabla^2 \Rm \rangle \\
		& \; \; \; + 4 \langle \nabla^2 \mathcal{B}, \nabla^2 \Rm \rangle + 4 \lambda_t |\nabla^2 \Rm |^2,
	\end{align*}
	where $\mathcal{B}_{ijkl} = B_{ijkl} - B_{ijlk} + B_{ikjl} - B_{iljk}$ and $B_{ijkl} = - \Rm_{pij}^q \Rm_{qlk}^p$. We use the product rule for the covariant derivative and the Cauchy-Schwarz inequality in order to get that
	\[ \langle \nabla^2 \mathcal{B}, \nabla^2\Rm \rangle \leq 8 (|\nabla^2 \Rm |^2 |\Rm| + |\nabla^2 \Rm| |\nabla \Rm|^2 ).
	\]
	Thus the Proposition ensues from the previous Lemma. For the second inequality note that
	\[ \langle \nabla^3 \mathcal{B}, \nabla^3 \Rm \rangle \leq 8 |\nabla^3 \Rm |^2 | \Rm | + 24 |\nabla^3 \Rm| | \nabla^2 \Rm | |\nabla \Rm |.
	\]
	The rest is completely analogous.
\end{proof}

\section{A priori derivative estimates for Einstein manifolds with bounded curvature operator}
This section is devoted to the proof of Theorem C. Using the previous sections, this is basically an application of the proof of the Shi estimates. Nevertheless, we will include the proof, not only for the readers conviction but also to make sure that we do not lose any of the dreadful constants appearing here.
\begin{thm}\label{quashi}
	Let $K, \lambda >0$ be constants. Then for any $n$-dimensional Einstein manifold $(M^n,g)$ with $\Ric_g = \lambda g$ and $|\Rm| \leq K$ we have
	\begin{align*}
		\text{(i)} \; &\; 	|\nabla \Rm |_g \leq (2K-\lambda)^{3/2} \left(  33\cdot \left(\tfrac{25}{4} + \sqrt{n}\right) \right)^{1/2} \\
		\text{(ii})	\; & |\nabla^2 \Rm |_g \leq (2K-\lambda)^2 \cdot \Big[33(35 + 4 \sqrt{n})(1+ (12+ 2 \sqrt{n})(25+4 \sqrt{n})) \\ & \hspace{4.3cm} + \frac{33^2(24+9 \sqrt{n})^2(12.5+2 \sqrt{n})^2}{272+32 \sqrt{n}}\Big]^{1/2} \\
		\text{(iii)} \;	& |\nabla^3 \Rm |_g \leq (2K-\lambda)^{5/2} \cdot \Bigg[ 33(95+8 \sqrt{n})(24 +2 \sqrt{n})(25+4 \sqrt{n}) \\ 
		& \; \; \; + \frac{1089(35+4 \sqrt{n})(1+(12+2 \sqrt{n})(25+4 \sqrt{n}))(25+4 \sqrt{n}) (41+ 14 \sqrt{n})^2}{92 + 8 \sqrt{n}} \\
		& \; \; \; + \frac{33^2(25+4 \sqrt{n})(24+9 \sqrt{n})^2(12.5+2 \sqrt{n})^2(41+ 14 \sqrt{n})^2}{(92 + 8 \sqrt{n})(272+ 32 \sqrt{n})} \\
		& \; \; \; + 66(95 + 8 \sqrt{n})(34 + 4 \sqrt{n})(35+4 \sqrt{n})(1+(12+2 \sqrt{n})(25+4 \sqrt{n})) \\
		& \; \; \; + \frac{1089(95+8 \sqrt{n})(24+9 \sqrt{n})^2(12.5+2 \sqrt{n})^2(69+8 \sqrt{n})}{272 + 32 \sqrt{n}} + 66(95+ 8 \sqrt{n}) \Bigg]^{1/2}
	\end{align*}
\end{thm}
\begin{rem}
	\begin{itemize}
		\item[(1)] It is basically clear that Theorem C can be deduced by plugging in the desired values for $n$.
		\item[(2)]The numbers presented in the Theorem are not that important. It is more important that the numbers are computable and we can use them as they are presented in the theorem. Moreover, we do not claim that the bounds are sharp. During the whole process of computation we often used the Cauchy Schwarz inequality and some basic estimates for tensors, which, in general, will not give optimal bounds. 
		\item[(3)] Using the inequality $\lambda \leq K$ we are clearly able to obtain an inequality for any Einstein manifold with bounded curvature operator. This is not surprising. Furthermore, the decomposition of the curvature operator at $p \in M$ in $\Rm_p = \left(\Rm_{I}\right)_p + \left(\Rm_W \right)_p$ yields that, for Einstein manifold with bounded Weyl curvature operator, the estimates above can be described only by the Einstein constant $\lambda$. 
		\item[(4)] For abbreviation, we will denote for $i=1,2,3$ the constants $\textbf{C}_i(n)$ by the constants appearing in the theorem, i.e. 
		\[ |\nabla^i \Rm |_g \leq (2K- \lambda)^{1+i/2} \textbf{C}_i(n).
		\]
		\item[(5)] Again, we want to make sure that the reader is not confused by the fact that the norm on the space of curvature operators $S_B^2(\son)$ does not coincide with the 2-norm on the bundle $T_4^0(M)$. Nevertheless, if $R \in T_4^0(\R^n)$ is a tensor satisfying the identities (\ref{eq:identitesalgeb}) with corresponding curvature operator $\mathcal{R} \in S_B^2(\son)$, i.e. $R_{e_ie_je_ke_l} = \mathcal{R}_{e_i \wedge e_j, e_k \wedge e_l}$, we have $2||\mathcal{R}|| = ||R||$. Hence, if $(M,g)$ is a Riemannian manifold with bounded curvature operator, we can also use the theorem, but we have to make sure to bring in this factor of $2$ in order to get the correct estimates.
	\end{itemize}
\end{rem}
We now turn to the proof of the theorem.
\begin{proof}
	Let $g(t)= (1-2\lambda t)g$ be the maximal solution to the Ricci flow on $[0,T)$ for $T = \frac{1}{2 \lambda}$ with $g(0) = g$. Then the metric $g(t)$ is just a rescaled version of $g$. So we obtain
	\begin{align*}
		|\Rm|_{g(t)} = \frac{1}{1-2 \lambda t} | \Rm |_g \leq 2K,
	\end{align*}
	for $t \leq \frac{1}{4 K}$. \\
	\textbf{First estimate:} Consider
	\[ G_1(t) = t |\nabla \Rm |^2 + \alpha | \Rm|^2,
	\]
	where $\alpha = 12.5 + 2 \sqrt{n}$ for $t \leq \frac{1}{4K}$. The effect of this constant will be clear in the next computation. We are now about to use the maximum principle and compute, using Proposition \ref{firstderivative}, that
	\begin{align*}
		G_1'(t) & \leq \Delta G_1(t) + \left( (48+8 \sqrt{n})t |\Rm| + 1 - 2 \alpha \right) |\nabla \Rm |^2 + 16 \alpha |\Rm|^3 \\
		&	\leq \Delta G_1(t) + (25 + 4 \sqrt{n}-2 \alpha) |\nabla \Rm |^2 + 16 \alpha |\Rm|^3 \\
		& 	= \Delta G_1(t) + 16 \alpha |\Rm|^3 \\
		& \leq \Delta G_1(t) + 128 \alpha K^3,
	\end{align*}
	because $t |\Rm | \leq 2Kt \leq \frac{1}{2}$. Since $G_1(0) \leq \alpha K^2$, we get
	\[G_1(t) \leq \alpha K^2 +  128 \alpha K^3 t \leq  33 \alpha K^2. \]
	Reordering the terms, we obtain
	\begin{align} \label{shi1} |\nabla \Rm |_{g(t)} \leq \frac{\sqrt{33 \alpha} K}{t^{1/2}} = \frac{(412.5+66 \sqrt{n})^{1/2}K}{t^{1/2}} = \frac{\textbf{A}_1(n) K}{t^{1/2}}, \end{align}
	where $\textbf{A}_1(n) = \sqrt{33 \alpha}$. In order to obtain the desired estimate we again note that, since the metric $g(t)$ is just a rescaled version of $g$,
	\[ |\nabla \Rm |_{g(t)} = \frac{1}{(1-2 \lambda t)^{3/2}} |\nabla \Rm |_g
	\]
	and we finally get
	\[ 
	|\nabla \Rm |_g \leq \frac{(1-2 \lambda t )^{3/2}}{t^{1/2}} \textbf{A}_1(n) K
	\]
	for each $t \leq \frac{1}{4K}$. By plugging in $t = \frac{1}{4K}$, which is optimal, since the function is monotonically decreasing, we get
	\[ |\nabla \Rm |_g \leq (2K- \lambda)^{3/2} \sqrt{\frac{1}{2}} \textbf{A}_1(n).
	\]
	\textbf{Second estimate:}
	Consider
	\[ G_2(t) = t^2 |\nabla^2 \Rm |^2 + \beta t |\nabla \Rm |^2 + \beta |\Rm |^2
	\]
	with $\beta = 35 + 4 \sqrt{n}$.
	Before estimating the derivative again, it will be advantagous to simplify the reaction term of $\frac{d}{dt} |\nabla^2 \Rm |^2$. As in Proposition \ref{secondandthirdderivative}, together with (\ref{shi1}), we have
	\begin{align} 
		& \; \; \; \; (68 + 8 \sqrt{n}) |\Rm| |\nabla^2 \Rm|^2 + (48+ 18 \sqrt{n}) |\nabla \Rm |^2 |\nabla^2 \Rm | \nonumber \\
		& \leq K \cdot \left( (136 + 16 \sqrt{n})  |\nabla^2 \Rm|^2 + (48 + 18 \sqrt{n}) \textbf{A}_1(n)^2 | \nabla^2 \Rm | \frac{K}{t} \right) \label{square1} \\
		& \leq K \cdot \left( (272 + 32 \sqrt{n}) |\nabla^2 \Rm |^2 + \frac{ ((24 + 9 \sqrt{n})\textbf{A}_1(n)^2)^2}{68 + 8 \sqrt{n}} \frac{K^2}{t^2} \right), \nonumber
	\end{align}
	where we just completed the square and used $(a+b)^2 \leq 2(a^2+b^2)$ in the last inequality. Now we compute that
	\begingroup
	\allowdisplaybreaks
	\begin{align*}
		G_2'(t) & \leq \Delta G_2(t) + 2t |\nabla^2 \Rm |^2 + \beta |\nabla \Rm|^2 \\
		& \; \; \;	- 2t^2 |\nabla^3 \Rm|^2 - 2\beta t |\nabla^2 \Rm |^2 - 2\beta |\nabla \Rm|^2 \\ 
		& \; \; \; + K t^2 \cdot \left( (272 + 32 \sqrt{n}) |\nabla^2 \Rm |^2 + \frac{ ((24 + 9 \sqrt{n})\textbf{A}_1(n)^2)^2}{68 + 8 \sqrt{n}} \frac{K^2}{t^2} \right) \\
		& \; \; \; + \beta t (48 + 8 \sqrt{n}) |\Rm| |\nabla \Rm|^2 + 16 \beta |\Rm|^3 \\
		&\leq  \Delta G_2(t) +2t(35+ 4 \sqrt{n} -\beta) |\nabla^2 \Rm |^2 \\
		& \; \; \; + \left(\frac{((24 + 9 \sqrt{n})\textbf{A}_1(n)^2)^2}{68 + 8 \sqrt{n}}+ \beta(96+16 \sqrt{n})\textbf{A}_1(n)^2+128\beta\right)K^3 \\
		& = \Delta G_2(t) + \left(\frac{((24 + 9 \sqrt{n})\textbf{A}_1(n)^2)^2}{68 + 8 \sqrt{n}}+ \beta(96+16 \sqrt{n})\textbf{A}_1(n)^2+128\beta\right)K^3.
	\end{align*}
	\endgroup
	Since $G_2(0) \leq \beta K^2$, we obtain, using the maximum principle, that
	\begin{align*} G_2(t) & \leq \beta K^2 + \left(\frac{((24 + 9 \sqrt{n})\textbf{A}_1(n)^2)^2}{68 + 8 \sqrt{n}}+ \beta(96+16 \sqrt{n})\textbf{A}_1(n)^2+128\beta\right)K^3t \\
		& \leq \left(\left(33+ (24+4 \sqrt{n})\textbf{A}_1(n)^2\right) \beta +\frac{((24+9 \sqrt{n})\textbf{A}_1(n)^2)^2}{272+32 \sqrt{n}} \right)K^2.
	\end{align*}
	This implies
	\begin{align} \label{shi2}
		|\nabla^2 \Rm |_{g(t)} \leq \frac{\textbf{A}_2(n)K}{t},
	\end{align}
	where $\textbf{A}_2(n)$ is defined through
	\[\textbf{A}_2(n) =\left( 33 \beta + 33(24+4 \sqrt{n})(12.5 + 2 \sqrt{n}) \beta +  \tfrac{1089(24+ 9 \sqrt{n})^2(12.5+2 \sqrt{n})^2}{272+ 32 \sqrt{n}} \right)^{1/2}.
	\]	
	By rescaling we get 
	\[ |\nabla^2 \Rm |_g \leq \frac{(1-2 \lambda t)^2}{t} \textbf{A}_2(n) K
	\]
	for all $t \leq \frac{1}{4K}$. Hence, plugging in $t = \frac{1}{4K}$ gives us the desired bound. \\
	\textbf{Third estimate:}
	Consider \[G_3(t) = t^3 |\nabla^3 \Rm|^2 + \gamma t^2 |\nabla^2 \Rm|^2 + 2 \gamma t |\nabla \Rm|^2 + 2 \gamma |\Rm|^2,\]
	where $\gamma = 47.5 + 4 \sqrt{n}$. We proceed as before and simplify the reaction term of $\frac{d}{dt} |\nabla^3 \Rm |^2$. We use (\ref{shi1}) and (\ref{shi2}) to calculate that
	\begin{align*}
		& \; \; \; (92+ 8 \sqrt{n}) |\nabla^3 \Rm |^2 |\Rm| + (164+56 \sqrt{n}) |\nabla^3 \Rm | |\nabla^2 \Rm | |\nabla \Rm | \\
		& \leq K \cdot \left( (368 + 32 \sqrt{n})|\nabla^3 \Rm|^2 + \frac{\textbf{A}_1(n)^2 \textbf{A}_2(n)^2(82+ 28 \sqrt{n})^2}{92 + 8 \sqrt{n}} \frac{K^2}{t^3} \right),
	\end{align*}
	where we again completed the square and used $(a + b)^2 \leq 2(a^2 + b^2)$ afterwards. Using the second a priori estimate (\ref{shi2}) we can write the reaction term of $\frac{d}{dt} |\nabla^2 \Rm|^2$ together with (\ref{square1}) as
	\begin{align*}
		& \; \; \; (68 + 8 \sqrt{n})|\Rm||\nabla^2 \Rm|^2 + (48+18 \sqrt{n}) |\nabla \Rm|^2 |\nabla^2 \Rm| \\
		& \leq \frac{K^3}{t^2} \cdot \left( (272+32 \sqrt{n})\textbf{A}_2(n)^2+ \frac{((24+9 \sqrt{n})\textbf{A}_1(n)^2)^2}{68 + 8 \sqrt{n}} \right).
	\end{align*}
	\newpage
	We derive that
	\begingroup
	\allowdisplaybreaks
	\begin{align*}
		G_3'(t) & \leq \Delta G_3(t) + 3t^2 |\nabla^3 \Rm|^2 + 2 \gamma t |\nabla^2 \Rm|^2 + 2 \gamma |\nabla \Rm|^2 \\
		& \; \; \; - 2t^3 |\nabla^4 \Rm|^2 - 2 \gamma t^2 |\nabla^3 \Rm|^2 - 4 \gamma t |\nabla^2 \Rm|^2 - 4 \gamma |\nabla \Rm|^2 \\
		& \; \; \;	+  Kt^3 \cdot \left( (368 + 32 \sqrt{n})|\nabla^3 \Rm|^2 + \frac{\textbf{A}_1(n)^2 \textbf{A}_2(n)^2(82+ 28 \sqrt{n})^2}{92 + 8 \sqrt{n}} \frac{K^2}{t^3} \right) \\
		& \; \; \; + K^3 \gamma \cdot \left( (272+32 \sqrt{n})\textbf{A}_2(n)^2+ \frac{((24+9 \sqrt{n})\textbf{A}_1(n)^2)^2}{68 + 8 \sqrt{n}} \right) \\
		& \; \; \; + 2 \gamma t (48+8 \sqrt{n}) |\Rm | |\nabla \Rm|^2 + 32 \gamma |\Rm|^3 \\
		& \leq \Delta G_3(t) + (95 + 8 \sqrt{n}-2 \gamma) t^2 |\nabla^3 \Rm|^2 + \Bigg[ \frac{\textbf{A}_1(n)^2 \textbf{A}_2(n)^2(82+ 28 \sqrt{n})^2}{92 + 8 \sqrt{n}} \\
		& \; \; \; + \gamma \cdot \bigg( (272+32 \sqrt{n})\textbf{A}_2(n)^2+ \frac{((24+9 \sqrt{n})\textbf{A}_1(n)^2)^2}{68 + 8 \sqrt{n}} \\
		& \hspace{1.2cm} + (192 + 32 \sqrt{n}) \textbf{A}_1(n)^2 + 256 \bigg) \Bigg]K^3
	\end{align*}
	\endgroup
	Using the value of $\gamma$ and the fact that $G_3(0) \leq 2 \gamma K^2$ we obtain
	\begin{align*}
		G_3(t) & \leq 2 \gamma K^2 + \Bigg[ \frac{\textbf{A}_1(n)^2 \textbf{A}_2(n)^2(82+ 28 \sqrt{n})^2}{92 + 8 \sqrt{n}} \\
		& \; \; \; + \gamma \cdot \bigg( (272+32 \sqrt{n})\textbf{A}_2(n)^2+ \frac{((24+9 \sqrt{n})\textbf{A}_1(n)^2)^2}{68 + 8 \sqrt{n}} \\
		& \hspace{1.2cm} + (192 + 32 \sqrt{n}) \textbf{A}_1(n)^2 + 256 \bigg) \Bigg]K^3t \\
		& \leq \Bigg[ \frac{\textbf{A}_1(n)^2 \textbf{A}_2(n)^2(41+ 14 \sqrt{n})^2}{92 + 8 \sqrt{n}} + (95 + 8 \sqrt{n})(34 + 4 \sqrt{n}) \textbf{A}_2(n)^2 \\
		& \; \; \; + (95+8 \sqrt{n})(24 +2 \sqrt{n})\textbf{A}_1(n)^2 + \frac{(95+8 \sqrt{n})((24+9 \sqrt{n})\textbf{A}_1(n)^2)^2}{544 + 64 \sqrt{n}} \\
		& \; \; \; + 33(95+ 8 \sqrt{n}) \Bigg]K^2
	\end{align*} \newpage
	Hence, we find that
	\[ |\nabla^3 \Rm|_{g(t)} \leq \frac{\textbf{A}_3(n)K}{t^{3/2}},
	\]
	where $\textbf{A}_3(n)$ is given by
	\begingroup
	\allowdisplaybreaks
	\begin{align*}
		\textbf{A}_3(n) =	&\Bigg[ \frac{\textbf{A}_1(n)^2 \textbf{A}_2(n)^2(41+ 14 \sqrt{n})^2}{92 + 8 \sqrt{n}} + (95 + 8 \sqrt{n})(34 + 4 \sqrt{n}) \textbf{A}_2(n)^2 \\
		& \; \; \; + (95+8 \sqrt{n})(24 +2 \sqrt{n})\textbf{A}_1(n)^2 + \frac{(95+8 \sqrt{n})((24+9 \sqrt{n})\textbf{A}_1(n)^2)^2}{544 + 64 \sqrt{n}} \\
		& \; \; \; + 33(95+ 8 \sqrt{n}) \Bigg]^{1/2}
	\end{align*}
	\endgroup
	To this end, we use again that $g(t)$ is just scaling of $g$ and obtain, as before,
	\begin{align*}
		|\nabla^3 \Rm|_g \leq (2K-\lambda)^{5/2}\sqrt{2}\textbf{A}_3(n).
	\end{align*}
\end{proof}

\chapter{The potential in dimensions below 12}
In this chapter we start with analyzing the equation
\[ \int_M |\nabla \Rm|^2 d\mu_g = 8 \int_M ||\mathcal{R}_W||^3 \left( P_{\nor}(\mathcal{R}_W(p)) - \sqrt{\frac{2(n-1)}{n}} \frac{\cos \alpha_p}{\sin \alpha_p} \right) d \mu_g(p)
\]
for any Einstein manifold $(M,g)$ with $\alpha_p = \angle(\mathcal{R}_p, \Id_n)$. Recall that the goal of this thesis is to show that any Einstein manifold with positive scalar curvature and curvature operator satisfying $\alpha_p < \alpha_n$ for any $p \in M$ has to be isometric to the round sphere, up to scaling. Clearly, the left hand side of the identity shows that there are no Einstein manifolds that satisfy 
\begin{align}\label{eq:uninteresting} P_{\nor}(\mathcal{R}_W(p)) - \sqrt{\frac{2(n-1)}{n}} \frac{\cos(\alpha_p)}{\sin(\alpha_p)} < 0 \end{align} for all $p \in M$. In contrast, this explains that manifolds, where (\ref{eq:uninteresting}) has a different sign at some points are of particular interest. In this chapter we are going to show for $n = 10, 11$ that any Einstein curvature operators $\mathcal{R}$ satisfying
\begin{align}\label{eq:interesting} P_{\nor}(\mathcal{R}_W) - \sqrt{\frac{2(n-1)}{n}} \frac{\cos(\alpha)}{\sin(\alpha)} > 0 \end{align}
has to be close to $\lambda \Id_n + W$ for a certain range of $\lambda > 0$ and $W \in \SO(n).W_{\mathbb{CP}^2}$, provided that $\angle(\mathcal{R}_{\mathbb{CP}^2}^{\crit}, \Id_n) \leq \alpha < \angle(\mathcal{R}_{\sym}^{\crit}, \Id_n)$. Recall that we denote by \[\mathcal{R}_{\mathbb{CP}^2}^{\crit} = \tfrac{1}{n-1} \sqrt{\tfrac{3}{2}} \Id_n + W_{\mathbb{CP}^2}\] and by $\mathcal{R}_{\sym}^{\crit}$ the curvature operator of $S^{\lceil n/2 \rceil} \times S^{\lfloor n/2 \rfloor}$. \\ In order to do so we will analyze the potential around $W_{\mathbb{CP}^2}$ using representation theory and Conjecture A. \\
In the end, we will be able to give an explicit description of curvature operators satisfying the property (\ref{eq:interesting}) above and prove Theorem A. \\
We would like to mention that we have benefitted from unpublished notes of Christoph B\"ohm concering the computation of the Hessian of the potential at $W_{\mathbb{CP}^2}$.
\section{Preliminaries from representation theory}
In the next few sections we will have to deal with representation theory. In order to prepare this it will be advantgeous to prove a few results that will be used over and over. The results used here are well known and can for example be found in \cite{representationtomdieckbroecker} or \cite{scott1996linear}. The general question we are adressing is the following: \\
Let $V$ be a finite dimensional real representation of a compact Lie group $G$, let $H \subset G$ be a compact Lie subgroup. Furthermore, let
\[ V = V_1 \oplus \dots \oplus V_k
\]
be an irreducible decomposition of $V$ under the action of $H$. Under which conditions on the representation is $V$ itself irreducible under the representation of $G$? \\
For the whole section we will denote by $G$ a compact Lie group, $H \subset G$ a compact Lie subgroup, by $V$ a finite dimensional real inner product space and by $\rho : G \to \text{Aut}(V)$ a representation. To begin, we start with the following
\begin{lem}\label{replemma}
	Let $V = \bigoplus_{i=1}^n V_i^{d_i}$ be an irreducible decomposition of $V$ under the action of $G$, where $\text{dim}(V_i^{d_i})=d_i$. Suppose that $V_i^{d_i} \ncong V_{j}^{d_j}$ as $G$-modules for $i \neq j$. If $W$ is an irreducible finite dimensional $G$-module and $\varphi : W \to V$ a nontrivial homomorphism of $G$-modules, then there exists $i_0 = 1, \dots, n$ such that $\text{dim}(W) = d_{i_0}$ and $\varphi(W) = V_i^{d_{i_0}}$.
\end{lem}
\begin{proof}
	Since $W$ is irreducible, the image $\varphi(W)$ is an invariant submodule of $V$, that is isomorphic to $W$. We obtain
	\[ 1 \geq  \text{dim}(\text{Hom}(W,V)) = \sum_{i=1}^n \text{dim}(\text{Hom}(W,V_i^{d_i})),
	\]
	i.e. there exists $i_0 = 1, \dots, n$ such that $\text{Hom}(W,V_{i_0}^{d_{i_0}})$ is nontrivial. By irreducibility and Schurs Lemma, we obtain that $\text{dim}(W) = \text{dim}(V_{i_0}^{d_{i_0}})$ and moreover that $\varphi(W)$ is isomorphic to $V_i^{d_i}$ as $G$-modules. Furthermore, by
	\[ \text{Hom}(W,V) = \bigoplus_{i=1}^n \text{Hom}(W,V_{i_0}^{d_{i_0}})
	\]
	the homomorphism $\varphi$ can be written as $\varphi = \sum_{i=1}^n \varphi_i$ with $\text{Im}(\varphi_i) \subset V_i^{d_i}$ but again by Schurs Lemma we obtain that $\text{Im}(\varphi_i) = \{ 0 \}$ for $i \neq i_0$ and finally that $\text{Im}(\varphi) = \text{Im}(\varphi_{i_0}) = V_{i_0}^{d_{i_0}}$.
\end{proof}
Most of the proof can be taken over to the case, that not all $V_i$ are inequivalent. But the same result is certainly not true, since $W$ can be contained non-canonically in $\bigoplus_{i=1}^m V_i$, where all $V_i$ and $W$ are pairwise equivalent as $G$-modules. We obtain the following general result.
\begin{lem}\label{generalreplemma}
	Let $V = \bigoplus_{i=1}^n W_i$ be an invariant decomposition of $V$ such that $W_i = \bigoplus_{j=1}^{k_i} V_j^i$,
	where $V_j^i$ is isomorphic to $V_k^i$ for each $j,k = 1, \dots k_i$ and each $i = 1, \dots, n$ in the way that
	\[ V = \bigoplus_{i=1}^n \bigoplus_{j=1}^{k_i} V_j^i \]
	is an irreducible decomposition of $V$ under the action of $G$. If $W$ is an irreducible representation of $G$ and $\varphi : W \to V$ a nontrivial homomorphism of $G$-module, there exists an $i_0 = 1, \dots, n$ such that $\text{dim}(W) = \text{dim}(V_j^{i_0})$ for one (then for all) $j = 1, \dots, k_{i_0}$ and $\varphi(W) \subset W_{i_0}$.
\end{lem}
The previous lemma deals with the fact that there is probably more than one possibility decomposing a finite dimensional vector space in its irreducible components. This question is also adressed in \cite[Ex. 2.8]{scott1996linear}. We now come to the original question posed at the beginning of the section.
\begin{lem}\label{representationlemma}
	Let $V = \bigoplus_{i=1}^n V_i$ be an irreducible decomposition of $V$ under the restricted action of $H$, such that $V_i \ncong V_j$ as $G$-modules for $i \neq j$. Then for any irreducible subspace $W \subset V$ under the action of $G$ there exists a set $J \subset \{ 1, \dots, n\}$ such that $W = \bigoplus_{i \in J} V_i$.
\end{lem}
\begin{proof}
	Since $W$ is $G$-invariant, it is also $H$-invariant, i.e. we can decompose it in $H$-irreducible submodules $W = \bigoplus_{i=1}^m W_i$. Now, by Lemma \ref{replemma}, the $W_i$ have to be pairwise inequivalent and moreover $W_i = V_j$ for some $j \in \{ 1, \dots, n\}$. This proves the lemma.
\end{proof}
The previous lemma gives us a precise strategy on how to find out if such a representation is irreducible under the full group, provided we can decompose it under some smaller subgroup. This will lead to an inductive scheme, as we will see later.
\begin{rem}Note that the fact, that the subrepresentations $V_i$ are pairwise inequivalent is pretty important. Consider for example the representation of $\SO(4)$ on $\Lambda^2(\R^4)$. As have we seen in Chapter 2, $\Lambda^2(\R^4)$ is not irreducible as an $\SO(4)$-module, but one can decompose $\Lambda^2(\R^4) = \Lambda^2_+(\R^4) \oplus \Lambda^2_-(\R^4)$. Note that these irreducible subspaces are equivalent as $\SO(4)$-modules. Now consider the Lie subgroup $\SO(3) \subset \SO(4)$. Then,
	\[ \Lambda^2(\R^4) = \R^3 \oplus \Lambda^2(\R^3)
	\]
	decomposes irreducibly as $\SO(3)$-modules, but clearly non of the $\Lambda^2_{\pm}(\R^4)$ coincides with $\R^3$ or $\Lambda^2(\R^3)$. The map $\varphi : \R^3 \to \Lambda^2(\R^3)$ that is induced by $\varphi(e_1) = e_2 \wedge e_3$, $\varphi(e_2) = e_3 \wedge e_1$, $\varphi(e_3) = e_1 \wedge e_2$ is an isomorphism of $\R^3$ and $\Lambda^2(\R^3)$ as $\SO(3)$-modules.
\end{rem}
Hence, in the case that there exists equivalent submodules under the corresponding action of the subgroup, we have to work around this problem.
\section[Irreducible decomposition of $\Weyl _n$ under the representation \texorpdfstring{\\}{} of $\SO(k) \times \SO(l)$]{Irreducible decomposition of $\Weyl _n$ under the representation \\ of $\SO(k) \times \SO(l)$}
The first aim of this section is to understand how the Hessian of the potential at $W_{\mathbb{CP}^2}$ looks like. In order to do so, it will be advantageous to compute the irreducible decomposition of $S^2(\son)$, while restricting the action of $\SO(n)$ to certain subgroups. More explicitely, we will restrict the representation to the stabilizer of $W_{\mathbb{CP}^2}$. The advantage of this will be clear in the next sections. The first step is the restriction to $\SO(4) \times \SO(n-4)$, since the stabilizer is contained in this group, as we will see later. We will do this in a slightly generalized version. Let $k,l \in \mathbb{N}$ such that $k+l =n$. We restrict the action of $\SO(n)$ to $\SO(k) \times \SO(l) \subset \SO(n)$, i.e. matrices of the form $\diag(A,B)$ for $A \in \SO(k)$ and $B \in \SO(l)$. During the whole section we will use the following elemantary lemma, which helps to keep the overview.
\begin{lem}\label{tensor_product_lemma}
	There is an one-to-one correspondence between $\R^k \otimes \R^l$ and self-adjoint (skew-adjoint) maps $\overline{\varphi}:\R^k \oplus \R^l \to \R^k \oplus \R^l$ with $\text{Im}(\overline{\varphi}\vert_{\R^k}) \subset \R^l$. Moreover, let $G \to \text{GL}(\mathbb{R}^k \oplus \mathbb{R}^l)$ be a representation of $\R^k \oplus \R^l$ that is invariant under the inner product on $\mathbb{R}^k \oplus \mathbb{R}^l$ such that $\R^k$ and $\R^l$ are subrepresentations. Then for the corresponding representation $G \to \text{GL}(S^2(\R^k \oplus \R^l))$ the spaces $S^2(\R^k)$, $S^2(\R^l)$ and $\R^k \otimes \R^l$ are subrepresentations with $(g.\varphi)(v) = g^{-1}.\varphi(g.v)$ for $\varphi \in \R^k \otimes \R^l$ and $v \in \R^l$.
\end{lem}
\begin{proof}
	For an element $A \in \R^k \otimes \R^l$ we consider $\left( \begin{array}{cc} 0 & A \\ A^t & 0 \end{array} \right) \in S^2(\R^k \oplus \R^l)$.  For $g \in G$ the corresponding representative $g. \in \text{GL}(\R^k \oplus \R^l)$ is of the form $g= \text{diag}(g_1,g_2)$, where $g_1 \in \text{GL}(\R^k)$, $g_2 \in \text{GL}(\R^l)$ by the invariance of $\R^k$ and $\R^l$.  Thus by definition, for $A \in \R^k \otimes \R^l$ we get that
	\begin{align*} g. \left( \begin{array}{cc} 0 & A \\ A^t & 0 \end{array} \right) \left( \begin{array}{c} v_1 \\ v_2 \end{array} \right) &= g^{-1} \left( \begin{array}{cc} 0 & A \\ A^t & 0 \end{array} \right) \left( \begin{array}{c} g_1v_1 \\ g_2 v_2 \end{array} \right) \\ & = \left( \begin{array}{cc} 0 & g_1^{-1}Ag_2\\ g_2^{-1}A^tg_1 & 0 \end{array} \right) \left( \begin{array}{c} v_1 \\ v_2 \end{array} \right). \end{align*}
	That proves the claim.
\end{proof} 
Note that there is an analogue statement for $\Lambda^2(\R^k \oplus \R^l)$. \\
Now, we come to the decomposition of $S^2(\son)$ into submodules.
Clearly, the invariant spaces described in Proposition \ref{standarddecomp} stay invariant under the restricted action but the irreducible decomposition will be finer this time. It is worth noting that there is a natural embedding $S^2(\mathfrak{so}(k)) \hookrightarrow S^2(\mathfrak{so}(n))$. In particular, it is easy to see that this embedding respects the restriction to the Weyl curvature part $\Weyl_k \hookrightarrow \Weyl_n$ and also $\Lambda^4\R^k \hookrightarrow \Lambda^4\R^n$ but it does not respect the restriction to the other submodules, since for $A \in S^2(\mathbb{R}^k)$ the map $A \wedge \id_{\mathbb{R}^n}$ is of the form
\[ A \wedge \id_{\mathbb{R}^n} = \left( \begin{array}{ccc} A \wedge \id_{\mathbb{R}^k} &  &  \\  & 0 &  \\  &  & \tfrac{1}{2} A \otimes \id_{\R^l} \end{array} \right).
\]
with respect to the decomposition $\son = \so(k) \oplus \so(l) \oplus\left( \R^k \otimes \R^l \right)$. Thus $S^2(\son)$ decomposes as
\begin{align*}
	S^2(\son)  & = S^2\left( \mathfrak{so}(k) \oplus \mathfrak{so}(l) \oplus \left(\R^k \otimes \R^l \right) \right) \\
	& = S^2(\mathfrak{so}(k)) \oplus S^2(\mathfrak{so}(l)) \oplus S^2(\R^k \otimes \R^l) \oplus \left( \mathfrak{so}(k) \otimes \mathfrak{so}(l) \right) \\ 
	& \; \; \; \oplus\left(  \mathfrak{so}(k) \otimes \left( \R^k \otimes \R^l \right) \right) \oplus \left( \mathfrak{so}(l) \otimes \left( \R^k \otimes \R^l \right) \right).
\end{align*}
In the following, we will describe the occuring spaces step by step and will decompose them if possible. 
\\
Any $B \in \SO(l)$ stabilizes $S^2(\so(k))$, so $S^2(\so(k))$ decomposes irreducibly as
\[ S^2(\so(k)) = \langle \id_{\mathbb{R}^k} \rangle \oplus S^2_0(\R^k) \oplus \Weyl_k \oplus \Lambda^4(\R^k).
\]
Similary, $S^2(\so(l))$ decomposes irreducibly as
\[ S^2(\so(l)) = \langle \id_{\R^l} \rangle \oplus S^2_0(\R^l) \oplus \Weyl_l \oplus \Lambda^4(\R^l).
\]
The space $S^2(\R^k \otimes \R^l)$ denotes the space of self adjoint maps on $\R^k \otimes \R^l \subset \so(k+l)$, as described in Lemma \ref{tensor_product_lemma}. It decomposes invariantly as 
\[ S^2(\R^k \otimes \R^l) = \left( S^2(\R^k) \otimes S^2(\R^l) \right) \oplus \left( \Lambda^2(\R^k) \otimes \Lambda^2(\R^l) \right).
\]
Here $S^2(\R^k) \otimes S^2(\R^l)$ denotes the elements of the form $A \otimes B$ for $A \in S^2(\R^k)$ and $B \in S^2(\R^l)$, where $(A \otimes B)(\varphi) v = \left( A \circ \varphi \circ B \right)(v)$ for any $\varphi \in \R^k \otimes \R^l$, identified with linear maps $\R^l \to \R^k$, and $v \in \R^l$. The space $\Lambda^2(\mathbb{R}^k) \otimes \Lambda^2(\mathbb{R}^l)$ denotes the elements of the form $\nu \otimes \eta$ for $\nu \in \Lambda^2(\mathbb{R}^k)$ and $\eta \in \Lambda^2(\R^l)$, defined in the same way as for the symmetric case. Note that $\nu \otimes \eta$ is indeed symmetric, since
\begin{align*} \langle (\nu \otimes \eta) (v \otimes x), w \otimes y \rangle & = \langle \nu(v), w \rangle \cdot \langle \eta(x), y \rangle = \left(-\langle \nu(w),v \rangle \right) \cdot \left( - \langle \eta(y),x \rangle \right) \\ & = \langle (\nu \otimes \eta)(w \otimes y), v \otimes x \rangle.
\end{align*}
Moreover these subspaces are invariant since for $A \otimes B \in S^2(\mathbb{R}^k) \otimes S^2(\R^l)$ and $g= \text{diag}(g_1,g_2) \in \SO(k) \times \SO(l)$ we have that
\begin{align} \label{uncoupled} \Ad_g^{\tr} (A \otimes B) \Ad_g = (Ad_{g_1}^{\tr}A \otimes Ad_{g_2}^{\tr}B)
\end{align}
and analogously for $\nu \otimes \eta \in \Lambda^2(\R^k) \otimes \Lambda^2(\R^l)$. \\
Clearly, $\SO(k) \times \SO(l)$ acts irreducibly on $\Lambda^2(\R^k) \otimes \Lambda^2(\R^l)$ for $k, l \neq 4$. If $k = 4$ the action of $\SO(4)$ on $\Lambda^2(\R^4)$ decomposes irreducible as $\Lambda^2(\R^4) = \Lambda^2_+(\R^4) \oplus \Lambda^2_-(\R^4)$, where $\Lambda^2_{\pm}(\R^4)$ denotes the eigenspace of the eigenvalue $\pm 1$ of the Hodge-star operator $\ast : \Lambda^2(\R^4) \to \Lambda^2(\R^4)$ corresponding to the standard orientation of $\R^4$. Note that this exactly corresponds to the decomposition of $\so(4) = \so(4)_+ \oplus \so(4)_-$, as described in Chapter 2. Hence if $k =4, l \neq 4$ we obtain $\Lambda^2(\R^k) \otimes \Lambda^2(\R^l)= \left( \Lambda^2_+(\R^k) \otimes \Lambda^2(\R^l) \right) \oplus \left( \Lambda^2_-(\R^k) \otimes \Lambda^2(\R^l)\right)$. For $l=4$ we obtain a similar decomposition. \\
For $S^2(\R^k) \otimes S^2(\R^l)$ we find that
\begin{align*} S^2(\R^k) \otimes S^2(\R^l) & =\left( \id_{\mathbb{R}^k} \oplus S^2_0(\mathbb{R}^k) \right) \otimes \left( \id_{\R^l} \oplus S^2_0(\R^l) \right) \\
	& = \id_{\R^k \otimes \R^l} \oplus \widetilde{S^2_0(\mathbb{R}^k)} \oplus \widetilde{S^2_0(\R^l)} \oplus \left( S^2_0(\R^k) \otimes S^2_0(\R^l) \right).
\end{align*}
Here we have to mention several things: The standard generator of $\id_{\R^k \otimes \R^l}$ does not correspond to the identity of $S^2(\son)$ but is of the form
\[ \id_{\mathbb{R}^k \otimes \R^l} = \left( \begin{array}{ccc} 0 &  &  \\  & 0 &  \\  &  & \id_{\R^k \otimes \R^l} \end{array} \right)
\]
with respect to the decomposition $\son = \so(k) \oplus \so(l) \oplus\left( \R^k \otimes \R^l \right)$. Furthermore, elements of $\widetilde{S^2_0(\mathbb{R}^k)}$ are of the form $A \otimes \id_{\R^l}$ for $A \in S^2_0(\R^k)$. \\
Now we come to $\so(k) \otimes \so(l)$. These denote the maps $R : \son \to \son$ of the form
\[ R = \left( \begin{array}{ccc} 0 & A & 0 \\
	A^t & 0 & 0 \\ 0 & 0 & 0 \end{array} \right),
\]
for an endomorphism $A : \so(l) \to \so(k)$. At first, note that this space does not correspond to $\Lambda^2(\R^k) \otimes \Lambda^2(\R^l) \subset S^2(\R^k \otimes \R^l)$, despite the fact that these spaces are isomorphic. The corresponding representation of $\SO(k) \times \SO(l)$ on $\so(k) \otimes \so(l)$ is given by
$gA = \Ad_{g_1}^{tr} A \Ad_{g_2}$ for $g = \diag(g_1, g_2)$. In order to see whether this representation is irreducible, it is useful regarding it in coordinates, i.e. for $A = \eta \otimes \theta$ for some $\eta \in \so(k)$ and $\theta \in \so(l)$. We find that
\[ g( \eta \otimes \theta) = \Ad_{g_1}^{tr} \eta \otimes \Ad_{g_2}^{\tr} \theta .
\]
This is clearly irreducible if $k,l \neq 4$, as we have noted before. For $k = 4$ we can get the same decomposition $\so(4) = \so(4)_+ \oplus \so(4)_-$ as before and find analogous irreducible decompositions. \\ It is left to consider $\so(k) \otimes (\R^k \otimes \R^l)$. It will be more convenient to consider this as $\Lambda^2(\R^k) \otimes (\R^k \otimes \R^l)$ or equivalently as maps $\R^k \otimes \R^l \to \Lambda^2(\R^k)$. These correspond to operators of the form
\[ R = \left( \begin{array}{ccc} 0 & 0 & A \\ 0 & 0 & 0 \\A^t & 0 & 0 \end{array} \right).
\]
Now we fix $p \in \{1, \dots, l\}$ and consider the map $e_i \mapsto A(e_i \otimes e_p)$. This reduces this case to linear maps $\R^k \to \Lambda^2(\R^k)$. We consider the map
\begin{align*} & \Phi: \Lambda^2(\R^k) \otimes \R^k \to \Lambda^2(\R^k) \otimes \R^k, \\
	&	v \wedge w \otimes x \mapsto v \wedge w \wedge x,
\end{align*}
where $v \wedge w \wedge x = \frac{1}{3}\left(v \wedge w \otimes x + x \wedge v \otimes w + w \wedge x \otimes v\right)$. $\Phi$ is a projection onto $\Lambda^3(\mathbb{R}^k)$ with $\R^k \subset \ker(\Phi)$, where $\R^k$ shall be identified with the vectors of the form
\[ \sum_{i=1}^n x \wedge e_i \otimes e_i. \]
Moreover by dimension comparison one easily checks that a basis of $\ker(\Phi)$ is given by
\begin{align*}
	& e_i \wedge e_j \otimes e_j, & \text{ for } i \neq j, \\
	& e_i \wedge e_j \otimes e_m + e_i \wedge e_m \otimes e_j, & \text{ for } i < j < m \text { or } j < m < i.
\end{align*}
If we denote the orthogonal complement of $\R^k$ in $\ker(\Phi)$ by $X_k$ we have the decomposition
\[  \Lambda^2(\R^k) \otimes \R^k = \Lambda^3(\R^k) \oplus \R^k \oplus X_k.
\]
Although it seems to be a standard result in representation theory that this decomposition is irreducible for any $k \geq 3$ with $k \neq 4,6$, we give a direct proof of this in Appendix I due to the missing of any source that is readable for the author of this thesis.
\\
We will denote the subspace of $\Lambda^2(\R^k) \otimes \R^k \otimes \R^l$ that is isomorphic to $\R^k \otimes \R^l$ by $\left(\R^k \otimes \R^l \right)_k$ and similary the subspace of $\Lambda^2(\R^l) \otimes \R^l \otimes \R^k$ that is isomorphic to $\R^l \otimes \R^k$ by $\left(\R^k \otimes \R^l \right)_l$.
Summarized we have the following decomposition
\begin{prop}\label{sok_sol_decomposition}
	Let $k,l \in \mathbb{N}$ with $k+l = n$ and $k,l \geq 5$, $k,l \neq 6$. Then the space $S^2(\son)$ decomposes as
	\begin{align} S^2(\son) = & \nonumber \langle \id_{\R^k} \rangle \oplus \langle \id_{\R^l}\rangle \oplus \langle \id_{\R^k \otimes \R^l} \rangle \oplus S^2_0(\R^k) \oplus \widetilde{S^2_0(\R^k)} \oplus S^2_0(\R^l)  \\
		& \oplus \widetilde{S^2_0(\R^l)} \oplus \left( S^2_0(\R^k) \otimes S^2_0(\R^l) \right) \oplus \left( \Lambda^2(\R^k) \otimes \Lambda^2(\R^l) \right) \label{soksoldecomp}  \\
		& \nonumber \oplus \left( \so(k) \otimes \so(l) \right) \oplus \left( \Lambda^3(\R^k) \otimes \R^l \right) \oplus \left( \Lambda^3(\R^l) \otimes \R^k \right) \\
		& \nonumber \oplus\left( \R^k \otimes \R^l \right)_k \oplus \left( \R^k \otimes \R^l \right)_l \oplus \left( X_k \otimes \R^l \right) \oplus \left( X_l \otimes \R^k \right) \\
		& \nonumber \oplus \Weyl_k \oplus \Weyl_l \oplus \Lambda^4(\R^k) \oplus \Lambda^4(\R^l) .
	\end{align}
	This decomposition is irreducible under the $\SO(k) \times \SO(l)$- action. Furthermore, 
	\begin{enumerate}
		\item[(a)] within the space $\langle \id_{\R^k} \rangle \oplus \langle \id_{\R^l} \rangle \oplus \langle \id_{\R^k \otimes \R^l} \rangle$ there is a 1-dimensional subspace, generated by $W_{S^k \times S^l}$ that belongs to $\Weyl_n$.
		\item[(b)] within the space $\left(\R^k \otimes \R^l \right)_k \oplus \left(\R^k \otimes \R^l\right)_l$ there is a $k \cdot l$-dimensional subspace, generated by the elements $[\ad_v,W_{S^k \times S^l}]$ for $v \in \R^k \otimes \R^l$ that corresponds to the tangent space of $SO(n).W_{S^k \times S^l}$ at $W_{S^k \times S^l}$.
		\item[(c)] within the space $S_0^2(\R^k) \oplus \widetilde{S^2_0(\R^k)}$ there is a subspace that is isomorphic to $S^2_0(\R^k)$ that belongs to the space $\Weyl_n$. Its orthogonal complement belongs to $S^2(\R^n)$ and is also isomorphic to $S^2_0(\R^k)$. For $S_0^2(\R^l) \oplus \widetilde{S^2_0(\R^l)}$ an analogue result holds.
		\item[(d)] within the space $\left( \so(k) \otimes \so(l) \right) \oplus \left( \Lambda^2(\R^k) \otimes \Lambda^2(\R^l) \right)$ there is a non canonically embedded subspace, isomorphic to $\Lambda^2(\R^k) \otimes \Lambda^2(\R^l)$ that is contained in $\Weyl_n$. The orthogonal complement is contained in $\Lambda^4(\R^n)$.
		\item[(e)] the spaces $\Weyl_k, \Weyl_l, X_k \otimes \R^l, X_l \otimes \R^k$ and $S^2_0(\R^k) \otimes S^2_0(\R^l)$ belong to $\Weyl_n $.
		\item[(f)] The rest of the spaces do not belong to $ \Weyl_n $.
	\end{enumerate}
	Summarized, the space of Weyl curvature operators decomposes as
	\begin{align*}
		\Weyl_n = & \langle W_{S^k \times S^l} \rangle \oplus  \Weyl_k \oplus \Weyl_l \oplus \left( \R^k \otimes \R^l \right) \oplus S^2_0(\R^k) \oplus S^2_0(\R^l) \\ & \left( S^2_0(\R^k) \otimes S^2_0(\R^l) \right) \oplus \left( \Lambda^2(\R^k) \otimes \Lambda^2(\R^l) \right) \oplus X_k \otimes \R^l \oplus X_l \otimes \R^k.
	\end{align*}
	This decomposition is orthogonal and irreducible under the $\SO(k) \times \SO(l)$-action for any $k,l \geq 5$. Furthermore, it is irreducible under the action of $O(k) \times O(l)$ for any $k, l \geq 3$.
\end{prop}
\begin{rem}
	\begin{itemize}
		\item[(i)] On the first sight these decompositions may look confusing. For example, the space $\langle \Id_n \rangle$ does not show up in the decomposition of $S^2(\son)$, although it is obviously also irreducible under the $\SO(k) \times \SO(l)$ representation. It still shows up in a non canonical way, lying in $\langle \id_{\R^k} \rangle \oplus \langle \id_{\R^l} \rangle \oplus \langle \id_{\R^k \otimes \R^l} \rangle$. The reason for this is that irreducible decompositions of representations are in general not unique, but unique up to isomorphism, cf. \cite{scott1996linear}.
		\item[(ii)] There is a similar decomposition, if $k, l \leq 4$. The only thing that one has to keep in mind here is that if $k,l=4$ there is always the decomposition $\Lambda^2(\R^4) = \Lambda^2_+(\R^4) \oplus \Lambda^2_-(\R^4)$, $\Weyl_4 = \Weyl_4^+ \oplus \Weyl_4^-$ and $X_4 = X_4^+ \oplus X_4^-$ as described in Appendix I. In low dimensions some of these spaces may just disappear, e.g. $\Weyl_k = \{ 0\}$ for $k \leq 3$. We will comment on that in the next section more explicitely.
		\item[(iii)] The decomposition above leads to a new proof of the fact that $\Weyl_n$ is $\SO(n)$-irreducible for $n \geq 5$. It is attached in Appendix II of this thesis.
	\end{itemize}
\end{rem}
\begin{proof}
	The decomposition has been shown before. What is left to show is the correspondence of the subspaces to the Weyl curvature operators. In order to show this, it is reasonable to decompose the spaces $S^2(\R^k \oplus \R^l), \Lambda^4(\R^k \oplus \R^l) \subset S^2(\son)$ with respect to the $\SO(k) \times \SO(l)$-action. First of all, we have
	\begin{align*} S^2(\R^k \oplus \R^l) & = S^2(\R^k) \oplus S^2(\R^l) \oplus \left( \R^k \otimes \R^l \right) \\ 
		& = \langle \widetilde{\id_{\R^k}} \rangle \oplus S^2_0(\R^k) \oplus \langle \widetilde{\id_{\R^l}} \rangle \oplus S^2_0(\R^l) \oplus \left( \R^k \otimes \R^l \right),
	\end{align*}
	where $\widetilde{\id_{\R^k}}$ corresponds to the map $\id_{\R^k} \wedge \id_{\R^n} =\left( \begin{array}{ccc} \Id_{\so(k)} &  &  \\  & 0 &  \\  &  & \tfrac{1}{2} \Id_{\R^k \otimes \R^l} \end{array} \right)$ and similar for $\widetilde{\id_{\R^l}}$. Furthermore, an element $A \in S^2_0(\R^k)$ corresponds to $A \wedge \id_{\R^n} \in S^2(\son)$. An straightforward computation shows that an element $C \in \R^k \otimes \R^l$ belongs to $\left( \so(k) \otimes (\R^k \otimes \R^l) \right) \oplus \left( \so(l) \otimes (\R^k \otimes \R^l) \right)$ in the decomposition (\ref{soksoldecomp}) and by invariance and dimension comparison we see that $\R^k \otimes \R^l \subset (\R^k \otimes \R^l)_k \oplus (\R^k \otimes \R^l)_l$.
	Then next, we get the invariant decomposition
	\begin{align*}
		\Lambda^4(\R^k \oplus \R^l) = & \Lambda^4(\R^k) \oplus \Lambda^4(\R^l) \oplus \left( \Lambda^3(\R^k) \otimes \R^l \right) \oplus \left( \Lambda^3(\R^l) \otimes \R^k \right) \\ &  \oplus \left( \Lambda^2(\R^k) \otimes \Lambda^2(\R^l) \right)
	\end{align*}
	Here, we directly see how all of these spaces, except for $\Lambda^2(\R^k) \otimes \Lambda^2(\R^l)$, canonically embed into the decomposition (\ref{soksoldecomp}) of $S^2(\son)$ above. The space $\Lambda^2(\R^k) \otimes \Lambda^2(\R^l)$ is generated by elements of the form
	\[ e_i \wedge e_j \otimes e_p \wedge e_q = e_i \wedge e_j \wedge e_p \wedge e_q
	\]
	for $1 \leq i < j \leq k$ and $k+1 \leq p < q \leq k+l=n$. Hence, $\Lambda^2(\R^k) \otimes \Lambda^2(\R^l)$ is non-canonically contained in $\left( \so(k) \otimes \so(l) \right) \oplus \left( \Lambda^2(\R^k) \otimes \Lambda^2(\R^l) \right)$, as we can find using the decomposition (\ref{soksoldecomp}) above. By dimension comparison we obtain the decomposition of $\Weyl_n$. It is left to prove (a), (b) and (c). The first claim is clear. In order to prove the identity for the tangent space of $\SO(n).W_{S^k \times S^l}$ at $W_{S^k \times S^l}$ we note that
	\[ T_{W_{S^k \times S^l}}\SO(n).W_{S^k \times S^l} = \{ [\ad_v, W_{S^k \times S^l}] \mid v \in \so(n) \}
	\]
	by definition. Now we may write 
	\[W_{S^k \times S^l} = \left( \begin{array}{ccc} \lambda \Id_{\so(k)} & & \\ & \mu \Id_{\so(l)} & \\ & & \eta \Id_{\R^k \otimes \R^l} \end{array} \right) \] for suitable constants $\lambda \neq \mu \neq \eta$, that we determined in Chapter 2, Section 2.3. A computation shows that for $v \in \so(k) \oplus \so(l)$ we get $[\ad_v, W_{S^k \times S^l}] = 0$. Furthermore, one calculates that 
	\[ [\ad_v, W_{S^k \times S^l}]w = \begin{cases} ( \lambda- \eta )\ad_vw ,& \text{ if } w \in \so(k), \\
		(\mu - \eta) \ad_vw, & \text { if } w \in \so(l), \end{cases}
	\]
	which certainly does not vanish. \\ Moreover, this computation also shows that $T_{W_{S^k \times S^l}}\SO(n).W_{S^k \times S^l} \cong \R^k \otimes \R^l$. In order to establish (c) we note that by the computation of $S^2(\R^k \oplus \R^l)$ there has to be a constant $c \notin \{ 0, \tfrac{1}{2} \}$ such that the map $f_c: S^2_0(\R^k) \to S^2(\son)$, given by
	\begin{align} \label{eq:S2expl}
		f_c(A) = \left( \begin{array}{ccc} A \wedge \id_{\R^k} & & \\ & 0 & \\ & & c \cdot A \otimes \id_{\R^l} \end{array} \right)
	\end{align}
	is an embedding with image in $\Weyl_n$. A straightforward computation, using that $\Ric(f_c(A)) = 0$ is sufficient to be in $\Weyl_n$, shows that $c= - \frac{k-2}{2l}$. This proves the result.
\end{proof}
\section[Irreducible decomposition of $\Weyl _n$ under the representation \texorpdfstring{\\}{} of the stabilizer of $W_{\mathbb{CP}^2}$]{Irreducible decomposition of $\Weyl _n$ under the representation of the stabilizer of $W_{\mathbb{CP}^2}$}
Our next aim will be to decompose $\Weyl_n$ even further under the representation of the stabilizer of $W_{\mathbb{CP}^2}$. We recall that the Weyl curvature operator of $\mathbb{CP}^2$ is given by
\[ W_{\mathbb{CP}^2} = \sqrt{\tfrac{1}{6}} \text{diag}(0,0,0,2,-1,-1)
\]
up to scaling and the action of $\SO(n)$. Here we use the matrix representation under the basis of $\so(4) = \syp(1)_+ \oplus \syp(1)_-$, which was discussed in Chapter 2 in detail. Recall also that $W_{\mathbb{CP}^2}$ can be seen as an $n$-dimensional curvature operator under the embedding $\Weyl_4 \hookrightarrow \Weyl_n$. With this in mind, it is clear that each element in $\SO(n-4)$ stabilize $W_{\mathbb{CP}^2}$. But it is also easy to see that not all elements of $\SO(4) \times \SO(n-4)$ stabilizes $W_{\mathbb{CP}^2}$. In order to examine the whole stabilizer we denote by $\SU(2)_{\pm} \subset \SO(4)$ the subgroup of $\SO(4)$ with Lie algebra $\su(2)_{\pm} = \syp(1)_{\pm}$. We will give a detailed description of $\SU(2)_-$. In its standard form $\SU(2)$ is given by
\[ \SU(2) = \left\{ \left( \begin{array}{cc} z & w \\
	-\overline{w} & \overline{z} 
\end{array} \right) \in \Mat_2(\mathbb{C}) \; \bigg \vert z,w \in \mathbb{C}, |z|^2+|w|^2 = 1 \right\}
\] and subsequently  $\su(2)$ is given by
\[	 \su(2) = \left\{ \left( \begin{array}{cc} ia & z \\
	-\overline{z} & -ia
\end{array} \right) \in \Mat_2(\mathbb{C}) \; \bigg \vert z \in \mathbb{C}, a \in \mathbb{R} \right\}.
\]
Since $\su(2)_- \subset \so(4)$ is in fact the matrix representation of $\su(2)$ under the usual representation of $\mathbb{C}$ in $\Mat_2(\R)$, we find that an explicit description of $\SU(2)_-$ is given by
\[ \SU(2)_- = \left. \left\{ \left( \begin{array}{cccc} 
	a & b & c & d \\
	-b & a &-d & c \\
	-c & d & a & -b \\
	-d & -c & b & a \end{array} \right)\;   \right\vert \; a,b,c,d \in \R,  a^2+b^2+c^2+d^2 = 1 \right\},
\]
which direcly can be identified with $S^3 \subset \mathbb{H}$, where $\mathbb{H}$ denote the Quaternions and 
\[ S^3 = \{ a+ib+jc+kd \in \mathbb{H} \mid a^2+b^2+c^2+d^2 = 1\}.
\]
Quaternionic multiplication is defined as usual, i.e. $i^2=j^2=k^2 = ijk= -1$. We denote the subgroup $\Spin(2) \subset \SU(2)_-$ by
\[ \Spin(2) = \{ a+ib \in \mathbb{H} \mid a^2+b^2 = 1\}
\]
and denote $\Pin(2) = \Spin(2) \cup j \cdot \Spin(2)$. It is clear that $\Pin(2)$ is a $1$-dimensional subgroup of $\SU(2)_-$ with two connected components. Moreover, $\Pin(2) \subset \SU(2)_-$ is maximal as we will show next.
\begin{lem}\label{pin2maximal}
	Let $G \subset \SU(2)_-$ be a nontrivial Lie subgroup such that $\Pin(2)$ is contained in $G$. Then $G = \Pin(2)$.
\end{lem}
\begin{proof}
	Since $\Pin(2)$ has maximal rank in $\SU(2)$ and there are no proper connected subgroups with maximal rank that contain $\Spin(2)$ (see \cite[Theorem 8.10.9]{wolf2011spaces}, we see that $G$ has to normalize its maximal torus $\Spin(2)$, i.e. 
	\[G \subset \{ ghg^{-1} \in \SU(2) \mid g \in \SU(2), h \in \Spin(2)\}. \]
	This is because all maximal tori are conjugated. But it is a straightforward calculation that this normalizer is simply given by $\Pin(2)$.
\end{proof}
Next we calculate the adjoint action of $\Spin(2)$.
\begin{lem}\label{ActionofSpin2}
	The adjoint action of $\SU(2)_-$ restricted to $\Spin(2)$ corresponds to a rotation in the $\{j_-,k_-\}$-plane. More explicitely, for $g = a+ib$ we have
	\[ \Ad_g =\left( \begin{array}{ccc} 1 & 0 & 0 \\ 0 & a^2-b^2 & -2ab \\ 0 & 2ab & a^2-b^2  \end{array} \right).
	\]
\end{lem}
\begin{proof}
	Let $g = a+ib \in \Spin(2)$. Choose $\varphi \in [0,2\pi)$ with $g = \cos(\varphi) + i \sin(\varphi)$. Then we calculate that in the basis $\{ i_-, j_-, k_- \}$ of $\syp(1)_-$
	\[ \Ad_g= \left( \begin{array}{ccc} 1 & & \\ & a^2-b^2 & -2ab \\ & 2ab & a^2-b^2 \end{array} \right) = \left( \begin{array}{ccc} 1 & & \\ & \cos(2\varphi) & \sin(-2 \varphi) \\ & \sin(2\varphi) & \cos(2\varphi) \end{array} \right).
	\]	\end{proof}
With this in mind we can compute the stabilizer of $W_{\mathbb{CP}^2}$.
\begin{lem}\label{stabilizercp2}
	The stabilizer of $W_{\mathbb{CP}^2} = \sqrt{\frac{1}{6}}\diag(0,0,0,2,-1,-1) \in \Weyl_n$ is given by $\SU(2)_+ \cdot \Pin(2) \times \SO(n-4)$.
\end{lem}
\begin{proof}
	At first, we consider the desired Weyl curvature operator as an operator of $\so(4)$ and show that its stabilizer is given by $\SU(2)_+ \cdot \Pin(2)$. For $g \in \SU(2)_+$ its adjoint action $\Ad_g : \so(4) \to \so(4)$ restricts to the identity on $\syp(1)_-$, i.e. for any $v \in \syp(1)_-$ we have
	\[ g.W_{\mathbb{CP}^2}(v) = \Ad_g^{\text{tr}} W_{\mathbb{CP}^2} \Ad_g(v) = \Ad_{g}^{\text{tr}} W_{\mathbb{CP}^2}(v) = W_{\mathbb{CP}^2}(v),
	\]
	where the last equality follows by $\text{Im}(W_{\mathbb{CP}^2}) = \syp(1)_-$. Now for any $v \in \syp(1)_+$ we have $W_{\mathbb{CP}^2}(v)=0$ by definition and since $\syp(1)_+ \subset \so(4)$ is in ideal we get $\Ad_g(v) \in \syp(1)_+$. Hence $W_{\mathbb{CP}^2}(\Ad_gv) = 0$.
	Now we check that $\Pin(2)$ stabilizes $W_{\mathbb{CP}^2}$. For abbreviation, we denote by $W$ the Weyl curvature operator $W_{\mathbb{CP}^2}$ restricted to $\syp(1)_-$. It suffices to check that $\Pin(2)$ stabilizes $W$. Indeed, for $\syp(1)_+$, the adjoint action of $\SU(2)_-$ restricts to the identity and $W_{\mathbb{CP}^2}(v) =0$ for any $v \in \syp(1)_+$. Therefore, let $g =a+ib \in \Spin(2)$. By Lemma \ref{ActionofSpin2} we find that
	\begin{align*} \text{Ad}_g^{\text{tr}} W \text{Ad}_g & = \left( \begin{array}{ccc} 1 & 0 & 0 \\ 0 & a^2-b^2 & 2ab \\ 0 & -2ab & a^2-b^2  \end{array} \right) \cdot W \cdot \left( \begin{array}{ccc} 1 & 0 & 0 \\ 0 & a^2-b^2 & -2ab \\ 0 & 2ab & a^2-b^2  \end{array} \right) \\
		& = \sqrt{\tfrac{1}{6}} \left( \begin{array}{ccc} 2 & 0 & 0 \\ 0 & -(a^2-b^2)^2-4a^2b^2 & 0 \\ 0 & 0 & -4a^2b^2 - (a^2-b^2)^2 \end{array} \right) \\
		& 	= \sqrt{\tfrac{1}{6}} \text{diag} ( 2 , -1, -1)  = W,
	\end{align*}
	since $a^2+b^2=1$. Now for $j \in \Pin(2)$ we have that $\Ad_j = \text{diag}(-1,1,-1)$. Thus we can conclude that $\Ad_j^{\text{tr}} W \Ad_j = W$. By Lemma \ref{pin2maximal} we now get $\text{Stab}(W_{\mathbb{CP}^2}) = \SU(2)_+ \cdot \Pin(2)$, if we can show that $W_{\mathbb{CP}^2}$ is not stabilized by the whole $\SU(2)_-$. This is easily done since for $h=\sqrt{\frac{1}{2}}(1+k_-) \in \SU(2)_- \setminus \Pin(2)$ we obtain that
	\[ \Ad_h^{\text{tr}}W\Ad_h = \left( \begin{array}{ccc} -1 && \\ & 2 & \\& & -1 \end{array} \right).
	\]
	Now we come to the general case $W_{\mathbb{CP}^2} \in \Weyl_n$. Since $W_{\mathbb{CP}^2}$ vanishes on \\ $\so(n-4) \oplus \left( \R^4 \otimes \R^{n-4} \right)$, we directly obtain that any $g \in \SO(n-4)$ stabilizes $W_{\mathbb{CP}^2}$. Now the result follows again directly be the classification of maximal rank subgroups, cf. \cite[Theorem 8.10.9]{wolf2011spaces}. This is because $\SO(4) \times \SO(n-4)$ has maximal rank in $\SO(n)$ and we can easily construct an element $g \in \SO(n) \setminus \SO(4) \times \SO(n-4)$ that does not stabilize $W_{\mathbb{CP}^2}$. Take for example $g \in \SO(n)$ with $ge_1 = e_n$ and $ge_i =e_i$ for $i \in \{2,3,4 \}$. Then, using the isomorphism $\Lambda^2(\R^n) \cong \son$, we obtain
	\[ gW_{\mathbb{CP}^2}(e_1 \wedge e_2,e_1 \wedge e_2) = W_{\mathbb{CP}^2}(ge_1 \wedge ge_2, ge_1 \wedge ge_2) = W_{\mathbb{CP}^2}(e_2 \wedge e_n, e_2 \wedge e_n) = 0,
	\]
	but $W_{\mathbb{CP}^2}(e_1 \wedge e_2, e_1 \wedge e_2) = \tfrac{1}{\sqrt{6}}$, as computed in Chapter 2, Section 2.4.2. The result follows now from the special case above.
\end{proof}
We now come to decomposition $\Weyl_n$ into irreducible subspaces under the representation of $\Pin(2)\cdot \SU(2)_+ \times \SO(n-4)$. By Proposition \ref{sok_sol_decomposition} we obtain an invariant decomposition
\begin{align*}
	\Weyl_n & =  \langle W_{S^4 \times S^{n-4}} \rangle \oplus  \Weyl_4 \oplus \Weyl_{n-4}\oplus T_{W_{S^4 \times S^{n-4}}}\SO(n).W_{S^4 \times S^{n-4}} \oplus S^2_0(\R^4) \\ & \oplus S^2_0(\R^{n-4}) \oplus \left( S^2_0(\R^4) \otimes S^2_0(\R^{n-4}) \right)  \oplus \left( \Lambda^2_+(\R^4) \otimes \Lambda^2(\R^{n-4}) \right) \\ & \oplus \left( \Lambda^2_-(\R^4) \otimes \Lambda^2(\R^{n-4}) \right)  \oplus \left( X_4 \otimes \R^{n-4} \right) \oplus \left( X_{n-4} \otimes \R^4 \right).
\end{align*}
for any $n \geq 4$. Note that this decomposition is not irreducible under \\ $\SO(4) \times \SO(n-4)$ and for $n \leq 7$ some of these spaces just vanish. In order to obtain an irreducible decomposition under $\Pin(2) \cdot \SU(2)_+ \times \SO(n-4)$ we start with
\[ \Weyl_4 = \Weyl_4^+ \oplus \Weyl_4^- =  \Weyl_4^+ \oplus \langle W_{\mathbb{CP}^2} \rangle \oplus \Weyl_4^{-,1} \oplus \Weyl_4^{-,2},
\]
where $\Weyl_4^{\pm}$ denote Weyl curvature operators $W : \so(4) \to \so(4)$ with kernel $\syp(1)_{\mp}$. For example, elements $W \in \Weyl_4^{+}$ are of the form	
\[ W = \left( \begin{array}{cc} A &  \\ & 0 \end{array} \right)
\]
with $A \in S^2_0(\syp(1)_+)$. Thus, we will just denote $W = A$, if it is clear which space is meant. Now 
\begin{align*}
	& \Weyl_4^{-,1} = \left. \left\{ \left( \begin{array}{ccc} 0 & a & b \\ 
		a & 0 & 0 \\
		b & 0 & 0 \end{array} \right) \in \Weyl_4^- \; \right \vert a,b \in \R \right\},\\
	& \Weyl_4^{-,2} = \left. \left\{ \left( \begin{array}{ccc} 0 & 0 & 0 \\ 
		0 & a & b \\
		0 & b & -a \end{array} \right) \in \Weyl_4^- \; \right \vert a,b \in \R \right\}.
\end{align*}
It is clear that this decomposition is invariant under the representation of \\ $\Pin(2) \cdot \SU(2)_+$. Moreover, it is easily seen that it is also irreducible. \\
We now come to the decomposition of $S^2_0(\R^4)$. Recall that any operator $W \in S^2_0(\R^4) \subset \Weyl_n$ is uniquely determined by $A \wedge \id_{\R^4} \in S_B^2(\so(4))$ for some $A \in S^2_0(\R^4)$. Following Section 2.2, $A \wedge \id_{\R^4}$ is an operator of traceless Ricci type, i.e. is in the basis $\so(4) = \syp(1)_+ \oplus \syp(1)_-$ of the form
\[ A \wedge \id_{\R^4} = \left( \begin{array}{cc} 0 & B \\ B^t & 0 \end{array} \right)
\]
for some $B \in \syp(1)_+ \otimes \syp(1)_-$. Let $g = g_1 \cdot g_2 \in \Pin(2) \cdot \SU(2)_+$. Then
\[ \Ad_g^{\tr} \left( \begin{array}{cc} 0 & B^{\tr} \\ B & 0 \end{array} \right) \Ad_g = \left( \begin{array}{cc} 0 & \Ad_{g_2}^{\tr}B \Ad_{g_1} \\
	\Ad_{g_1}^{\tr} B^{\tr} \Ad_{g_2}& 0 \end{array} \right). 
\]
By Lemma \ref{ActionofSpin2} the action of $\Pin(2)$ fixes the $i_-$-plane. Since $\SU(2)_+$ acts irreducible on $\su(2)_+ = \syp(1)_+$ we obtain an irreducible decomposition of $S^2_0(\R^4) = S^2_1 \oplus S^2_2$ with
\begin{align*} & S^2_1 = \left. \left\{ \left( \begin{array}{ccc} a & 0 & 0 \\ 
		b & 0 & 0 \\
		c & 0 & 0 \end{array} \right) \in \syp(1)_+ \otimes \syp(1)_- \; \right \vert  a,b,c \in \R \right\} \\
	& S^2_2 = \left. \left\{ \left( \begin{array}{ccc} 0 & d & m \\ 
		0 & e & p \\
		0 & f & q \end{array} \right) \in \syp(1)_+ \otimes \syp(1)_- \; \right \vert d,e,f,m,p,q \in \R \right\}.
\end{align*}
For $S^2_0(\R^4) \otimes S^2_0(\R^{n-4})$ we proceed similary. By (\ref{uncoupled}), $\SO(4) \times \SO(n-4)$ acts component wise on $S^2_0(\R^4) \otimes S^2_0(\R^{n-4})$ and $\SO(4)$ acts the same way on \\ $S^2_0(\R^4) \otimes S^2_0(\R^{n-4})$ than it acts on $S^2_0(\R^4)$. This is because the map \\ $\_ \wedge \id_{\R^4} : S^2_0(\R^4) \to S^2(\Lambda^2\R^4)$ is $\SO(4)$-equivariant. Hence we obtain an irreducible decomposition for
\[ S^2_0(\R^4) \otimes S^2_0(\R^{n-4}) = \left( S^2_1 \otimes S^2_0(\R^{n-4}) \right) \oplus \left( S^2_2 \otimes S^2_0(\R^{n-4} \right).
\]
We now come to $X_4 \otimes \R^{n-4}$. As described in Appendix I we have an irreducible decomposition of $X_4 = X_4^+ \oplus X_4^-$ under the action of $\SO(4)$. Also note that we explicitely wrote down a basis there. As $\SU(2)_+$ acts irreducibly on $\su(2)_+ = \syp(1)_+$, the space $X_4^+ \otimes \R^{n-4}$ is irreducible but $X_4^-$ will not be irreducible under the action of $\Pin(2) \cdot \SU(2)_+$, because $\Pin(2)$ stabilizes the $i_-$-axis. Consider again the basis of $X_4^-$. This time it will be advantageous to make it orthogonal. It is given by
\begin{align*}
	&\hspace{0.5cm} i_- \otimes e_2 - \tfrac{1}{2} j_- \otimes e_3 - \tfrac{1}{2} k_- \otimes e_4, \hspace{1.5cm} j_- \otimes e_3 - k_- \otimes e_4 \\
	& -i_- \otimes e_1 - \tfrac{1}{2} j_- \otimes e_4 + \tfrac{1}{2} k_- \otimes e_3, \hspace{1.5cm} j_- \otimes e_4 + k_- \otimes e_3, \\
	& -i_- \otimes e_4 + \tfrac{1}{2} j_- \otimes e_1 - \tfrac{1}{2} k_- \otimes e_2, \hspace{1.5cm}  j_- \otimes e_1 + k_- \otimes e_2, \\
	& \hspace{0.5cm} i_- \otimes e_3 + \tfrac{1}{2} j_- \otimes e_2 + \tfrac{1}{2} k_- \otimes e_1,\hspace{1.2cm} -j_- \otimes e_2 + k_- \otimes e_1.
\end{align*}
Decompose $X_4^- = X_4^{-,1} \oplus X_4^{-,2}$, where $X_4^{-,1}$ is generated by
\begin{align*}
	j_- \otimes e_3 - k_- \otimes e_4, \; \; j_- \otimes e_4 + k_- \otimes e_3, \; \; j_- \otimes e_1 + k_- \otimes e_2, \; \; -j_- \otimes e_2 + k_- \otimes e_1.
\end{align*}
and $X_4^{-,2}$ is generated by
\begin{align*}
	&\hspace{0.5cm} i_- \otimes e_2 - \tfrac{1}{2} j_- \otimes e_3 - \tfrac{1}{2} k_- \otimes e_4, \; \; -i_- \otimes e_1 - \tfrac{1}{2} j_- \otimes e_4 + \tfrac{1}{2} k_- \otimes e_3, \\
	& -i_- \otimes e_4 + \tfrac{1}{2} j_- \otimes e_1 - \tfrac{1}{2} k_- \otimes e_2, \; \; \hspace{0.35cm} i_- \otimes e_3 + \tfrac{1}{2} j_- \otimes e_2 + \tfrac{1}{2} k_- \otimes e_1.
\end{align*}
Since $\Pin(2)$ acts irreducible on the $\{ j_-, k_-\}$-plane we obtain that the action of $\Pin(2) \cdot \SU(2)_+$ on $X_4^{-,1}$ is irreducible. Moreover, it is easy to see that the action on $X_4^{-,2}$ is also irreducible, because for any $v \in X_4^{-,2}$ the space generated by \\ $\Pin(2) \cdot \SU(2)_+ \cdot v$ is at least $3$-dimensional. If there would be a $3$-dimensional invariant subspace $V \subset X_4^{-,2}$, then its orthogonal complement $V^{\perp} \subset X_4^{-,2}$ was also invariant. But this space is $1$-dimensional, which is a contradiction. \\
In the end, decompose $\Lambda^2_-(\R^4)$ as 
\[ \Lambda^2_{-,1}(\R^4) \oplus \Lambda^2_{-,2}(\R^4),
\]
where $\Lambda^2_{-,1}(\R^4) = \text{span} \{ i_-\}$, $\Lambda^2_{-,2}(\R^4) = \text{span} \{ j_-, k_-\}$.
\\
Finally, we are able to write down the decomposition of $\Weyl_n$.
\begin{prop}\label{pindecomp} Let $n \geq 9$. Then the space $\Weyl_n$ decomposes orthogonally as
	\begin{align*}
		\Weyl_n  = & \langle W_{S^4 \times S^{n-4}} \rangle \oplus \Weyl_4^+ \oplus \Weyl_4^{-,1} \oplus \Weyl_4^{-,2} \oplus \langle W_{\mathbb{CP}^2} \rangle \oplus \Weyl_{n-4} \\ 
		&\oplus T_{W_{S^4 \times S^{n-4}}}\SO(n).W_{S^4 \times S^{n-4}} \oplus S^2_1 \oplus S^2_2 \oplus S^2_0(\R^{n-4}) \\
		& \oplus \left( S^2_1 \otimes S^2_0(\R^{n-4}) \right)\oplus \left( S^2_2 \otimes S^2_0(\R^{n-4}) \right) 
		\oplus \left( \Lambda^2_+(\R^4) \otimes \Lambda^2(\R^{n-4}) \right) \\ 
		& \oplus \left( X_4^+ \otimes \R^{n-4} \right) \oplus \left( X_4^{-,1} \otimes \R^{n-4} \right) \oplus \left( X_4^{-,2} \otimes \R^{n-4} \right) \oplus \left( \R^4 \otimes X_{n-4} \right)\\
		& \oplus \left( \Lambda^2_{-,1}(\R^4) \otimes \Lambda^2(\R^{n-4}) \right) \oplus \left( \Lambda^2_{-,2}(\R^4) \otimes \Lambda^2(\R^{n-4}) \right).
	\end{align*} 
	The decomposition is irreducible under the action of $\Pin(2) \cdot \SU(2)_+ \times \SO(n-4)$.
\end{prop}
\begin{rem} For completeness we will also write down the irreducible decompositions for $n = 5,6,7,8$ under the action of $\Pin(2) \cdot \SU(2)_+ \times \SO(n-4)$. \\
	For $n=8$ we obtain 
	\begin{align*}
		\Weyl_8  = & \langle W_{S^4 \times S^{4}} \rangle \oplus \Weyl_4^+ \oplus \Weyl_4^{-,1} \oplus \Weyl_4^{-,2} \oplus \langle W_{\mathbb{CP}^2} \rangle   \oplus \Weyl_4^+\\ 
		& \oplus \Weyl_4^- \oplus T_{W_{S^4 \times S^{4}}}\SO(4).W_{S^4 \times S^{4}} \oplus S^2_1 \oplus S^2_2 \oplus S^2_0(\R^{4}) \\
		& \oplus \left( S^2_1 \otimes S^2_0(\R^{4}) \right)\oplus \left( S^2_2 \otimes S^2_0(\R^{4}) \right) 
		\oplus \left( \Lambda^2_+(\R^4) \otimes \Lambda^2(\R^{4}) \right) \\ 
		& \oplus \left( X_4^+ \otimes \R^{4} \right) \oplus \left( X_4^{-,1} \otimes \R^{4} \right) \oplus \left( X_4^{-,2} \otimes \R^{4} \right) \oplus \left( \R^4 \otimes X_{4}^+ \right)\\
		& \oplus \left( \R^4 \otimes X_{4}^- \right) \oplus \left( \Lambda^2_{-,1}(\R^4) \otimes \Lambda^2_+(\R^{4}) \right) \oplus \left( \Lambda^2_{-,1}(\R^4) \otimes \Lambda^2_-(\R^{4})  \right)  \\
		&  \oplus \left( \Lambda^2_{-,2}(\R^4) \otimes \Lambda^2_+(\R^{4}) \right) \oplus \left( \Lambda^2_{-,2}(\R^4) \otimes \Lambda^2_-(\R^{4}) \right).
	\end{align*} 
	For $n = 7$ we obtain
	\begin{align*}
		\Weyl_7  = & \langle W_{S^4 \times S^3} \rangle \oplus \Weyl_4^+ \oplus \Weyl_4^{-,1} \oplus \Weyl_4^{-,2} \oplus \langle W_{\mathbb{CP}^2} \rangle  \\ 
		&\oplus T_{W_{S^4 \times S^{3}}}\SO(7).W_{S^4 \times S^{3}} \oplus S^2_1 \oplus S^2_2 \oplus S^2_0(\R^{3}) \\
		& \oplus \left( S^2_1 \otimes S^2_0(\R^{3}) \right)\oplus \left( S^2_2 \otimes S^2_0(\R^{3}) \right) 
		\oplus \left( \Lambda^2_+(\R^4) \otimes \Lambda^2(\R^{3}) \right) \\ 
		& \oplus \left( X_4^+ \otimes \R^{3} \right) \oplus \left( X_4^{-,1} \otimes \R^{3} \right) \oplus \left( X_4^{-,2} \otimes \R^{3} \right) \oplus \left( \R^4 \otimes X_{3} \right)\\
		& \oplus \left( \Lambda^2_{-,1}(\R^4) \otimes \Lambda^2(\R^{3}) \right) \oplus \left( \Lambda^2_{-,2}(\R^4) \otimes \Lambda^2(\R^{3}) \right).
	\end{align*} 
	For $n=6$ we obtain
	\begin{align*}
		\Weyl_6  = & \langle W_{S^4 \times S^{2}} \rangle \oplus \Weyl_4^+ \oplus \Weyl_4^{-,1} \oplus \Weyl_4^{-,2} \oplus \langle W_{\mathbb{CP}^2} \rangle \\ 
		&\oplus T_{W_{S^4 \times S^{2}}}\SO(6).W_{S^4 \times S^{2}} \oplus S^2_1 \oplus S^2_2 \oplus S^2_0(\R^{2}) \\
		& \oplus \left( S^2_1 \otimes S^2_0(\R^{2}) \right)\oplus \left( S^2_2 \otimes S^2_0(\R^{2}) \right) 
		\oplus \left( \Lambda^2_+(\R^4) \otimes \Lambda^2(\R^{2}) \right) \\ 
		& \oplus \left( X_4^+ \otimes \R^{2} \right) \oplus \left( X_4^{-,1} \otimes \R^{2} \right) \oplus \left( X_4^{-,2} \otimes \R^{2} \right) \\
		& \oplus \left( \Lambda^2_{-,1}(\R^4) \otimes \Lambda^2(\R^{2}) \right) \oplus \left( \Lambda^2_{-,2}(\R^4) \otimes \Lambda^2(\R^{2}) \right).
	\end{align*} 
	Any finally, for $n=5$ most of the spaces above vanish and we obtain
	\begin{align*}
		\Weyl_5  = & \Weyl_4^+ \oplus \Weyl_4^{-,1} \oplus \Weyl_4^{-,2} \oplus \langle W_{\mathbb{CP}^2} \rangle \\ 
		& \oplus S^2_1 \oplus S^2_2  \oplus X_4^+ \oplus X_4^{-,1} \oplus X_4^{-,2}.
	\end{align*} 
\end{rem}
\section{The Hessian of the potential at $W_{\mathbb{CP}^2}$}
In this section, we are going to compute the Hessian of the potential at $W_{\mathbb{CP}^2}$. Recall that the potential is given by $P(\mathcal{R}) = \frac{1}{||\mathcal{R}||^3} \langle \mathcal{R}^2+\mathcal{R}^{\#},\mathcal{R} \rangle$ for any curvature operator $\mathcal{R} \in S_B^2(\son)$. In order to estimate the potential for any Weyl curvature operator that is close to $W_{\mathbb{CP}^2}$ we will compute the linear map
\[ F= F_{\mathbb{CP}^2}: \Weyl_n \to \Weyl_n,
\]
given by $F(W) = Q(W_{\mathbb{CP}^2},W) = \frac{1}{2} (W_{\mathbb{CP}^2} \circ W + W \circ W_{\mathbb{CP}^2}) + W_{\mathbb{CP}^2} \# W$. More explicitely, we will use the decomposition, obtained in Section 3 in order to compute all the eigenvalues of $F$. As a first result we will show that $F$ is equivariant under the action of the stabilizer at $W_{\mathbb{CP}^2}$.
\\ 
Let $g \in \text{Stab}(W_{\mathbb{CP}^2})$. Then for any $W \in \Weyl_n$, we see that
\begin{align}\label{eq:equi} F(gW) = Q(W_{\mathbb{CP}^2},gW) = Q(gW_{\mathbb{CP}^2},gW) = gQ(W_{\mathbb{CP}^2},W) = gF(W),
\end{align}
where we used that $Q: \Weyl_n \otimes \Weyl_n \to \Weyl_n$ is $\SO(n)$-equivariant. We will now justify the need of Propositition \ref{pindecomp}.
\begin{lem} \label{irreducibilitylemma} Let $(V, \langle \cdot , \cdot \rangle)$ be a real inner product space and $G \to \GL(V)$ be an irreducible representation of $V$. If $F: V \to V$ is a $G$-equivariant, selfadjoint map, then there exists $\lambda \in \mathbb{R}$ such that $F = \lambda \cdot \id_V$.
\end{lem}
\begin{proof} Since $F$ is selfadjoint, there exists a real eigenvalue $\lambda \in \mathbb{R}$. We consider the eigenspace $V_{\lambda} = \{ v \in V \mid F(v) = \lambda v \}$. Since $F$ is $G$-equivariant, we have that
	\[ F(gv) = g(Fv) = g(\lambda v) = \lambda \cdot gv
	\] 
	for any $v \in V_{\lambda}$ and $g \in G$.
	Hence $V_{\lambda}$ is a nontrivial subrepresentation. By irreducibility we obtain $V_{\lambda} = V$.
\end{proof}
Using Lemma \ref{irreducibilitylemma} and (\ref{eq:equi}), we see that any $\Pin(2) \cdot \SU(2)_+ \times \SO(n-4)$-irreducible subspace is an eigenspace of $F_{\mathbb{CP}^2}$ for a certain eigenvalue. Thus, we can just pick our favorite cuvature operator in this space and compute the image of $F$ under this element. We get the following
\begin{thm}\label{Hessian}
	For $n\geq 9$ the eigenvalues of the map $F_{W_{\mathbb{CP}^2}}: \Weyl_n \to \Weyl_n$ are given as follows:
	\begin{itemize}
		\item[(a)] $W \in \langle W_{\mathbb{CP}^2} \rangle$ is an eigenvector of eigenvalue $\sqrt{\tfrac{3}{2}}$.
		\item[(b)] $W \in \Weyl_4^{-,1} \oplus \left( X_4^{-,2} \otimes \R^{n-4} \right)$ is an eigenvector of eigenvalue $\frac{1}{2} \sqrt{\frac{3}{2}}$.
		\item[(c)] $W \in  S^2_1 \oplus \left( S^2_1 \otimes \R^{n-4} \right) \oplus \left( \Lambda^2_{-,1}(\R^4) \otimes \Lambda^2(\R^{n-4}) \right)$ is an eigenvector \\ of eigenvalue $\frac{1}{3} \sqrt{\frac{3}{2}}$.
		\item[(d)] The kernel of $F$ is given by \begin{align*} &\; \; \; \Weyl_4^+ \oplus S^2_0(\R^{n-4}) \oplus \left( \Lambda^2_+(\R^{n-4}) \otimes \Lambda^2(\R^{n-4}) \right) \oplus \Weyl_{n-4} \\ & \oplus  T_{W_{S^4 \times S^{n-4}}}\SO(n).W_{S^4 \times S^{n-4}} \oplus \left( \R^4 \otimes X_{n-4} \right) \oplus \left( X_4^+ \otimes \R^{n-4} \right). \end{align*}
		\item[(e)] $W \in S^2_2 \oplus \left( S^2_2 \otimes S^2_0(\R^{n-4}) \right) \oplus \left( \Lambda^2_{-,2}(\R^4) \otimes \Lambda^2(\R^{n-4}) \right)$ is an eigenvector \\ of eigenvalue $-\frac{1}{6} \sqrt{\frac{3}{2}}$.
		\item[(f)] $W \in X_4^{-,1} \otimes \R^{n-4}$ is an eigenvector of eigenvalue $-\frac{1}{2} \sqrt{\frac{3}{2}}$.
		\item[(g)] $W \in \Weyl_4^{-,2}$ is an eigenvector of eigenvalue $-\sqrt{\frac{3}{2}}$.
	\end{itemize}
\end{thm}
Before we start with the proof, we will state a lemma, which will be helpful during the proof of the theorem.
\begin{lem}\label{tracevanish} Let $V \subset \Weyl_n$ be a $\SU(2)_-$-invariant subspace of $\Weyl_n$. Then $\tr\left((F_{W_{\CP^2}})_{\vert V}\right)$ vanishes.
\end{lem}
\begin{proof} Let $\{ R_i \in V \mid i =1, \dots, k\}$ be an orthonormal basis of $V$. Since the $\SO(n)$-action is isometric and $V$ is $\SU(2)_-$-invariant, the set $\{ g. R_i \in V \mid i=1, \dots, k \}$ is also an orthonormal basis of $V$ for each $g \in SU(2)_-$. Since the trilinear map $\tri: \Weyl_n \times \Weyl_n \times \Weyl_n \to \R$, given by
	\[ \tri(R_1,R_2,R_3) = \langle Q(R_1,R_2), R_3 \rangle
	\] is fully symmetric, we find that
	\begin{align*}
		\tr\left((F_{W_{\CP^2}})_{\vert V}\right) & = \sum_{i=1}^k \langle F_{W_{\CP^2}} (R_i), R_i \rangle = \sum_{i=1}^k \langle Q(W_{\CP^2}, R_i), R_i \rangle \\ & = \sum_{i=1}^k\langle R_i^2+R_i^{\#}, W_{\CP^2} \rangle = \sum_{i=1}^k \langle (g. R_i)^2 + (g. R_i)^{\#}, g. W_{\CP^2} \rangle \\ & = \sum_{i=1}^k \langle F_{g. W_{\CP^2}}(g . R_i), g . R_i \rangle = \tr\left((F_{g. W_{\CP^2}})_{\vert V}\right) 
	\end{align*}
	for all $g \in \SU(2)_-$. More precisely,
	\[ \tr\left((F_{W_{\CP^2}})_{\vert V}\right) = \Biggl\langle \sum_{i=1}^k R_i^2+R_i^{\#} , W_{\CP^2} \Biggl\rangle = \Biggl\langle \sum_{i=1}^k R_i^2+R_i^{\#} , g. W_{\CP^2} \Biggl\rangle.
	\]
	We need to show that this vanishes. For abbreviation we denote by \\ $\overline{T} = \sum_{i=1}^k R_i^2 + R_i^{\#}$. Moreover, since $W_{\CP^2} \in \Weyl_4$ and $W_{\mathbb{CP}^2}(\su(2)_+) = 0$, the right hand side does only depend on $T = \pr_{\Weyl_4}(\overline{T})_{\vert \su(2)_-}$. By the definition of the scalar product we might assume that $T = \diag(a,b,c)$
	with $a+b+c=0$. Computing the trace, we get $\langle T, W_{\CP^2} \rangle = \sqrt{\frac{1}{6}} \left( 2a-b-c \right)$. We consider $g_1 = \sqrt{\frac{1}{2}} (i+j)$ and $g_2 = \sqrt{\frac{1}{2}} (1+j) $ and compute that
	\[  (g_1. W_{\CP^2})_{\vert \su(2)_-} = \sqrt{\tfrac{1}{6}} \diag \left(-1,2,-1\right)
	\]
	and
	\[  (g_2. W_{\CP^2})_{\vert \su(2)_-} = \sqrt{\tfrac{1}{6}} \left( -1,-1,2 \right)
	\]
	Thus $\langle T, g_1. W_{\CP^2} \rangle = \sqrt{\frac{1}{6}} (-a+2b-c)$ and $ \langle T, g_2. W_{\CP^2} \rangle = \sqrt{\frac{1}{6}} (-a-b+2c)$. Because of equality we easily get $a=b=c = 0 $. So the trace vanishes.
\end{proof}
\begin{proof}[Proof of Theorem \ref{Hessian}]
	(a): We already computed this in Chapter 2.4. \\
	(d): The spaces $S^2_0(\R^{n-4}), \Weyl_{n-4}, T_{W_{S_{4,n-4}}} \SO(n).W_{S_{4,n-4}}$ and $\R^4 \otimes \R^{n-4}$ are invariant under $\SO(4)$. Thus they are clearly also invariant under $\SU(2)_-$. Since the trace is the sum of the eigenvalues and all eigenvalues within the space are the same by Lemma \ref{irreducibilitylemma}, we directly get the result for these spaces with Lemma \ref{tracevanish}. For the spaces $\Weyl_4^+$, $\Lambda^2_+(\R^{n-4}) \otimes \Lambda^2(\R^{n-4})$ and $X_4^+ \otimes \R^{n-4}$ the action of $\SU(2)_-$ is trivial by construction. Hence these spaces are also invariant under the action of $\SU(2)_-$ and we obtain the result in the same way than above. \\
	(b): Consider 
	\[ W_1 = \left( \begin{array}{ccc} 0 & 1 & 0 \\
		1 & 0 & 0 \\
		0 & 0 & 0 \end{array} \right) \in \Weyl_4^{-,1}.
	\] We see that
	\[ W_1^2+ W_1^{\#} = \text{diag}(1,0,0).\]
	Since $W_1$ is an eigenvector of $F$ we get for $\lambda \in \R$ with $F(W_1) = \lambda W_1$ that
	\[ \lambda = \langle F(W_1), W_1 \rangle = \langle Q(W_1), W_{\mathbb{CP}^2} \rangle = \tfrac{2}{\sqrt{6}} = \tfrac{1}{2} \sqrt{\tfrac{3}{2}}.
	\]
	For $X_4^{-,2} \otimes \R^{n-4}$ we just consider 
	\[ W = i_- \otimes (e_2 \otimes e_5) - \tfrac{1}{2} j_- \otimes (e_3 \otimes e_5) - \tfrac{1}{2} k_- \otimes (e_4 \otimes e_5) \in X_4^{-,2}
	\]
	and are about to compute $F(W)$. Since we know that $W$ is an eigenvector of $F$ we have that, $F(W) = \alpha W$ for some $\alpha \in \R$. Hence it suffices to compute $F(W)(e_2 \otimes e_5) = \alpha i_-$ in order to determine the eigenvalue. We clearly see that
	\[ \tfrac{1}{2} \left( W \circ W_{\mathbb{CP}^2}(e_2 \otimes e_5) + W_{\mathbb{CP}^2} \circ W(e_2 \otimes e_5)\right) = \tfrac{1}{\sqrt6} i_-.
	\]
	So we know that there exists $\tilde{\alpha} \in \R$, such that $W \# W_{\mathbb{CP}^2} (e_2 \otimes e_5)= \tilde{\alpha} i_-$. Together with Lemma \ref{properties of sharp}, we obtain that
	\begin{align*} 2 \tilde{\alpha} & = \langle W \# W_{\mathbb{CP}^2}(e_2 \otimes e_5), i_- \rangle \\
		&	= \sum_{\alpha, \beta} \langle \left[W(b_{\alpha}), W_{\mathbb{CP}^2}(b_\beta)\right], e_2 \otimes e_5 \rangle \langle [b_{\alpha}, b_{\beta}], i_- \rangle \\
		& = \frac{1}{2\sqrt{6}}  \langle [e_2 \wedge e_3, e_3 \wedge e_5] + [e_2 \wedge e_4, e_4 \wedge e_5], e_2 \wedge e_5 \rangle = \tfrac{1}{\sqrt6},
	\end{align*}
	where the last equality is just Lemma \ref{formulabracket} and the orthonormal basis \\ $\left\{ b_{\alpha} \in \son \mid \alpha=1, \dots,  \tfrac{n(n-1)}{2}\right\}$ is chosen with respect to the decomposition \[\son = \syp(1)_+ \oplus \syp(1)_- \oplus \so(n-4) \oplus (\R^4 \otimes \R^{n-4}). \]
	(f): The space $X_4^- = X_4^{-,1} \oplus X_4^{-,2}$ is by construction invariant under the action of $\SU(2)_-$. The claim follows together with Lemma \ref{irreducibilitylemma}, Lemma \ref{tracevanish} and by $\text{dim}(X_4^{-,1}) = \text{dim}(X_4^{-,2})$. \\
	(g): We have that $\Weyl_4^{-} = \langle W_{\mathbb{CP}^2} \rangle \oplus \Weyl_4^{-,1} \oplus \Weyl_4^{-,2}$ and can apply the same method as in (f). \\
	(c): Consider $B = \text{diag}(1,0,0) \in S^2_1$.
	A straightforward computation shows that if $A \wedge \id_{\R^4} = \left( \begin{array}{cc} & B \\ B^t & \end{array} \right)$, then $A= \text{diag}(1,1,-1,-1)$. So we find that $B$ corresponds to
	\[ W = \left( \begin{array}{ccc} A \wedge \id_{\R^4} & & \\ & 0 & \\ & & c \cdot A \otimes \id_{\R^{n-4}} \end{array} \right) \in \Weyl_n,
	\]
	where $c = \tfrac{-2}{2(n-4)}$, as described in (\ref{eq:S2expl}). Let $\alpha$ be the eigenvalue of $F$ such that $F(W) = \alpha W$. Then we know that $\alpha = \tfrac{1}{|W|^2} \langle W^2 + W^{\#}, W_{\mathbb{CP}^2} \rangle$. We use Lemma \ref{StandardSquareandSharp} in order to see that $\langle W^2 + W^{\#}, W_{\mathbb{CP}^2} \rangle = \frac{1}{\sqrt{6}} (4(n-4)c^2+2)$ and we also get $|W|^2= 4(n-4)c^2+2$. Thus we conclude $\alpha = \frac{1}{\sqrt6} = \frac{1}{3} \sqrt{\frac{3}{2}}$. \\
	Next, we consider $B \otimes S \in S^2_1 \otimes S^2_0(\R^{n-4})$ with $S = \text{diag}(1,-1, 0, \dots, 0) \in S^2_0(\R^{n-4})$ and $B$ as above. We find that for $i=1,2$
	\[ B \otimes S (e_i \otimes e_p) = \begin{cases}
		e_i \otimes e_5, & \text{ if } p=5, \\
		-e_i \otimes e_6, & \text{ if } p=6, \\
		0, & \text{ else}
	\end{cases}
	\] and for $i = 3,4$ that
	\[ B \otimes S (e_i \otimes e_p) = \begin{cases}
		-e_i \otimes e_5, & \text{ if } p=5, \\
		e_i \otimes e_6, & \text{ if } p=6,\\
		0, & \text{ else.}
	\end{cases}
	\]
	Let $\alpha \in \R$ be the eigenvalue of $F$ such that $F(B \otimes S) = \alpha B \otimes S$. Then we obtain together with Lemma \ref{StandardSquareandSharp} that
	\[ (B \otimes S)^{\#}_{\vert \so(4)} = \text{diag}(2,-2,-2,-2,-2,2)
	\]
	in the standard basis of $\so(4)$. Recall that the representation of $W_{\mathbb{CP}^2}$ in the standard basis of $\so(4)$ is given by
	\[ W_{\mathbb{CP}^2} = \scriptsize \tfrac{1}{2\sqrt6} \left( \begin{array}{cccccc} 2 & & & & & -2 \\
		& -1 & & & -1 & \\ & &-1 & 1 & & \\ & & 1 & -1 & & \\ & -1 & & & -1 & \\ -2 & & & & & 2 \end{array} \right).
	\] Since $(B \otimes S)^2_{\vert \so(4)} = 0$, we directly get 
	\[ \langle (B \otimes S)^2 + (B \otimes S)^{\#} , W_{\mathbb{CP}^2} \rangle = \tfrac{8}{\sqrt6}.
	\]
	We compute that $|B \otimes S|^2 = 8$ to obtain the result. \\
	In the end, we consider $ R= i_- \otimes e_5 \wedge e_6 \in \Lambda^2_{-,1}(\R^{4}) \otimes \Lambda^2(\R^{n-4})$. This curvature operator satisfies
	\begin{align*} & R(e_i \otimes e_5) = \begin{cases} e_2 \otimes e_6, & \text{ if } i=1, \\
			-e_1 \otimes e_6, & \text{ if } i=2, \\
			-e_4 \otimes e_6, & \text{ if } i=3, \\
			e_3 \otimes e_6, & \text{ if } i=4. 
		\end{cases} \\
		&R(e_i \otimes e_6) = \begin{cases} -e_2 \otimes e_5, & \text{ if } i=1, \\
			e_1 \otimes e_5, & \text{ if } i=2, \\
			e_4 \otimes e_5, & \text{ if } i=3, \\
			-e_3 \otimes e_5, & \text{ if } i=4. 
		\end{cases}
	\end{align*}
	Note that it does not vanish on $\so(4) \oplus \so(n-4)$, but we are just interested in its values on $\R^4 \otimes \R^{n-4}$. Let $\alpha \in \R$ such that $F(R) = \alpha R$. Let $v = \frac{1}{\sqrt2} (e_1 \otimes e_5 + e_2 \otimes e_6) \in \R^4 \otimes \R^{n-4}$. Then it is fairly easy to see that $v$ is an eigenvector of $R$ with eigenvalue $1$. On the one hand, we obtain that $\langle F(R)v ,v \rangle = \alpha$, but on the other hand we find that
	\begin{align*}
		\langle F(R)v, v \rangle & = \langle W_{\mathbb{CP}^2} \# R(v), v \rangle = - \tfrac{1}{2} \tr(W_{\mathbb{CP}^2} \circ \ad_v \circ R \circ \ad_v) \\
		& = - \tfrac{1}{2} \tr( \ad_v \circ R \circ \ad_v \circ W_{\mathbb{CP}^2}) \\
		& = \tfrac{1}{2\sqrt6} \langle R \ad_v i_-, \ad_vi_- \rangle - \tfrac{1}{4 \sqrt6} \langle R \ad_v j_-, \ad_vj_-  \rangle \\ & \; \; \; - \tfrac{1}{4 \sqrt6} \langle R \ad_v k_-, \ad_vk_- \rangle,  
	\end{align*}
	where we used that $W_{\mathbb{CP}^2}(R(v)) = R(W_{\mathbb{CP}^2}(v)) =0$ and that the image of $W_{\mathbb{CP}^2}$ is contained in $\syp(1)_-$. Now we can just compute the expression to obtain the result. \\
	(e): This follows directly by (c) and Lemma \ref{tracevanish}.
\end{proof}
\begin{rem}
	Note that any $W \in T_{W_{\mathbb{CP}^2}} \SO(n).W_{\mathbb{CP}^2}$ is a Weyl curvature operator. In fact, we have $F(W) = \frac{1}{2} \sqrt{\frac{3}{2}}W$. In order to see that, let \\ $W \in T_{W_{\mathbb{CP}^2}} \SO(n).W_{\mathbb{CP}^2}$, $W(t)$ be a curve in the orbit of $W_{\mathbb{CP}^2}$ with $W(0) = W_{\mathbb{CP}^2}$ and $W'(0) = W$. Then 
	\begin{align*}F(W) & = \tfrac{1}{2} (W'(0)W_{\mathbb{CP}^2} + W_{\mathbb{CP}^2}W'(0)) + W'(0) \# W_{\mathbb{CP}^2} \\
		& = \tfrac{1}{2} \tfrac{d}{dt}_{\vert t=0} Q(W(t)) = \frac{1}{2} \sqrt{\tfrac{3}{2}} \tfrac{d}{dt}_{\vert t=0} W(t) =  \tfrac{1}{2} \sqrt{\tfrac{3}{2}} W.
	\end{align*}
	Furthermore, we have
	\[ T_{W_{\mathbb{CP}^2}} \SO(n).W_{\mathbb{CP}^2} = \{ [\ad_v, W_{\mathbb{CP}^2}] \mid v \in \son\}.
	\] Plugging in $i_- \in \so(4)$ and $e_1 \otimes e_5$, we obtain by dimension comparison that the tangent space above is given by $\Weyl_4^{-,1} \oplus X_4^{-,2} \otimes \R^{n-4}$.
\end{rem}

\section{An invariant set for the gradient flow of the potential}
We prove the following theorem which helps us to show that any curvature operator with high potential has to be close to the Weyl curvature of $\mathbb{CP}^2$. Here we frequently use Conjecture A. More explicitely, we use the statement that \[ \max \{(P(W)) \mid W \in \Weyl_n, ||W|| = 1 \} = \sqrt{\tfrac{3}{2}} \] and that this maximum is only achieved by $W_{\mathbb{CP}^2}$ in dimensions between $8$ and $12$. Let
\[
\Theta = \max \{ P(W) \mid W \in \Weyl_n \setminus \R W_{\mathbb{CP}^2} \text{ is a critical point of } P \text{ with } ||W|| = 1\}. \]
By Conjecture A, we know that $\Theta = \theta_n$ (compare Proposition \ref{potentialofsphere}), but we do not need that here.
\begin{thm}\label{decreasingpotentialhigh}
	Let $10 \leq n \leq 11$. Then for any unit Weyl curvature operator $W \in \Weyl(n)$ that is perpendicular to $\mathbb{R}W_{\mathbb{CP}^2} \oplus T_{W_{\mathbb{CP}^2}}(\SO(n). W_{\mathbb{CP}^2})$ the function $f(\varphi)$, defined by
	\[ f(\varphi) = \sqrt{2/3}P(\cos(\varphi) W_{\mathbb{CP}^2} + \sin(\varphi) W)
	\]
	is strictly decreasing on $(0, \pi/6]$ with $f(\pi/6) < \sqrt{2/3} \Theta$.
\end{thm}
The proof makes use of the Hessian at $W_{\mathbb{CP}^2}$. Recall that by Proposition \ref{potentialofsphere} we have
\begin{align*}
	&	P(\mathcal{W}_{\sym,11}^{\crit}) =  \frac{2 \sqrt{30}}{9} \approx 1.217161239,\\
	& P(\mathcal{W}_{\sym,10}^{\crit}) = \frac{6}{5} = 1.2,
\end{align*}
where $\mathcal{W}_{\sym,n}^{\crit}$ denotes the Weyl curvature of $S^{\lceil n/2 \rceil} \times S^{\lfloor n/2 \rfloor}$.
\begin{proof} After we have shown that $f'(\varphi) < 0$ for $\varphi \in (0, \pi/6)$ it clearly suffices to show that $f(\varphi) < \sqrt{2/3} P(W_{S_n})$. Since $\text{tri}(R,S,T) = \tr( (R \circ S + S \circ R + 2 R \# S)T )$ is symmetric in all three arguments, we may write for abbreviation
	\[ f(\varphi) = \cos^3(\varphi) + 3 \cos(\varphi) \sin^2(\varphi) \alpha + \sin^3(\varphi) \gamma,
	\]
	where $\alpha = \sqrt{2/3} \langle Q(W), W_{\mathbb{CP}^2} \rangle $ and $\gamma = \sqrt{2/3} P(W)$. Proposition \ref{Hessian} tells us that $\alpha \leq \frac{1}{3}$. We now make an estimate for $\gamma$ in terms of $\alpha$. The fact that $f(\varphi) \leq 1$ on $[0, \pi]$ shows that 
	\[ g(\varphi) = \frac{1-\cos^3(\varphi) - 3 \cos(\varphi) \sin^2(\varphi) \alpha }{\sin^3(\varphi)} \geq \gamma
	\]
	By minimizing $g(\varphi)$, we see that a minimizer $\varphi_{0}$ for $g$ has to suffice
	\[ 0 = 3 \alpha \sin^4(\varphi_0) + 3(1+ \alpha) \cos^2(\varphi_0) \sin^2(\varphi_0) + 3 \cos^4(\varphi_0) - 3 \cos(\varphi_0)
	\]
	Using $\cos^2(\varphi) + \sin^2(\varphi) = 1$ and solving for $\cos(\varphi_0)$ gives us that $\cos(\varphi_0) = \frac{\alpha}{1-\alpha}$. Hence we obtain $\gamma \leq \sqrt{1-2 \alpha}(1+ \alpha)$. Now we prove that the derivative of $f$ is negative on $(0, \pi/6)$. As a first step we derive
	\begin{align*} f'(\varphi) & = -3 \cos^2(\varphi)\sin(\varphi) -3 \alpha \sin^3(\varphi) + 6 \alpha \cos^2(\varphi) \sin(\varphi) + 3 \sin^2(\varphi) \cos(\varphi) \gamma \\
		& \leq 3\sin(\varphi)\left( (2\alpha -1) \cos^2(\varphi)  - \alpha \sin^2(\varphi) + \sin(\varphi) \cos(\varphi) \sqrt{1-2\alpha}(1+\alpha) \right) \\
		& = 3 \sin(\varphi) \left( (3 \alpha -1) \cos^2(\varphi) + \sin(2\varphi) \frac{1}{2} \sqrt{1-2 \alpha}(1+\alpha) - \alpha \right)
	\end{align*}
	Because we have $\alpha \leq \frac{1}{3}$, we immediately see that 
	\[ \tilde{f}(\varphi) = (3 \alpha -1) \cos^2(\varphi) + \sin(2\varphi) \frac{1}{2} \sqrt{1-2 \alpha}(1+\alpha) - \alpha
	\] is increasing. Hence, it suffices to show that $\tilde{f}(\frac{\pi}{6}) < 0$ for $\alpha < \frac{1}{3}$. We find that
	\begin{align*}
		\tilde{f}(\tfrac{\pi}{6}) =  \frac{3}{4}(3 \alpha -1 ) + \frac{\sqrt3}{4}  \sqrt{1-2 \alpha} (1+ \alpha) - \alpha = g_1(\alpha),
	\end{align*}
	\begin{figure}
		\centering
		\begin{minipage}{.5\textwidth}
			\centering
			\includegraphics[width=.9\linewidth]{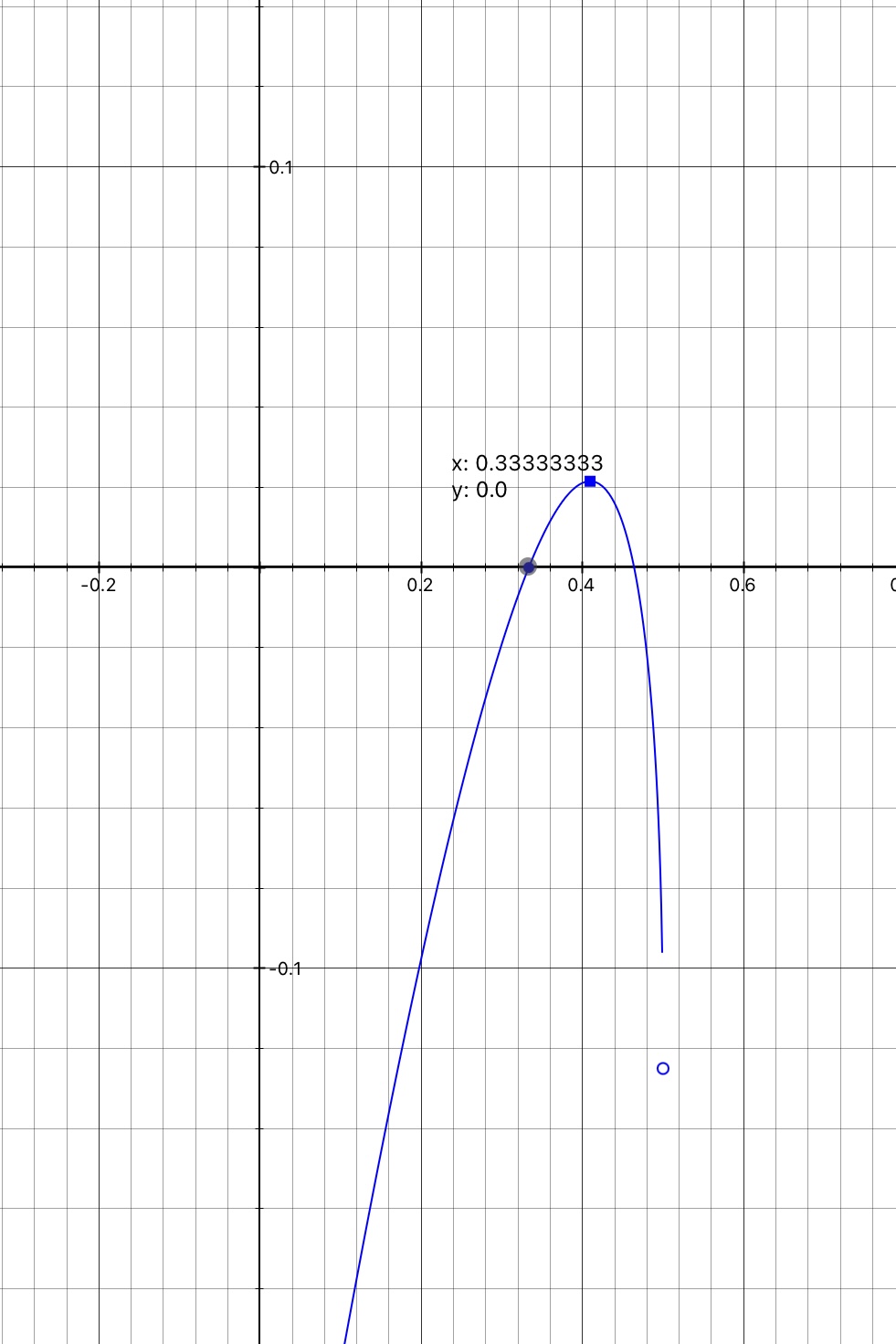}
			\captionof{figure}{The graph of $g_1(\alpha)$}
			\label{plot1}
		\end{minipage}%
		\begin{minipage}{.5\textwidth}
			\centering
			\includegraphics[width=.9\linewidth]{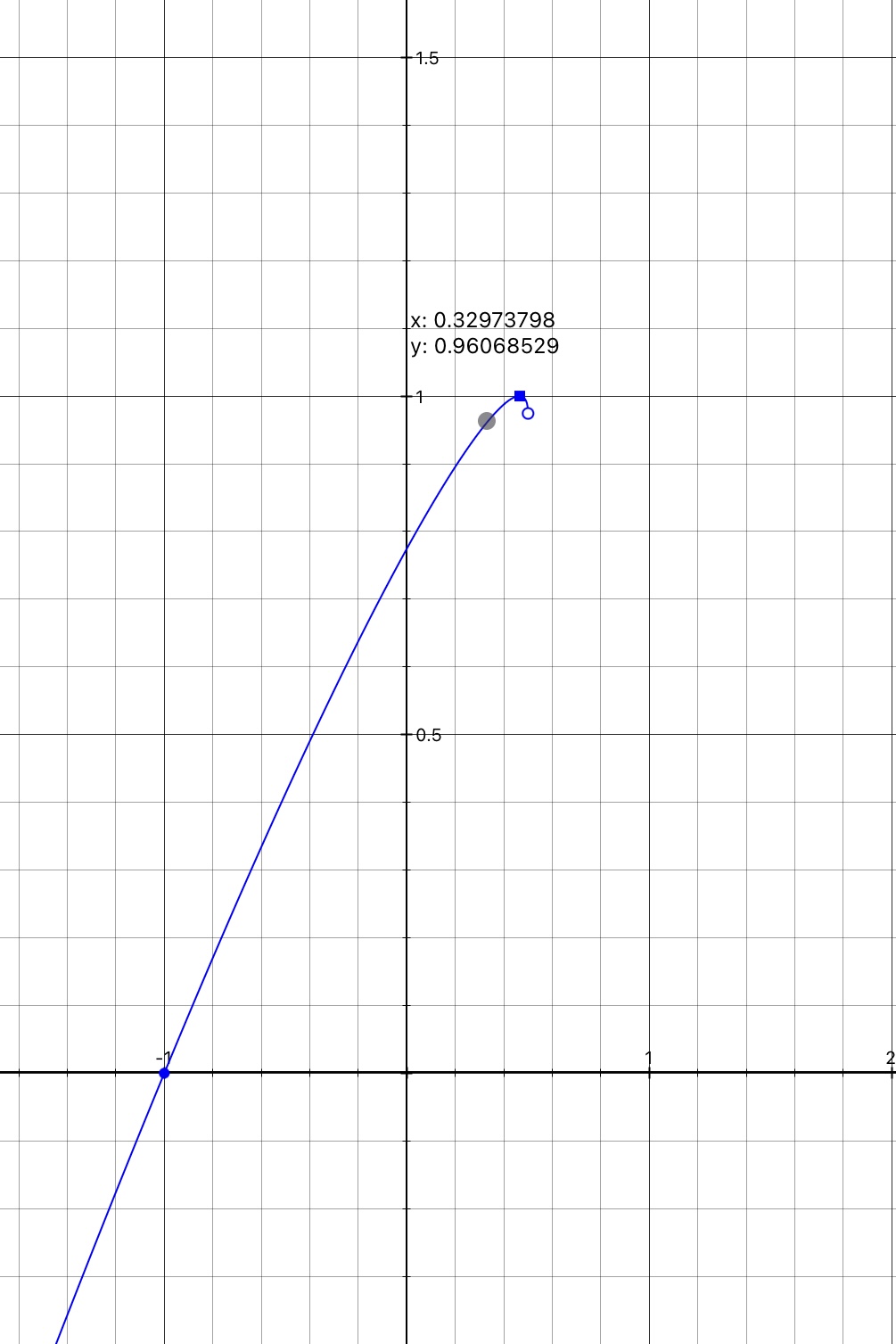}
			\captionof{figure}{The graph of $f(\pi/6)$}
			\label{plot2}
		\end{minipage}
	\end{figure}
	which is negative for $\alpha < \frac{1}{3}$, as we see in Figure \ref{plot1}. Now it is immediate to compute that $f(\frac{\pi}{6}) \leq \frac{3\sqrt{3}}{8} (1+ \alpha) + \frac{1}{8} \sqrt{1-2\alpha}(1+ \alpha) =g_2(\alpha)$, which is increasing in $\alpha < \tfrac{1}{3}$, as we see in Figure \ref{plot2}. Hence we compute, using that $\alpha < \frac{1}{3}$ that we have $f(\frac{\pi}{6}) \leq g_2(1/3) = \frac{5 \sqrt{3}}{9} \approx 0.9622504486$. This proves the result.
\end{proof}
\begin{rem}
	Note that Theorem \ref{decreasingpotentialhigh} shows immediatly that $f(\varphi)$ is also strictly descreasing in dimensions $n \leq 9$, but the estimates given above do not provide that $f(\varphi) < \sqrt{\frac{2}{3}} P(\mathcal{W}_{\sym}^{\crit})$ for $n \leq 9$. Nevertheless, it provides the correct estimate for small ranges of $\alpha = \sqrt{2/3} \langle Q(W), W_{\mathbb{CP}^2} \rangle $. Thus, it would suffice to prove the correct estimate for $\alpha \approx \frac{1}{3}$. For instance, we might assume that $\alpha > 0.317$, if $n = 9$. We conclude, that after subdividing $W = \cos(\beta) W_1 + \sin(\beta) W_2$ in the way that $W_1 \perp W_2$ are two unit Weyl curvature operators with $F_{W_{\mathbb{CP}^2}}(W_1) = \frac{1}{3}\sqrt{\frac{3}{2}}$ and $\langle F_{W_{\mathbb{CP}^2}}(W_2), W_2 \rangle \leq 0$, that the angle $\beta$ is very small. For instance, for $n = 9$ we might assume that $\beta < 0.223$. \\
	Now a strategy on proving the estimate, is to estimate $\gamma = \sqrt{2/3} P(W)$ for $W$ as above. Unfortunately, we were not able to do this yet.
\end{rem}
\section{Weyl curvature operators with big potential}
This section is devoted to the proof of Theorem B. Let $n = 10,11$. We will specify an explicit tubular neighbourhood $U_{W_{\mathbb{CP}^2}} \subset \Weyl_n^1 = \{ W \in \Weyl_n \mid ||W|| = 1\}$ of $\SO(n).W_{\mathbb{CP}^2}$ such that the potential satisfies \[P(W) < P(\mathcal{W}_{\sym}^{\crit}) \] for $W \in \Weyl_n^1 \setminus U_{W_{\mathbb{CP}^2}}$. Together with Conjecture A this shows that for any $W \in \Weyl_n^1 \setminus U_{W_{\mathbb{CP}^2}}$ we have
\begin{align*} P(W) - \sqrt{\tfrac{2(n-1)}{n}} \tfrac{\cos(\alpha)}{\sin(\alpha)} < 0,
\end{align*}
if $\angle(\mathcal{R}_{n, \lambda_{\text{crit}}}, \Id_n) < \alpha < \angle(\mathcal{R}_{S_n}, \Id_n)$. Thus it will suffice to consider curvature operators that are in $U_{W_{\mathbb{CP}^2}}$ in order to prove the Main Theorem. \\
The following result clearly proves Theorem B. To be precise, the following assertion is even stronger.
\begin{thm}\label{diffthmb} Let $10 \leq n \leq 11$. Then for any $W \in \Weyl_n^1 \setminus U_{W_{\mathbb{CP}^2}}$ the potential satisfies $P(W) \leq P(\mathcal{W}_{\sym}^{\crit}) = \theta_n$, where
	\[ U_{W_{\mathbb{CP}^2}} = \{ W \in \Weyl_n^1 \mid d_{\Weyl_n^1}(W, \SO(n).W_{\mathbb{CP}^2}) < \gamma_n\}
	\]
	with $\gamma_{11} = 0.13 $ and $\gamma_{10} = 0.26 $. The distance function $d_{\Weyl_n^1}$ denotes the inner metric on $\Weyl_n^1$.
\end{thm}
\begin{rem}
	These estimates are very odd, but for simplicity one might also choose $\gamma_{11} = \frac{\pi}{24}$ and $\gamma_{10} = \frac{\pi}{12}$, but clearly the choices as in the theorem are slightly better.
\end{rem}
Note that, since $\Weyl_n^1$ is a sphere in the vector space $\Weyl_n$, the inner metric on $\Weyl_n^1$ is given by
\[ d_{\Weyl_n^1}(W_1, W_2) = \angle(W_1, W_2).
\]
\begin{proof}
	Let $W \in \Weyl_n^1$ be a Weyl curvature operator with $d_{\Weyl_n^1}(W, \SO(n).W_{\mathbb{CP}^2}) > \gamma_n$. After rotating $W$ via the $\SO(n)$-action we might assume that 
	\[d= d_{\Weyl_n^1}(W, \SO(n).W_{\mathbb{CP}^2}) = d_{\Weyl_n^1}(W,W_{\mathbb{CP}^2}).
	\]
	We can also assume that $d \leq \frac{\pi}{6}$. If not, we consider the gradient flow of the potential $W'(t) = \nabla P(W(t))$ with $W(0) = W$ and find that $P(W(t))$ is stricly increasing in $t > 0$, until $W(t_0)$ becomes a critical point of the potential. We consider two \vspace{3mm} cases. \\
	\underline{1st case. $W(t_0) \in \SO(n).W_{\mathbb{CP}^2}$:} Then by Theorem \ref{decreasingpotentialhigh}, we know that for some $t < t_0$ we have $d(W(t),W_{\mathbb{CP}^2}) \leq \frac{\pi}{6}$. Because the potential is strictly increasing in $t >0$ we might assume \vspace{3mm} that $d(W,W_{\mathbb{CP}^2}) \leq \frac{\pi}{6}.$    \\
	\underline{2nd case. $W(t_0) \notin \SO(n).W_{\mathbb{CP}^2}$:} Then we find that $P(W) \leq P(W(t_0)) \leq P(\mathcal{W}_{\sym}^{\crit})$ by \vspace{0.2cm} Conjecture A. \\
	Choose now a unit speed geodesic $c :[0,d]  \to \Weyl_n^1$ with $c(0) = W_{\mathbb{CP}^2}$ and $c(d) = W$. Since $c$ starts orthogonally to $\SO(n).W_{\mathbb{CP}^2}$ by the first variation formula we find that $c$ is of the form
	\[c(\varphi) = \cos(\varphi)W_{\mathbb{CP}^2} + \sin(\varphi) W_1, \] where $W_1 \in \left( \R \cdot W_{\mathbb{CP}^2} \oplus T_{W_{\mathbb{CP}^2}} \SO(n).W_{\mathbb{CP}^2}\right)^{\perp}$ is of unit length. Now consider $f(\varphi) = \sqrt{\frac{2}{3}} P(c(\varphi)) = \sqrt{\frac{2}{3}} P(\cos(\varphi) W_{\mathbb{CP}^2}+ \sin(\varphi)W_1)$. By Theorem \ref{decreasingpotentialhigh} we get that $f$ is strictly decreasing and as in Theorem \ref{decreasingpotentialhigh} we find that
	\[ f(\varphi) \leq \cos^3(\varphi) + 3 \alpha \cos(\varphi) \sin^2(\varphi) + \sin^3(\varphi) \sqrt{1-2 \alpha}(1+ \alpha),
	\]
	\begin{figure}[H]
		\centering
		\begin{minipage}{.5\textwidth}
			\centering
			\includegraphics[width=.8\linewidth]{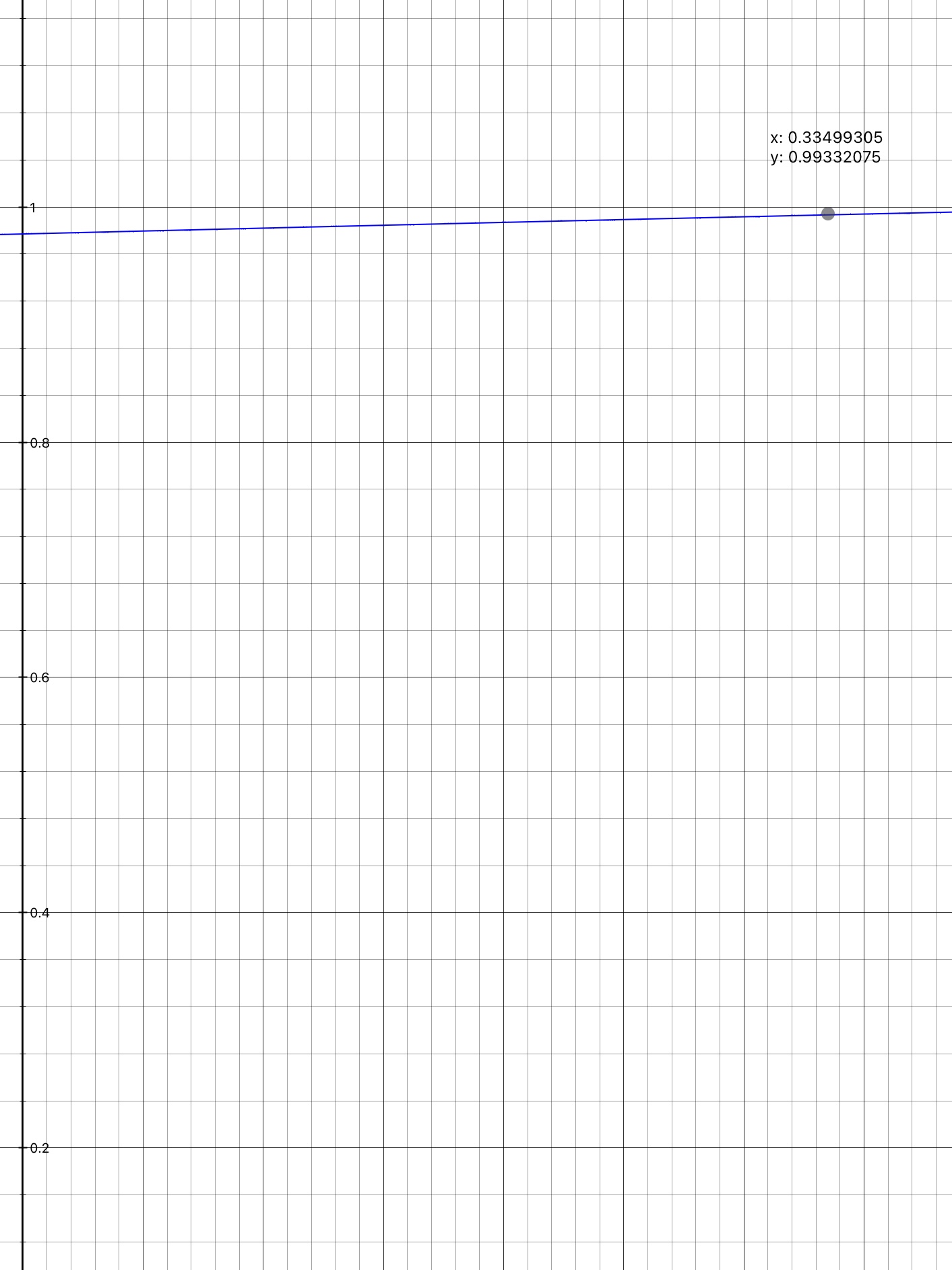}
			\captionof{figure}{The estimate for $f(\gamma_{11})$} 
			\label{plot3}
		\end{minipage}%
		\begin{minipage}{.5\textwidth}
			\centering
			\includegraphics[width=.8\linewidth]{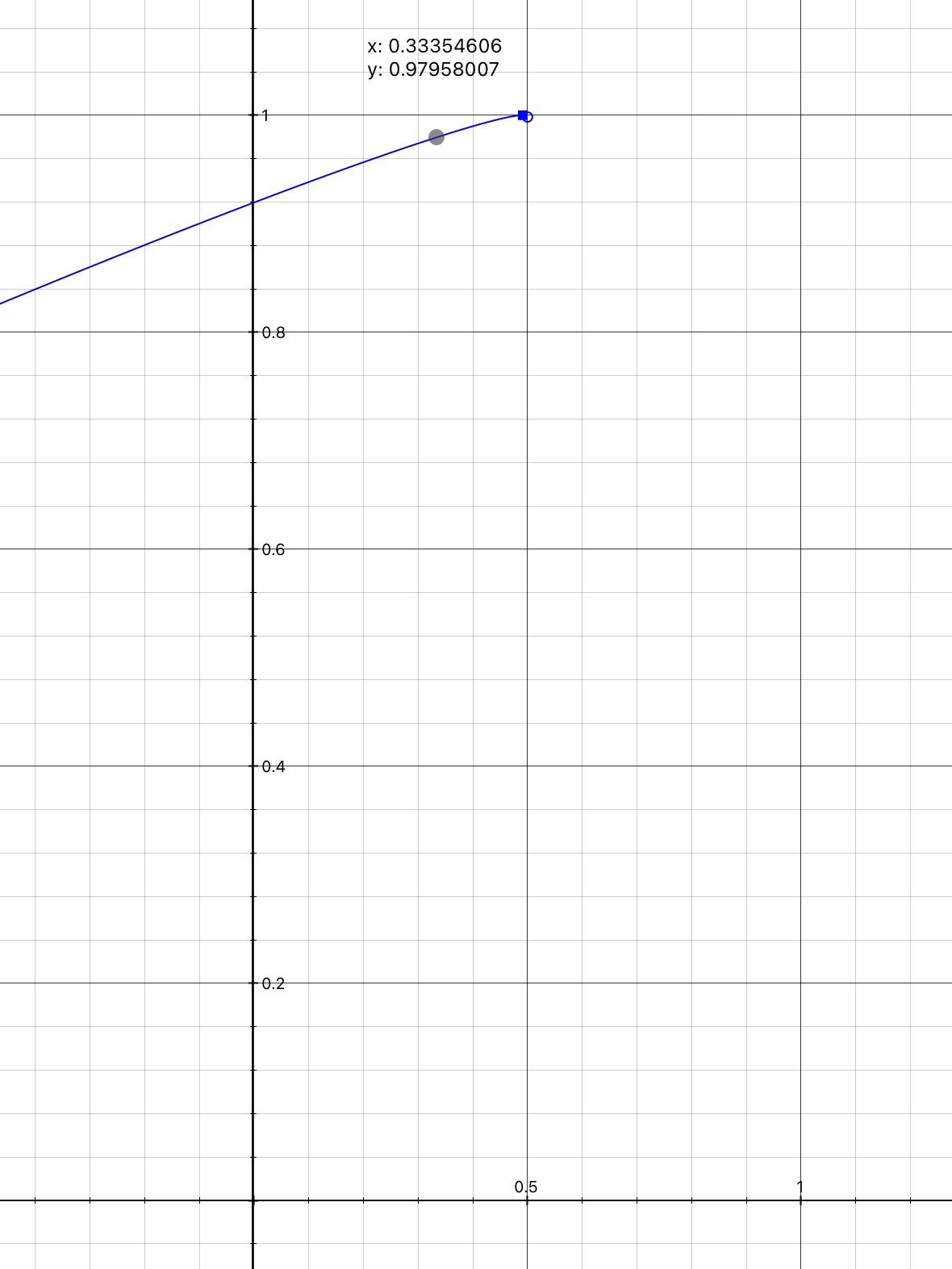}
			\captionof{figure}{The estimate for $f(\gamma_{10})$}
			\label{plot4}
		\end{minipage}
	\end{figure}
	where $\alpha = \sqrt{2/3}\langle Q(W), W_{\mathbb{CP}^2} \rangle$. If we plug in $\gamma_{11} = 0.13$, we find that
	\begin{align*}
		\sqrt{\frac{2}{3}} P(W) = f(d) \leq f(\gamma_{11}) & \leq 0.9749 + 0.05 \alpha + 0.002179 \sqrt{1-2 \alpha}(1+ \alpha) \\
		& \leq 0.9934 < \frac{2}{9} \sqrt{20} = \sqrt{\frac{2}{3}} P(\mathcal{W}_{\sym,11}^{\crit}).
	\end{align*}
	If we plug in $\gamma_{10} = 0.26$, we find that
	\begin{align*} \sqrt{\frac{2}{3}} P(W) = f(d) \leq f(\gamma_{10}) & \leq 0.9026 +0.1916 \alpha + 0.01699 \sqrt{1-2\alpha}(1+\alpha) \\
		& \leq 0.9796 < \frac{2 \sqrt{6}}{5} = \sqrt{\frac{2}{3}} P(\mathcal{W}_{\sym,10}^{\crit}).
	\end{align*}
\end{proof}
In order to prove Theorem B, we just have to note that by Theorem \ref{diffthmb} any curvature operator in $U_{W_{\mathbb{CP}^2}}$ is of the form
\[ W = \cos(\varphi) W_{\mathbb{CP}^2} + \sin(\varphi) W_1
\]
with $\varphi \leq \gamma_n$ and hence we might compute that
\[ \angle(\mathcal{R}, \mathcal{R}_{\mathbb{CP}^2}^{\crit}) \leq \begin{cases}
	\tfrac{\pi}{32}, & \text{ if } n = 11, \\
	\tfrac{\pi}{16}, & \text{ if } n = 10,
\end{cases}
\]
where $\mathcal{R} = \overline{\lambda} \Id_n + W$ and $\overline{\lambda}$ is intermediate, compare Definition \ref{peculiar}.

\addtocontents{toc}{\protect\newpage}
\chapter{Algebraically symmetric spaces}
The notion of (locally) symmetric spaces is made in an analytic fashion, in the way that it is not possible to decide if a Riemannian manifold is (locally) symmetric by just considering its curvature operator at each point individually.
In this chapter we will introduce the weaker notion of \textit{semi symmetric spaces}. In order to do that we will require the antisymmetrization of the second covariant derivative of the curvature tensor to vanish. Suprisingly, this quantity is algebraic. This suggests that it is possible to define the notion of semi symmetric spaces in a purely algebraic way. We will do this by defining the \textit{algebraic symmetry operator} $D^2\mathcal{R} \coloneqq [\mathcal{R}, \ad_{\mathcal{R}}]$ for any algebraic curvature operator $\mathcal{R} \in S_B^2(\son)$. On the one hand any semi symmetric space has vanishing algebraic symmetry operator and on the other hand any Riemannian manifold with vanishing algebraic symmetry operator is semi symmetric. That is, why we will also call these manifolds \textit{algebraically symmetric spaces}. \\
In this chapter we are more interested in spaces, that are not algebraically symmetric. In fact, we will prove that, assuming an lower bound on the algebraic symmetry, we are able to give a quantitative integral bound for the second covariant derivative of the curvature tensor. In the end, we will apply this technique to specific curvature operators, that are Einstein with Weyl operator close to $W_{\mathbb{CP}^2}$. 
\section{Definitions and background}
In a very vague way, a Riemannian manifold is called a \textit{symmetric space}, if at each point an observer cannot decide whether he looks forwards or backwards. To be precise, a Riemannian manifold $(M,g)$ is a symmetric space, if for each $p \in M$ there exists an isometry $\varphi_p : M \to M$ that fixes $p$, such that $(d\varphi_p)_p = - \id_{T_pM}$. These isometries are often called \textit{reflections}. Because isometries are completely described by their value and their differential at some point, we immediatly get $\varphi_p^2 = \id_M$, so $\varphi_p$ is an involution. The first simple observation. we directly find, is that each symmetric space $(M,g)$ is in fact a complete homogeneous space. This can easily be seen, since for $p, q \in M$ an isometry that sends $p$ to $q$ is given by a reflection at the midpoint of $p$ and $q$. We call $(M,g)$ locally symmetric, if for each $p \in M$ there exists a ball $B_r(p)$ around $p$ and a reflection $\widetilde{\varphi_p}: B_r(p) \to B_r(p)$. Since $\widetilde{\varphi_p}^*\Rm_p = \Rm_p$, it is straightforward to show that the curvature tensor of each locally symmetric manifold is parallel, i.e. $\nabla \Rm = 0$. Moreover, each manifold with parallel curvature curvature tensor is actually parallel. \\
Before we come to the notion of \textit{semi symmetric spaces} and the \textit{algebraic symmetry operator} of a curvature operator, we describe how these definitions naturally arise out of (locally) symmetric spaces. All of this is standard and can for example be found in \cite{do1992riemannian,helgason1979differential,petersen2006riemannian}. 
\subsection{Symmetric Pairs}
At first, we make sure to write $M$ as a quotient. In order to do so, we denote by $G= \Isom(M)$ the isometry group of $M$ and by $K=\Iso_{p_0}(M)$ the isotropy group at $p_0 \in M$, defined by all isometries that fix $p_0$. The Myers-Steenrod theorem states that G is actually a finite dimensional Lie group, cf. \cite{myers}. Moreover, the map $f: G/K \to M$, given by $f(gK) = g(p_0)$ is a diffeomorphism. On the other hand one may ask, given a Lie group $G$ and a closed subgroup $K$ in $G$, under which conditions is the quotient $G/K$ a symmetric space. In order to give an complete answer to this, we have to introduce symmetric pairs. The pair $(G,K)$ is called a \textit{symmetric pair} if there is an involutive automorphism $\varphi : G \to G$ such that
\begin{enumerate}
	\item[(i)] $G_{\varphi_0} \subset K \subset G_{\varphi}$, where $G_{\varphi}$ denotes the subgroup of elements of $G$ that are fixed by $\varphi$ and $G_{\varphi_0}$ denotes it connected component of the identity.
	\item[(ii)] $\Ad_K$ is a compact subgroup of $\text{GL}(\mathfrak{g})$, where $\mathfrak{g}$ denotes the Lie algebra of $G$.
\end{enumerate} 
Now, it is basically clear that any symmetric space $(M,g)$ leads to a symmetric pair. We may just choose $G$ to be identity component of the isometry group of $M$ and $K \subset G$ be the isotropy group at $p_0$. The involutive group automorphism $\varphi : G \to G$ is then given by $\varphi(g) = \varphi_{p_0} \circ g \circ \varphi_{p_0}$. Moreover, the map $\Psi = f \circ \pi : G \to M$, defined by $\Psi(g) = g(p_0)$ is a Riemannian submersion, such that $T_{p_0}M$ can be identitfied with the Lie algebra $\mathfrak{p} = \{X \in \mathfrak{g} \mid (d\varphi)_e(X) = -X\}$ via the differential of $\Psi$. This is because one may split $\mathfrak{g} = \mathfrak{k} \oplus \mathfrak{p}$, where $\mathfrak{k}$ denotes the Lie algebra of $K$, which is explicitly given by all elements that are fixed by $(d\varphi)_e$. Then $\Psi$ is constant on $K$. Hence $\ker(d\Psi)_e= \mathfrak{k}$.  \\ On the other hand, any symmetric pair $(G,K)$ defines us a symmetric space as follows: Denote by $\varphi: G \to G$ the involutive automorphism that comes with $(G,K)$ and define $M = G/K$. Then we can define a $G$-invariant metric on $M$ as follows: Consider the usual $\Ad(K)$-invariant inner product on $\mathfrak{p} = \{ X \in \mathfrak{g} \mid (d\varphi)_e(X) = -X\}$ that exists since $\Ad(K)$ is compact. By transfering it via $\left((d\pi)_{e}\right)_{\vert \mathfrak{p}}$ to $T_{eK}M$ we obtain an $\Ad(K)$-invariant product $\langle \cdot, \cdot \rangle$ on $T_{eK}M$. Now we get a well defined Riemannian metric $(\cdot, \cdot)_{gK}$ on $M$ by
\[ (v,w)_{gK} = \langle (dL_{g^{-1}})_ev, dL_{g^{-1}})_ew \rangle,
\] 
which $G$-invariant by construction. The involution $\varphi_{eK}(gK) = \varphi(g)eK$ defines a reflection at $eK$. Since $G/K$ is a homogeneous space, we obtain the reflection at other points by translation along $G/K$. Thus $G/K$ is a symmetric space. 
\subsection{Symmetric spaces and parallel transport}
Since any symmetric space $(M,g)$ has parallel curvature tensor, also the second covariant derivative of the curvature vanishes identically. This clearly means, that the antisymmetrical part of the second covariant derivative $(\nabla^2_{v,w}-\nabla^2_{w,v} \Rm_p)(a,b,c,d)$ vanishes for each $v,w,a,b,c,d \in T_pM$. Even though this seems to be a trivial observation, we will bring this in a geometric setting and derive the definition of semi symmetric spaces from this. Let $\gamma : [0,1] \to M$ be a differentiable curve in $M$ and $v \in T_pM$. Then the parallel transport along $\gamma$ of $v$ is defined as
\[ \Par_{\gamma}(v) = V(1),
\]
where $V \in \Gamma(\gamma^*TM)$ is supposed to be the unique parallel vector field along $\gamma$ with $V(0) = v$. We define the holonomy group $\Hol(M,p)=Hol_p$ of $M$ at $p \in M$ by
\[ \Hol(M,p) = \{ \Par_{\gamma} \mid \gamma \subset M \text{ is a curve that is closed in } p.\}.
\]
It is well known, that $\Hol_p \subset \SO(n)$ is a Lie group. We denote its Lie algebra by $\mathfrak{hol}_p \subset \son$. It is now possible to characterize curvature map $\Rm(v,w): T_pM \to T_pM$ as elements in the holonomy algebra as follows. Consider a chart $x : U \to V$ around $p \in M$ with $\partial x_1 \vert_p =  v$, $\partial x_2 \vert_p = w$ and the family of piecewise differentiable loops $\gamma_t(s)$ defined by the parallelogram which is spanned by $\sqrt{t}v$ and $\sqrt{t}w$. This is clearly well defined for $t$ small enough. Denote by $\Par_{\gamma_t}$ the parallel transport along $\gamma_t$. Then,
\[ \Rm(u,v)w = \frac{d}{dt}_{\vert t=0} \Par_{\gamma_t}(w),
\]
as described in \cite{besse}. Hence any endomorphism $\Rm(u,w)$ is in fact an element in the holonomy algebra $\mathfrak{hol}_p$. The Ambrose Singer theorem states more:
\begin{thm}[Ambrose-Singer, \cite{besse}]
	Let $(M,g)$ be a locally symmetric space. Then the holonomy algebra $\mathfrak{hol}_p$ is generated by curvature transformations $\Rm(v,w)$ for $v,w \in T_pM$.
\end{thm}
In order to establish this, we have to consider all contractible loops emanating at $p \in M$. For that we fill any contractible loop by small rectangles, as described before, and connect them to $p$ via some curve $\alpha : [0,1] \to M$ with $\alpha(0) = p$, $\alpha(1) = q$, see Figure \ref{lassos}.
\begin{figure}[h]
	\centering
	\includegraphics[scale=0.2]{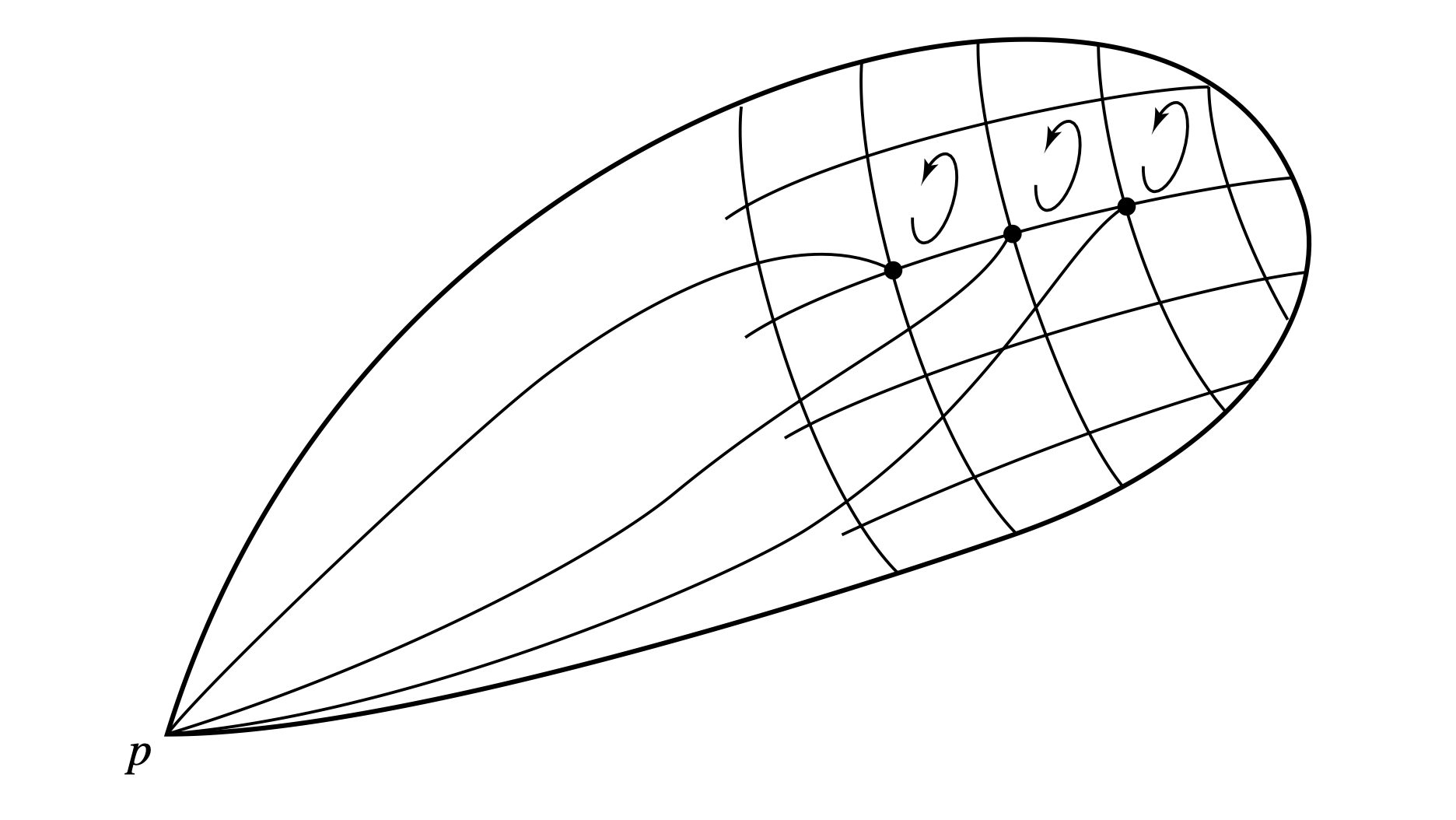}
	\caption{A contractible loop filled with lassos (cf. \cite{petersen2006riemannian})}
	\label{lassos}
\end{figure}
It is then easy to see that $\mathfrak{hol}_p$ is generated by elements of the form $\Par_{\alpha}^{-1} \circ \Rm(\Par_{\alpha}(v), \Par_{\alpha}(w)) \circ \Par_{\alpha}$. It is left to show, that the curvature is invariant under parallel transport. Let $\beta : [0,1] \to M$ be some curve, denote by $\Par_{\beta}(t): T_{\beta(0)}M \to T_{\beta(t)}M$ its parallel transport until $\beta(t)$. Let $v,w,x,y \in T_{\beta(0)}M$. Denote $V(t)= \Par_{\beta}(t)(v)$, $W(t)= \Par_{\beta}(t)(w)$, $X(t)= \Par_{\beta}(t)(x)$, $Y(t)= \Par_{\beta}(t)(y)$. These vector fields are parallel along $\beta$. Thus $\frac{d}{dt} \left( \Rm_{\beta(t)}(V(t), W(t), X(t), Y(t)) \right) =0$, since $M$ is locally symmetric. Hence the curvature is invariant under parallel transport. \\
Now we turn to the geometric description of $(\nabla^2_{v,w} - \nabla^2_{w,v})\Rm$:
\begin{lem} \label{mixedsecondcurvature}
	Let $(M,g)$ be a Riemannian manifold, $p \in M$. Then
	\begin{align*} & (\nabla^2_{v,w} - \nabla^2_{w,v} ) \Rm(a,b,c,d) \\
		& = \Rm(\Rm(v,w)a,b,c,d) + \Rm(a, \Rm(v,w)b,c,d) \\
		& \; \; \; 	+ \Rm(a,b, \Rm(v,w)c,d) + \Rm(a,b,c, \Rm(v,w)d)
	\end{align*}
	for any $v,w,a,b,c,d \in T_pM$.
\end{lem}
\begin{proof}
	In order to simplify the proof, extend $v,w \in T_pM$ to vector fields $V,W$, such that $[V,W]_p = 0$, i.e. $\nabla_VW(p) = \nabla_WV(p)$. We might also extend $,a,b,c,d$ to local vector fields $A,B,C,D$ around $p$. Since $\nabla^2_{v,w} \Rm(a,b,c,d) = \nabla_v (\nabla_w \Rm)(a,b,c,d) - (\nabla_{\nabla_VW}\Rm)(a,b,c,d)$ we obtain
	\begin{align*}		
		&	\nabla^2_{V,W} \Rm(A,B,C,D) = V(\nabla_W \Rm)(A,B,C,D) - (\nabla_W \Rm)(\nabla_VA, B,C,D)  \\ 
		& \hspace{3.5cm} - \dots - (\nabla_W \Rm)(A,B,C,\nabla_VD) - (\nabla_{\nabla_VW}\Rm)(A,B,C,D)\\
		& \hspace{3.1cm} = VW (\Rm)(A,B,C,D) - V(\Rm)(\nabla_WA,B,C,D) - \dots \\ 
		& \hspace{3.5cm} - V(\Rm)(A,B,C, \nabla_WD) - W\Rm(\nabla_VA,B,C,D) \\
		& \hspace{3.5cm} + \Rm(\nabla_W \nabla_V A,B,C,D) + \dots + \Rm(\nabla_VA,B,C, \nabla_WD) \\
		& \hspace{3.5cm} - \dots - W \Rm(A,B,C, \nabla_VD) + \Rm(\nabla_WA,B,C,\nabla_VD) \\
		& \hspace{3.5cm} + \dots + \Rm(A,B,C, \nabla_W \nabla_V D) -(\nabla_{\nabla_VW}\Rm)(A,B,C,D).
	\end{align*}
	If we use the exact same formula for $\nabla^2_{w,v} \Rm(a,b,c,d)$, subtract them and plug in $p \in M$ we directly see that $(\nabla_{\nabla_VW} \Rm)(A,B,C,D)$ and $(\nabla_{\nabla_WV} \Rm)(A,B,C,D)$ cancel out. Moreover, the mixed terms that are for example of the form \\ $\Rm(\nabla_VA, \nabla_WB,C,D)$ cancel out, because the come with different signs. We are left with
	\begin{align*}
		& 	\nabla^2_{V,W} \Rm(A,B,C,D)(p) - \nabla^2_{W,V} \Rm(A,B,C,D)(p) = (VW-WV)(p)\Rm(a,b,c,d) \\
		& \hspace{4.5cm} + \Rm(\nabla_W \nabla_VA,B,C,D)(p) - \Rm(\nabla_V \nabla_WA, B,C,D)(p)  \\ 
		& \hspace{3.5cm} + \dots + \Rm(A,B,C, \nabla_W \nabla_V D)(p) - \Rm(A,B,C,\nabla_V \nabla_W D)(p).
	\end{align*}
	Using the definition of $[V,W] = VW-WV$ and the definition of the curvature tensor $\Rm(V,W)A = \nabla_W \nabla_V A - \nabla_V \nabla_W  A + \nabla_{[V,W]}A$ we obtain the result.
\end{proof}
Let $\gamma_t \in M$ be, as before, the boundary curve obtained by the parallelogram spanned by $\sqrt{t}v$ and $\sqrt{t}w$. Then, $\frac{d}{dt}_{\vert t=0} \Par_{\gamma_t}(a) = \Rm(v,w)a$. So we have
\begin{align*}
	(\nabla^2_{v,w}- \nabla^2_{w,v} \Rm)(a,b,c,d) & =  \Rm\left(\frac{d}{dt}_{\vert t=0} \Par_{\gamma_t}a,b,c,d\right) + \Rm\left(a, \frac{d}{dt}_{\vert t=0} \Par_{\gamma_t}b,c,d\right) \\
	& 	+ \Rm\left(a,b, \frac{d}{dt}_{\vert t=0} \Par_{\gamma_t}c,d\right) + \Rm\left(a,b,c, \frac{d}{dt}_{\vert t=0} \Par_{\gamma_t}d\right)
\end{align*}
If we interpret $\Par_{\gamma_t}a$, $\Par_{\gamma_t}b$, $\Par_{\gamma_t}c$, $\Par_{\gamma_t}d$ as vector fields along the constant curve $p \in M$ and use that along constant curve $\nabla_{\partial_t} = \frac{d}{dt}$ we get that
\begin{align*} (\nabla^2_{v,w} - \nabla^2_{w,v}) \Rm )(a,b,c,d) & = \frac{d}{dt}_{\vert t=0} \Rm(A(t), B(t), C(t), D(t))  \\
	& - (\nabla_{\partial_t} \Rm)(A(t), B(t), C(t), D(t))(0).
\end{align*}
This shows that any Riemannian manifold with vanishing antisymmetrical second covariant curvature derivative is a locally symmetric space, if and only if it is invariant under parallel transport. Furthermore, for a locally symmetric space, the (tautological) condition $(\nabla^2_{v,w} - \nabla^2_{w,v}) \Rm = 0$ exactly means, that the curvature operator is invariant under parallel transport. Another direct consequence of Lemma \ref{mixedsecondcurvature} is the following
\begin{cor}
	Let $(M,g)$ be a Riemannian manifold, $p \in M$. Then for each $v,w \in T_pM$ the $(4,0)$-tensor $(\nabla^2_{v,w} - \nabla^2_{w,v}) \Rm$ just depends on the Riemannian curvature tensor at $p \in M$ and $v,w \in T_pM$.
\end{cor}
We will now investigate this even further.
\subsection{Semi symmetric spaces}
In this section we consider spaces, that antisymmetrical part of the second covariant derivative of the curvature vanishes. These spaces were, for example, studied by Szab\'o. In 1985, he classified them in \cite{szabo1}, \cite{szabo2} in terms of symmetric spaces. In 1996, Boeckx studied these spaces together with additional structure, see \cite{boeckx}. We will explain their results later. 
\begin{defi}
	A Riemannian manifold $(M,g)$ is called \textit{semi symmetric} if \[ \Rm(v,w)\Rm=(\nabla^2_{v,w} - \nabla^2_{w,v}) \Rm =0 \] for each $v,w \in T_pM$ and each $p \in M$.
\end{defi}
In a series of two papers, Szab\'o proved the following local structure theorem.
\begin{thm}[Szab\'o, \cite{szabo1}] \label{szabosfirstthm}
	Let $(M^n,g)$ be a Riemannian semi symmetric space. Then there exists an everywhere dense, open subset $U \subset M$, such that locally around any point in $U$ this space is locally isometric to 
	\[ \R^k \times M_1 \times \dots \times M_l,
	\]
	where each $M_i$ is either a symmetric space, a two-dimensional manifold, a real cone, a Kaehlerian cone or a Riemannian space foliated by Euclidean leaves of codimension two.
\end{thm}
He also proved the following global classification theorem.
\begin{thm}[Szab\'o, \cite{szabo2}] \label{szabossecondthm}
	Let $(M^n,g)$ be a connected, simply connected, and complete irreducible semi symmetric space. Then it is a 
	\begin{itemize}
		\item[(1)] symmetric space, if $\nu(p) = 0$ for each $p \in M$,
		\item[(2)] a real cone, if $\nu(p) = 1$ for each $p \in M$,
		\item[(3)] a Kaehlerian cone, if $\nu(p) = 2$ for each $p \in M$
		\item[(4)] a Riemannian manifold foliated by Euclidean leaves of codimension two, if $\nu(p) = n-2$ for each $p \in M$.
	\end{itemize}
	Here $\nu(p) = \dim\{ v \in T_pM \mid R(v,w)x = 0 \text{ for all } w,x \in T_pM\}$ denotes the nullity of $T_pM$.
\end{thm}
We are not going to explain these results further. \\
As mentioned by Boeckx in \cite{boeckx}, it is implicitely proven in \cite{szabo2}, that any semi symmetric Einstein manifold is in fact locally symmetric, if we just go through the definitions and see that any of the occuring manifolds cannot be Einstein. Since we have not defined any of these manifolds in his result explicitely, we will quote a slightly stronger result in the same article, that is proven in exact the same way.
\begin{thm}[Boeckx, \cite{boeckx}] \label{boeckxresult}
	Let $(M^n,g)$ be a semi symmetric is cyclic paralell, i.e. $\nabla \Ric(X,X) = 0$ for each $X \in \Gamma(TM)$. Then $(M^n,g)$ is locally symmetric.
\end{thm}
\begin{rem} \
	\begin{enumerate}
		\item[(a)] In fact, any locally homogeneous semi symmetric space is locally symmetric, cf. \cite{boeckx}.
		\item[(b)] This whole section does not deal with the natural question: Is any semi symmetric space locally symmetric? The answer to this is negative. The first example is given by Takagi in \cite{takagi}. He constructed a global graph $M^3 \subset \R^4$, which is semi symmetric but not locally symmetric. Since then there were found many examples in all dimensions $n \geq 3$.
		\item[(c)] Calvaruso also used the classification and showed in \cite{Calvaruso2005} that any conformally flat semi symmetric space is either locally symmetric or locally irreducible and isometric to a semi symmetric Euclidean cone.
		\item[(d)] Nevertheless, sometimes these manifolds are called manifolds with curvature operator, that come from a symmetric space. The reason for this will be clear in the next section, where we show with the methods of Cartan that in fact, for any curvature operator, that satisfies the semi symmetric condition, there exists a symmetric space with this curvature operator at each point.
	\end{enumerate}
\end{rem}
\subsection{The algebraic symmetry operator}
As mentioned and seen before, the operator $(\nabla^2_{v,w} - \nabla^2_{w,v})\Rm$ only depends on the curvature tensor $\Rm$ at the given point and $v,w \in T_pM$. Thus it makes sense, to bring the whole operator in an algebraic setting.
\begin{lem}
	Let $(M,g)$ be a Riemannian manifold. The map \\ $D^2 \Rm_p : T_pM \times T_pM \to \mathcal{T}_{4,0}(T_pM)$ given by
	\[  D^2_{v,w} \Rm_p = (\nabla^2_{v,w}- \nabla^2_{w,v}) \Rm_p
	\]
	is a skew symmetric map whose image is contained in $S_B^2(\Lambda^2(T_pM))$.
\end{lem}
\begin{proof}
	The only thing, that is nontrivial is the Bianchi identity. Let \\ $a,b,c,d \in T_pM$. Then 
	\begin{align*}
		&	\; \; \; \;  D^2_{v,w} \Rm_p(a,b,c,d) + D^2_{v,w} \Rm_p(c,a,b,d) + D^2_{v,w} \Rm_p(b,c,a,d) \\
		& 	= \Rm(\Rm(v,w)a,b,c,d) + \Rm(a,\Rm(v,w)b,c,d) \\ & \; \; \;  +\Rm(a,b,\Rm(v,w)c,d) + \Rm(a,b,c, \Rm(v,w)d) \\
		& \; \; \;  +\Rm(\Rm(v,w)c,a,b,d) + \Rm(c,\Rm(v,w)a,b,d) \\
		& \; \; \; +  \Rm(c,a, \Rm(v,w)b,d) + \Rm(c,a,b, \Rm(v,w)d) \\
		& \; \; \; + \Rm(\Rm(v,w)b,c,a,d) + \Rm(b, \Rm(v,w)c,a,d) \\
		& \; \; \; + \Rm(b,c,\Rm(v,w)a,d) + \Rm(b,c,a, \Rm(v,w)d) \\
		& = \Rm(\Rm(v,w)a,b,c,d) + \Rm(c,\Rm(v,w)a,b,d) \\
		& \; \; \; + \Rm(b,c, \Rm(v,w)a,d) + \Rm(a, \Rm(v,w)b,c,d) \\
		& \; \; \; + \Rm(c,a \Rm(v,w)b,d) + \Rm(\Rm(v,w)b,c,a,d) \\ 
		& \; \; \; + \Rm(a,b,\Rm(v,w)c,d) + \Rm(\Rm(v,w)c,a,b,d) \\
		& \; \; \; + \Rm(b, \Rm(v,w)c,a,d) + \Rm(a,b,c, \Rm(v,w)d) \\
		& \; \; \; + \Rm(c,a,b, \Rm(v,w)d) + \Rm(b,c,a,\Rm(v,w)d)  \\
		& = 0,
	\end{align*}
	where we used the Bianchi identity four times in the last equality.
\end{proof}
The last Lemma suggests that it makes sense to rewrite the operator $D^2 \Rm$ in a complete algebraic way as an operator on $\son$. In the following, we will describe this procedure. Consider for linear map $T: \R^n \to \R^n$ the map \\ $T \wedge \id_{\R^n} : \Lambda^2\R^n \to \Lambda^2\R^n$, induced by
\[ T \wedge \id_{\R^n} (v \wedge w) = \frac{1}{2} \left( Tv \wedge w + v \wedge Tw \right).
\]
Then, as an operator in $\Lambda^2(T_pM)$, we have that
\begin{align*}
	D^2_{v \wedge w} \Rm(a \wedge b,c \wedge d) & = 2 \Rm((\Rm(v \wedge w) \wedge \id)(a \wedge b),c \wedge d) \\ & + 2 \Rm(a \wedge b,(\Rm(v \wedge w) \wedge \id)(c \wedge d)),
\end{align*}
where $\id = \id_{T_pM}$. Now observe that, since $\Rm(v \wedge w) \in \so(T_pM)$, it is straightforward to see that $\Rm(v \wedge w) \wedge \id \in \so(\Lambda^2T_pM)$. Thus as a map \\ $D^2_{v \wedge w} \Rm : \Lambda^2(T_pM) \to \Lambda^2(T_pM)$ we obtain
\begin{align*} 
	D^2_{v \wedge w} \Rm(a \wedge b) & = 2\left( \Rm \circ (\Rm(v \wedge w) \wedge \id)(  a \wedge b) - (\Rm(v \wedge w) \wedge \id) \circ \Rm (a \wedge b) \right) \\
	& = 2[\Rm, \Rm(v \wedge w) \wedge \id](a \wedge b),
\end{align*}
where $[\cdot, \cdot]$ denotes the commutator of maps. In the end, we want to put this into the context of $\so(n)$, i.e. we would like to see the map $D^2 \Rm$ as a map
\[ D^2\Rm : \so(T_pM) \to S_B^2(\so(T_pM)).
\]
The missing step is the interpretation of $\Rm(v \wedge w) \wedge \id$.
\begin{lem}
	Let $(V, \langle \cdot , \cdot \rangle)$ be a finite dimensional real inner product space. Consider the isometric isomorphism $\varphi: \Lambda^2(V) \to \so(V)$, defined as in Lemma \ref{isomorphismson}. Let $A \in \so(V)$ be a skew symmetric linear map on $V$. Then the following diagram commutes
	\[ \begin{tikzcd}
		\Lambda^2(V) \arrow{r}{2A \wedge \id_V} \arrow{d}{\varphi} & \Lambda^2(V) \arrow{d}{\varphi} \\
		\so(V) \arrow{r}{\ad^*_A} & \so(V)
	\end{tikzcd} 
	\]
	Here, $\ad^*: \so(V) \to \Lambda^2(\so(V))$ denotes the adjoint of $\ad: \Lambda^2(\so(V)) \to \so(V)$.
\end{lem}
\begin{proof}
	As we have seen before, $\ad^*_A(B) = [A,B]$ for any $A,B \in \so(V)$. Let $v, w \in V$. Then we compute that for $x \in V$
	\begin{align*}
		\ad^*_A \circ \varphi(v \wedge w)x & =  [A, \varphi(v \wedge w)]x \\ & = A(\langle w,x \rangle v) - A(\langle v,x \rangle w) - \langle w, Ax \rangle v + \langle v, Ax \rangle w.
	\end{align*}
	And also
	\begin{align*}
		\varphi \circ(2 A \wedge \id_V)(v \wedge w)(x) & = \varphi( Av \wedge w + v \wedge Aw)(x) \\
		& = \langle w,x \rangle Av - \langle Av,x \rangle w + \langle Aw,x \rangle v - \langle v,x \rangle Aw \\
		& = A(\langle w,x \rangle v) + \langle Ax, v \rangle w - \langle Ax,w \rangle v - A(\langle v,x \rangle w),
	\end{align*}
	where we haved used the skew symmetry of $A$ in the last step. This completes the proof.
\end{proof}
This leads to the following central definition.
\begin{defi}
	Let $\mathcal{R} \in S_B^2(\son)$ be a curvature operator. We define the algebraic symmetry operator of $\mathcal{R}$ as $D^2R : \son \to S_B^2(\son)$, where
	\[ D^2_{v}\mathcal{R}= [\mathcal{R},\ad_{\mathcal{R}v}]
	\]
	for $v \in \son$.
\end{defi}
For later use, it will also we advantageous to introduce
\[ D^2_v(\mathcal{R},\mathcal{S}) = \frac{1}{2}\left( [\mathcal{R}, \ad_{\mathcal{S}v}] + [\mathcal{S}, \ad_{\mathcal{R}v}] \right)
\]
for two curvature operators $\mathcal{R},\mathcal{S} \in S_B^2(\son)$. Clearly, $D^2_v(\mathcal{R},\mathcal{R}) = D^2_v(\mathcal{R})$ and
\[ D^2_v(\mathcal{R}+\mathcal{S}) = D^2_v(\mathcal{R}) + 2 D^2_v(\mathcal{R},\mathcal{S}) + D^2_v(\mathcal{S}).
\]
\begin{rem}\
	\begin{enumerate}
		\item[(a)]It is natural to ask, whether the operator \[
		D^2: S_B^2(\son) \to \son^* \otimes S_B^2(\son)
		\] is $\SO(n)$-equivariant. This is in fact the case. Let $\mathcal{R} \in S_B^2(\Lambda^2\R^n)$, $v \wedge w, a  \wedge b, c \wedge d \in \Lambda^2(\R^n)$. Let $g \in \SO(n)$. Then we compute that,
		\begin{align*} & D^2_{v \wedge w}(g.\mathcal{R})(a \wedge b, c \wedge d)  \\ &= g.\mathcal{R}(g.\mathcal{R}(v \wedge w)a \wedge b, c \wedge d) + g.\mathcal{R}(a \wedge g.\mathcal{R}(v \wedge w)b, c \wedge d) \\
			& \; \; \;+ g.\mathcal{R}(a \wedge b, g.\mathcal{R}(v \wedge w)c \wedge d) + g.\mathcal{R}(a \wedge b, c \wedge g.\mathcal{R}(v \wedge w)d) \\
			&	= \mathcal{R}(\mathcal{R}(gv \wedge gw)ga \wedge gb, gc \wedge gd) + \mathcal{R}(ga \wedge \mathcal{R}(gv \wedge gw)gb, gc \wedge gd) \\
			& \; \; \; + \mathcal{R} (ga \wedge gb, \mathcal{R}(gv \wedge gw)gc \wedge gd) + \mathcal{R}(ga \wedge gb, gc \wedge \mathcal{R}(gv \wedge gw)gd )	\\
			& = gD^2_{g(v \wedge w)}\mathcal{R}(a \wedge b, c \wedge d).
		\end{align*}
		\item[(b)] For $v \wedge w, a \wedge b, c \wedge d \in \Lambda^2(\R^n)$ we have
		\begin{align*}
			& \; \; \; \langle [\mathcal{R}, 2\mathcal{S}(v \wedge w) \wedge \id_{\R^n}](a \wedge b), c \wedge d \rangle \\
			& = \mathcal{R}(\mathcal{S}(v \wedge w)a \wedge b, c \wedge d) + \mathcal{R}(a \wedge \mathcal{S}(v \wedge w)b, c \wedge d) \\
			& \; \; \; + \mathcal{R}(a \wedge b, \mathcal{S}(v \wedge w)c \wedge d) + \mathcal{R}(a \wedge b, c \wedge \mathcal{S}(v \wedge w)d).
		\end{align*}
	\end{enumerate}
\end{rem}
Using this formula, we obtain a simple way to estimate the algebraic symmetry operator in terms of the curvature operator.
\begin{lem}\label{estimatesymmetry}
	Let $\mathcal{R},\mathcal{S} \in S_B^2(\son)$ be two curvature operators and $v,w \in T_pM$ with $||v|| = ||w|| =1$. Then
	\[ ||D^2_{v \wedge w}(\mathcal{R},\mathcal{S}) || \leq 8 ||\mathcal{R}|| \cdot ||\mathcal{S}||
	\]
\end{lem}
\begin{proof}
	Denote by $R,S \in T_4^0(\R^n)$ the corresponding $(4,0)$-tensors. By Lemma \ref{differencescalarproduct} we find that
	\begin{align*}
		& ||D^2_{v \wedge w} (\mathcal{R},\mathcal{S}) ||^2 = \frac{1}{4} \sum_{i,j,k,l = 1}^n \left(D^2_{v \wedge w}(R,S)(e_i,e_j, e_k, e_l)\right)^2\\
		& = \frac{1}{4} \sum_{i,j,k,l=1}^n \Biggl( \frac{1}{2} \Big( R(S(v,w)e_i,e_j,e_k,e_l) 
		+ \dots R(e_i, e_j,e_k,S(v,w)e_l) \\ 
		& \hspace{2cm} +S(R(v,w)e_i,e_j,e_k,e_l) + \dots + S(e_i, e_j, e_k, R(v,w)e_l) \Bigl) \Biggl) ^2
	\end{align*}
	The fact that we obtain $64$ different summands and an application of Lemma \ref{tracelemma} shows us that
	\[ ||D^2_{v \wedge w} (\mathcal{R},\mathcal{S}) ||^2 \leq \frac{64}{16} ||R||^2 ||S||^2 = 64 ||\mathcal{R}||^2 ||\mathcal{S}||^2.
	\]
	This is what we have stated.
\end{proof}
In order to justify the name of the algebraic symmetry operator let $\mathcal{R} : \son \to \son$ be an algebraic curvature operator with $D^2\mathcal{R} = 0$. Denote by $\mathfrak{g} = \text{Im}(\mathcal{R})$ the image of $\mathcal{R}$ in $\son$. We can define a symmetric pair as follows: \\
Introduce a Lie bracket on $\mathfrak{g} \oplus \R^n$ by $[v,x] = [-x,v] = vx$ for $v \in \mathfrak{g} \subset \son$ and $x \in \R^n$. The assumption $D^2\mathcal{R} =0$ implies that $\mathfrak{g} \oplus \R^n$ is a Lie algebra. Indeed, it is easy to see that $(\mathfrak{g} \oplus \R^n, \mathfrak{g})$ is a symmetric pair. This shows by Section 5.1. that there exists a symmetric space $G/K$ with holonomy algebra $\mathfrak{g}$ and curvature operator $\mathcal{R}$. \\
To the end, we will give an alternative proof of the fact that semi symmetric Einstein manifolds are in fact locally symmetric. Here we use an extension to general algebraic curvature operators of the Bochner technique that was first used by Petersen and Wink, cf. \cite{Petersen:2021vh}. Although we are sure that the previously mentioned authors are aware of this result, we could not find it anywhere in the literature. Let $(M,g)$ be an Einstein manifold, such that $D^2\mathcal{R} = [\mathcal{R}, \ad_{\mathcal{R}}] = 0$. Using that the curvature operator of $M$ is harmonic, we follow \cite{Petersen:2021vj} and find that
\[ \Delta \tfrac{1}{2} |\mathcal{R} |^2 = |\nabla \mathcal{R}|^2 + \langle \mathfrak{R}(\mathcal{R}), \mathcal{R} \rangle,
\]
where $\langle \mathfrak{R}(\mathcal{R}), \mathcal{R} \rangle = \sum_{\alpha, \beta, \gamma} \lambda_{\gamma}(\lambda_{\alpha} - \lambda_{\beta})^2 \left(c_{\alpha, \beta}^{\gamma}\right)^2$, where $\mathcal{R} = \diag(\lambda_1, \dots, \lambda_N)$ is supposed to be diagonal in a suitable eigenbasis $b_1, \dots, b_N$ of $\son$ and $c_{\alpha, \beta}^{\gamma}$ denotes the structure constants of $\son$ with respect to that eigenbasis, cf. \cite{Petersen:2022ve}. A straightforward computation, using Lemma \ref{properties of sharp}, shows that
\[ \langle \mathfrak{R}(\mathcal{R}), \mathcal{R} \rangle  = 4 \langle \mathcal{R}^2 \# \Id - \mathcal{R}^{\#} , \mathcal{R} \rangle.
\]
We wish to show that this vanishes to use the maximum principle. This would show that $\mathcal{R}$ is parallel, i.e. $(M,g)$ is locally symmetric. \\
So let $v \in \son$. Using Lemma \ref{properties of sharp} again and the assumption that $(M,g)$ is semi symmetric, we find that
\begin{align*}
	0 = \langle [\mathcal{R}, \ad_{\mathcal{R}(v)}], [\mathcal{R},\ad_v] \rangle = 4 \langle (\mathcal{R}^2 \# \Id- \mathcal{R}^{\#})(v), \mathcal{R}(v) \rangle.
\end{align*}
This immediatly implies that $\langle \mathfrak{R}(\mathcal{R}), \mathcal{R} \rangle = 0$.
\section{The algebraic symmetry operator of $\mathcal{R}_{\lambda,n}$}
In this section we compute the algebraic symmetry operator for the generalized curvature operator of $\mathbb{CP}^2$ in several dimension. More explicitely, we consider \[ \mathcal{R}_{\lambda, n} = \overline{\lambda} \Id_n + W_{\mathbb{CP}^2} \] and will compute $||D^2_{b_i}\mathcal{R}_{\lambda,n}||$ for a suitable basis $b_1, \dots b_N$ of $\so(n)$, where $N = \frac{n(n-1)}{2}$. The first thing we have to see is that any operator clearly commutes with the identity operator, so it will suffice to compute expressions of the form
\[ [W_{\mathbb{CP}^2}, \ad_{\mathcal{R}_{\lambda,n}b_i}].
\]
Consider the orthonormal basis $\{ e_i \wedge e_j \in \son \mid 1\leq i < j \leq n \}$ of $\son$, where $e_i$ denotes the standard basis of $\R^n$. Denote $b_{ij} = e_i \wedge e_j$. We are going to compute $||[W_{\mathbb{CP}^2}, \ad_{\mathcal{R}_{\lambda,n}b_{ij}}]||$. Let $d_1, \dots, d_N$ be an orthonormal basis of $\son$, that is possibly not the same as before. Then 
\begin{align*}
	||D^2_v\mathcal{R}||^2 = \tr(D^2_v\mathcal{R} \circ D^2_v\mathcal{R}) = \sum_{i \leq N} \langle D^2_v\mathcal{R} \circ D^2_v\mathcal{R}(d_i), d_i \rangle = \sum_{i \leq N} || D^2_v\mathcal{R}(d_i)||^2,
\end{align*}
since $D^2_v\mathcal{R} \in S^2(\son)$.
Since the image of $W_{\mathbb{CP}^2}$ is $\syp(1)_-$ and the kernel is $\syp(1)_-^{\perp}$ it will be advantageous to choose a basis with respect to the decomposition $\son = \syp(1)_+ \oplus \syp(1)_- \oplus \so(n-4) \oplus \left( \R^4 \otimes \R^{n-4} \right)$. More explicitely, we consider the standard orthonormal basis 
\[\Bigl\{ \tfrac{1}{\sqrt2}i_{\pm}, \tfrac{1}{\sqrt2}j_{\pm}, \tfrac{1}{\sqrt2}k_{\pm} \Bigl\}
\]
of $\syp(1)_{\pm}$, the standard orthonormal basis $\{ e_a \wedge e_b \in \so(n) \mid 5 \leq a < b \leq n\}$ of $\so(n-4) \subset \son$ and the standard orthonormal basis
$\{ e_k \wedge e_l \in \son \mid 1 \leq k \leq 4 \text{ and } 5 \leq l \leq n \}$ of $\R^4 \otimes \R^{n-4} \subset \son$ and combine them to an orthonormal basis of $\so(n)$. We first begin with the kernel of $D^2\mathcal{R}_{(\lambda,n)}$.
\begin{lem}\label{kernelofCP2}
	Suppose that $D^2_v\mathcal{R}_{(\lambda,n)}(w) \neq 0$. Then either
	$v \in \so(4)$ and $w \in \so(4)$ or $v \in \R^4 \otimes \R^{n-4}$ and $w \in \so(4) \oplus \R^4 \otimes \R^{n-4}$.
\end{lem}
\begin{proof} 
	At first, let  $v \in \so(4)$. Then we have $\mathcal{R}_{(\lambda,n)}v \in \so(4)$. Hence for $w \in \so(n-4)$ we have $D^2_v\mathcal{R}_{(\lambda, n)}w = 0$ and for $w \in \R^4 \otimes \R^{n-4}$:
	\[ \left[ W_{\mathbb{CP}^2}, \ad_{\mathcal{R}_{(\lambda,n)}v} \right]w =  W_{\mathbb{CP}^2} ( \ad_{\mathcal{R}_{(\lambda,n)}v}w),
	\]
	which vanishes because $\ad_{\mathcal{R}_{(\lambda,n)}v}w \in \R^4 \otimes \R^{n-4}$.
	Now let $v \in \so(n-4)$. Then $W_{\mathbb{CP}^2}(v) = 0$. Hence $D^2_v\mathcal{R}_{(\lambda,n)} = \overline{\lambda} \cdot W_{\mathbb{CP}^2} \circ \ad_v$. Now for $w \in \so(4)$ we have $\ad_vw = 0$ and for  $w \in \so(4)^{\perp}$ we have $\ad_vw \in \so(4)^{\perp}$, thus $W_{\mathbb{CP}^2}(\ad_vw) = 0$. 
	In the end, let $v \in \R^4 \otimes \R^{n-4}$ and $w \in \so(n-4)$. Then, as before $\ad_{\mathcal{R}_{(\lambda,n)}v}w \in \R^4 \otimes \R^{n-4}$, so $D^2_v\mathcal{R}_{(\lambda,n)}(w) = 0$.
\end{proof}
Our aim is now to compute $|| D^2_v \mathcal{R}_{(\lambda, n)} ||$ for $v \in \so(4)$:
\begin{lem}\label{algsym1}The following table lists the norm of $D^2_v\mathcal{R}_{(\lambda, n)}$ for $v \in \so(4)$: \begin{center}
		\begin{tabular}[h]{c|c|c|c|c}
			$v$ & $e_1 \wedge e_2$ & $e_1 \wedge e_3$ & $e_1 \wedge e_4$ & $e_2 \wedge e_3$ \\
			\hline
			$|| D^2_v\mathcal{R}_{(\lambda, n)}||$& $0$ & $ \sqrt{2} \cdot \left| \frac{1}{2} - \frac{3 \overline{\lambda}}{\sqrt 6} \right| $ & $ \sqrt{2} \cdot \left| \frac{1}{2} - \frac{3 \overline{\lambda}}{\sqrt 6} \right| $ & $ \sqrt{2} \cdot \left| \frac{1}{2} - \frac{3 \overline{\lambda}}{\sqrt 6} \right| $ \\
		\end{tabular}
	\end{center}
	\hspace{0.91cm}
	\begin{tabular}[h]{c|c|c}
		$v$ & $e_2 \wedge e_4$ & $e_3 \wedge e_4$ \\
		\hline
		$|| D^2_v\mathcal{R}_{(\lambda, n)}||$ & $ \sqrt{2} \cdot \left| \frac{1}{2} - \frac{3 \overline{\lambda}}{\sqrt 6} \right| $ & $0$ \\
	\end{tabular}
\end{lem}
\begin{proof}
	For distinction we will use $|| \cdot ||$ for the norm on $S^2(\son)$ and $| \cdot |$ for the norm on $\son$.
	Since $\syp(1)_{+} \subset \so(4)$ is an ideal we have $\ad_vw \in \syp(1)_+$ for $v \in \so(4)$ and $w \in \syp(1)_+$. Hence
	\[ D^2_v\mathcal{R}_{(\lambda, n)}w = \overline{\lambda} \left[ W_{\mathbb{CP}^2} ,\ad_v \right]w + \left[ W_{\mathbb{CP}^2}, \ad_{W_{\mathbb{CP}^2}v} \right]w = 0,
	\]
	for $v \in \so(4)$, $w \in \syp(1)_+$, since $\syp(1)_+ \subset \text{ker}W_{\mathbb{CP}^2}$ and $\text{Im}W_{\mathbb{CP}^2} = \syp(1)_-$. This together with the last lemma yields
	\[ ||D^2_v\mathcal{R}_{(\lambda,n)} ||^2 = \frac{1}{2} \left( |D^2_v\mathcal{R}_{(\lambda, n)}i_-|^2 + |D^2_v\mathcal{R}_{(\lambda, n)}j_-|^2 + |D^2_v\mathcal{R}_{(\lambda, n)}k_-|^2 \right)
	\]
	for each $v \in \so(4)$. The rest is now a simple computation. For persuasion, we will do this for $v = e_1 \wedge e_2, e_1 \wedge e_3$. For abbreviaton we write $\mathcal{R} = \mathcal{R}_{(\lambda, n)}$. We have
	\begin{align*}
		||D^2_{e_1 \wedge e_2}\mathcal{R} ||^2 & = \frac{1}{2} \left( |D^2_{e_1 \wedge e_2}\mathcal{R}(i_-)|^2 + |D^2_{e_1 \wedge e_2}\mathcal{R}(j_-)|^2 + |D^2_{e_1 \wedge e_2}\mathcal{R}(k_-)|^2 \right) \\ 
		& = \frac{1}{2} \left| \left[ W_{\mathbb{CP}^2}, \ad_{\mathcal{R}(e_1 \wedge e_2)} \right] i_- \right| ^2 + \frac{1}{2} \left| \left[ W_{\mathbb{CP}^2}, \ad_{\mathcal{R}(e_1 \wedge e_2)} \right] j_- \right| ^2 \\
		& \; \; \; + \frac{1}{2} \left| \left[ W_{\mathbb{CP}^2}, \ad_{\mathcal{R}(e_1 \wedge e_2)} \right] k_- \right| ^2  \\
		&  = \frac{1}{2} \left| \overline{\lambda} \left[ W_{\mathbb{CP}^2}, \ad_{e_1 \wedge e_2} \right] i_- + \tfrac{1}{\sqrt 6} \left[ W_{\mathbb{CP}^2}, \ad_{i_-}\right]i_-  \right| ^2 \\ 
		& \; \; \; \; +  \frac{1}{2} \left| \overline{\lambda} \left[ W_{\mathbb{CP}^2}, \ad_{e_1 \wedge e_2} \right] j_- + \tfrac{1}{\sqrt 6} \left[ W_{\mathbb{CP}^2}, \ad_{i_-}\right]j_-  \right| ^2 \\
		& \; \; \; \; +  \frac{1}{2} \left| \overline{\lambda} \left[ W_{\mathbb{CP}^2}, \ad_{e_1 \wedge e_2} \right] k_- + \tfrac{1}{\sqrt 6} \left[ W_{\mathbb{CP}^2}, \ad_{i_-}\right]k_-  \right| ^2
	\end{align*}
	Now the first norm vanishes. That is, because $\ad_{e_1 \wedge e_2} i_- = \ad_{i_-} i_- = 0$. One computes that $\ad_{e_1 \wedge e_2} j_- = k_-$ and $\ad_{e_1 \wedge e_2}k_- = -j_-$. Thus the latter expression simplifies to
	\begin{align*}
		&	= \frac{1}{2} \left| \overline{\lambda} W_{\mathbb{CP}^2} k_- + \tfrac{\overline{\lambda}}{\sqrt6} \ad_{e_1 \wedge e_2} j_- + \tfrac{2}{\sqrt6} W_{\mathbb{CP}^2}k_- + \tfrac{1}{6} \ad_{i_-}j_- \right|^2 \\
		& \; \; \; \; + \frac{1}{2} \left| -\overline{\lambda} W_{\mathbb{CP}^2} j_- + \tfrac{\overline{\lambda}}{\sqrt6} \ad_{e_1 \wedge e_2}k_- - \tfrac{2}{\sqrt6} W_{\mathbb{CP}^2}j_- + \tfrac{1}{6} \ad_{i_-}k_- \right| ^2 \\
		& = \frac{1}{2} \left| - \tfrac{\overline{\lambda}}{\sqrt6}k_- + \tfrac{\overline{\lambda}}{\sqrt6}k_- - \frac{1}{3} k_- + \frac{1}{3}k_- \right|^2 + \frac{1}{2} \left| \tfrac{\overline{\lambda}}{\sqrt6} j_- - \tfrac{\overline{\lambda}}{\sqrt6}j_- + \tfrac{1}{3} j_- - \tfrac{1}{3} j_- \right|^2 = 0
	\end{align*}
	By exactly the same computation we get
	\begin{align*}
		||D^2_{e_1 \wedge e_3} \mathcal{R}||^2 & = \tfrac{1}{2} \left| \left( \tfrac{3\overline{\lambda}}{\sqrt6} - \tfrac{1}{2} \right)k_- \right| ^2 + \tfrac{1}{2} \left| \left( \tfrac{3 \overline{\lambda}}{\sqrt6} - \tfrac{1}{2} \right) i_- \right|^2 \\
		& = 2 \left( \tfrac{3 \overline{\lambda}}{\sqrt6} - \tfrac{1}{2} \right)^2	
	\end{align*}
	This finishes the proof.
\end{proof}
Next we compute $||D^2_v\mathcal{R}_{(\lambda, n)}||$ for $v \in \R^4 \otimes \R^{n-4}$:
\begin{lem}\label{algsym4}
	For $v = e_p \wedge e_q$ with $1 \leq p \leq 4$ and $5 \leq q \leq n$ we have \[||D^2_v\mathcal{R}_{(\lambda, n)}|| =\overline{\lambda}. \]
\end{lem}
\begin{proof}
	For abbreviation write $\mathcal{R} = \mathcal{R}_{(\lambda, n)}$. Denote by $\mathcal{B}$ the orthonormal basis of $\son$, corresponding to the decomposition \[ \son= \syp(1)_+ \oplus \syp(1)_- \oplus \so(n-4) \oplus \left( \R^4 \otimes \R^{n-4} \right),\] as described before. By Lemma \ref{kernelofCP2} we obtain
	\begin{align*}
		||D^2_{e_p \wedge e_q}\mathcal{R}||^2 & = \sum_{b \in \mathcal{B} \cap \so(4)} \left|D^2_{e_p \wedge e_q} \mathcal{R}(b)\right|^2 + \sum_{b \in \mathcal{B} \cap \R^4 \otimes \R^{n-4}} \left| D^2_{e_p \wedge e_q} \mathcal{R}(b) \right|^2
	\end{align*}
	Now it is easy to see, that for $b \in \so(4)$
	\begin{align*}
		D^2_{e_p \wedge e_q} \mathcal{R}(b)  & = \left[ W_{\mathbb{CP}^2} , \ad_{\mathcal{R}(e_p \wedge e_q)} \right] (b) = \overline{\lambda} \cdot [W_{\mathbb{CP}^2}, \ad_{e_p \wedge e_q} ](b) \\
		& = \overline{\lambda} \cdot \left( W_{\mathbb{CP}^2}(\ad_{e_p \wedge e_q}b) - \ad_{e_p \wedge e_q}(W_{\mathbb{CP}^2}b) \right) \\
		& = - \overline{\lambda} \ad_{e_p \wedge e_q} (W_{\mathbb{CP}^2} b),
	\end{align*}
	since $\ad_{e_p \wedge e_q}b \in \R^4 \otimes \R^{n-4} \subset \text{ker}W_{\mathbb{CP}^2}$. Moreover, the latter expression completely vanishes for $b \in \syp(1)_+$.
	Furthermore, for $b \in \R^4 \otimes \R^{n-4}$
	\begin{align*}
		D^2_{e_p \wedge e_q}\mathcal{R}(b) & = \overline{\lambda} \cdot \left( W_{\mathbb{CP}^2}(\ad_{e_p \wedge e_q}b) - \ad_{e_p \wedge e_q}(W_{\mathbb{CP}^2}b)\right) \\
		& = \overline{\lambda} W_{\mathbb{CP}^2}(\ad_{e_p \wedge e_q} b).
	\end{align*}
	Hence we get
	\begin{align*}
		||D^2_{e_p \wedge e_q} \mathcal{R}||^2 =  \overline{\lambda}^2\left(\sum_{b \in \mathcal{B} \cap \syp(1)_-} \left|\ad_{e_p \wedge e_q} (W_{\mathbb{CP}^2}b)\right|^2 + \sum_{b \in \mathcal{B} \cap \R^4 \otimes \R^{n-4}} \left| W_{\mathbb{CP}^2}(\ad_{e_p \wedge e_q} b) \right|^2 \right)
	\end{align*}
	The first sum simplifies to
	\begin{align*} \sum_{b \in \mathcal{B} \cap \syp(1)_-} \left|\ad_{e_p \wedge e_q} (W_{\mathbb{CP}^2}b)\right|^2 & = \frac{1}{2} \left| \ad_{e_p \wedge e_q}(W_{\mathbb{CP}^2}i_-)\right|^2 + \frac{1}{2} \left| \ad_{e_p \wedge e_q}(W_{\mathbb{CP}^2}j_-)\right|^2 
		\\
		& \; \; \; + \frac{1}{2} \left| \ad_{e_p \wedge e_q}(W_{\mathbb{CP}^2}k_-)\right| ^2 \\
		&= \frac{1}{2} \left( \frac{2}{3} |\ad_{e_p \wedge e_q}i_-|^2 + \frac{1}{6} | \ad_{e_p \wedge e_q} j_-  |^2 + \frac{1}{6} |\ad_{e_p \wedge e_q} k_-| \right)
	\end{align*}
	A direct calculation using shows that $|\ad_{e_p \wedge e_q} i_-| = |\ad_{e_p \wedge e_q} j_-|  = |\ad_{e_p\wedge e_q} k_-| = 1$. This gives $\sum_{b \in \mathcal{B} \cap \syp(1)_-} \left|\ad_{e_p \wedge e_q} (W_{\mathbb{CP}^2}b)\right|^2 = \frac{1}{2}$.
	We turn to the second sum. A necessary condition for nonvanishing $\ad_{e_p \wedge e_q} (e_r \wedge e_s)$ for $1 \leq p,r \leq 4$ and $5 \leq q,s \leq n$ is that $p=r$ or $q=s$. Thus
	\begin{align*}  \sum_{b \in \mathcal{B} \cap \R^4 \otimes \R^{n-4}} \left| W_{\mathbb{CP}^2}(\ad_{e_p \wedge e_q} b) \right|^2 & = \sum_{r=1}^4 |W_{\mathbb{CP}^2}(\ad_{e_p \wedge e_q}e_r \wedge e_q)|^2 + \sum_{s=5}^n |W_{\mathbb{CP}^2}(\ad_{e_p \wedge e_q}e_p \wedge e_s)|^2 \\
		& = \sum_{r=1}^4 |W_{\mathbb{CP}^2}(e_p \wedge e_r)|^2 + \sum_{s=5}^n |W_{\mathbb{CP}^2}(e_q \wedge e_s)|^2 \\ &
		= \sum_{r=1}^4 |W_{\mathbb{CP}^2}(e_p \wedge e_r)|^2 = \frac{1}{2}.
	\end{align*}
	Here, the last equality is just a straightforward calculation for each $p = 1, \dots, 4$. This gives the desired result.
\end{proof}
For our purposes it will not be enough to know how the algebraic symmetry operator of $\mathcal{R}_{\lambda,n}$ looks like, but we have to vary this operator a bit. In order to do that properly we introduce
\[ \mathcal{R}_{\lambda,n}^{\varphi} = \overline{\lambda} \Id_n + \cos(\varphi) W_{\mathbb{CP}^2}
\]
where $\varphi \geq 0$ denotes some angle. \\
The reason for that is that we will have to estimate the algebraic symmetry of curvature operators of the form
\[ \mathcal{R}_{\lambda,n}^{\varphi} + \sin(\varphi)W = \overline{\lambda} \Id_n + \cos(\varphi)W_{\mathbb{CP}^2} + \sin(\varphi) W
\] later on. Here, $W$ denotes some unit Weyl operator, which is perpendicular to $W_{\mathbb{CP}^2}$.
\begin{lem}\label{algsym2}
	The following table lists the norm of $D^2_v\mathcal{R}_{(\lambda, n)}^{\varphi}$ for $v \in \so(4)$ and $\varphi \leq \frac{\pi}{2}$: \begin{center}
		\begin{tabular}[h]{c|c|c|c}
			$v$ & $e_1 \wedge e_2$ & $e_1 \wedge e_3$ & $e_1 \wedge e_4$ \\
			\hline
			$|| D^2_v\mathcal{R}_{(\lambda, n)}^{\varphi}||$& $0$ & $ \sqrt{2} \cos(\varphi) \cdot \left| \frac{\cos(\varphi)}{2} - \frac{3 \overline{\lambda}}{\sqrt 6} \right| $ & $ \sqrt{2} \cos(\varphi) \cdot \left| \frac{\cos(\varphi)}{2} - \frac{3\overline{\lambda}}{\sqrt 6} \right| $ \\
		\end{tabular}
	\end{center}
	\hspace{0.45cm}
	\begin{tabular}[h]{c|c|c|c}
		$v$ & $e_2 \wedge e_3$ & $e_2 \wedge e_4$ & $e_3 \wedge e_4$ \\
		\hline
		$|| D^2_v\mathcal{R}_{(\lambda, n)}^{\varphi}||$  & $ \sqrt{2} \cos(\varphi) \cdot \left| \frac{\cos(\varphi)}{2} - \frac{3 \overline{\lambda}}{\sqrt 6} \right| $ & $ \sqrt{2} \cos(\varphi)\cdot \left| \frac{\cos(\varphi)}{2} - \frac{3 \overline{\lambda}}{\sqrt 6} \right| $ & $0$ \\
	\end{tabular}
\end{lem}
\begin{proof}
	We wish to apply Lemma \ref{algsym1}. Note that
	\[ \mathcal{R}_{\lambda,n}^{\varphi} = \cos(\varphi) \mathcal{R}_{\frac{\lambda}{\cos(\varphi)},n}.
	\]
	Thus we find that
	$D^2_v\mathcal{R}^{\varphi}_{\lambda,n} = \cos^2(\varphi)D^2_v \mathcal{R}_{\frac{\lambda}{\cos(\varphi)},n}$. The claim follows now directly, since $\cos(\varphi) \geq 0$ for $\varphi \leq \frac{\pi}{2}$.
\end{proof}
With the same method we obtain that
\begin{lem} \label{algsym3}
	For $v = e_p \wedge e_q$ with $1 \leq p \leq 4$ and $5 \leq q \leq n$ and $\varphi \leq \frac{\pi}{2}$ we have \[||D^2_v\mathcal{R}_{(\lambda, n)}^{\varphi}|| =\cos(\varphi) \cdot \overline{\lambda}. \]
\end{lem}
\section{From algebraic to analytic symmetry and applications}
The aim of this section is a result, which shows that the algebraic symmetry $D^2\mathcal{R}$ can be used to obtain a bound for the analytic quantity $\nabla^2 \Rm$ in some sense. More explicitly, we show that a Riemannian manifold $(M,g)$ with high algebraic symmetry $||D^2\mathcal{R}_p|| \geq C$ at some $p \in M$ obtains high second covariant derivative curvature in an integral sense in some neighbourhood around $p$, i.e. $\int_{B_r(p)} |\nabla^2 \Rm | d\mu \geq \tilde{C}$ for some $r, \tilde{C} > 0$. For that purpose, note again that the natural norm on $\mathcal{T}_{4}^0(M)$ differs from the norm on $S^2(\so(TM))$, in the way that, 
\[ |\Rm_p|_{g_p} = 2 || \mathcal{R}_p ||,
\]
where $\Rm_p$ denotes the curvature tensor of $(M,g)$ at $p \in M$ and $\mathcal{R}_p$ denotes the corresponding curvature operator at $p \in M$ as an element in $S^2(\so(T_pM))$. In order to avoid confusion we will denote the norm on $\mathcal{T}_{4,0}(M)$ at $p \in M$ by $| \cdot |_{g_p}$ and the norm on $S^2(\son)$ by $|| \cdot ||$. Moreover, we will denote the usual norm on $\son$ by $| \cdot |$. \\
Furthermore, we will denote the Lebesgue measure on $(T_pM,g_p)$ by $\lambda_n$ and the induced measure on the sphere $S_{r}(T_pM) = \{ v \in T_pM \mid |v|_{g_p} = 1 \}$ of radius $r >0$ by $\lambda^{n-1}_{S_r(T_pM)}$. For abbreviation, we write $S_1(T_pM) = S(T_pM)$.
\begin{thm}\label{algtoanaly}
	Let $(M,g)$ be a Riemannian manifold with curvature operator $\mathcal{R}_p \in S^2_B\so(T_pM)$ at $p \in M$ and $C, \mu_0 > 0$. Suppose that there exists a subset $U = K \cup L \subset S(T_pM)$, where $K, L \subset S(T_pM)$ are open subsets with $K \cap L = \emptyset$, such that $\lambda^{n-1}_{S(T_pM)}(K) = \lambda^{n-1}_{S(T_pM)}(L) = \mu_0 > 0$ in a way that $||D^2_{v \wedge w} \mathcal{R}_p|| \geq C$ for all $v \in K$, $w \in L$. Then 
	\[ \int_{K \cup L} | \nabla^2_{x, \cdot} \Rm |_{g_p} d\lambda^{n-1}_{S(T_pM)}(x) \geq C \mu_0.
	\]
	Here, $|\nabla^2_{x, \cdot} \Rm |_{g_p}$ denotes the norm of the $(5,0)$-tensor \[(y,a,b,c,d) \mapsto \nabla^2_{x,y} \Rm(a,b,c,d).
	\]
\end{thm}
\begin{proof}
	At first, we recall that by Lemma \ref{mixedsecondcurvature} \[\langle D^2_{v \wedge w} \mathcal{R}_p(a \wedge b),c \wedge d \rangle = (\nabla^2_{v,w} \Rm)_p(a,b,c,d) - (\nabla^2_{w,v}\Rm)_p(a,b,c,d). \]
	Hence, we obtain for any $v \in K$, $w \in L$ that $|\left(\nabla^2_{v,w} - \nabla^2_{w,v} \right) \Rm_p |_{g_p} \geq 2C$. That clearly implies that either $|\nabla^2_{v,w} \Rm_p |_{g_p} \geq C$ or $|\nabla^2_{w,v} \Rm_p|_{g_p} \geq C$. Now we consider for any $v \in K$ and $w \in L$ the following sets:
	\begingroup
	\allowdisplaybreaks
	\begin{align*}
		& K_1^{w} = \{ x \in K \mid |\nabla^2_{x,w} \Rm_p |_{g_p} \geq C \} \\
		& K_2^{w} = \{ x \in L \mid | \nabla^2_{w,x} \Rm_p|_{g_p} \geq C \} \\
		& L_1^{v} = \{ y \in K \mid | \nabla^2_{y,v} \Rm_p |_{g_p} \geq C \} \\
		& L_2^{v} = \{ y \in L \mid | \nabla^2_{v,y} \Rm_p |_{g_p} \geq C\}.
	\end{align*}
	\endgroup
	Note that, $K_1^w \cup K_2^w = K$ and $L_1^v \cup L_2^v = L$. Hence, either $\lambda^{n-1}_{S(T_pM)}(K_1^w) \geq \frac{\mu_0}{2}$ or $\lambda^{n-1}_{S(T_pM)}(K_2^w) \geq \frac{\mu_0}{2}$ and $\lambda^{n-1}_{S(T_pM)}(L_1^v) \geq \frac{\mu_0}{2}$ or $\lambda^{n-1}_{S(T_pM)}(L_2^v) \geq \frac{\mu_0}{2}$. We distinguish the following two cases. \\
	\underline{1st case.} There exists a $v \in K$ or a $w \in L$ such that $K_2^w$ or $L_2^v$ are empty: Then clearly $\lambda^{n-1}_{S(T_pM)}(K_1^w) = \mu_0$ or $\lambda^{n-1}_{S(T_pM)}(L_1^v) = \mu_0$. We only consider the case that $\lambda^{n-1}_{S(T_pM)}(K_1^w) = \mu_0$, since the other one works equally and leads to the same result. We obtain
	\begin{align*}
		& \int_{K \cup L} | \nabla^2_{x, \cdot} \Rm_p |_{g_p} d\lambda_{S(T_pM)}^{n-1}(x) \geq \int_{K_1^w} | \nabla^2_{x,w} \Rm_p |_{g_p} d\lambda_{S(T_pM)}^{n-1}(x) \geq C^2 \mu_0,
	\end{align*}
	where we haved used that $|\nabla^2_{x, \cdot} \Rm_p |_{g_p} \geq |\nabla^2_{x,w} \Rm_p|_{g_p}$ for any $w \in S(T_pM)$. \\
	\underline{2nd case.} For all $v \in K$ and $w \in L$ the sets $K_2^w$ and $L_2^v$ are nonempty: Then the standard estimate leads to
	\[ \int_{K \cup L} |\nabla^2_{x, \cdot} \Rm_p |_{g_p} d\lambda_{S(T_pM)}^{n-1}(x) \geq 2 \min_{x \in K \cup L} | \nabla^2_{x, \cdot} \Rm_p |_{g_p} \cdot \mu_0
	\]
	Let $x \in K \cup L$. We only consider the case $x \in K$. Again, the other one works equally and leads to the same result. Then $L_2^x$ is by assumption nonempty and there exists an element $y \in L$ such that $|\nabla^2_{x,y} \Rm_p|_{g_p} \geq C$. The result follows.
\end{proof}
We are about the apply the latter result to manifolds with curvature operator of the form $\mathcal{R} = \overline{\lambda} \Id_n + W_{\mathbb{CP}^2}$ or more generally $\mathcal{R} = \overline{\lambda} \Id_n + \cos(\varphi) W_{\mathbb{CP}^2} + \sin(\varphi) W_1$ for some unit Weyl operator $W_1$ that is perpendicular to $\R \cdot W_{\mathbb{CP}^2} \oplus T_{W_{\mathbb{CP}^2}}\SO(n). W_{\mathbb{CP}^2}$ for some small angles $\varphi$. In order to do that we have to specify the areas $G, R \subset S(T_pM)$ and find suitable estimates for $||D^2_{v \wedge w}\mathcal{R}||$ within this area. We will now describe this process.\\ 
Consider spherical caps $K$ and $L$ of radius $\psi$ around $e_1$, $e_3 \in T_pM$ respectively, see Figure \ref{fig:areas}. Then we have that by the computations of Lemma \ref{kernelofCP2} - Lemma \ref{algsym3} \[ \psi \mapsto ||D^2_{\left( \cos(\psi) e_1 + \sin(\psi)v \right) \wedge \left( \cos(\psi)e_3 + \sin(\psi)w \right) }\mathcal{R}^{\varphi}_{\lambda,n}||\] is minimal when $v,w \in S(T_pM)$ are in that way that $||D^2_{v \wedge w} \mathcal{R}_{\lambda,n} || = 0$, e.g. $v = e_5, w= e_6$ among all $\psi \leq \tfrac{\pi}{4}$. Here we denote $\mathcal{R}_{\lambda,n}^{\varphi} = \overline{\lambda} \Id_n + \cos(\varphi)  W_{\mathbb{CP}^2}$, as in the last section. This shows that
\begin{align*}
	& \; \; \; \; \;	||D^2_{\cos(\psi)e_1 + \sin(\psi)v \wedge \cos(\psi)e_3 + \sin(\psi) w} \mathcal{R}_{\lambda,n}^{\varphi} ||^2 \\ & = || \cos^2(\psi) D^2_{e_1 \wedge e_3} \mathcal{R}_{\lambda,n}^{\varphi} + \cos(\psi)\sin(\psi)(D^2_{e_1, w}\mathcal{R}_{\lambda,n}^{\varphi} + D^2_{v,e_2} \mathcal{R}^{\varphi}_{\lambda,n} ) + \sin(\psi) D^2_{v,w} \mathcal{R}^{\varphi}_{\lambda,n} ||^2 \\
	& \geq \left( 2\cos^4(\psi) \cos^2(\varphi) \left( \tfrac{\cos(\varphi)}{2} - \tfrac{3\overline{\lambda}}{\sqrt 6} \right)^2 + \cos^2(\psi) \sin^2(\psi) \cos^2(\varphi) \overline{\lambda}^2 \right),
\end{align*}
where we just used that $D^2_{v,w} \mathcal{R}^{\varphi}_{\lambda,n} =0$ and also $||(D^2_{v,e_3} \mathcal{R}_{\lambda,n^{\perp}}^{\varphi})_{\vert \syp(1)_{-}}|| = \tfrac{1}{2} \cos(\varphi) \overline{\lambda}$ as in the computations of Lemma \ref{algsym4}, i.e.
\begin{align*} ||D^2_{v,w} \mathcal{R}_{\lambda,n}^{\varphi} || 
	\geq \left( 2\cos^4(\psi) \cos^2(\varphi) \left( \tfrac{\cos(\varphi)}{2} - \tfrac{3\overline{\lambda}}{\sqrt 6} \right)^2 + \cos^2(\psi) \sin^2(\psi) \cos^2(\varphi) \overline{\lambda}^2 \right)^{1/2}
\end{align*}
for any $v \in K$, $w \in L$. \\ With that we can prove the following application of Theorem \ref{algtoanaly}.
\begin{figure}[h]
	\centering	\begin{tikzpicture}
		\node (A1) at (0.66,-0.6) {};
		\node (A2) at (0.6,3.6) {};	
		\node (A3) at (-0.8,3.1) {};
		\node (A4) at (2, 3.1) {};
		\node (A5) at (2.1, 3) {};
		\node (A6) at (2.1, 0) {};
		\draw (A5) to [bend angle= 15, bend right](A6);
		\draw (A3) to [bend angle=15, bend right](A4);
		\draw[dashed] (A1) to[bend angle=50, bend right] (A2);
		\draw (A2) to[bend angle=50, bend right] (A1);
		\draw[thick] (2.6,1.5) arc (0:360:2cm);
		\draw[fill=black](2.6,1.5)circle(1.5pt);
		\node[above right] at (2.9,1.4) {\small $e_1$};
		\node[above right] at (2.4, 0.5) {\small $K$};
		\draw[fill=black](0.6,3.5)circle(1.5pt);
		\node[above right] at (0.6,3.7) {\small $e_3$};
		\node[above right] at (-1,2.3) {\small $L$};
	\end{tikzpicture}
	\caption{Description of the spherical caps around $e_1, e_3$.}
	\label{fig:areas}
\end{figure}
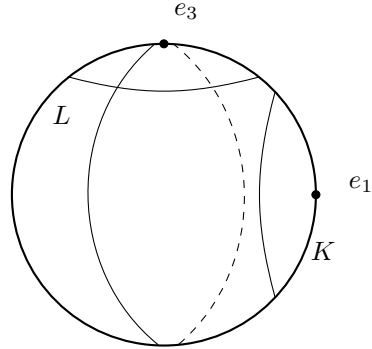
\begin{thm}\label{thmd}
	Let $\lambda > 0$ be some constant and $\varphi \geq 0$, $0 < \psi < \frac{\pi}{4}$ be some angles. Further, let $W \in \Weyl_n$ be an algebraic curvature operator of unit length. Denote by \[ \mathcal{R} = \overline{\lambda} \Id_n + \cos(\varphi) W_{\mathbb{CP}^2} + \sin(\varphi) W. \] 
	Then any $n$-dimensional Riemannian manifold $(M,g)$ with curvature operator $\mathcal{R}$ at $p \in M$ satisfies
	\[ \int_{S(T_pM)} | \nabla^2_{x, \cdot} \Rm |_{g_p} d\mu^{n-1}_{S(T_pM)}(x)  \geq 4 \vol(S^{n-2}) \cdot \Sn(n-2, \pi/4) \cdot G(\lambda,\pi/4,\varphi)
	\]
	where $G(\lambda,\psi, \varphi)$ is defined as
	\begin{align*} 
		G(\lambda, \psi, \varphi) = &  
		\left( 2\cos^4(\psi) \cos^2(\varphi) \left( \tfrac{\cos(\varphi)}{2} - \tfrac{3\overline{\lambda}}{\sqrt 6} \right)^2 + \cos^2(\psi) \sin^2(\psi) \cos^2(\varphi) \overline{\lambda}^2 \right)^{1/2} \\ & - \tfrac{1}{2} \sin(\varphi) \cdot \left( \sqrt{ \tfrac{\lambda n}{2(n-1)} + \cos^2(\varphi) } + \sin(\varphi) \right)
	\end{align*}
	and 
	\[\Sn(n, \psi) = \int_0^{\psi} \sin^n(\theta) d\theta. \]
	This quantity is positive if $\varphi$ is small enough.
\end{thm}
\begin{proof}
	Write $\mathcal{R} = \mathcal{R}_{\lambda,n}^{\varphi} + \sin(\varphi) W$, where $\mathcal{R}_{\lambda,n}^{\varphi} = \overline{\lambda} \Id_n + \cos(\varphi)W_{\mathbb{CP}^2}$. Then by Lemma \ref{estimatesymmetry} we find that
	\begin{align*}
		||D^2_{v \wedge w} \mathcal{R} || & \geq ||D^2_{v \wedge w} \mathcal{R}_{\lambda,n}^{\varphi}|| - || D^2_{v \wedge w}(\mathcal{R}_{\lambda,n}^{\varphi},\sin(\varphi)W) || - ||D^2_{v \wedge w} (\sin(\varphi)W) || \\
		&	\geq \tilde{G}(\lambda,\psi, \varphi) - \tfrac{1}{2} \sin(\varphi) || \mathcal{R}_{\lambda,n}^{\varphi} || - \tfrac{1}{2} \sin^2(\varphi)
	\end{align*}
	vor each $v \in F_1^{\psi}, w \in F_3^{\psi}$, where $F_i^{\psi}$ denotes the spherical cap of radius $\psi$ around $e_i$. Here, $\tilde{G}(\lambda, \psi, \varphi)$ is defined through
	\[\tilde{G}(\lambda,\psi, \varphi) = \left( 2\cos^4(\psi) \cos^2(\varphi) \left( \tfrac{\cos(\varphi)}{2} - \tfrac{3\overline{\lambda}}{\sqrt 6} \right)^2 + \cos^2(\psi) \sin^2(\psi) \cos^2(\varphi) \overline{\lambda}^2 \right)^{1/2}. \]
	Note that 
	\[ ||\mathcal{R}_{\lambda,n}^{\varphi}|| = \sqrt{ \tfrac{\lambda n}{2(n-1)} + \cos^2(\varphi) }
	\]
	We apply Theorem \ref{algtoanaly} to the spherical caps $F_1^{\psi}$, $F_3^{\psi}$ and $F_2^{\psi}$, $F_4^{\psi}$ and their antipodals and plug all the constants, computed in Lemma \ref{algsym2} and Lemma \ref{algsym3}, in the computations.
\end{proof}

\chapter{Towards the second best Einstein metric in dimensions below 12}
In this final chapter we will now give a proof of the Main Theorem, using the results we established before. Recall that, by Proposition \ref{keyequality}, any $n$-dimensional Einstein manifold with $\Ric = \lambda g$ satisfies
\begin{align} \label{eq:estimatech6}  \int_M |\nabla \Rm|^2_g \text{dvol}_g = 8 \int_M ||\mathcal{R}_W||^3 \left( P_{nor}(\mathcal{R}_W) - \sqrt{\tfrac{2(n-1)}{n}} \tfrac{\cos(\alpha)}{\sin(\alpha)} \right) \text{dvol}_g.
\end{align}
After localizing this equality quantitatively in the first part of this chapter using the geodesic flow and consequently proving Theorem A, our goal is to manipulate this equality in order to show that locally around any point $p \in M$, such that the right integrand is positive, we find a quantitative lower bound for all $r \geq r_0 > 0$ on
\[ \dashint_{B_r(0_p)} |\nabla \Rm|_{\exp_p(v)} d\lambda^n(v)
\]
that will be bigger than the right hand side. Note that we will have to choose the radius $r > 0$ independent on the point $p \in M$. This leads to a contradiction assuming that $\angle(\mathcal{R}_p, \Id_{\so(T_pM)}) \leq \alpha_0$ for a suitable $\alpha_0 > \angle(\mathcal{R}_{\mathbb{CP}^2}^{\crit}, \Id_n)$. \\
In the end we will briefly comment on how better estimates in Theorem C can help to prove Conjecture B.
\section{The geodesic flow}
This section is supposed to sum up well known results of the geodesic flow on $TM$. More explictitely, we will use the basics of the symplectic structure on the cotangent bundle $T^*M$ in order to state that the geodesic flow is meausure preserving with respect to a natural measure, that arises on $TM$ pulling back the symplectic form on $T^*M$ by the musical isomorphism. \\
Using this and Fubinis theorem we are able to transform $\int_M |\nabla \Rm|^2_g \dvol_g$ into a double integral, so we can compare both integrands in (\ref{keyequality}) pointwise. The other advantage of this is going to be that the second integral will be performed $T_pM$, so that we do not have to deal with any injectivity radius problems. \\
The setup of this section is not new and can mostly be found in \cite{besse2}. \\

Let $M$ be a differentiable  manifold and $(x_1, \dots, x_n): U \to \R^n$ be local coordinates around $p \in M$. Then the tangent bundle $TM$ provides a smooth atlas as follows: \\
For $v \in T_pM \subset TM$ we consider $U_T = \{w \in T_qM \mid q \in U\} \subset TM$ and define
\[ \varphi(w) = (x_1(q), \dots, x_n(q), w(x_1), \dots, w(x_n)) \in \R^{2n}, \]
where $w \in T_qM \subset U$. Similary, it is possible to define a smooth structure on $T^*M$. \\
Let $\alpha \in \mathcal{T}_1^0(T^*M)$ be the tautological $1$-form on $T^*M$, i.e. uniquely defined by the property that
\[ \lambda^*(\alpha) = \lambda
\]
for any $\lambda \in \mathcal{T}_1^0(M)$, i.e. $\alpha_{\lambda(p)}(d\lambda_p)(v) = \lambda_p(v)$ for any $v \in T_pM$, $p \in M$. It is easy to show that one can explicitely define $\alpha_q(v) = q(d\pi_q(v))$ for any $q \in T^*M$, where $\pi : T^*M \to M$ denotes the usual projection map. In local coordinates $(x_1, \dots, x_n, X^1, \dots X^n)$ on $T^*M$ one finds that
\[ \alpha = \sum_{i=1}^n  X^i dx_i.
\]
This shows that $d\alpha = \sum_{i=1}^n dX^i \wedge dx^i$ is a closed, nondegenerate $2$-form. Hence $T^*M$ admits a natural symplectic structure, with volume element $\omega = (d \alpha)^n$. \\
Now, let $g$ be a Riemannian metric on $M$. Define a Hamilton function $H : T^*M \to \R$ by
\[ H(v) = \frac{1}{2}||v||^2
\]
Then the flow $F_{v_H}: (- \varepsilon, \varepsilon) \times T^*M \to T^*M$ that corresponds to the Hamiltonian vector field $v_H$, implicitely defined by
\begin{align}  \label{eq:hamilton} d\alpha(v_H, \cdot) = dH
\end{align}
preserves the volume form $\omega$ by Liouville's Theorem, see \cite[Prop 1.56]{besse2}. Now we can use the musical isomorphisms
\[ \cdot^{\musFlat{}}: TM \to T^*M
\]
and
\[ \cdot^{\musSharp{}}: T^*M \to TM
\]
given by $X^{\musFlat{}}(v) = g(X,v)$ in order to transfer the given symplectic structure of $T^*M$ to $TM$. It is well known, cf. \cite[1.51]{besse2}, that the flow of (\ref{eq:hamilton}) transfers to the geodesic flow on $TM$, i.e. the map
\[ G_s: TM \to TM,
\]
given by $G_s(v) = \tfrac{d}{dt}_{\vert t = s} \exp(tv)$. Denote by $G = G_1$. We will now compute the corresponding volume form $\omega^{\musSharp{}}$ on $TM$. By definition
\[ \omega^{\musSharp{}}(v_1, \dots v_{2n}) = \omega(v_1^{\musFlat{}}, \dots, v_{2n}^{\musFlat{}}).
\]
If $x_1, \dots x_n, x^1, \dots, x^n$ are coordinates on $T_{(p,v)}TM$  
\[ x_i^{\musFlat{}} = x_i
\]
and it is also straightforward to see that
\[ (x^i)^{\musFlat{}} = \sum_{j=1}^n g^{ij}_p X^j.
\]
Hence, we find that
\[ \omega^{\musSharp{}}(x_1, \dots, x_n, x^1, \dots, x^n) = \sqrt{|\det(g_{ij}(p))|},
\]
i.e.
\[ \omega^{\musSharp{}}_{(p,v)} = \sqrt{|\det(g_{ij}(p))|}  dx_1 \wedge \dots \wedge dx_n \wedge dx^1 \wedge \dots \wedge dx^n.
\]
We can conclude that locally on $TM$
\[ \dvol_{TM}(p,v) = \omega^{\musSharp{}}_{(p,v)}= \dvol_g(p) \wedge d\lambda^n(v),
\]
where $\lambda^n$ denotes the $n$-dimensional Lebesgue-measure on $T_pM$. The geodesic flow will preserve this volume form, by construction.
\section{The pointwise estimate}
We will use the fact that the geodesic flow is measure preserving map on $TM$ with respect to $\dvol_{TM}$ in order to transform the original integral $\int_M |\nabla \Rm |^2 \dvol_g$ into a double integral on $T^{<r}M = \{ (p,v) \in TM \mid p \in M, |v| < r\}$. More explicitely, we prove the following
\begin{thm}\label{double}
	Let $(M,g)$ be a closed Riemannian manifold and $f:M \to \R$ be a differentiable map on $M$. Then for any $r > 0$ we have that
	\[ \int_M f(p) \dvol_g(p) = \int_M \dashint_{B_r(0_p)} f(\exp(v)) d\lambda^n(v) \dvol_g(p).
	\]
\end{thm}
\begin{proof}
	We consider the geodesic flow $G : TM \to TM$ and the set \[T^{<r}M = \{ (p,v) \in TM \mid p \in M, |v| < r\} = \bigcup_{p \in M}B_r(0_p).\]
	Then it is immediate that
	\[ G(T^{<r}M) = \bigcup_{p \in M} \{v \in T_qM \mid \exp_q(-v)=p, |v| < r\} = T^{<r}M,
	\]
	since we can just reverse the geodesic.
	Now, using that the geodesic flow is a measure preserving map we have that
	\begingroup
	\allowdisplaybreaks
	\begin{align*}
		\int_M f(p) \dvol_g(p) &= \int_M \int_{B_r(0_{\vert p})} \tfrac{1}{B_r(0_{\vert p})} f(p) d\lambda^n(v) \dvol_g(p) \\
		&= \int_{T^{<r}M} \tfrac{1}{\vol(B_r(0_{\vert p}))} f(p) \dvol_{TM}(p,v) \\
		& = \int_{T^{<r}M} \tfrac{1}{\vol(B_r(0_{\vert \exp_p(v)}))} f(\exp(v)) \dvol_{TM}(p,v) \\
		& = \int_M \int_{B_r(0_{\vert p})} \tfrac{1}{\vol(B_r(0_{\vert \exp_p(v)}))} f(\exp(v)) d\lambda^n(v) \dvol_g(p)
	\end{align*}
	\endgroup
	Since $\vol(B_r(0_{\vert p})) = \vol(B_r(0_{\vert \exp(v)}))$ is just the usual Lebesgue volume, we obtain the result.
\end{proof}
Applying Theorem \ref{double} to the left hand side of (\ref{eq:estimatech6}) we obtain Theorem A.
\section{The result}
In this section we are going go give a proof on the main theorem. We will do the general part in this section and postpone the computation of the exact constant to the next section in order to keep a better overview.  
Recall that
\[
\mathcal{R}_{\sym}^{\crit} = \tfrac{1}{n-1} \sqrt{\tfrac{3}{2}} \Id_n + W_{\mathbb{CP}^2}
\]
and $\mathcal{R}_{\sym}^{\crit}$ denotes the curvature operator of $S^{\lceil n/2 \rceil} \times S^{\lfloor n/2 \rfloor}$, endowed with the symmetric Einstein metric. \\
We will prove the following general result in a quantitative way.
\begin{main}
	Let $n=10, 11$. Then there exists an angle $\alpha_0 > 0$ with \[\angle(\mathcal{R}^{\crit}_{\mathbb{CP}^2}, \Id_\son) < \alpha_0 < \angle(\mathcal{R}_{\sym}^{\crit}, \Id_\son) \] such that the following holds: Any simply connected Einstein manifold $(M,g)$ with positive scalar curvature that satisfies
	$\angle(\mathcal{R}(p), \Id_n) < \alpha_0$
	for all $p \in M$ is isometric to the round sphere up to scaling. 
\end{main}
Let $n \leq 11$ and	let $(M,g)$ be a simply connected $n$-dimensional Einstein manifold with $\Ric = \lambda g$, where $\lambda > 0$ and $\angle(\mathcal{R}(p), \Id_{\so(T_pM)}) < \angle(\mathcal{R}_{\sym}^{\crit}, \Id_n)$ for each $p \in M$. Assume that $M$ is not isometric to the round sphere, i.e. we find points $p \in M$ such that $||\mathcal{R}_W || > 0$. After rescaling we might assume that $||\mathcal{R}_W|| \leq 1$. Using Proposition \ref{keyequality} we find that
\[ \int_M |\nabla \Rm|^2_g \text{dvol}_g = 8 \int_M ||\mathcal{R}_W||^3 \left( P_{nor}(\mathcal{R}_W) - \sqrt{\tfrac{2(n-1)}{n}} \tfrac{\cos(\alpha)}{\sin(\alpha)} \right) \text{dvol}_g.
\]
We start with the special case that $\angle(\mathcal{R}(p), \Id_{\so(T_pM)}) < \angle(\mathcal{R}_{\mathbb{CP}^2}^{\crit}, \Id_n) = \beta_n$ for all $p \in M$ where $\lambda_{\crit} = \sqrt{\frac{3}{2}}$. Then we find, using Conjecture A and Section 2.5, at any point with $||\mathcal{R}_W|| > 0$ that
\begin{align*} P_{\nor}(\mathcal{R}_W) - \sqrt{\tfrac{2(n-1)}{n}} \tfrac{\cos(\alpha)}{\sin(\alpha)} & < P(W_{\mathbb{CP}^2}) - \sqrt{\tfrac{2(n-1)}{n}} \tfrac{ \cos(\beta_n)}{\sin(\beta_n)} \\
	& = \sqrt{\tfrac{3}{2}} - \sqrt{\tfrac{2(n-1)}{n}} \cdot \sqrt{\tfrac{3n}{4(n-1)}} 
	= 0.
\end{align*}
This is an immediate contradiction.
Thus we might assume that there exists points $p \in M$ with $ \angle(\mathcal{R}_{\mathbb{CP}^2}^{\crit}, \Id_n) \leq \angle(\mathcal{R}(p), \Id_{\so(T_pM)}) < \angle(\mathcal{R}_{\sym}^{\crit}, \Id_n)$. Hence we have that $\lambda \leq \sqrt{\frac{3}{2}}$ by the assumption that $||\mathcal{R}_W|| \leq 1$. This yields an upper sectional curvature bound for $M$ by
\[ \sec(M) \leq ||\mathcal{R}|| \leq \sqrt{1+ \lambda^2 \tfrac{n}{2(n-1)}} \leq \sqrt{\tfrac{7n-4}{4(n-1)}} = \kappa_0(n).
\]
Using Theorem \ref{double} we find for any $r > 0$ that
\begin{align*}
	\int_M |\nabla \Rm |^2 \dvol_g & = \int_M \dashint_{B_r(0_p)} |\nabla \Rm|^2_{\exp_p(v)} d\lambda^n(v) \dvol_g(p) 
\end{align*}
We claim now the following, which directly proves the main theorem: 
\begin{thm}\label{last} Let $n=10,11$. Then there exists an angle $\alpha_0$ with $\angle(\mathcal{R}_{\mathbb{CP}^2}^{\crit}, \Id_n) < \alpha_0 \leq  \angle(\mathcal{R}_{\sym}^{\crit}, \Id_n)$ such that the following holds. If $(M,g)$ is a simply connected $n$-dimensional Einstein manifold with positive scalar curvature and $\angle(\mathcal{R}, \Id_n) \leq \alpha_0$ for all $p \in M$, then there exists $r_0 > 0$ such that we have
	\[  \dashint_{B_r(0_p)} |\nabla \Rm|_{\exp_p(v)}^2 d \lambda^n(v) \geq 
	8 ||\mathcal{R}_W(p)||^3 \left( P_{nor}(\mathcal{R}_W(p)) - \sqrt{\tfrac{2(n-1)}{n}} \tfrac{\cos(\alpha_p)}{\sin(\alpha_p)} \right)
	\]
	for all $r \geq r_0$ and for all $p \in M$.
	Furthermore, there exist points, where this inequality is strict.
\end{thm} 
\begin{proof}
	Let $p \in M$. We might assume without loss of generality that $|\mathcal{R}_W(p)| > 0$. After rescaling we might also assume that $|\mathcal{R}_W(p)| = 1$. This is because the inequality above is invariant under scaling, compare (\ref{eq:derivativescale}).
	We might assume that 
	\[  P_{nor}(\mathcal{R}_W(p)) - \sqrt{\tfrac{2(n-1)}{n}} \tfrac{\cos(\alpha_p)}{\sin(\alpha_p)} > 0.
	\] Using Theorem B we know that
	\[ \mathcal{R} = \overline{\lambda} \Id_n + \cos(\varphi) W_{\mathbb{CP}^2} + \sin(\varphi) W,
	\]
	with $\overline{\lambda} = \tfrac{1}{n-1} \lambda_0$ and $\varphi \leq \varphi_0$, where
	\begin{align} \label{einsteinconstant} \lambda_0 \geq \begin{cases} \tfrac{1}{n}\sqrt{2(n-1)(n-2)}, & \text{ if } n \text{ is even,} \\
			\sqrt{2} \cdot\sqrt{\tfrac{(n-1)(n-3)}{(n+1)(n-2)}}, & \text{ if } n \text{ is odd}
		\end{cases}
	\end{align}
	and $\varphi_0$ as described in Theorem B. \\
	Introduce polar coordinates on $B_r(0_p)$ through \[ (t, v) \in (0,r) \times S(T_pM) \subset T_pM \setminus \{ 0\}, \] where $S(T_pM) = \{ v \in T_pM \mid ||v||^2 =1 \}$. Now clearly on $B_r(0_p)$ the bundle of $(k,l)$-tensors $T_k^l(B_r(0_p))$ is trivial. Thus we might assume that after local trivialization
	the section $\nabla \Rm$ is a map \[ \nabla \Rm: (0,r) \times S(T_pM) \to \left( \otimes^5\ T_pM \right)^*. \]
	Let $0 < t < r$, $v_0 \in S(T_pM)$. We do a Taylor expansion to see that there exists $\xi \in (0,r)$ such that
	\[\nabla \Rm(t,v_0) = \nabla \Rm_p + t \tfrac{d}{dt}_{\vert (t,v)= (0,v_0)} \nabla \Rm + \tfrac{1}{2} t^2 \tfrac{d^2}{(dt)^2} _{\vert (t,v)= (\xi,v_0)} \nabla \Rm. \]
	Now by choosing a parallel orthonormal frame along $\exp(tv_0)$ we obtain 
	\[ \tfrac{d}{dt}_{\vert (t,v) = (0,v_0)} \nabla \Rm = \nabla_{v_0} \nabla \Rm_p.
	\]
	And similary we compute 
	\[ \tfrac{d^2}{(dt)^2} _{\vert (t,v)= (\xi,v_0)} \nabla \Rm = \nabla^2_{v_0,v_0} \nabla \Rm_{(\xi,v_0)}.
	\]
	If we write $g_{\text{eucl}} = dr^2 + t^{n-1} g_{S^{n-1}}$, we find that
	\begin{align*}
		& \; \; \; \; \; \dashint_{B_r(0_p)} |\nabla \Rm|^2  d\lambda^n \\ &=\tfrac{1}{\vol B_r(0_p)} \int_0^r \int_{S(T_pM)}  \left| \nabla \Rm_p + t \nabla_v \nabla \Rm_p + \tfrac{1}{2} t^2 \nabla^2_{v,v} \nabla \Rm_{(\xi,v)} \right|^2 t^{n-1} d \mu_{S(T_pM)}(v) dt \\
		& = \tfrac{1}{\vol B_r(0_p)} \int_{0}^r \int_{S(T_pM)} \Bigl[ |\nabla \Rm_p|^2 + 2 \langle t \nabla_v \nabla \Rm_p, \nabla \Rm_p \rangle  \\ &
		\hspace{5cm} + \langle t^2 \nabla^2_{v,v} \Rm_{(\xi,v)}, \nabla \Rm_p \rangle  \Bigr] t^{n-1} d \mu_{S(T_pM)}(v) dt \\
		& \; \; \; + \tfrac{1}{\vol B_r(0_p)} \int_0^{r} \int_{S(T_pM)} \left| t \nabla_v \nabla \Rm_p + \tfrac{1}{2} t^2 \nabla^2_{v,v} \nabla \Rm_{(\xi,v)} \right|^2 t^{n-1} d \mu_{S(T_pM)}(v) dt
	\end{align*}
	At first, we deal with the first integral. Since $v \mapsto \langle t \nabla_v \nabla \Rm_p, \nabla \Rm_p \rangle$ is anti-symmetric with respect to reflexion at $0_p$ its integral over $S(T_pM)$ vanishes for fixed $t \in (0,r)$. For the term that is left over we might write
	\[ \nabla^2_{v,v} \nabla \Rm_{\xi,v} = c \nabla \Rm_p + L,
	\]
	where $L$ is orthogonal to $\nabla \Rm_p$ and find that
	\begin{align*} & \; \; \; \; \int_0^r \int_{S(T_pM)} \left( |\nabla \Rm_p|^2 + t^2 \langle \nabla^2_{v,v} \Rm_{(\xi,v)}, \nabla \Rm_p \rangle \right) t^{n-1} d\mu_{S(T_pM)}(v) dt \\
		& = \int_0^r \int_{S(T_pM)} \left( |\nabla \Rm_p |^2 + t^2 c |\nabla \Rm_p|^2 \right) t^{n-1} d\mu_{S(T_pM)}(v) dt
	\end{align*}
	is nonnegative for $r \leq \sqrt{2(n+2)}$ by a standard minimization process for $c$.\\
	Hence we find that
	\begin{align*}
		& \; \; \;	\dashint_{B_r(0_p)} |\nabla \Rm|^2 d\lambda^n \\
		&	\geq \tfrac{1}{\vol B_r(0_p)} \int_0^{r} \int_{S(T_pM)} t\left| \nabla_v \nabla \Rm_p + \tfrac{1}{2} t \nabla^2_{v,v} \nabla \Rm_{(\xi,v)} \right|^2 t^{n-1} d \mu_{S(T_pM)}(v) dt
	\end{align*}
	Note that in the end, we will choose $r > 0$ small enough, such that
	\[ | \nabla_v \nabla \Rm_p | \geq \tfrac{1}{2}t |\nabla^2_{v,v} \nabla \Rm_{(\xi,v)}| 
	\]
	for all $t < r$. Using Theorem C we find that
	\begin{align*} & \; \; \; \dashint_{B_r(0_p)} |\nabla \Rm|^2 d\lambda^n \\
		& \geq \tfrac{1}{\vol{B_r(0_p)}} \int_0^r \int_{S(T_pM)} t^2 \left( \left| \nabla_v \nabla \Rm_p \right| - \tfrac{1}{2}t |\nabla^2_{v,v} \nabla \Rm_{(\xi,v)} | \right)^2 t^{n-1} d\mu_{S(T_pM)}(v) dt \\
		&	\geq \tfrac{1}{\vol{B_r(0_p)}} \int_0^r \int_{S(T_pM)} t^2 \left( \left| \nabla_v \nabla \Rm_p \right| - \tfrac{1}{2}t (2\kappa_0-\lambda)^{5/2} \textbf{C}_3(n) \right)^2 \ t^{n-1} d\mu_{S(T_pM)}(v) dt \\
		&	\geq \tfrac{1}{\vol{B_r(0_p)}} \int_{0}^r \int_{S(T_pM)} \left( \left| \nabla_v \nabla \Rm_p \right| - \tfrac{1}{2}t (2\kappa_0-\lambda)^{5/2} \textbf{C}_3(n) \right)^2 \ t^{n+1} d\mu_{S(T_pM)}(v) dt 
	\end{align*}
	Now we use Theorem \ref{thmd} for the spherical caps around $e_1, e_3$ and $e_2, e_4$ and $-e_1, -e_3,$ and $-e_2, -e_4$ of radius $\tfrac{\pi}{4}$ to finally get
	\begingroup
	\allowdisplaybreaks
	\begin{align*}
		& \; \; \; \dashint_{B_r(0_p)} |\nabla \Rm|^2 d\lambda^n \\
		& 	\geq \tfrac{1}{\vol{B_r(0_p)}} \int_{0}^r \int_{S(T_pM)} \left( \left| \nabla_v \nabla \Rm_p \right| - \tfrac{1}{2}t (2\kappa_0-\lambda)^{5/2} \textbf{C}_3(n) \right)^2 \ t^{n+1} d\mu_{S(T_pM)}(v) dt \\
		& \geq  \tfrac{4\vol(S^{n-2}) \Sn(n-2, \pi/4)}{\vol{B_r(0_p)}} \int_{0}^r \left( G(\lambda,\pi/4, \varphi) - \tfrac{1}{2}t (2\kappa_0-\lambda)^{5/2} \textbf{C}_3(n) \right)^2 t^{n+1} dt.
	\end{align*}
	\endgroup
	Here $G$ is defined as in Theorem \ref{thmd} and $\textbf{C}_3(n)$ is defined as in Theorem C. Observe that $G(\lambda,\pi/4, \varphi) \geq G(\lambda_0, \pi/4, \varphi_0)$, where $\varphi_0 > 0$ is slightly bigger than $0$ and $\lambda_0$ is defined as in (\ref{einsteinconstant}). \\
	Now we can conclude the Lemma as follows: Clearly the whole expression is positive for $r > 0$ small enough. Since $P_{\nor}(\mathcal{R}_W(p)) - \sqrt{\tfrac{2(n-1)}{n}} \tfrac{\cos(\alpha_p)}{\sin(\alpha_p)}$ is nonpositive for $\alpha_p \leq \angle(\mathcal{R}_{\mathbb{CP}^2}^{\crit}, \Id_n)$ we can find $\alpha_0$ slightly larger than $\angle(\mathcal{R}_{\mathbb{CP}^2}^{\crit}, \Id_n)$ such that the claim holds.
\end{proof} 
\subsection{The explicit estimate}
In this section we will specify on a possible value for $\alpha_0$ in order to justify that it is possible to compute it. Let $\alpha_0 = \angle(\mathcal{R}_{\mathbb{CP}^2}^{\crit}, \Id_n) + 1.015 \cdot 10^{-15}$.
The first lemma is a quantitative improvement of Theorem B assuming that the angle is just slightly bigger than $\angle(\mathcal{R}_{\mathbb{CP}^2}^{\crit}, \Id_n)$.
\begin{lem}\label{lastlemma}
	Let $10 \leq n \leq 11$. Then for any $W \in \Weyl_n^1 \setminus \widetilde{U_{W_{\mathbb{CP}^2}}}$ the potential satisfies $P(W) \leq \sqrt{\tfrac{3}{2}} - 2.66 \cdot 10^{-15}$, where
	\[ \widetilde{U_{W_{\mathbb{CP}^2}}} = \{ W \in \Weyl_n^1 \mid d_{\Weyl_n^1}(W, \SO(n).W_{\mathbb{CP}^2}) < \gamma_0\}
	\]
	with $\gamma_0=10^{-6}$. The distance function $d_{\Weyl_n^1}$ denotes the inner metric on $\Weyl_n^1$.
\end{lem}
\begin{proof}
	We proceed as in the proof of Theorem \ref{diffthmb} and just will show the last step. \\
	So let $f(\varphi) = P(\cos(\varphi) W_{\mathbb{CP}^2}+ \sin(\varphi)W_1)$ be defined as in Theorem \ref{diffthmb}. Then we find as in the exact same way that
	\[ P(W) \leq f(\varphi) \leq f(\gamma_0) \leq \sqrt{\tfrac{3}{2}} - 2.66 \cdot 10^{-15}.
	\]
\end{proof}
\begin{thm}
	Let $n = 10,11$. Then any simply connected $n$-dimensional Einstein manifold with positive scalar curvature and $\angle(\mathcal{R}(p), \Id_n) \leq \alpha_0$ for all each $p \in M$ is isometric to the sphere up to scaling.
\end{thm}
The proof works exactly the same for both dimensions. We will just do $n=11$, since for $n=10$ the constants are quite better and hence we will get at least the same result.
\begin{proof}
	We exactly proceed as in Theorem \ref{last} and find that at each point with huge potential that
	\begingroup
	\allowdisplaybreaks
	\begin{align*}  & \; \; \; \dashint_{B_r(p)} |\nabla \Rm|^2  \dvol_g \\
		& \geq  \tfrac{4\vol(S^9) \Sn(9, \pi/4)}{\vol{B_r(0_p)}} \int_0^r \left( G(\lambda_0,\pi/4, \varphi_0) - \tfrac{1}{2}t (2\kappa_0-\lambda_0)^{5/2} \textbf{C}_3(11) \right)^2 t^{12} dt \\
		& = \tfrac{4\vol(S^9) \Sn(9, \pi/4)}{\vol{B_r(0_p)}} \left( \tfrac{G^2}{13}r^{13} - \tfrac{G\cdot \textbf{C}}{14} r^{14} + \tfrac{\textbf{C}^2}{60}r^{15} \right) \\
		& 	=  \tfrac{4\vol(S^9) \Sn(9, \pi/4)}{\vol{B_1(0_p)}}  \left( \tfrac{G^2}{13} - \tfrac{G\cdot \textbf{C}}{14} r + \tfrac{\textbf{C}^2}{60}r^2 \right) r^2,
	\end{align*} 
	\endgroup
	where $G = G(\sqrt{40/27}, \pi/4, 10^{-6}) \approx 0.303088$ (see Lemma \ref{lastlemma}) and \\ $\textbf{C} = (2\sqrt{73/40}-\sqrt{40/27})^{5/2}\textbf{C}_3(11) \approx 1035846$.
	Now we have to choose $r$ such that
	\[ |\nabla_v\nabla \Rm| \geq \tfrac{1}{2}t(2 \kappa_0 - \lambda)^{5/2} \textbf{C}_3(11)
	\]
	for all $t < r$. Thus we might choose $r = 5.86 \cdot 10^{-7}$. Summing up, this implies
	\[
	\dashint_{B_r(0_p)} |\nabla \Rm|^2 \dvol_g \geq 2.86 \cdot 10^{-15}.
	\]
	
	Now we just have to compare this quantity to the right hand side, i.e. we have to estimate
	\[
	8\left( P_{\text{nor}}(\mathcal{R}_W(p)) - \sqrt{\tfrac{2(n-1)}{n}} \tfrac{\cos(\alpha_p)}{\sin(\alpha_p)} \right)
	\]
	We find that, plugging in $P_{\text{nor}}(\mathcal{R}_W) \leq \sqrt{3/2}$ and the desired angle that
	\[
	8	\left( P_{\text{nor}}(\mathcal{R}_W(p)) - \sqrt{\tfrac{20}{11}} \tfrac{\cos(\alpha_p)}{\sin(\alpha_p)} \right) \approx 2.6 \cdot 10^{-15}.
	\]
	This proves the theorem.
\end{proof}
\section{A final comment on the improvement of the angle}
In this short final section we are going to comment on how improvements on the quantitative Shi estimates can help in order to come closer to prove Conjecture B in dimensions in dimension $11$. \\
Note that we have proven in Chapter 3 that
\[ |\nabla^3 \Rm| \leq (2K-\lambda)^{5/2} \textbf{C}_3(n)
\] 
for any Einstein manifold $(M,g)$ with $\Ric = \lambda g$ and $|\Rm| \leq K$. We have shown that we can choose $\textbf{C}_3(n)$ such that
\[
\textbf{C}_3(n) \leq 385661.
\]
This is clearly not optimal, since we basically used the Cauchy Schwarz inequality and an improved version for special tensor products of tensors (see Lemma \ref{tracelemma}) in order to obtain the estimate. \\
In this section, we will assume that we might choose 
\[ \textbf{C}_3(n) = 100
\]
This would be an improvement by about a factor of $3800$. Then we find that
$\textbf{C} = (2\sqrt{73/40}-\sqrt{40/27})^{5/2}\textbf{C}_3(11) \approx 268.589$.
This allows to pick $r \approx 2 \cdot 10^{-3}$ and consequently we find that
$	\dashint_{B_r(0_p)} |\nabla \Rm|^2 \dvol_g \geq 1.26 \cdot 10^{-7}$, which extends the angle to $\angle(\mathcal{R}_{\mathbb{CP}^2}^{\crit}, \Id_n) + 5 \cdot 10^{-8}$. Although it still does not prove Conjecture B for $n = 11$, it is possible to get closer to the desired angle.

\appendix

\chapter*{Appendix}
\addcontentsline{toc}{chapter}{Appendix}  
\pagestyle{fancy}		
\fancyfoot[C]{\small -- \thepage \hspace{0.25pt} --}		
\fancyhead[R]{}
\fancyhead[L]{\small Appendix}	
\renewcommand{\headrulewidth}{0.4pt}	
\renewcommand{\footrulewidth}{0pt}
\fancypagestyle{plain}{ 
	\fancyfoot[C]{\small -- \thepage \hspace{0.25pt} --}
	\fancyhead[R]{}
	\fancyhead[L]{}
	\renewcommand{\headrulewidth}{0.4pt}
}
\section*{Appendix I. The $\SO(k)$-irreducible decomposition of $\Lambda^2(\R^k) \otimes \R^k$}
\addcontentsline{toc}{section}{II. \hspace{0.3cm} The $\SO(k)$-irreducible decomposition of $\Lambda^2(\R^k) \otimes \R^k$}
In the first appendix we give a direct proof of the irreducible decomposition of $\Lambda^2(\R^k) \otimes \R^k$ under the usual $\SO(k)$-action, induced by
\[ g(v \wedge w \otimes x) = gv \wedge gw \otimes gx.
\]
for any $k \geq 3$. Consider the $\SO(k)$-invariant map $\Phi : \Lambda^2(\R^k) \otimes \R^k \to \Lambda^3(\R^k)$ given by
\[ \Phi(v \wedge w \otimes x) = v \wedge w \wedge x.
\] and $\R^k = \text{span} \Bigl\{ \sum_{j=1}^k e_i \wedge e_j \otimes e_j \mid i = 1, \dots, k \Bigl\} \subset \ker(\Phi)$. Then for $X_k = \left( \R^k \right)^{\perp} \subset \ker(\Phi)$ we prove the following
\begin{thm*}
	Let $k = 3,5$ or $k \geq 7$. Then the following decomposition
	\begin{align} \label{eqtensor} \Lambda^2(\R^k) \otimes \R^k = \Lambda^3(\R^k) \oplus \R^k \oplus X_k \end{align}
	for the standard product representation of $\SO(k)$ on $\Lambda^2(\R^k) \otimes \R^k$ is irreducible. \\
	For $k = 4$ there is a decomposition
	\[ \Lambda^2(\R^4) \otimes \R^4 = (X_4^+ \oplus \R^4_+) \oplus (X_4^- \oplus \R^4_-),
	\]
	that is due to the $\SO(4)$-invariant decomposition $\Lambda^2(\R^4) = \Lambda^2(\R^4)_+ \oplus \Lambda^2(\R^4)_-$. \\ For $k = 6$ there is a decomposition
	\[ \Lambda^2(\R^6) \otimes \R^6 = \Lambda^3(\R^6)_+ \oplus \Lambda^3(\R^6)_- \oplus \R^k \oplus X_k,
	\]
	that is due to the $\SO(6)$-invariant decomposition $\Lambda^3(\R^6) = \Lambda^3(\R^6)_+ \oplus \Lambda^3(\R^6)_-$. \\
	Furthermore, the decomposition (\ref{eqtensor}) is $O(k)$-irreducible for any $k \geq 3$.
\end{thm*}
Since $\Phi$ is an equivariant map, it is barely easy to check, that the subspaces above are invariant. The only thing, that is worth mentioning here, is that 
\[ gx = \sum_i gx \wedge ge_i \otimes ge_i = \sum_i gx \wedge e_i \otimes e_i.
\]
This comes from the fact that $g$ maps orthonormal bases to orthonormal bases. Hence $\R^k$ is well defined. So we just have to show that the decomposition is irreducible. We first show that the representation of $O(k)$ on $\Lambda^3(\R^k)$ is irreducible. 
\begin{lem*} $\Lambda^3(\R^k)$ is an irreducible $\SO(k)$-module for any $k \geq 3$ and $k \neq 6$. For $k = 6$ we have $\Lambda^3(\R^k) = \Lambda^3(\R^k)_+ \oplus \Lambda^3(\R^k)_-$.
\end{lem*} 
\begin{proof}For $k = 3$ this is clear. Now let $k > 3$. If we restrict the representation to $\SO(k-1)$ we obtain an invariant decomposition
	\[ \Lambda^3(\R^k) = \Lambda^3(\R^{k-1}) \oplus \Lambda^2(\R^{k-1}).
	\]
	To see that, just decompose $\R^k = \R^{k-1} \oplus \R$ and find a representative $e \in \R$. Then for $v = v_1 + \lambda e$, $w = w_1 + \mu e$ and $x = x_1 + \nu e$ we find
	\begin{align*}v \wedge w \wedge x &= v_1 \wedge w_1 \wedge x_1 + \lambda w \wedge x \wedge e \\
		&	\; \; \; - \mu v_1 \wedge x_1 \wedge e + \nu v_1 \wedge w_1 \wedge e.
	\end{align*}
	This decomposition is irreducible by induction hypothesis for $k \neq 5$. For $k = 5$ we obtain an irreducible decomposition as
	\[ \Lambda^3(\R^5) = \Lambda^3(\R^4) \oplus \Lambda^2(\R^4) \] under the representation of $O(4)$. We start with the case $k \neq 6$. We embedd $\SO(k-1) \subset \SO(k)$ as usual for $k \neq 5,7$ and $O(k-1) \subset \SO(k)$ for $k =5,7$ by 
	\[ A \mapsto \left( \begin{array}{cc} A & \\ & \text{det}A \end{array} \right).
	\]If we relax the representation to the whole $\SO(k)$, it is left to show that $\Lambda^3(\R^{k-1})$ is not invariant under this representation by Lemma \ref{representationlemma}. Choose an orthonormal basis $\{e_1, \dots e_{k-1},e \}$ of $\R^k$. Let $g \in \SO(k)$ be an orientation preserving isometry with $ge_1 = e$. Then clearly,
	$g(e_1 \wedge e_2 \wedge e_3) \notin \Lambda^3(\R^{k-1})$, which proves the claim. \\
	For $k = 6$ we consider the standard basis $(e_1, \dots, e_6)$ of $\R^6$ and the spaces
	\[ \Lambda^3(\R^6)_{\pm} = \text{span}\{ e_{i_1} \wedge e_{i_2} \wedge e_{i_3} \pm e_{i_4} \wedge e_{i_5} \wedge e_6 \mid (e_{i_1}, \dots, e_{i_5}) \text{ is positively oriented}\}
	\]
	It is easy to see that these spaces are invariant under the action of $\SO(6)$. A detailed proof of this can be found in \cite[Ch. IV]{Knapp+2016}. Moreover, it is clear that these spaces are irreducible, since $\Lambda^3(\R^6) = \Lambda^3(\R^5) \oplus \Lambda^2(\R^5)$ is irreducible under the action of $\SO(5)$, i.e. all irreducible representations of $\Lambda^3(\R^6)$ have to be $10$-dimensional. The matrix $\text{diag}(1,1,1,1,1,-1) \in O(6)$ interchanges $\Lambda^3(\R^6)_+$ and $\Lambda^3(\R^6)_-$. Hence the representation of $O(6)$ on $\Lambda^3(\R^6)$ is irreducible.
\end{proof}
Now we show that the representation of $\SO(k)$ on $X_k$ is irreducible. Here we also proceed inductively. Before doing so, we fix some notation. Denote by $A_k = \text{span} \{ e_i \wedge e_j \otimes e_j \mid i \neq j\}$ and $B_k = \text{span} \{ e_i \wedge e_j \otimes e_m + e_i \wedge e_m \otimes e_j \mid i < j < m \text{ or } j < m < i.\}$. Since these elements are linearly independent we find that $\text{dim}A_k = k(k-1)$ and $\text{dim}B_k = k(k-1)(k-2)/3$. Furthermore, we denote $Y_k = A_k \oplus B_k$. Note that
\[ Y_k = \text{ker}(\Phi) = \R^k \oplus X_k.
\]We start with the base case $k = 3$.
\begin{lem*} $X_3$ is an irreducible $\SO(3)$-module.
\end{lem*}
\begin{proof}
	We decompose $X_3$ in irreducible $\SO(2)$-modules as follows: Consider the spaces
	\begin{align*}
		&	\R^2 \hspace{0.75cm} = \text{span} \{ e_1 \wedge e_2 \otimes e_2 + e_1 \wedge e_3 \otimes e_3, e_2 \wedge e_1 \otimes e_1 + e_2 \wedge e_3 \otimes e_3\}, \\
		& \R \hspace{0.925cm}= \text{span} \{ e_3 \wedge e_1 \otimes e_1 + e_3 \wedge e_2 \otimes e_2 \}, \\
		& S^2_0(\R^2) = \text{span} \{ e_3 \wedge e_1 \otimes e_2 + e_3 \wedge e_2 \otimes e_1, e_3 \wedge e_1 \otimes e_1 - e_3 \wedge e_2 \otimes e_2\}, \\
		&  \widetilde{\R^2} \hspace{0.75cm} = \text{span}\{e_1 \wedge e_2 \otimes e_2 - e_1 \wedge e_3 \otimes e_3, e_2 \wedge e_1 \otimes e_1 - e_2 \wedge e_3 \otimes e_3 \}, \\
		& \Lambda^2(\R^2) = \text{span}\{e_1 \wedge e_2 \otimes e_3 +  \tfrac{1}{2} e_3 \wedge e_2 \otimes e_1 + \tfrac{1}{2} e_1 \wedge e_3 \otimes e_2\}.
	\end{align*}
	These spaces are pairwise orthogonal and generate $X_3 \oplus \R^3$. Moreover, it is easy to see that $X_3 = S^2_0(\R^2) \oplus \widetilde{\R^2} \oplus \Lambda^2(\R^2)$ is an irreducible decomposition under the action of $\SO(2)$. Together with Lemma \ref{representationlemma} it is enough to show that non of these summands is invariant under the action of $\SO(3)$ and that there does not exist a $2$-dimensional invariant subspace $V \subset S^2_0(\R^2) \oplus \widetilde{\R^2}$. We first show the second claim. \\
	Assume that there exists a vector $v \in S^2_0(\R^2) \oplus \widetilde{\R^2}$ such that $gv \in S^2_0(\R^2) \oplus \widetilde{\R^2}$ for all $g \in \SO(3)$. We can decompose $v= \sum_{i=1}^4 \lambda_i v_i$ for some $\lambda_i \in \R$, where 
	\begin{align*}
		& 	v_1 = e_1 \wedge e_2 \otimes e_2 - e_1 \wedge e_3 \otimes e_3 \\
		& v_2 = e_2 \wedge e_1 \otimes e_1 - e_2 \wedge e_3 \otimes e_3 \\
		& v_3 = e_3 \wedge e_1 \otimes e_2 + e_3 \wedge e_2 \otimes e_1 \\
		& v_4 = e_3 \wedge e_1 \otimes e_1 - e_3 \wedge e_2 \otimes e_2.
	\end{align*}
	We now choose elements $g_j \in \SO(3)$, such that only one of the vector $g_jv_i$ for $i= 1, \dots, 4$ has a summand, that contains a nonzero multiple of $e_1 \wedge e_2 \otimes e_3$. Since by assumption $g_j v \perp \Lambda^2(\R^2)$, we can deduce $\lambda_i = 0$ for some $i = 1, \dots, 4.$\\
	Choose 
	\[g_1 = \left( \begin{array}{ccc} & & 1 \\ & 1 & \\ -1 & & \end{array} \right).\] A computation shows that $\lambda_3 = 0$. Now choose 
	\[ g_2= \frac{1}{\sqrt2} \left( \begin{array}{ccc} 1 & 0 & -1 \\ 0 & \sqrt{2} & 0 \\ 1 & 0 & 1 \end{array} \right), g_3 = \frac{1}{\sqrt{2}} \left( \begin{array}{ccc} \sqrt{2} & 0 & 0 \\ 0 & 1 & 1 \\ 0 & -1 & 1 \end{array} \right).
	\]
	to find that $\lambda_2 = \lambda_1 = 0$. This is a contradiction to the assumption that $V$ is $2$-dimensional. Moreover, the proof above shows that $S^2_0(\R^2)$, $\widetilde{\R^2}$ and $S^2_0(\R^2) \oplus \widetilde{\R^2}$ are not $\SO(3)$-invariant. Thus $\Lambda^2(\R^2) = \left(S^2_0(\R^2) \oplus \widetilde{\R^2} \right)^{\perp}$ is also not $\SO(3)$-invariant.
\end{proof}
We proceed with the case $k=4$. Consider the decomposition
\[ \Lambda^2(\R^4) \otimes \R^4 = \left( \Lambda^2(\R^4)_+ \otimes \R^4 \right) \oplus \left( \Lambda^2(\R^4)_- \otimes \R^4 \right).
\]
and the $\SO(4)$-equivariant map $\varphi_{\pm} : \Lambda^2(\R^4)_{\pm} \otimes \R^4 \to \R^4$ given by $\varphi_{\pm}(A \otimes v) = A(v)$. This yields an invariant decomposition \[
\Lambda^2(\R^4)_{\pm} \otimes \R^4 = \ker(\varphi_{\pm}) \oplus \R^4_{\pm},
\] 
where $\R^4_{\pm} = \text{span} \{ \widetilde{e_i}_{\pm} \in \Lambda^2(\R^4)_+ \otimes \R^4 \mid i=1, \dots 4 \}$, where $\widetilde{e_i}_{\pm} \in \text{Im}(\varphi_{\pm})$ is defined via the standard orthonormal basis of $\Lambda^2(\R^4)_{\pm} \otimes \R^4$, e.g.
\[ \widetilde{e_1}_{-} = \frac{1}{3} (i_- \otimes e_2 + j_- \otimes e_3 + k_- \otimes e_4).
\]
Denote $X_4^{\pm} = \text{ker}(\varphi_{\pm})$. These spaces are irreducible, since by the case $k=3$ we obtain an $\SO(3)$-irreducible decomposition of $X_4 \otimes \R^4$ as follows
$X_4 = X_3 \oplus S^2_0(\R^3) \oplus \R^3 \oplus \Lambda^3(\R^3)$. An straightforward computation shows that both, $X_4^+$ and $X_4^-$ are non canonically embedd in the whole $X_4$ and therefore cannot contain an invariant subspace, which is $3$ or $5$-dimensional. Furthermore this shows that
\[ X_4 = X_4^+ \oplus X_4^-.
\]
decomposes irreducibly under $\SO(4)$.
We can easily write down a basis of $X_4^{\pm}$ as follows: First of all, $X_4^-$ is generated by
\begin{align*}
	&\hspace{0.5cm} i_- \otimes e_2 - \tfrac{1}{2} j_- \otimes e_3 - \tfrac{1}{2} k_- \otimes e_4, \hspace{1.5cm} - \tfrac{1}{2}i_- \otimes e_2 + j_- \otimes e_3 - \tfrac{1}{2} k_- \otimes e_4, \\
	& -i_- \otimes e_1 - \tfrac{1}{2} j_- \otimes e_4 + \tfrac{1}{2} k_- \otimes e_3, \hspace{1.85cm} \tfrac{1}{2}i_- \otimes e_1 + j_- \otimes e_4 + \tfrac{1}{2} k_- \otimes e_3 \\
	& -i_- \otimes e_4 + \tfrac{1}{2} j_- \otimes e_1 - \tfrac{1}{2} k_- \otimes e_2, \hspace{1.85cm}  \tfrac{1}{2}i_- \otimes e_4 - j_- \otimes e_1 - \tfrac{1}{2} k_- \otimes e_2 \\
	& \hspace{0.5cm} i_- \otimes e_3 + \tfrac{1}{2} j_- \otimes e_2 + \tfrac{1}{2} k_- \otimes e_1,\hspace{1.5cm}  -\tfrac{1}{2}i_- \otimes e_3 - j_- \otimes e_2 + \tfrac{1}{2} k_- \otimes e_1 
\end{align*} 	
Then, $X_4^+$ is generated by
\begin{align*}
	&\hspace{0.5cm} i_+ \otimes e_2 - \tfrac{1}{2} j_+ \otimes e_3 + \tfrac{1}{2} k_+ \otimes e_4, \hspace{1.5cm} - \tfrac{1}{2}i_+ \otimes e_2 + j_+ \otimes e_3 + \tfrac{1}{2} k_+ \otimes e_4, \\
	& -i_+ \otimes e_1 + \tfrac{1}{2} j_+ \otimes e_4 + \tfrac{1}{2} k_+ \otimes e_3, \hspace{1.85cm} \tfrac{1}{2}i_+ \otimes e_1 - j_+ \otimes e_4 + \tfrac{1}{2} k_+ \otimes e_3 \\
	& \hspace{0.5cm} i_+ \otimes e_4 + \tfrac{1}{2} j_+ \otimes e_1 - \tfrac{1}{2} k_+ \otimes e_2, \hspace{1.5cm}  -\tfrac{1}{2}i_+ \otimes e_4 - j_+ \otimes e_1 - \tfrac{1}{2} k_+ \otimes e_2 \\
	&  -i_+ \otimes e_3 - \tfrac{1}{2} j_+ \otimes e_2 - \tfrac{1}{2} k_+ \otimes e_1,\hspace{1.85cm}  \tfrac{1}{2}i_+ \otimes e_3 + j_+ \otimes e_2 - \tfrac{1}{2} k_- \otimes e_1 
\end{align*} 
In the end, note that the map $D = \text{diag}(1,1,1,-1) \in O(4)$ interchanges $X_4^+$ and $X_4^-$. Hence $X_4$ is a irreducible $O(4)$-module. We come to the gerneral case of $k \geq 5$ and conclude the theorem with the following 
\begin{lem*} $X_k$ is an irreducible $\SO(k)$-module for any $k \geq 5$.
\end{lem*}
\begin{proof} We start with the following \\
	\underline{Claim.} For any $k \geq 3$ there is an orthogonal decomposition
	\[ X_k = X_{k-1} \oplus \R^{k-1} \oplus S^2_0(\R^{k-1}) \oplus \Lambda^2(\R^{k-1})
	\] which is invariant under the action of $\SO(k-1)$. \\
	\underline{Proof of Claim.} In order to find the invariant decomposition for $X_k$ we decompose $X_k \oplus \R^k$ at first. We find that $X_k \oplus \R^k = A_k \oplus B_k$ with $A_k$ and $B_k$ as above and decompose further $A_k = A_{k-1} \oplus V_{k-1} \oplus W_{k-1}$ and $B_k = B_{k-1} \oplus Y_{k-1} \oplus Z_{k-1}$, where
	\begin{align*}
		& V_{k-1} = \text{span}\{ e_i \wedge e_k \otimes e_k \mid i=1, \dots, k-1 \}, \\
		& W_{k-1} = \text{span}\{ e_k \wedge e_i \otimes e_i \mid i = 1, \dots, k-1 \}, \\
		& Y_{k-1} = \text{span}\{ e_k \wedge e_i \otimes e_j + e_k \wedge e_j \otimes e_i \mid 1 \leq i < j \leq k-1\}, \\
		& Z_{k-1} = \text{span}\{ e_i \wedge e_j \otimes e_k + \tfrac{1}{2} e_k \wedge e_j \otimes e_i + \tfrac{1}{2} e_i \wedge e_k \otimes e_j \mid 1 \leq i < j \leq k-1 \}.
	\end{align*}
	Note that there are isomorphisms $V_{k-1} \cong \R^{k-1}$, $W_{k-1} \oplus Y_{k-1} \cong S^2(\R^{k-1})$ and $Z_k \cong \Lambda^2(\R^{k-1})$ as $\SO(k-1)$-modules, which directly yields the desired decomposition, together with $\R^k = \R^{k-1} \oplus \R$, where $\R$ is represented by $\text{span} \bigl\{ \sum_{i=1}^k e_k \wedge e_i \otimes e_i \bigl\}\subset S^2(\R^{k-1})$. \\ We now begin with the proof of the irreducibility of $X_k$. For $k =5$ we obtain $X_5 = X_4 \oplus \R^4 \oplus S^2_0(\R^4) \oplus \Lambda^2(\R^4)$, which is irreducible under the action $O(4)$ by the work we did before. By Lemma \ref{representationlemma} we obtain the result, when we embed $O(4) \subset \SO(5)$ as follows: We can push any basis element in $\R^5$, $S^2_0(\R^5)$ or $\Lambda^2(\R^5)$ easily in $X_5$ and also send any element in $X_5$ easily in $\R^5\oplus S^2_0(\R^5) \oplus \Lambda^2(\R^5)$ by rotating the $\text{span}\{e_i, e_k\}$-plane for some $i = 1, \dots, k-1$ by $\pi/2$.  Thus any possible direct sum is not invariant. Because any of the subrepresentations above are pairwise inequivalent, we obtain the result. For the induction step let $k > 5$. As above, we obtain an invariant decomposition
	\[ X_k = X_{k-1} \oplus \R^{k-1} \oplus S^2_0(\R^{k-1}) \oplus \Lambda^2(\R^{k-1}), 
	\]
	which is irreducible under the representation of $\SO(k-1)$. We obtain the result exactly the same as in the case $k=5$.
\end{proof}
\section*{Appendix II. A new proof for the invariance of $\Weyl _n$}
\addcontentsline{toc}{section}{III. \; A new proof for the irreducibility of $\Weyl _n$} 
In this note we will give a direct proof of the well known fact that the space of Weyl curvature operators $\Weyl_n$ is irreducible under the representation of $\SO(n)$, given by
\[ g.W = \Ad_g^{\tr} W \Ad_g
\]
for $g \in \SO(n)$, $W \in \Weyl_n$. The usual proofs of this, given for example in \cite[Ch 10.3.2]{goodman2009symmetry} or \cite[p~82 f.]{bergermazet}, use heighest weight theory. For this proof we only need basic linear algebra, including Proposition \ref{sok_sol_decomposition}.
\begin{thm*}
	For $n \geq 5$ the space of Weyl curvature operators $\Weyl_n$ is an irreducible $\SO(n)$-module.
\end{thm*}
Before starting with the proof, note that $\Weyl_4$ is not irreducible under $\SO(4)$. As indicated in Chapter 2, it decomposes as
\[
\Weyl_4 = \Weyl_4^+ \oplus \Weyl_4^-
\]
where $\Weyl_4^+$ denotes the operators of the form $\left( \begin{array}{cc} A & \\ & 0 \end{array} \right)$ for $A \in S^2_0(\so(3)_+)$ and $\Weyl_4^-$ denotes the operators of the form $\left( \begin{array}{cc} 0 & \\ & B \end{array} \right)$ for $B \in S^2_0(\so(3)_-)$. Recall also that the element $D= \text{diag}(1,1,1,-1) \in O(4)$
switches $\Weyl_4^+$ and $\Weyl_4^-$ with respect to the standard basis. Hence $\Weyl_4$ is indeed an irreducible $\text{O}(4)$-module.
\begin{proof}
	We proceed by induction on $n \geq 5$. Since we will use some of the following techniques also in the induction step, we will do the base case $n =5$ at the end. So let $n \geq 6$. By the proof of Proposition \ref{sok_sol_decomposition} and induction hypothesis we obtain an irreducible decomposition of $\Weyl_n$ under the representation of $\SO(n-1)$ as
	\begin{align*}
		\Weyl_n = \Weyl_{n-1} \oplus S^2_0(\R^{n-1}) \oplus X_{n-1},
	\end{align*}
	where $S^2_0(\R^{n-1})$ denotes operators of the form 
	\begin{align}  \label{eq:form1} \left( \begin{array}{cc} A \wedge \id_{\R^{n-1}}  & \\ & -\frac{n-3}{2} A \end{array}
		\right) \in S^2(\son)
	\end{align}
	with $A \in S^2_0(\R^{n-1})$. Furthermore, $X_{n-1}$ denotes the irreducible component of dimension $\frac{1}{3} (n+1)(n-1)(n-3)$ in the decomposition of $\Lambda^2(\R^{n-1}) \otimes \R^{n-1} = \Lambda^3(\R^{n-1}) \oplus \R^{n-1} \oplus X_{n-1}$, which is also described in Chapter 4 more explicitely. It is clear, that $\Weyl_{n-1}$ is not invariant under the action of the full $\SO(n)$. \\ We now show that we find $R_1 \in S^2_0(\R^{n-1})$ and $g_1 \in \SO(n) \setminus \SO(n-1)$ such that $g_1R_1 \notin S^2_0(\R^{n-1})$ but $g_1R_1 \in \Weyl_{n-1} \oplus S^2_0(\R^{n-1})$.
	Choose $A = \text{diag}(2,-2,0,...,0) \in S^2_0(\R^{n-1})$ and denote
	\begin{align*} R_1 = \left( \begin{array}{cc} A \wedge \id_{\R^{n-1}} & \\ & -\frac{n-3}{2}A \end{array} \right) 
	\end{align*}
	Then for $g_1 \in \SO(n)$ defined by $g_1(e_1) = -e_n$, $g_1(e_i)= e_i$ for $i = 2, \dots, {n-1}$ and $g_1(e_n) =e_1$ we obtain
	\begin{align*} & g_1R_1(e_1 \wedge e_n,e_1 \wedge e_n) = -(n-3), \\
		& g_1R_1(e_2 \wedge e_n, e_2 \wedge e_n) = 0,\\
		& g_1R_1(e_i \wedge e_n, e_i \wedge e_n) \; =  1, \hspace{1.8cm} \text{for } 3 \leq i \leq n-1.
	\end{align*}
	Now assume that $g_1R_1 \in S^2_0(\R^{n-1})$, i.e. is in the form (\ref{eq:form1}) for some $A_1 \in S^2_0(\R^{n-1})$ as above. Then we find that $A_1$ must satisfy
	\[A_1 = \tfrac{1}{n-3} \text{diag}(2(n-3),0,-2, \dots, -2).\]
	So on the one hand, $A_1 \wedge \id_{\R^{n-1}}(e_1 \wedge e_2, e_1 \wedge e_2) = 1$, but one the other hand $g_1R_1(e_1 \wedge e_2, e_1 \wedge e_2) = R_1(e_2 \wedge e_n, e_2 \wedge e_n) = -\frac{n-3}{2}A(e_2,e_2) =(n-3)$, which shows that $S^2_0(\R^{n-1})$ is not invariant under the action of $\SO(n)$. \\ 
	In order to finish the proof, it is now left to show that there exists $R_2 \in X_{n-1}$ and $g_2 \in \SO(n)$ such that $g_2 R_2 \notin X_{n-1}$. Pick $R_2 = e_1 \wedge e_2 \otimes e_3 + e_1 \wedge e_3 \otimes e_2 \in X_{n-1}$. This corresponds to the symmetric map $R_2: \so(n) \to \so(n)$ defined by 
	\[ R_2(e_i \wedge e_j) = \begin{cases} e_3 \wedge e_n & \text{ if } (i,j)=(1,2) \\
		e_1 \wedge e_2 & \text{ if } (i,j)=(3,n) \\
		e_2 \wedge e_n & \text{ if } (i,j)=(1,3) \\
		e_1 \wedge e_3 & \text{ if } (i,j)=(2,n) \\
		0 & \text{ else.}
	\end{cases}
	\]
	Now define $g_2 \in \SO(n)$ through $g_2(e_4) = e_n$, $g_2(e_n) = -e_4$ and $g_2(e_i)= e_i$ for $i \neq 4,n$.
	Then, since $g_2R_2(e_1 \wedge e_2, e_3 \wedge e_4) = 1$, we have $g_2R_2 \notin X_{n-1}$, which finishes the proof of the induction step by applying lemma \ref{representationlemma}. \\ We are left with the base case $n=5$. Note that $O(4)$ is a subgroup of $\SO(5)$ via the embedding $A \mapsto \left( \begin{array}{cc} A & \\ & \det(A) \end{array} \right).$ By the proof of Proposition \ref{sok_sol_decomposition} we obtain an irreducible decomposition of $\Weyl_5$ under the action of $O(4)$ as
	\[ \Weyl_5 = \Weyl_4 \oplus S^2_0(\R^4) \oplus X_4.
	\]
	The same examples as in the induction step show that none of these spaces is invariant under the action of the whole $\SO(5)$. This proves the claim together with Lemma \ref{representationlemma}.
\end{proof}
\begin{rem*}
	Note that one could also decompose 
	\[\Weyl_5 = \Weyl_4^- \oplus \Weyl_4^+ \oplus S^2_0(\R^4) \oplus X_4^+ \oplus X_4^{-}
	\]
	under the action of $\SO(4)$. But since $\Weyl_4^{\pm}$ and $X_4^{\pm}$ are isomorphic as $\SO(4)$ modules, one could not use Lemma \ref{representationlemma} in this case. This is why we use the other approach of considering $O(4)$ as a subgroup of $\SO(5)$.
\end{rem*}

 \addcontentsline{toc}{chapter}{Bibliography}
 
 \pagestyle{fancy}		
 \fancyfoot[C]{\small -- \thepage \hspace{0.25pt} -- }		
 \fancyhead[R]{}
 \fancyhead[L]{\small Bibliography}
 \renewcommand{\headrulewidth}{0.4pt}	
 \renewcommand{\footrulewidth}{0pt}
 \fancypagestyle{plain}{ 
 	\fancyfoot[C]{\small -- \thepage \hspace{0.25pt} --}
 	\fancyhead[R]{}
 	\fancyhead[L]{}
 	\renewcommand{\headrulewidth}{0.4pt}
 }

\bibliographystyle{alpha}
\bibliography{bib/literaturverzeichnis}

\end{document}